\documentclass[english]{amsart}
\usepackage{amsmath}

\usepackage{amssymb}
\usepackage{svn}

\usepackage{mathabx}

\usepackage{stackrel}

\usepackage{amscd}
\usepackage[all,cmtip,line]{xy}

\usepackage{csquotes}

\newenvironment{itquote}
{\begin{quote}\itshape}
{\end{quote}}

\SVN $LastChangedDate: 2012-10-30 15:27:28 +0000 (Tue, 30 Oct 2012) $
\SVN $LastChangedRevision: 28 $

\def\AA{\mathbb{A}}

\def\CC{\mathbb{C}}

\def\FF{\mathbb{F}}

\def\QQ{\mathbb{Q}}
\def\RR{\mathbb{R}}

\def\TT{\mathbb{T}}

\def\ZZ{\mathbb{Z}}

\newcommand{\onto}{\twoheadrightarrow}

\newcommand{\lra}{\longrightarrow}

\newcommand{\mat}[4]{\left( \begin{array}{cc} {#1} & {#2} \\ {#3} & {#4}
\end{array} \right)}

\newcommand{\smat}[4]{{\mbox{\scriptsize $\mat{{#1}}{{#2}}{{#3}}{{#4}}$}}}
\newlength{\ownl}

\newcommand{\Ann}{{\operatorname{Ann}\,}}

\newcommand{\Aut}{{\operatorname{Aut}\,}}

\newcommand{\Coind}{{\operatorname{Coind}\,}}
\newcommand{\coker}{{\operatorname{coker}\,}}

\newcommand{\cyc}{{\operatorname{cyc}\,}}

\newcommand{\Fil}{{\operatorname{Fil}\,}}

\newcommand{\Frob}{{\operatorname{Frob}}}

\newcommand{\Gal}{{\operatorname{Gal}\,}}

\newcommand{\gr}{{\operatorname{gr}\,}}

\newcommand{\Hom}{{\operatorname{Hom}\,}}

\newcommand{\Ind}{{\operatorname{Ind}\,}}

\newcommand{\JL}{{\operatorname{JL}\,}}

\newcommand{\rec}{{\operatorname{rec}}}

\newcommand{\Res}{{\operatorname{Res}}}

\newcommand{\WD}{{\operatorname{WD}}}

\newcommand{\Spec}{{\operatorname{Spec}\,}}

\newcommand{\tr}{{\operatorname{tr}\,}}

\newcommand{\GL}{\operatorname{GL}}

\newcommand{\SL}{\operatorname{SL}}
\newcommand{\SO}{\operatorname{SO}}

\newcommand{\ab}{{\operatorname{ab}}}

\newcommand{\A}{{\mathbb{A}}}

\newcommand{\C}{{\mathbb{C}}}

\newcommand{\F}{{\mathbb{F}}}

\newcommand{\Q}{{\mathbb{Q}}}
\newcommand{\R}{{\mathbb{R}}}

\newcommand{\U}{{\mathbb{U}}}

\newcommand{\Z}{{\mathbb{Z}}}

\newcommand{\CA}{{\mathcal{A}}}

\newcommand{\CE}{{\mathcal{E}}}
\newcommand{\CF}{{\mathcal{F}}}
\newcommand{\CG}{{\mathcal{G}}}
\newcommand{\CH}{{\mathcal{H}}}
\newcommand{\CI}{{\mathcal{I}}}

\newcommand{\CL}{{\mathcal{L}}}
\newcommand{\CM}{{\mathcal{M}}}
\newcommand{\CN}{{\mathcal{N}}}
\newcommand{\CO}{{\mathcal{O}}}

\newcommand{\CS}{{\mathcal{S}}}
\newcommand{\CT}{{\mathcal{T}}}

\newcommand{\CV}{{\mathcal{V}}}

\newcommand{\CY}{{\mathcal{Y}}}
\newcommand{\CZ}{{\mathcal{Z}}}

\newcommand{\gM}{{\mathfrak{M}}}
\newcommand{\gN}{{\mathfrak{N}}}

\newcommand{\gP}{{\mathfrak{P}}}

\newcommand{\gS}{{\mathfrak{S}}}

\newcommand{\gd}{{\mathfrak{d}}}

\newcommand{\gm}{{\mathfrak{m}}}
\newcommand{\gn}{{\mathfrak{n}}}

\newcommand{\gp}{{\mathfrak{p}}}
\newcommand{\gq}{{\mathfrak{q}}}

 \newcommand{\barrho   }{{\overline{\rho}}}

 \newcommand{\rhobar   }{{\overline{\rho}}}

\newcommand{\Qbar}{{\overline{\Q}}}

 \newcommand{\St}{\operatorname{St}}

\newcommand{\Qpbar}{\overline{\Q}_p}

\newcommand{\Fpbar}{\overline{\F}_p}

\newcommand{\Zpbar}{\overline{\Z}_p}

\newcommand{\Sym}{\operatorname{Sym}}

\newcommand{\sub}{\operatorname{sub}}

\def\smallmat#1#2#3#4{\bigl(\begin{smallmatrix}{#1}&{#2}\\{#3}&{#4}\end{smallmatrix}\bigr)}

\usepackage{hyperref}

\newcommand{\f}{\mathbf{f}}

\newcommand{\Nm}{\mathrm{Nm}}
\newcommand{\uhp}{\mathfrak{H}}

\newcommand{\crys}{\mathrm{crys}}

\newcommand{\ord}{{\operatorname{ord}}}

\newcommand{\tot}{{\operatorname{tot}}}
\newcommand{\im}{{\operatorname{im}}}

\theoremstyle{plain}
\newtheorem{theorem}{Theorem}[subsection]
\newtheorem{proposition}[theorem]{Proposition}
\newtheorem{corollary}[theorem]{Corollary}
\newtheorem{lemma}[theorem]{Lemma}
\newtheorem{conjecture}[theorem]{Conjecture}
\newtheorem{ithm}{Theorem}
\newtheorem{iconj}[ithm]{Conjecture}

\theoremstyle{definition}
\newtheorem{remark}[theorem]{Remark}
\newtheorem{definition}[theorem]{Def\/inition}

\newcommand{\epf}{{\ifhmode\unskip\nobreak\hfil\penalty50 \hskip1em
\else\nobreak\fi \nobreak\mbox{}\hfil\mbox{$\square$} \parfillskip=0pt
\finalhyphendemerits=0 \par\vskip5pt}}

\begin{document}

\title[On Galois representations associated to Hilbert modular forms]{On Galois representations associated to mod $p$ Hilbert modular forms}
\author{Fred Diamond}
\email{fred.diamond@kcl.ac.uk}
\address{Department of Mathematics,
King's College London, WC2R 2LS, UK}
\author{Shu Sasaki}
\email{s.sasaki.03@cantab.net}
\address{School of Mathematical Sciences, Queen Mary University of London, E1 4NS, UK}

%\date{November 2021}% (preliminary version, not for circulation)}
%\subjclass[2010]{11G18 (primary), 11F32, 11F41, 14G35 (secondary)}

\begin{abstract}
We consider mod $p$ Hilbert modular forms for a totally real field $F$, viewed as sections of automorphic line bundles on Hilbert modular varieties in prime characteristic $p$.  For a Hecke eigenform of arbitrary weight, we prove the existence of an associated two-dimensional representation of the absolute Galois group of $F$.  Furthermore, for any such irreducible Galois representation, we formulate a conjecture predicting the set of weights of eigenforms from which it arises.  This generalizes Edixhoven's variant of the weight part of Serre's Conjecture (in the case $F = \QQ$), and removes the restriction that $p$ be unramified in $F$ from prior work in this direction.  We also establish one direction of a conjectural relation with the algebraic analogue of the weight part of Serre's Conjecture in this context.  Finally, we prove results towards our conjecture  in the case of partial weight one for real quadratic fields $F$ in which $p$ is ramified.
\end{abstract}

%\begin{abstract}

%\end{abstract}

\maketitle

%testing $\mathbf{l} \mathbf{1}$

%\end{document}

\section{Introduction} 
\subsection{Background and motivation}
In \cite{serre}, Serre formulated a remarkable conjecture, later remarkably proved by Khare and Wintenberger~\cite{KW1,KW2}, predicting the modularity of certain two-dimensional Galois representations in prime characteristic $p$. 
More precisely, the conjecture (now theorem) asserts that every odd, continuous, irreducible representation of the form $\rho: \Gal(\Qbar/\QQ) \to \GL_2(\Fpbar)$ arises as the reduction of a $p$-adic Galois representation associated to a Hecke eigenform $f \in S_k(\Gamma_1(N))$, where $N = N_\rho$ is the (prime-to-$p$) Artin conductor of $\rho$ and $k = k_\rho$ is described explicitly in terms of the restriction of $\rho$ to a decomposition group at $p$.

In his original formulation, Serre restricts his attention to modular forms in characteristic zero and of weight $k \ge 2$, but remarks \cite[pp.197--8]{serre}:
\begin{itquote}
\hspace*{0.2cm} Au lieu de d\'efinir les formes paraboliques \`a coefficients dans $\F_p$ par réduction \`a partir de la caract\'eristique $0$, comme nous l'avons fait, nous aurions pu utiliser la d\'efinition de Katz [23],\footnote{\cite{katz} in the references herein.} qui conduit à un espace a priori plus grand\ldots \; Il serait \'egalement int\'eressant d'\'etudier de ce point de vue le cas $k = 1$, que nous avons exclu jusqu'ici; peut-\^etre la d\'efinition de Katz donne-t-elle alors beaucoup plus de repr\'esentations $\rho_f$?
\end{itquote}

The referenced definition of Katz interprets the space of modular forms in characteristic $p$ more directly as sections of line bundles on the reduction mod $p$ of the modular curve $X_1(N)$.
The question above was addressed by Edixhoven in his paper~\cite{edixhoven} proving the weight part of Serre's Conjecture (for odd $p$), i.e., that if $\rho$ arises from a Hecke eigenform of some weight $k \ge 2$ and level $N$ prime to $p$, then in fact it arises from one of weight $k_\rho$ and level $N$ (in the classical sense above or equivalently that of Katz).  
To answer Serre's question, Edixhoven furthermore proves that such a $\rho$ is modular of weight one and level $N$ in Katz's sense if and only if $\rho$ is unramified at $p$.  (In general, only the forward implication holds using the classical sense of modularity in the original formulation.)

In his paper, Serre also goes on to ask \cite[3.4(2)]{serre}: 
\begin{itquote}
\hspace*{0.2cm} Peut-on reformuler ces conjectures dans le cadre d'une th\'eorie des repr\'esentations (mod $p$) des groupes ad\'eliques? Autrement dit, existe-t-il une ``philosophie de Langlands modulo p'', comme le demandent Ash et Stevens dans [2]\footnote{\cite{AS} in the references herein}? Si oui, cela permettrait peut-\^etre:\\
\hspace*{0.2cm} de donner une d\'efinition plus naturelle du poids $k$ attach\'e \`a $\rho$;\\
\hspace*{0.2cm} de remplacer $\GL_2$ par $\GL_N$, ou m\^eme par un groupe r\'eductif;\\
\hspace*{0.2cm}
de remplacer $\QQ$ par d'autres corps globaux.
\end{itquote}

These questions have received a great deal of attention, thanks in part to their significance in the context of the $p$-adic Langlands Programme.  Firstly, it turns out that the most natural characterizations of the weight $k_\rho$ are provided by $p$-adic Hodge theory; furthermore, this weight can be viewed as a manifestation of a finer invariant associated to $\rho$ by the Breuil--M\'ezard Conjecture~\cite{BM} (see \cite{hanneke_short}, for example). The question of replacing $\GL_2$ by $\GL_N$ was first considered by Ash and Sinnott~\cite{AS}, and that of replacing $\QQ$ by a number field by the first author with Buzzard and Jarvis~\cite{BDJ}.  Many people have continued work in these directions, formulating (and proving results towards) Serre weight and Breuil--M\'ezard Conjectures in the context of Galois representations associated to cohomological automorphic forms.  See especially~\cite{GHS} for a formulation in significant generality and a discussion of the history of the problem, and \cite{LeLeHung} and the references therein for more recent results.

There has, however, been relatively little consideration in more general settings of the variant addressed by Edixhoven.
In view of this, we initiated a study in \cite{DS1} of this problem in the context of Galois representations associated to Hilbert modular forms, where $F$ is a totally real number field.
In particular, we formulated a conjecture that predicts the set of weights of all mod $p$ Hilbert modular eigenforms giving rise to a fixed $\rho:\Gal(\overline{F}/F) \to \GL_2(\Fpbar)$, under the hypothesis that $p$ is unramified in $F$.  While it is still the case that the set of weights is conjecturally determined (via $p$-adic Hodge theory) by the local behavior of $\rho$ at primes over $p$, several new features emerge.  To begin with, the interpretation of modularity in terms of sections of automorphic bundles on Shimura varieties allows for the consideration of non-cohomological weights, which are necessarily excluded by the framework of \cite{BDJ} and its generalizations discussed above. In the setting of classical modular forms considered by Edixhoven (i.e., the case $F = \QQ$), this extension ultimately pertains only to weight one modular forms and Galois representations unramified at $p$, but the analogous extension in the Hilbert modular setting involves richer structure and behavior on both the automorphic and Galois sides of the relation.  Another new feature revealed in \cite{DS1} is the determinative role played by a certain cone of weights, which we call the {\em minimal cone}.  Input from the geometry of Shimura varieties also enables generalizations of weight entailment phenomena that were provided in the classical setting by Hasse invariants and $\Theta$-operators.

The primary motivation for this paper is to remove the restriction from the conjectures and results in \cite{DS1} that $p$ be unramified in $F$.  Thus for an arbitrary prime $p$ and totally real field $F$, we consider a continuous, irreducible, totally odd\footnote{By {\em totally odd,} we mean that $\det(\rho(c)) = - 1$ for every complex conjugation $c$.} representation
$$ \rho: \Gal(\overline{F}/F) \to \GL_2(\Fpbar) $$
and formulate conjectures predicting the weights of 
mod $p$ Hilbert modular eigenforms of level prime to $p$ 
giving rise to $\rho$ in terms of its local behavior at primes over $p$.  Since $p$ may be ramified in $F$, we use the Pappas--Rapoport model for the Hilbert modular variety to define Hilbert modular forms in characteristic $p$. In order for the conjectures to be meaningful, we must first prove the existence of Galois representations associated to such eigenforms, thus removing parity hypotheses on the weight from prior results of Reduzzi and Xiao~\cite{RX}. To that end, we adopt a different approach from the one used in~\cite{DS1} (for $p$ unramified in $F$); in this paper we make more systematic use of congruences involving forms that are neither $L$-algebraic nor $C$-algebraic, along with the Jacquet--Langlands correspondence.  This perspective also enables us to prove one direction of a conjectural relationship between the sets of weights giving rise to $\rho$ in the sense of this paper and that of \cite{BDJ}.  (This is new even when $p$ is unramified in $F$.)  

The geometric Serre weight conjecture formulated in this paper is qualitatively similar to the one from \cite{DS1} that it generalizes, but it clarifies how features such as minimal cone and weight-shifting phenomena manifest when $p$ is ramified in~$F$.  In particular, our formulation in terms of line bundles on the Pappas--Rapoport model captures aspects of the numerology, including that of weight entailment phenomena, implicit in Gee's generalization (\cite{Gtypes}) of the Buzzard--Diamond--Jarvis Conjecture to the case where $p$ is ramified in $F$.
Furthermore, the ramified setting more readily exhibits how geometric Serre weights provide a finer invariant of local Galois representations than their algebraic analogue.  For example, if $p$ is highly totally ramified in $F$, then there is a class of representations $\rho$ as above whose set of algebraic Serre weights consists of {\em all} those with the same central character (see \cite{GS}), but whose (conjectural) geometric Serre weights are distinguished by a hierarchy of possible minimal weights.

One direction of our conjecture is an assertion about the local behavior at primes over $p$ of all Galois representations $\rho$ arising from mod $p$ Hilbert modular forms of a given weight.  It can thus be interpreted as a local-global compatibility statement in the context of Langlands correspondences associating Galois representations to automorphic forms in characteristic $p$.  The other direction asserts that {\em all} continuous, irreducible, totally odd $\rho$ satisfying those local conditions indeed arise from forms of that weight.  One can envision the two directions as providing a starting point to using the approach of Calegari--Geraghty~\cite{CG} to propagate such an equivalence to suitable deformations of $\rho$.

As in \cite{DS1}, we provide evidence for our conjecture by proving strong results describing which Galois representations (a priori modular of {\em some} weight) arise from forms of partial weight one for real quadratic $F$, except that we now we assume $p$ is ramified in $F$.  Our overall approach to this is similar to that of \cite{DS1}, but we must now make extensive use of the $\Theta$-operators constructed and studied in the ramified case in \cite{theta}.

We remark that the results from \cite{DS1} in the quadratic unramified case have since been strengthened by Yang (see~\cite[Cor.~4.31]{SY}) and partly generalized to higher degree by Wiersema~\cite{Hanneke_GM, hanneke_new}.  Further results in the unramified case are expected from forthcoming work of Kansal, Levin and Savitt (on local-global compatibility) and Wiersema (on geometric modularity in prescribed weight).

\subsection{Outline of the paper}
We now describe the contents of the paper in more detail.
 
Sections 2--4 mainly recall definitions and prior results we will use on Hilbert modular forms, and more generally automorphic forms for unit groups of quaternion algebras over a totally real field $F$.  Section 2 reviews the classical theory over $\CC$, including the Jacquet--Langlands correspondence.  In Section~3, we recall aspects of the theory with integral coefficients, or more precisely over the ring of integers of a finite extension of $\QQ_p$.  

In Section~4, we focus on our main objects of study: the spaces $M_{\vec{k},\vec{m}}(U,E)$ of Hilbert modular forms of weight $(\vec{k},\vec{m}) \in \ZZ^\Sigma \times \ZZ^\Sigma$ and level $U$ with coefficients in a finite field $E$ of characteristic $p$, where $\Sigma$ is the set of embeddings $F \to \Qpbar$.  The component $\vec{m}$ plays a secondary role related to twisting by characters, so for the purpose of the introduction, we assume it is trivial and suppress it from the notation.
The level $U$ is an open compact subgroup of $\GL_2(\AA_{F,\f})$, which we generally assume is prime to $p$ in the sense that it contains a maximal compact subgroup of $\GL_2(F_p)$ (where $F_p = F\otimes \QQ_p$).

The spaces $M_{\vec{k}}(U,E)$ are equipped with commuting endomorphisms ({\em Hecke operators}) $T_v$ and $S_v$ for all but finitely primes $v$ of $F$.  Furthermore these are compatible with weight-shifting maps defined by application of partial $\Theta$-operators and multiplication by partial Hasse invariants.
The partial Hasse invariants are elements $H_\sigma \in M_{\vec{h}_\sigma}(U,E)$ for certain weights $\vec{h}_\sigma$ indexed by $\sigma \in \Sigma$, while the partial $\Theta$-operators are certain differential operators (which, strictly speaking, also shift the component $\vec{m}$ we just suppressed).

Motivated by the conjectures being formulated for this paper, the first author and Kassaei proved in~\cite{DK} that every 
Hilbert modular form arises by multiplication by partial Hasse invariants from one whose weight lies in a certain cone $\Xi_{\min} \subset\ZZ_{\ge 0}^\Sigma$, which we call the {\em minimal cone}.  Using an idea of Deo, Dimitrov and Wiese~\cite{DDW}, we prove a refinement in this paper (Theorem~\ref{thm:positivity}) showing that $\Xi_{\min}$ can be replaced by $\Xi_{\min}^+ \cup \{\vec{0}\}$, where $\Xi_{\min}^+ = \Xi_{\min} \cap \ZZ^{\Sigma}_{>0}$.

In Section~5, we prove the existence of Galois representations associated to Hilbert modular eigenforms (Theorem~\ref{thm:galois}):
\begin{ithm} \label{ithm:galois} Suppose that $f \in M_{\vec{k}}(U,E)$ is an eigenform for $T_v$ (resp.~$S_v$) with eigenvalue $a_v$ (resp.~$d_v$) for all $v \not\in Q$, where $Q$ is a finite set of primes containing all those dividing the level or $p$.  Then there exists a continuous representation
$$\rho_f : \Gal(\overline{F}/F) \to \GL_2(E)$$
such that if $v\not\in Q$, then $\rho_f$ is unramified at $v$ and the characteristic polynomial of $\rho_f(\Frob_v)$ (where $\Frob_v$ is a Frobenius element at $v$) is
$$X^2 - a_v X + \Nm_{F/\QQ}(v)d_v.$$
\end{ithm}
This was proved in~\cite{RX} under the assumption that all components $k_\sigma$ of $\vec{k}$ have the same parity.  The idea is to use congruences to forms of higher weight and cohomological vanishing (provided by ampleness results) for automorphic bundles in order to lift $f$ to a characteristic zero eigenform for which one already has associated Galois representations. A difficulty that arises without the parity hypothesis is that such lifts will not have level prime to $p$, so one also needs cohomological vanishing results for certain Hilbert modular varieties with bad reduction.  Under the assumption that $p$ is unramified in $F$, we removed the parity hypothesis in \cite{DS1} using this approach, with the cohomological vanishing provided by \cite{DKS}.
These vanishing results were generalized to the ramified case in
\cite{FD:KS}, so that one can prove the above theorem by combining the methods of \cite{RX} and \cite{DS1}.  

Here however, we present a different proof, inspired by remarks of David Loeffler and George Boxer.  The idea is to lift (a twist of) $f$ to a (non-algebraic) characteristic zero eigenform of level prime to $p$.  We lack a Galois representation associated to the lift, but we can nonetheless apply the Jacquet--Langlands correspondence to transfer our problem to the setting of automorphic forms on quaternion algebras.  This allows us to construct the desired congruences to forms with primes over $p$ in the level, without having to appeal to the more delicate analysis of Hilbert modular varieties with bad reduction.

In \S\ref{ssec:ord}, we combine the approach with techniques from Hida theory in order to prove that if (a twist of) $f$ is ordinary at a prime $\gp$ over $p$, then the restriction of $\rho_f$ to $\Gal(\overline{F}_{\gp}/F_{\gp})$ is reducible, with subrepresentation and quotient determined by the (invertible) $T_{\gp}$-eigenvalue (and central character).

In Section~6, we formulate our conjecture predicting which Galois representations $\rho$ are {\em geometrically modular} of weight $\vec{k}$, in the sense that they arise from some eigenform $f$ of weight $\vec{k}$ as above.
\begin{iconj} \label{iconj} Suppose that $\vec{k} \in \Xi_{\min}^+$ and $\rho: \Gal(\overline{F}/F) \to \GL_2(\Fpbar)$ is continuous, irreducible and totally odd. Then $\rho$ is geometrically modular of weight $\vec{k}$ if and only if $\chi_\cyc \otimes \rho|_{\Gal(\overline{F}_\gp/F_\gp)}$ has a crystalline lift with labelled Hodge--Tate weights 
$(0,k_\sigma - 1)_{\sigma \in \Sigma_{\gp}}$ for each prime $\gp$ dividing $p$,  where $\chi_{\cyc}$ is the (mod $p$) cyclotomic character and $\Sigma_\gp$ is the set of embeddings $\sigma \in \Sigma$ factoring through $F_\gp$.
\end{iconj}
See \S\ref{subsec:cryslift} for notions from $p$-adic Hodge theory, including the definition of labelled weights, and
Conjecture~\ref{conj:geomweights} for a more refined version of the conjecture covering all weights $\vec{k}$.

In \S\ref{ssec:galoistwist}, we discuss what happens as we vary 
$\vec{m}$ (the weight component we assumed here to be trivial), and in particular its interaction with twisting and with partial $\Theta$-operators.

In \S\ref{ssec:algmod}, we discuss the relation with the notion of being modular of weight $\vec{k}$ in the sense of \cite{BDJ} and its generalizations, which we call {\em algebraic modularity}. In particular, we conjecture that the notions are equivalent for $\vec{k} \in \ZZ_{\ge 2}^\Sigma \cap \Xi_{\min}$.  Furthermore, using ingredients of the proof of Theorem~\ref{ithm:galois}, we prove one direction of this equivalence.  (For partial results in the other direction, under the assumption that $p$ is unramified in $F$, see \cite{SY}.)
More precisely, we prove the following (see Theorem~\ref{thm:defalggeom}), which allows one to transfer algebraic modularity results of Gee and others (\cite{GLS, GK, JN} into ones towards Conjecture~\ref{iconj}.
\begin{ithm} \label{ithm:alggeommod} If $\rho: \Gal(\overline{F}/F) \to \GL_2(\Fpbar)$ is irreducible and algebraically modular of weight $\vec{k} \in \ZZ_{\ge 2}^\Sigma$, then $\rho$ is geometrically modular of weight $\vec{k}$.
\end{ithm}

We also present a conjectural converse (Conjecture~\ref{conj:ord}) to our result on Galois representations associated to ordinary forms, i.e., that if $\rho$ is as in Conjecture~\ref{iconj}, then it arises from a form $f$ of weight $\vec{k}$ which is ordinary at all primes $\gp$ dividing $p$ such that $\rho|_{\Gal(\overline{F}_\gp/F_\gp)}$ has the requisite shape.  We also explain how results in this direction follow from analogous ones in the context of algebraic modularity.

In Section~7, we prove results towards Conjecture~\ref{iconj} in the case of partial weight one, where $F$ is a real quadratic field in which $p$ ramifies.  Note that such weights are outside the setting where results on algebraic modularity apply, so we use the approach developed in \cite{DS1} and \cite{HWPhD}.  The idea is to relate the properties of crystalline liftability and geometric modularity in these irregular weights to those properties in regular weights, so that one can apply the results of Gee and others on algebraic Serre weight conjectures.

Before specializing to the case of real quadratic $F$, we generalize the results we need from \cite{DS1} on stabilized eigenforms to the case where $p$ is ramified in $F$.  We also generalize results of \cite{GLS} and \cite{hanneke_new} on reductions of crystalline representations to the case of irregular weight. In \S\ref{subsec:hodge}, we specialize the $p$-adic Hodge theory results to the ramified quadratic case in order to relate crystalline liftability in irregular and regular weights.
Finally in the last section, we use partial Hasse invariants and 
$\Theta$-operators to obtain analogous relations for geometric modularity.  Putting everything together gives the following result:
\begin{ithm} \label{ithm:quadratic} Let $F$ be a real quadratic field in which $p$ is ramified, $\gp$ the prime of $F$ dividing $p$, and $\vec{k} = (1,w) \in \ZZ^\Sigma$ for some 
$w \in [2,p]$.  Suppose that $\rho$ is irreducible, geometrically modular of some weight, and satisfies a Taylor--Wiles hypothesis.  If $\rho|_{\Gal(\overline{F}_\gp/F_\gp)}$ has a crystalline lift of weight $\vec{k}$, and is reducible if $w = 2$, then $\rho$ is geometrically modular of weight $\vec{k}$.
\end{ithm}
The Taylor--Wiles hypothesis (on the image of $\rho$; see the statements of  Theorems~\ref{thm:3toPmod} and~\ref{thm:2mod} for details) is needed in order to apply results on algebraic Serre weight conjectures.  We also prove a converse (see Theorems~\ref{thm:3toPcrys} and~\ref{thm:2crys}), assuming our conjecture that geometric implies algebraic modularity for weights $\vec{k} \in \ZZ_{\ge 2}^\Sigma \cap \Xi_{\min}$.  On the other hand, we are able to use ad hoc arguments to remove the Taylor--Wiles hypothesis from our results in this direction.

Finally, we remark that the additional reducibility hypothesis on the local Galois representation in the case $w=2$ is needed in order obtain the relation between crystalline liftability, as well as geometric modularity, in irregular and regular weights.

\subsection{Notation and conventions}
We fix a totally real field $F \neq \QQ$ throughout the paper, and let $d = [F:\QQ]$ denote its degree.  We denote its ring of integers $\CO_F$ and let $\gd \subset \CO_F$ denote its different.

For any field $L$, we write $\overline{L}$ for its algebraic closure, and we view $\Qbar$ as a subfield of $\CC$.
We let $\Sigma$ denote the set of embeddings $F \hookrightarrow \Qbar$, which we often identify with the set $\Sigma_\infty$ of embeddings $F \hookrightarrow \RR$.

We write $F_+^\times$ for the subgroup of totally positive elements of $F^\times$.  More generally, we use the subscript ``+'' to denote the set of totally positive elements of any subset of $F_\infty = F \otimes \RR = \prod_{\sigma \in \Sigma_\infty} \RR$. We extend this usage to matrices via the determinant, so for example $\GL_2(F)_+$ denotes the set of elements of $\GL_2(F)$ with totally positive determinant.

For any finite place $v$ of $F$, we write $\CO_{F,v}$ (resp.~$F_v$) for the completion of $\CO_F$ (resp.~$F$) at $v$.  We let $\A$ denote the ring of adeles of $\QQ$ and $\A_F$ the adeles of $F$, so $\A_F = F\otimes \AA = F_\infty \times \A_{F,\f}$ where $\A_{F,\f} = F\otimes \widehat{\ZZ} \subset \prod_v F_v$, the last product being over all finite places $v$ of $F$.

We fix a rational prime $p$ and an embedding $\Qbar \to \Qpbar$, and let $S_p$ be the set of primes of $\CO_F$ containing $p$. We write $\CO_{F,(p)}$ for the localization $\CO_F\otimes \ZZ_{(p)}$ of $\CO_F$ at the prime ideal $(p) = p\ZZ$ of $\ZZ$.  We let $\AA_{F,\f}^{(p)}$ denote the prime-to-$p$ finite adeles of $F$; i.e., the kernel of the natural projection $\AA_{F,\f} \to F_p:= \prod_{\gp \in S_p} F_\gp$.

For each $\gp \in S_p$, write $\FF_\gp$ for the
residue field $\CO_F/\gp$. Denote the maximal unramified subextension of $F_\gp$ by $F_{\gp,0}$, and identify it with the field of fractions of the ring of Witt vectors $W(\FF_\gp)$. Let $f_\gp$ denote the residue degree $[F_{\gp,0}:\Q_p] = [\FF_\gp:\FF_p]$, and $e_\gp$ the ramification index $[F_\gp:F_{\gp,0}]$.
We choose a uniformizer $\varpi_\gp \in \CO_{F,(p),+}$ 
such that $\varpi_\gp\CO_{F,(p)} = \gp\CO_{F,(p)}$. On the other hand, if $v\not\in S_p$, then $\varpi_v$ will simply denote a uniformizer in $\CO_{F,v}$.

We let $\Sigma_{\gp}$ denote the set of embeddings
$F_{\gp} \to \Qpbar$, and let $\Sigma_{\gp,0}$ denote 
the set of embeddings $F_{\gp,0} \to \Qpbar$
(or equivalently $W(\FF_\gp) \to W(\Fpbar)$, or
$\FF_\gp \to \Fpbar$).
We fix a choice of embedding $\tau_{\gp,0}  \in \Sigma_{\gp,0}$,
and for $i \in \Z/f_\gp\Z$, we let $\tau_{\gp,i} = \phi^i\circ\tau_{\gp,0}$ where $\phi$
is the Frobenius automorphism of $\Fpbar$ (sending $x$ to $x^p$).   We also fix an ordering 
$\sigma_{\gp,i,1},\sigma_{\gp,i,2},\ldots,\sigma_{\gp,i,e_\gp}$ of the embeddings 
$\sigma \in \Sigma_\gp$ restricting to $\tau_{\gp,i}$.
Our choice of embedding $\Qbar \to \Qpbar$ allows us to identify $\Sigma$ with $\coprod_{\gp \in S_p} \Sigma_{\gp}$, so that
$$\Sigma = \{\,\sigma_{\gp,i,j}\,|\,\gp \in S_p, i \in \ZZ/f_\gp\ZZ, 1 \le j \le e_\gp\,\}.$$
We also let $\Sigma_0 = \coprod_{\gp \in S_p} \Sigma_{\gp,0}
= \{\,\tau_{\gp,i}\,|\,\gp \in S_p, i \in \ZZ/f_\gp\ZZ\,\} $,
and define a permutation $\varphi$ of $\Sigma$ by 
$\varphi(\sigma_{\gp,i,j}) = \sigma_{\gp,i,j+1}$ if $j \neq e_\gp$ and
$\varphi(\sigma_{\gp,i,e_\gp}) = \sigma_{\gp,i+1,1}$.

Let $K$ be a finite extension of $\Q_p$ (in $\Qpbar$), which we assume is sufficiently large to contain the images
of all $\sigma \in \Sigma$; let $\CO$ denote
its ring of integers, $\varpi$ a uniformizer, and $E$ its residue field.  %For $\tau \in \Sigma_{\gp,0}$, we let $E_\tau(u) \in \CO[u]$ be
%the image of $E_\gp$ under the homomorphism induced by $\tau$, so that
%$$\CO_F \otimes \CO  =  
%\bigoplus_{\tau \in \Sigma_0}  \CO[u]/(E_\tau).$$
%Similarly for any $\CO_F \otimes \CO$-module $M$, we have
%$M = \bigoplus_{\tau \in \Sigma_{0}} M_\tau$, and we write
%simply $M_{\gp,i}$ for $M_{\tau_{\gp,i}}$.  
%Finally if $S$ is a scheme over $\CO$ and
%$\CM$ is a quasi-coherent sheaf of 
%$\CO_F \otimes \CO_S$-modules on $S$,
%then we let $\CM_{\gp,i}$ denote
%the corresponding summand of $\CM$.

We write $\vec{k}$ for the element of $\ZZ^{\Sigma}$ (or $\ZZ^{\Sigma_\infty}$ or $\ZZ^{\Sigma_\gp}$) with $\sigma$-component $k_\sigma$.  We often abbreviate $k_{\sigma_{\gp,i,j}}$ by $k_{\gp,i,j}$. For $\vec{k} \in \ZZ^{\Sigma}$, we write $\vec{k}_{\gp}$ for its image in $\ZZ^{\Sigma_{\gp}}$ under the canonical projection (induced by the inclusion $\Sigma_\gp\hookrightarrow \Sigma$).
For $\sigma \in \Sigma$, we let $\vec{e}_\sigma$ denote the standard basis vector associated to $\sigma$ (so $e_{\sigma,\sigma'} = 1$ or $0$ according to whether or not $\sigma = \sigma'$).  A special role\footnote{As weights of partial Hasse invariants; see \S\ref{ssec:Hasse}} will be played by certain elements $\vec{h}_\sigma \in \ZZ^\Sigma$, defined as follows: Let $\vec{h}_\sigma = \nu_\sigma \vec{e}_{\varphi^{-1}\sigma} - \vec{e}_\sigma$, where $\nu_{\gp,i,j} = p$ or $1$ according to whether or not $j=1$.

For a non-zero ideal $\gn$ of $\CO_F$, we write $U(\gn)$ and $U_1(\gn)$ for the following standard open compact subgroups\footnote{By contrast, $U_0(\gn)$ (for $\gn$ prime) will be defined below relative to an open compact subgroup $U$, and only used in the context of quaternion algebras; see below.} of $\GL_2(\AA_{F,\f})$:  $U(\gn)$ denotes the kernel of the
projection $\GL_2(\widehat{\CO}_F)\to \GL_2(\CO_F/\gn)$ and
$$U_1(\gn) = \left\{\left.\,\smat{a}{b}{c}{d} \in \GL_2(\widehat{\CO}_F) \,\right|\,
 c,d-1 \in \gn\widehat{\CO}_F\,\right\}.$$
When we consider more general open compact subgroups $U$ of $\GL_2(\AA_{F,\f})$, they are often assumed to be of the form $\prod_v U_v$ with each $U_v \subset \GL_2(\CO_{F,v})$, in which case we write $U = U_pU^p$, where $U_p = \prod\limits_{\gp \in S_p} U_\gp \subset \GL_2(\CO_{F,p})$ and $U^p = \prod\limits_{v \not\in S_p} U_v \subset \GL_2(\AA_{F,\f}^{(p)})$.

For a (non-split) quaternion algebra $B$ over $F$, we let $\Sigma^B$ denote the set of places of $F$ at which $B$ is ramified, and let $\Sigma^B_\infty = \Sigma^B \cap \Sigma_\infty$. We choose $F_\sigma$-algebra isomorphisms $B_\sigma \cong M_2(\RR)$ for each $\sigma \in \Sigma_\infty - \Sigma^B$, and $B_\sigma\otimes_{\RR}\CC \cong M_2(\CC)$ for each $\sigma \in \Sigma^B_\infty$.  We also fix a maximal order $\CO_B$ in $B$ and $\CO_{F,v}$-algebra isomorphisms $\CO_{B,v} \cong M_2(\CO_{F,v})$ for each finite $v \not\in \Sigma^B$, and write $\CO_{B,(p)}$ for $\CO_B\otimes \ZZ_{(p)}$. We let $B_{\AA} = B\otimes \AA = B_\infty \times B_{\f}$, where $B_\infty = B\otimes \RR = \prod_{\sigma\in \Sigma_\infty} B_\sigma$ and $B_{\f} = B\otimes\widehat{\ZZ} = B\otimes_F \A_{F,\f} \subset \prod B_v$, and write $B_{\f}^{(p)}$ for $B\otimes_F \A^{(p)}_{F,\f} \subset \prod_{v\not\in S_p} B_v$.  For $\gamma \in B$, we write $\gamma_\infty$ (resp.~$\gamma_\f$, $\gamma_\f^{(p)}$) for its image in $B_\infty$ (resp.~$B_\f$, $B_{\f}^{(p)}$).  We write 
$\det$ (resp.~$\tr$) for the reduced norm (resp.~trace) map 
$B \to F$, as well as its extension to maps $B_v \to F_v$, $B_{\infty} \to F_\infty$, etc., and again use the subscript $+$ for subsets of $B$ (or $B_\infty$) to denote those elements with totally positive image under $\det$.  We again generally only consider open compact subgroups $U$ of $B_{\f}^\times$ of the form $\prod_v U_v$ with $U_v \subset \CO_{B,v}^\times$ for all $v$, and write $U = U_pU^p$ with $U_p \subset \CO_{B,p}^\times = \prod_{\gp\in S_p} \CO_{B,\gp}^\times$ and $U^p \subset \prod_{v\not\in S_p} \CO_{B,v}^\times \subset (B_{\f}^{(p)})^\times$. For $v \not\in \Sigma^B$ such that $U_v = \CO_{B,v}^\times$, we let $U_0(v) = U \cap hUh^{-1}$, where $h$ is any element of $\CO_{B,v}$ such that $\det(h)$ is a uniformizer in $\CO_{F,v}$ (so $h \in B_v^\times \subset B_{\f}^\times$, and $U_0(v)$ only depends up to conjugacy in $\CO_{B,v}^\times$ on the choice of $h$).

For any field $L$, we let $G_L$ denote its absolute Galois group.
If $L$ has characteristic different from $p$, then we let $\chi_{\cyc}:G_L \to \Z_p^\times$ denote the $p$-adic cyclotomic character (or the mod $p$ cyclotomic character $G_L \to \F_p^\times$ if clear from the context).
We use $\phi$ (as above) to denote the Frobenius automorphism in $G_{\F_p}$, as well as its lifts to automorphisms of $W(\Fpbar)$ and its field of fractions.  

If $L = F_v$ (or a finite extension thereof), then we write $I_L$ for the inertia subgroup of $G_L$.  We let $\Frob_v$ denote the {\em geometric} Frobenius element of $G_{F_v}/I_{F_v} = G_{\CO_F/v}$. (In particular, if 
$\gp \in S_p$, then $\Frob_\gp = \phi^{-f_{\gp}}$.)  We often view $G_{F_v}$ as a subgroup of $G_F$ via the inclusion $G_{F_v} \hookrightarrow G_F$, determined up to conjugacy in $G_F$ by a choice of $F$-embedding $\overline{F}\hookrightarrow \overline{F}_v$.  We let $W_{F_v}$ denote the Weil group of $F_v$, i.e., the preimage of $\langle \Frob_v \rangle$ in $G_{F_v}$, and normalize the Artin isomorphism $W_{F_v}^{\ab} \cong F_v^\times$ of local class field theory so that lifts of $\Frob_v$ correspond to uniformizers in $F_v^\times$.  In particular if $v\not\in S_p$, then the restriction of $\chi_{\cyc}$ corresponds to $|\cdot|_v$ (where $\varpi_v$ has absolute value $\Nm_{F/\Q}(v)^{-1}$).
We normalize the local Langlands correspondence for $\GL_2$ so that the principal series representation 
$I(\psi_1|\cdot|^{1/2},\psi_2|\cdot|^{1/2})$ of $\GL_2(F_v)$
(for $\psi_2\neq \psi_1|\cdot|^{\pm1}$) corresponds to the representation $\psi_1 \oplus \psi_2$ of $W_{F_v}$ (conflating continuous characters of $W_{F_v}$ and $F_v^\times$).

For $\sigma \in \Sigma_\gp$, we let $\omega_\sigma:I_{F_\gp} \to \Fpbar^\times$ denote the fundamental character associated to $\Sigma$, defined as the composite
$I_{F_\gp} \onto \CO_{F,\gp}^\times \onto \F_{\gp}^\times \stackrel{\sigma}{\hookrightarrow} \Fpbar^\times$
(where the first map is induced by the Artin isomorphism).  Note in particular that $\omega_\sigma$ depends only on the restriction of $\sigma$ to $F_{\gp,0}$.

For $\alpha \in F^\times$ and $\vec{k} \in \ZZ^\Sigma$, we let $\alpha^{\vec{k}}$ denote $\prod_\sigma \sigma(\alpha)^{k_\sigma} \in \Qbar^\times$. We use the same notation similarly in various related settings; for example, if $\alpha \in \CO_{F,(p)}^\times$, then we may view $\alpha^{\vec{k}}$ as an element of $\CO^\times$, or more generally in $R^\times$ for an $\CO$-algebra $R$ that should be clear from the context. We sometimes further extend this notation to allow exponents in $\frac{1}{2}\ZZ^\Sigma$ for totally positive $\alpha$.

Note that if $\vec{\ell} \in \frac{1}{2}\ZZ^\Sigma$ and $R$ is a an $\CO$-algebra such that $p^nR = 0$, then there exist continuous characters $\xi:(\AA_{F,\f}^{(p)})^\times \to R^\times$ such that $\xi(\alpha) = \alpha^{\vec{\ell}}$ for all $\alpha \in \CO_{F,(p),+}^\times$.  We make use of such characters, especially for $R = E$, in various constructions and arguments. We often use $e_\xi$ to denote a basis for a rank one $R$-module on which $\GL_2(\AA_{F,\f}^{(p)})$ or $(B_{\f}^{(p)})^\times$ acts via $\xi\circ\det$. Note that if $\vec{\ell} \in \ZZ^\Sigma$, then $\xi$ extends uniquely to a continuous character of $\AA_F^\times/F^\times F_{\infty,+}^\times$, which we also denote by $\xi$.  Furthermore, we have $\xi(u) = \prod_{\sigma\in \Sigma_\gp} \sigma(u)^{-\ell_\sigma}$ for all $\gp \in S_p$ and $u \in \CO_{F,\gp}^\times$. In particular, if $R=E$, then the restriction to $I_{F_{\gp}}$ of the character of $G_F$ corresponding to $\xi$ via class field theory is $\prod_{\sigma\in \Sigma_\gp} \omega_\sigma^{-\ell_\sigma}$.

For notation related to $p$-adic Hodge theory, see \S\ref{subsec:cryslift}.  We mention here only our convention that $\chi_{\cyc}:G_{\Q_p} \to \QQ_p^\times$ has Hodge--Tate weight $1$. 
\begin{remark} \label{rmk:conventions} We caution that our conventions in this paper with regard to Hecke actions differ by a twist from those in the prequel~\cite{DS1}. More precisely, the definition of the action of $g \in \GL_2(\A_{F,\f}^{(p)})$ on spaces of Hilbert modular forms in \cite[\S4.2]{DS1} (for $p$ unramified in $F$) involves a factor of $||\det(g)||$ which is not included here, so for example, our Hecke operator $T_v$ is $\Nm_{F/\Q}(v)$ times the one denoted $T_v$ in \cite{DS1}.  The reason for this is that the normalizations in \cite{DS1} were chosen for consistency with the more familiar setting of classical modular forms, but the ones in this paper were chosen for consistency with a more general framework of automorphic forms, including Hilbert modular forms over $\CC$ of non-paritious weight, of which we make repeated use.
In the same vein, our representation $D_{\vec{k},\vec{m},\Fpbar}$ of $\GL_2(\CO_F/p)$ is the twist by the character $\prod_\sigma (\sigma\circ\det^{-1})$ of the one denoted $V_{\vec{k},\vec{m}}$ in \cite{DS1}.
\end{remark}

\section{Hilbert modular forms over $\CC$} \label{sec:HMFC}

\subsection{Hilbert modular varieties over $\CC$} \label{ssec:HMVC}
For an open compact subgroup $U$ of $\GL_2(\A_{F,\f})$,
the Hilbert modular variety of level $U$ over $\CC$ is defined as the quotient
$$Y_U = \GL_2(F) \backslash \GL_2(\A_F) / UU_\infty =
\GL_2(F)_+ \backslash (\uhp^{\Sigma_\infty} \times \GL_2(\A_{F,\f})/U),
$$
where $U_\infty = \prod_{\sigma\in \Sigma_\infty} \R^\times\SO_2(\RR)
 \subset \GL_2(F_\infty)$, $\uhp$ is the complex upper half-plane, and $\GL_2(F)_+ \subset \GL_2(\RR)_+^{\Sigma_\infty}$ 
acts componentwise on $\uhp^{\Sigma_\infty}$ via linear fractional transformations.  Note that we have a natural right action of 
$\GL_2(\A_{F,\f})$ by right-multiplication on the inverse system of the $Y_U$ (for varying $U$).  More precisely, if 
$U$ and $U'$ are open compact subgroups of $\GL_2(\A_{F,\f})$,
and $h \in \GL_2(\A_{F,\f})$ is such that $U' \subset hUh^{-1}$,
then we have the map $\rho_h: Y_{U'} \to Y_U$ induced by
$(\tau, gU') \mapsto (\tau, ghU)$, and these satisfy
$\rho_h\circ\rho_{h'} = \rho_{h'h}$ 
(for $U'' \subset h'U'h'^{-1}$).
Furthermore the determinant map induces a bijection between the set of connected components of $Y_U$ and the set
$$\A_F^\times/F^\times F_{\infty,+}^\times \det(U),$$
compatibly with the obvious action of $\GL_2(\A_{F,\f})$.

If $U$ is sufficiently small, then $Y_U$ inherits the
structure of a $d$-dimensional complex manifold from $\uhp^{\Sigma_\infty}$.
Furthermore $Y_U$ also naturally has the structure of a quasi-projective variety over $\QQ$, but we shall be primarily interested in the special fibre of an integral model over a (sufficiently large) finite extension of $\QQ_p$, and only in the case when the level $U$ is prime to $p$ (see \S\ref{ssec:PRmodel}, below). To facilitate statements relating the complex and $p$-adic settings, we shall often view $Y_U$ as a variety over $\Qbar$.

The Hilbert modular variety $Y_U$ admits a minimal compactification $Y_U^{\min}$ (see for example \cite{compact}).
The complement of the open immersion
$i:Y_U \hookrightarrow Y_U^{\min}$  is the finite set of
{\em cusps} $C_U$ of $Y_U$, in canonical bijection with
$$B(F)_+\backslash \GL_2(\A_{F,\f})/U$$
where $B(F)_+$ denotes the subgroup of upper-triangular
matrices in $\GL_2(F)_+$.  We let $j:C_U \hookrightarrow Y_U^{\min}$ denote the closed immersion complementing $i$.

The description of the completion of
$Y_U^{\min}$ at each cusp is given for example in \cite[\S7.2]{theta}.  In particular if $U = U(\gn)$ for
some non-zero ideal $\gn \subset \CO_F$, then the completion 
$\CO_{Y_U^{\min},x}^{\wedge} = (i_* \CO_{Y_U})_x^{\wedge}$
at each cusp $x$ of $Y_U^{\min}$ is isomorphic to
$$\left\{\left. \sum_{m\in \gn^{-1}M_{x,+}\cup \{0\}} r_m q^m\,\right|\,
  \mbox{$r_m \in \CC$, $r_{m} = r_{\nu m}$ for all $\nu
 \in V_{\gn,+}$, $m \in \gn^{-1}M_{x,+} \cup \{0\}$} \right\},$$
where $M_x$ is a certain oriented invertible $\CO_F$-module and
$V_{\gn} = \ker(\CO_F^\times\to (\CO_F/\gn)^\times)$.
Similarly if $U = U_1(\gn),$
then we have the same description of each $\CO_{Y_U^{\min},x}^{\wedge}$, but with $\gn^{-1}M_x$ replaced by $M_x$
and $V_{\gn}$ replaced by $\CO_F^\times$.

\subsection{Hilbert modular forms over $\CC$} \label{ssec:HMFC}
Suppose now that $\vec{k} = (k_\sigma)_{\sigma \in \Sigma_\infty}
\in \ZZ^{\Sigma_\infty}$.  For $g = \smallmat{a}{b}{c}{d}
\in \GL_2(\RR)_+$, $\tau\in \uhp$,
let $j(g,\tau) = \det(g)^{-1/2}(c\tau + d) \in \CC$.

If $U$ is sufficiently small,
then we have the complex line bundle $\CL_{\vec{k}} = \CL_{\vec{k},U}$ on 
$Y_U$ defined by (holomorphic sections of)
$$\GL_2(F)_+ \backslash (\uhp^{\Sigma_\infty}  \times \CC
 \times \GL_2(\A_{F,\f})/U),$$
the action of $\gamma \in \GL_2^+(F)$
being given by the formula
$$\gamma\cdot(\tau,z,g) = 
(\gamma_\infty(\tau),j(\gamma_\infty,\tau)^{\vec{k}}, \gamma_\f g)$$
where $j(\gamma_\infty,\tau)^{\vec{k}} = \prod_{\sigma} 
j(\sigma(\gamma),\tau_\sigma)^{k_\sigma}$ for $\tau = (\tau_\sigma)_\sigma \in \uhp^{\Sigma_\infty}$.

A section $f \in M_{\vec{k}}(U,\CC) := H^0(Y_U,\CL_{\vec{k}})$ is called a {\em Hilbert modular form} of {\em weight} $\vec{k}$ and level $U$.  The natural action of
$\GL_2(\A_{F,\f})$ on the inverse system of $Y_U$ gives rise to
a left action of $\GL_2(\A_{F,\f})$ on
$$\CM_{\vec{k}} = \varinjlim_{U} M_{\vec{k}}(U,\C).$$
More precisely for $h \in \GL_2(\A_{F,\f})$ such that 
$U' \subset hUh^{-1}$, we have the isomorphism
$\rho_h^*\CL_{\vec{k},U} \stackrel{\sim}{\to} \CL_{\vec{k},U'}$
induced by $(\tau, gU',z) \mapsto (\tau, ghU,z)$, and
hence (injective) maps $M_{\vec{k}}(U,\C) \to M_{\vec{k}}(U',\C)$
which yield the action.  Furthermore we may identify 
$M_{\vec{k}}(U,\C)$ with $\CM_{\vec{k}}^U$.

By the Koecher Principle, the direct image 
$\CL_{\vec{k}}^{\min}:= i_*\CL_{\vec{k}}$ is a coherent sheaf
on $Y_U^{\min}$, and in fact a line bundle if $\vec{k}$ is parallel (i.e., if $k_\sigma$ is independent of $\sigma$).  The completion of $\CL_{\vec{k}}^{\min}$ at each
cusp of $Y_U^{\min}$ is as described in \cite[\S7.2]{theta}.
In particular if $x$ is a cusp of $Y_{U(\gn)}^{\min}$,
then $(\CL_{\vec{k}}^{\min})^{\wedge}_x$ is isomorphic\footnote{The isomorphism depends on a choice of basis for a certain one-dimensional complex vector space depending on the cusp representative and the weight $\vec{k}$; see \cite{theta} for a more canonical description.} to
$$\left\{\left. \sum_{m\in \gn^{-1}M_{x,+}\cup \{0\}} r_m q^m\,\right|\,
  \mbox{$r_m \in \CC$, $r_{m} = 
  \nu^{\vec{k}/2}r_{\nu m}$ for all $\nu
 \in V_{\gn,+}$, $m \in \gn^{-1}M_{x,+} \cup \{0\}$} \right\}$$
as an $\CO_{Y_{U(\gn)}^{\min},x}^{\wedge}$-module
(where $\nu^{\vec{k}/2} = \prod_\sigma \sigma(\nu)^{k_\sigma/2}$).  Again the description
with $U(\gn)$ replaced by $U_1(\gn)$ is the same, but with 
$\gn^{-1}M_x$ replaced by $M_x$ and $V_{\gn}$ replaced by $\CO_F^\times$.  Note in particular that $r_0=0$ unless $\vec{k}$ is parallel.
We recall also the $q$-expansion Principle, which states that if
$C$ is any set of cusps of $Y_U^{\min}$ containing at least one cusp on each connected component, then the natural map
$$M_{\vec{k}}(U,\C) \longrightarrow \bigoplus_{x\in C}
(\CL_{\vec{k}}^{\min})^{\wedge}_x$$
is injective.

We call $f \in M_{\vec{k}}(U,\CC)$ a {\em cusp form} if the
coefficient $r_0$ of its $q$-expansion at each cusp vanishes.
Note that this holds automatically if $\vec{k}$ is not parallel,
so letting $S_{\vec{k}}(U,\CC)$ denote the space of cusp forms of weight $\vec{k}$ and level $U$, we have
$$S_{\vec{k}}(U,\CC) = H^0(Y_U^{\min},\CL_{\vec{k}}^{\sub})
 \subset H^0(Y_U^{\min},\CL_{\vec{k}}^{\min}) = M_{\vec{k}}(U,\CC),$$
where $\CL_{\vec{k}}^{\sub} = \ker(\CL_{\vec{k}}^{\min} \to j_*j^* \CL_{\vec{k}}^{\min}$) if $\vec{k}$ is parallel, and
$\CL_{\vec{k}}^{\sub}=\CL_{\vec{k}}^{\min}$ otherwise.

The action of $\GL_2(\A_{F,\f})$
on $\CM_{\vec{k}}$ restricts to one on 
$$\CS_{\vec{k}} = \varinjlim_{U} S_{\vec{k}}(U,\C),$$
and $\CS_{\vec{k}} = 0$ unless $k_\sigma \ge 1$ for all $\sigma \in \Sigma$, in which case it may be described in terms of 
cuspidal automorphic representations of weight $\vec{k}$ of 
$\GL_2(\A_F)$. More precisely, for an integer $k$ greater than
(or equal to) $1$, let $D'_{k}$ be the irreducible unitary (limit of)
discrete series representation of $\GL_2(\R)$ of weight $k$.
We then have
$$ \CS_{\vec{k}} \cong \bigoplus_{\pi \in \mathcal{C}_{\vec{k}}} 
\pi_{\f},$$ where $\mathcal{C}_{\vec{k}}$ is the set of cuspidal
automorphic representations $\pi = \pi_\infty \otimes \pi_{\f} = \otimes'_v \pi_v$ of  $\GL_2(\AA_F)$  such that $\pi_\sigma \cong D'_{k_\sigma}$ for all $\sigma\in \Sigma_\infty$.

\subsection{The Jacquet--Langlands Correspondence} \label{ssec:JL}
We shall also make use of automorphic representations
of $B_{\AA}^\times$, where $B$ is a (non-split)
quaternion algebra over $F$.  Recall that $\Sigma^B$ denotes
the set of places of $F$ at which $B$ is ramified,
and $\Sigma^B_\infty = \Sigma^B \cap \Sigma_\infty$.
Recall also that for each (finite or infinite) place $v \not\in \Sigma^B$, we have chosen an isomorphism
$B_v \cong M_2(F_v)$ of $F_v$-algebras, as well as
an isomorphism
$B_\sigma \otimes_{\RR} \CC \cong M_2(\CC)$ of $\CC$-algebras for 
for each
$\sigma \in \Sigma^B_\infty$.  For $k \ge 2$, let $D_k$ denote the representation of $B_\sigma^\times$ defined by $\det^{(2-k)/2}\otimes \Sym^{k-2}\CC^2$.  (Note that $\det(B_\sigma^\times) \subset \R_+^\times$ if $\sigma \in \Sigma^B_\infty$, so this is well-defined.)

Recall that if $v \in \Sigma^B$, then the local Jacquet-Langlands correspondence (see \cite[\S8]{GJ}) defines a bijection $\JL_v$ between the sets of (isomorphism classes of) smooth irreducible representations of $B_v^\times$ and discrete series representations of 
$\GL_2(F_v)$. In particular, if $k \ge 2$
then $\JL(D_k) = D_k'$.

Suppose now that $\vec{k} \in \ZZ^{\Sigma_\infty}$ is such that
$k_\sigma \ge 1 $ for all $\sigma \in \Sigma_\infty$ and
$k_\sigma \ge 2 $ for all $\sigma \in \Sigma_\infty^B$.
We let $\mathcal{C}_{\vec{k}}^B$ denote the set of all infinite-dimensional automorphic representations $\pi = \otimes'_v \pi_v$ of $B_\AA^\times = (B\otimes\AA)^\times$ such
that $\pi_\sigma \cong D_{k_\sigma}$ if $\sigma \in \Sigma_\infty^B$ and $\pi_\sigma \cong D'_{k_\sigma}$ if $\sigma \not\in \Sigma_\infty^B$.

The (global) Jacquet--Langlands correspondence can then be stated as follows:
\begin{theorem} \label{thm:JL} There is an injection $\JL_{\vec{k}}^B:\mathcal{C}^B_{\vec{k}} \hookrightarrow \mathcal{C}_{\vec{k}}$ such that if 
$\Pi = \otimes'_v \Pi_v \in \mathcal{C}^B_{\vec{k}}$, then
$\JL_{\vec{k}}^B(\Pi) = \pi = \otimes'_v \pi_v$ is characterized by
\begin{itemize}
\item $\pi_v \cong \Pi_v$ for all $v\not\in \Sigma^B$;
\item $\pi_v \cong \JL_v(\Pi_v)$ for all $v \in \Sigma^B$.
\end{itemize}
Furthermore the image of $\JL_{\vec{k}}^B$ is the set of all 
$\pi = \otimes'_v \pi_v$ such that $\pi_v$ is a discrete
representation of $\GL_2(F_v)$ for all $v \in \Sigma^B$.
\end{theorem}

\subsection{Automorphic forms on totally definite quaternion algebras over $\CC$} \label{ssec:defQC}
We have the following explicit description of $\mathcal{C}_{\vec{k}}^B$ in the case that $B$ is a totally definite quaternion algebra, i.e.,
$\Sigma_{\infty}^B = \Sigma_\infty$.  For an open compact
subgroup $U$ of $B_{\f}^\times := (B\otimes \A_\f)^\times$, let
$$M^B_{\vec{k}}(U,\CC) = 
  \{\,f:B_{\AA}^\times \to D_{\vec{k}}\,|\, \mbox{$f(\gamma x u)
  = u_\infty^{-1} f(x)$ for all $\gamma\in B^\times$,
$x \in B_{\AA}^\times$,
$u \in B_\infty^\times U$} \,\},$$
where $D_{\vec{k}} = \otimes_{\sigma} D_{k_\sigma}$
(and we assume $k_\sigma \ge 2$ for all $\sigma$).
Note that restriction identifies $M^B_{\vec{k}}(U,\CC)$
with 
$$\{\,f:B_{\f}^\times \to D_{\vec{k}}\,|\, \mbox{$f(\gamma_{\f} x u)
  = \gamma_\infty f(x)$ for all $\gamma\in B^\times$, $x\in B_{\f}^\times$, 
$u \in U$}\,\}.$$
Furthermore if $U$ is sufficiently small, then $B^\times \cap U$ is contained in $\CO_{F,+}^\times$, and hence acts trivially on $D_{\vec{k}}$.  We may then identify $M^B_{\vec{k}}(U,\CC)$
with $H^0(Y_U^B, \mathcal{D}_{\vec{k}})$, where $Y_U^B$ is the finite
set 
$B^\times \backslash B_{\f}^\times / U \cong
B^\times \backslash B_{\AA}^\times / B_\infty^\times U$,
and $\mathcal{D}_{\vec{k}}$ is the sheaf on $Y_U^B$ defined by
$B^\times \backslash (D_{\vec{k}} \times B_{\f}^\times / U)$.

Just as for Hilbert modular forms, we have a natural action
of $B_{\f}^\times$ on the direct limit
$$\CM^B_{\vec{k}} := \varinjlim_{U} M^B_{\vec{k}}(U,\C).$$
If $k_\sigma = 2$ for all $\sigma$ (so $D_{\vec{k}}$ is the
trivial representation), then we define 
$I^B_{\vec{k}}(U,\C) \subset M_{\vec{k}}^B(U,\C)$
to be the subspace of functions $B_{\AA}^\times \to \CC$
factoring through the map
$$B_{\AA}^\times \stackrel{\det}{\longrightarrow}
  \AA_F^\times \longrightarrow 
 F^\times \backslash \A_F^\times / F_{\infty,+}^\times \det(U)$$
(where we use $\det$ to denote the reduced norm);
otherwise let $I_{\vec{k}}^B(U,\C) = 0$.
Letting $\CI_{\vec{k}}^B = \varinjlim_{U} I_{\vec{k}}^B(U,\C)$
and $\CS_{\vec{k}}^B = \varinjlim_{U} S_{\vec{k}}^B(U,\C)$,
where $S_{\vec{k}}^B(U,\C) = M_{\vec{k}}^B(U,\C)/I_{\vec{k}}^B(U,\C)$, we have
$$\CS_{\vec{k}}^B = \CM^B_{\vec{k}}/\CI^B_{\vec{k}} \cong \bigoplus_{\Pi \in \mathcal{C}_{\vec{k}}^B} \Pi_{\f}$$
as representations of $B_{\f}^\times$.  Furthermore we may identify $M_{\vec{k}}^B(U,\C)$ (resp.~$S_{\vec{k}}^B(U,\C))$ with 
$(\CM_{\vec{k}}^B)^U$ (resp.~$(\CS_{\vec{k}}^B)^U$).

\subsection{Automorphic forms on Shimura curves over $\CC$} \label{ssec:indefQC}

Suppose now that $|\Sigma^B_\infty| = d-1$, and let $\sigma_0$ be the unique element of $\Sigma_\infty - \Sigma_\infty^B$.
Let $Y_U^B$ denote the Shimura curve
$$B_+^\times \backslash ((\uhp \times B_{\f}^\times)/U),$$
where $B_+^\times$ acts on $\uhp$ via its embedding in $B_{\sigma_0}^\times$ (for which we fixed an isomorphism with
$\GL_2(\RR)$).  Thus $Y_U^B$ is a compact Riemann surface, and
if $U$ is sufficiently small, then the stabilizer of 
any point in $(\uhp \times B_{\f}^\times)/U$ under the action
of $B_+^\times$ is contained in $\CO_{F,+}^\times$.

We view $D_{\vec{k}} = \otimes_{\sigma} D_{k_\sigma}$ as
a representation of $B_{\infty,+}^\times$, and hence $B_+^\times$,
where $D_{k_{\sigma_0}}=\det^{(2-k_{\sigma_0})/2}\otimes \Sym^{k_{\sigma_0}-2}\CC^2$ with its action of
$B_{\sigma_0,+}^\times \cong \GL_2(\RR)_+$.
For sufficiently small $U$, we have a locally constant
sheaf on $Y_U^B$ defined by (sections of)
$$B^\times_+ \backslash ((\uhp  \times B_{\f}^\times \times D_{\vec{k}})/U),$$
which we denote by $\CV^B_{\vec{k},U}$.

Now consider the cohomology groups
$M_{\vec{k}}^B(U,\C):=H^1(Y_U^B,\CV^B_{\vec{k},U})$.
As in the case of Hilbert modular varieties and forms, for
$h \in B_{\f}^\times$ and $U$, $U'$ such that $U' \subset
hUh^{-1}$, we have morphisms $\rho_h:Y_{U'} \to Y_U$
and $\rho_h^* \CV_{\vec{k},U}^B \to \CV_{\vec{k},U'}^B$,
yielding maps $M^B_{\vec{k}}(U,\C)\to M^B_{\vec{k}}(U',\C)$ 
giving rise to an action of $B_{\f}^\times$ on
$$\CM_{\vec{k}}^B := \varinjlim_{U} M^B_{\vec{k}}(U,\C).$$
We then have
$$\CM_{\vec{k}}^B  \cong \bigoplus_{\Pi\in \mathcal{C}_{\vec{k}}^B}
\left(\Pi_{\f} \otimes \CC^2 \right)$$
as representations of $B_{\f}^\times$
(as in \cite[(4)]{DT}, for example).
Again we have the identification $(\CM_{\vec{k}}^B)^U =
M^B_{\vec{k}}(U,\C) $.

\section{Hilbert modular forms over $\CO$-algebras} \label{sec:HMFO}

%\subsection{Embeddings} \label{ssec:emb}

\subsection{Hilbert modular varieties over $\CO$} \label{ssec:PRmodel} 
Let $\A_{F,\f}^{(p)} = F \otimes \widehat{\ZZ}^{(p)}$, where 
$\widehat{\ZZ}^{(p)} = \prod_{\ell \neq p}\ZZ_\ell$  
denotes the prime-to-$p$ completion of $\ZZ$, so 
$\A_{F,\f} = \A_{F,\f}^{(p)} \times \prod_{\gp \in S_p} F_{\gp}$.
Let $U$ be a sufficiently small open compact subgroup of 
$\GL_2(\A_{F,\f})$ containing $\GL_2(\CO_{F,p})$, so that
$U = U_pU^p$ where $U_p = \GL_2(\CO_{F,p})$ and $U^p 
\subset \GL_2(\A_{F,\f}^{(p)})$.

Let $\widetilde{\CY}_U$ be the $\CO$-scheme representing
isomorphism classes of data $(A,\iota,\lambda,\eta,\CF^\bullet)$
as in \cite[\S2.2]{FD:KS}, i.e., abelian schemes with $\CO_F$-action and suitable quasi-polarization, level $U^p$-structure
and Pappas-Rapoport filtrations.  Recall that this is a smooth
quasi-projective scheme of relative dimension $d$ over $\CO$
equipped with an action of $\CO_{F,(p),+}^\times$, and the quotient $\CY_U = \CO_{F,(p),+}^\times \backslash \widetilde{\CY}_U$ is a smooth model for the Hilbert modular variety $Y_U$,
i.e., we have an identification of $\CY_{U,\Qpbar}$ 
with $Y_{U,\Qpbar}$ as complex manifolds.
For $U$, $U'$ as above and 
$h\in\GL_2(\A_{F,\f}^{(p)})$ such that $U'\subset hUh^{-1}$,
we have a morphism $\widetilde{\rho}_h:\widetilde{\CY}_{U'} \to \widetilde{\CY}_U$, and these morphisms define an action of
$\GL_2(\A_{F,\f}^{(p)})$ on the inverse system of the $\widetilde{\CY}_U$ (for varying $U$ as above).  Furthermore the
morphism $\widetilde{\rho}_h$ is compatible with the action of $\CO_{F,(p),+}^\times$, inducing a morphism $\rho_h:\CY_{U'} \to \CY_U$ which is compatible (in the obvious sense) with the one previously so denoted,

We will also make use of the minimal compactification $\CY_U^{\min}$ of $\CY_U$, as defined in \cite[\S7.2]{theta}.  The reduced complement of $\CY_U$ is a finite $\CO$-scheme $\CZ_U$, and we assume $\CO$ is sufficiently large that its components (as well as those of $\CY_U^{\min}$) are geometrically connected.  Thus $\CZ_U$ is a disjoint union of copies of $\Spec(\CO)$, which we also call cusps, corresponding to the cusps of $Y_U^{\min}$. 
The description of the completion of 
$\CY_U^{\min}$ at each such cusp is the same as for $Y_U^{\min}$, but with $\CC$ replaced by $\CO$.

We note also that the connected components of $\CY_U$ (or equivalently $\CY_U^{\min}$) correspond to those of $Y_U$.

\subsection{Hilbert modular forms over $\CO$-algebras}\label{ssec:HMFp}
For $\vec{k},\vec{m} \in \ZZ^\Sigma$, we have the automorphic line bundle
$$\widetilde{\CA}_{\vec{k},\vec{m}} = \bigotimes_{\sigma\in\Sigma}
  \CL_\sigma^{k_\sigma+m_\sigma} \otimes \CM_\sigma^{m_\theta}$$
on $\widetilde{\CY}_U$, where 
$\CL_{\sigma_{\gp,i,j}}$ is the subquotient 
$\CF_{\tau_{\gp,i}}^{(j)}/\CF_{\tau_{\gp,i}}^{(j-1)}$
of the Pappas-Rapoport filtration on the $\tau_{\gp,i}$-component
of the Hodge bundle of the universal abelian scheme over 
$\widetilde{\CY}_U$, and similarly $\CM_\sigma$ is a
certain invertible subquotient of its relative de Rham cohomology (see \cite[\S2.3]{FD:KS}).
The line bundles $\widetilde{\CA}_{\vec{k},\vec{m}}$ are equipped with a natural action of $\CO_{F,(p),+}^\times$ over its action on
$\widetilde{\CY}_U$, i.e., isomorphisms
$$\alpha_\mu:  \CA_{\vec{k},\vec{m}} 
\stackrel{\sim}{\longrightarrow} \psi_\mu^*\CA_{\vec{k},\vec{m}}$$
for $\mu \in \CO_{F,(p),+}^\times$ satisfying $\alpha_{\mu\mu'} = \psi_{\mu'}^*(\alpha_\mu)\circ \alpha_{\mu'}$ for $\mu,\mu'\in \CO_{F,(p),+}^+$, where $\psi_\mu$ and $\psi_{\mu'}$ are the underlying automorphisms of $\widetilde{\CY}_U$.

If $\vec{k}+2\vec{m}$ is parallel (in the sense that
$k_\sigma+2m_\sigma$ is independent of $\sigma$), then this action defines descent data to a line bundle $\CA_{\vec{k},\vec{m}}$ on $\CY_U$.
Furthermore for any $\CO$-algebra $R$ in which the image of
$\mu^{\vec{k}+2\vec{m}} := 
\prod_{\sigma} \sigma(\mu)^{k_\sigma + 2m_\sigma}$
is trivial for all $\mu \in U \cap \CO_F^\times$, we similarly
obtain a line bundle $\CA_{\vec{k},\vec{m},R}$ on $\CY_{U,R}$.
Note that if $p^N R = 0 $ for some $N \ge 1$, then this condition holds for all sufficiently small $U$.
We then define the space of {\em Hilbert modular forms} over $R$ of {\em weight $(\vec{k},\vec{m})$} and {\em level $U$} to be the $R$-module $$M_{\vec{k},\vec{m}}(U,R) : =H^0(\CY_{U,R},\CA_{\vec{k},\vec{m},R}).$$

For the purpose of arguments involving lifting to characteristic zero and applying the Jacquet-Langlands correspondence, we shall also need integral structures on spaces of Hilbert modular forms for which the $k_\sigma$ differ in parity.
To that end, suppose that $K$ is sufficiently large to contain $\sigma(\mu)^{1/2} \in \Qbar^\times \subset \Qpbar^\times$ for all $\mu \in \CO_{F,(p),+}^\times$, $\sigma\in \Sigma$.  We may then define another action of $\CO_{F,(p),+}^\times$ on $\widetilde{\CA}_{\vec{k},\vec{0}}$ by setting $\alpha'_\mu = \mu^{-\vec{k}/2}\alpha_\mu$, yielding well-defined descent data and hence a line bundle on $\CY_U$ which we denote by $\CL_{\vec{k},\CO}$.  This is compatible with our previous notation in the sense that the line bundle $\CL_{\vec{k},\Qpbar}$ over ${\CY}_{U,\Qpbar}$ may be identified with the base-change of the line bundle $\CL_{\vec{k}}$ on $Y_U$ (defined over $\Qbar$).  For any $\CO$-algebra $R$, we define
$$M_{\vec{k}}(U,R) : =H^0(\CY_{U,R},\CL_{\vec{k},R}).$$
The relation between the $M_{\vec{k},\vec{m}}(U,R)$ and
$M_{\vec{k}}(U,R)$ will be described below.

\subsection{$q$-expansions} \label{ssec:qexp}
Again letting $i:\CY_{U,R} \to \CY_{U,R}^{\min}$ denote the open immersion of $\CY_U$ in its minimal compactification, the direct images $\CA_{\vec{k},\vec{m},R}^{\min} = i_*\CA_{\vec{k},\vec{m},R}$ and $\CL_{\vec{k},R}^{\min} = i_*\CL_{\vec{k},R}$ 
are coherent sheaves on $\CY_{U,R}^{\min}$.  Furthermore their completion at cusps have a similar description to that of $\CL_{\vec{k}}^{\min}$, but with $\CC$ replaced by $\CO$ (and $\vec{k}/2$ replaced by $-\vec{m}$). Note however that formation of $\CA_{\vec{k},\vec{m},R}^{\min}$ (or $\CL_{\vec{k},R}^{\min}$) does not commute with base-change.  For example, $\CA_{\vec{k},\vec{m},\CO}^{\min}$
is a line bundle if and only if $\vec{k}$ (or equivalently $\vec{m}$) is parallel, but $\CA_{\vec{k},\vec{m},E}^{\min}$
is a line bundle for all $\vec{k}$, $\vec{m}$ and sufficiently small $U$.

The $q$-expansion Principle holds in this setting as well, i.e., if $C$ contains at least one cusp on each connected component of $\CY_{U}$, then the natural maps
$$M_{\vec{k,\vec{m}}}(U,R) \longrightarrow \bigoplus_{x\in C}
(\CA_{\vec{k},\vec{m},R}^{\min})^{\wedge}_x\quad\mbox{and}\quad
M_{\vec{k}}(U,R) \longrightarrow \bigoplus_{x\in C}
(\CL_{\vec{k},R}^{\min})^{\wedge}_x$$
are injective.

We shall often assume $U = U_1(\gn)$ or $U(\gn)$ for some $\gn$ prime to $p$, and take $C$ to be the set of cusps at $\infty$, in the sense of \cite[\S6.4]{compact}.  Such cusps have the form
$B(F)_+ \smat{t}{0}{0}{1}U$ for $t \in (\AA_{F,\f}^{(p)})^\times$ (or equivalently $B(\CO_{F,(p)})_+ \smat{t}{0}{0}{1}U^p$)
for $t \in (\AA_{F,\f}^{(p)})^\times$ 
and are in bijection with the set of connected components of $Y_U^{\min}$.  If $x \in C$ is the cusp associated to $t$, then the oriented invertible $\CO_F$-module $M_x$ is isomorphic to 
$t^{-1}\gd^{-1}\widehat{\CO}_F \cap F$, where $\gd$ is the different of $F$.  We then write\footnote{Here as well the coefficient is viewed as an element of $R$ after choosing a basis for a certain invertible $R$-module; see~\cite[\S7.2]{theta}} 
$r_m^t(f)$ for the coefficient of $q^m$ in the $q$-expansion of $f$ at the cusp corresponding to $t$ (with $r_m^t$ understood to be $0$ if, for example, $U = U_1(\gn)$ and $m \not\in M_{x,+} \cup \{0\}$).   The dependence on the choice of $t$ is given by the formula
\begin{equation} \label{eqn:qexpinv} 
r_m^{\alpha t u}(f) = \alpha^{\vec{m}} r_{\alpha m}^t(f)
\end{equation}
for $\alpha \in \CO_{F,(p),+}^\times$, 
$u \in (\widehat{\CO}^{(p)}_F)^\times$ if $U = U_1(\gn)$, 
(and further that $u \equiv 1\bmod \gn$ if $U = U(\gn)$), where $\alpha t u$ is understood as an element of $(\AA_{F,\f}^{(p)})^\times$ (see \cite[(6.1)]{compact}).

We define the spaces of cusp forms $S_{\vec{k},\vec{m}}(U,R)$ and $S_{\vec{k}}(U,R)$ as in the complex setting, via vanishing of the constant term of the $q$-expansion at all cusps.  We can again identify $S_{\vec{k},\vec{m}}(U,R)$ with the $R$-module of sections of a coherent subsheaf $\CA_{\vec{k},\vec{m},R}^{\sub}$ of $\CA_{\vec{k},\vec{m},R}^{\min}$; in this generality the kernel of a (possibly trivial) morphism to $j_*\CO_{\CZ_{U,R}}$ (writing $j$ for the closed immersion $\CZ_{U,R} \hookrightarrow \CY_{U,R}^{\min}$).  Similarly we have 
$S_{\vec{k}}(U,R) = H^0(\CY_{U,R},\CL_{\vec{k},R}^{\sub})$ where $\CL_{\vec{k},R}^{\sub}=\ker(\CL_{\vec{k},R}^{\min} \to j_*\CO_{\CZ_{U,R}})$.  Note that, unlike $\CA_{\vec{k},\vec{m},R}^{\min}$ and $\CL_{\vec{k},R}^{\min}$, formation of $\CA_{\vec{k},\vec{m},R}^{\sub}$ and $\CL_{\vec{k},R}^{\sub}$ does commute with base-change.

\subsection{Hecke operators} \label{ssec:Hecke}
Suppose now that $U$ and $U'$ are sufficiently small open compact subgroups containing 
$\GL_2(\CO_{F,p})$, and that $h \in \GL_2(\A_{F,\f}^{(p)})$ is
such that $U' \subset hUh^{-1}$.  We again have a canonical isomorphism
$\rho_h^*\CA_{\vec{k},\vec{m},R} \stackrel{\sim}{\longrightarrow}
\CA'_{\vec{k},\vec{m},R}$
of line bundles on $\CY_{U',R}$ (where $\CA'_{\vec{k},\vec{m},R}$ is defined as above with $U$ replaced by $U'$), giving rise to an $R$-linear morphism
$$M_{\vec{k},\vec{m}}(U,R) \longrightarrow
  M_{\vec{k},\vec{m}}(U',R).$$
Furthermore (assuming $\vec{k}+2\vec{m}$ is parallel or that
$p^N R = 0$ for some $N > 0$), these morphisms define an action of 
$\GL_2(\A_{F,\f}^{(p)})$ on
$$\CM_{\vec{k},\vec{m}}(R) := \varinjlim_U M_{\vec{k},\vec{m}}(U,R),$$
restricting to one on $\CS_{\vec{k},\vec{m}}(R) := \varinjlim_U S_{\vec{k},\vec{m}}(U,R)$.
If $U$ is such that $\mu^{\vec{k}+2\vec{m}}$ is trivial in $R$
for all $\mu \in U \cap \CO_F^\times$, then we may identify $M_{\vec{k},\vec{m}}(U,R)$ with $(\CM_{\vec{k},\vec{m}}(R))^{U^p}$; more generally for any $U$ containing $\GL_2(\CO_{F,p})$, we define $M_{\vec{k},\vec{m}}(U,R)$ to be $(\CM_{\vec{k},\vec{m}}(R))^{U^p}$, and similarly for $S_{\vec{k},\vec{m}}(U,R)$.  The space
$$M_{\tot}(U,R) := \bigoplus_{\vec{k},\vec{m}\in \ZZ^\Sigma} 
 M_{\vec{k},\vec{m}}(U,R)$$
then forms an $R$-algebra in which 
$S_{\tot}(U,R) := \bigoplus S_{\vec{k},\vec{m}}(U,R)$ is an ideal.

We may similarly define an action of $\GL_2(\A_{F,\f}^{(p)})$
on $\CM_{\vec{k}}(R):= \varinjlim_U M_{\vec{k}}(U,R)$ for any $\vec{k}\in \Z^{\Sigma}$ and $\CO$-algebra $R$.  As above, this restricts to an action on $\CS_{\vec{k}}(R):= \varinjlim_U S_{\vec{k}}(U,R)$, and we recover the previously defined
$M_{\vec{k}}(U,R)$ (and $S_{\vec{k}}(U,R)$) by taking $U^p$-invariants.

For any open compact subgroups $U$, $U'$ of 
$\GL_2(\A_{F,\f})$ containing $\GL_2(\CO_{F,p})$, and 
$h \in \GL_2(\A_{F,\f}^{(p)})$, we define the 
$R$-linear double coset 
operator 
$$[U'hU]: M_{\vec{k},\vec{m}}(U,R) = (\CM_{\vec{k},\vec{m}}(R))^U \longrightarrow 
(\CM_{\vec{k},\vec{m}}(R))^{U'} = M_{\vec{k},\vec{m}}(U',R)$$
as $\sum_{i=1}^n h_i$, where $U'hU = 
\coprod_{i=1}^n h_i U$, and the action of the $h_i$ on
$\CM_{\vec{k},\vec{m}}(R)$ is the one defined above.
(For $U$, $U'$ sufficiently small, this can be written
as a suitably normalized composite of pull-back and trace maps, 
as in \cite[\S4.3]{DS1}, but see Remark~\ref{rmk:conventions}
for a discussion of the difference between the conventions
here and in \cite{DS1}.)
Note that the operator $[U'hU]$ restricts to a map $S_{\vec{k},\vec{m}}(U,R) \longrightarrow 
S_{\vec{k},\vec{m}}(U',R)$.

In particular if $v$ is a prime not dividing $p$ and
$U$ contains $\GL_2(\CO_{F,v})$, then we have the usual Hecke operators
$$T_v = U\smat{\varpi_v}{0}{0}{1}U \quad\mbox{and}\quad
S_v = U\smat{\varpi_v}{0}{0}{\varpi_v}U$$
as endomorphisms of $M_{\vec{k},\vec{m}}(U,R)$ and $M_{\vec{k}}(U,R)$
(where $\varpi_v \in \CO_{F,v}$ is any uniformizer at $v$).
If $U = U_1(\gn)$ or $U(\gn)$, the effect of $T_v$ on Fourier expansions at cusps at $\infty$ is given by \cite[Prop~6.5.1]{compact}, which translated to our conventions becomes the
formula
$$r_m^t(T_vf) = \Nm_{F/\Q}(v) r_m^{\varpi_v t}(f)
  + r_m^{\varpi_v^{-1}t}(S_vf).$$
In the case that
$$ U = U_1(\gn) := 
 \left\{\,\left.\smat{a}{b}{c}{d} \in \GL_2(\widehat{\CO}_F)\,\right|\, 
  \mbox{$c_v \in \gn\CO_{F,v}$ for all $v$}\,\right\}$$
for some ideal $\gn$ of $\CO_F$ prime to $p$,
we also define $T_v$ as above for $v|\gn$.
The effect on Fourier expansions at cusps at $\infty$ is then
given by the formula 
$$r_m^t(T_vf) = \Nm_{F/\Q}(v) r_m^{\varpi_v t}(f).$$

Furthermore if $\gp \in S_p$ and
\begin{equation} \label{eqn:Tpinequality} \sum_{\sigma \in \Sigma_{\gp}} \min(m_\sigma + 1, m_\sigma + k_\sigma) \ge 0,\end{equation}
then we have the endomorphism $T_{\gp}$ of 
$M_{\vec{k},\vec{m}}(U,R)$ defined as in \cite[\S5.4]{FD:KS}
for arbitrary $U$ containing $\GL_2(\CO_{F,p})$.
More precisely, assuming  (\ref{eqn:Tpinequality})
instead of \cite[(54)]{FD:KS}, we may replace
$\psi^*$ by $p^{f_\gp}\psi^*$ in the definition of 
the morphism \cite[(55)]{FD:KS} in order
to conform to the conventions of this paper.
Similarly if $\sum_{\sigma \in \Sigma_{\gp}} (k_\sigma + 2m_\sigma) \ge 0$, then the operator $S_{\gp}$ is given 
by $\varpi_{\gp}^{\vec{k}+2\vec{m}} S_{\varpi_{\gp}}$,
where $p^{-f_{\gp}}\varpi_{\gp}^{\vec{1}-\vec{k}-2\vec{m}}$ 
is replaced by $\varpi_{\gp}^{-\vec{k}-2\vec{m}}$
in the definition of $S_{\varpi_{\gp}}$ in \cite[\S6.7]{compact}.
Translating the formula in \cite[Prop.~6.8.1]{compact}
for the effect of $T_{\gp}$ on $q$-expansions then gives
$$r_m^t(T_{\gp}f)  = \epsilon\varpi_{\gp}^{\vec{m}+ \vec{1}}
 r_{\varpi_{\gp}m}^{x^{-1}t}(f) + 
\varpi_{\gp}^{\vec{k}+\vec{m}}
 r_{\varpi_{\gp}^{-1}m}^{xt}(S_{\varpi_\gp}f),$$
where $\epsilon = p^{f_\gp}\Nm(\varpi_{\gp})^{-1}\in \ZZ_{(p)}^\times$ and $x = \varpi_\gp^{(p)}$.  In particular $T_\gp$ restricts to an endomorphism of $S_{\vec{k},\vec{m}}(U,R)$.

We also have the usual Hecke operator $T_{\gp}$ on 
$M_{\vec{k}}(U,K)$ (assuming also now that 
$\sigma(\varpi_\gp)^{1/2} \in K$ for all $\sigma\in \Sigma$),
whose effect on $q$-expansions is given by the same formula as above, but with $\vec{m}$ replaced by $-\vec{k}/2$
(with $S_{\varpi_\gp}$ being the usual $S_\gp$).  Note in
particular that it does not preserve integrality of $q$-expansions, but that $T_{\varpi_\gp} := 
p^{-f_\gp}\varpi_\gp^{\vec{k}/2}T_\gp$ does, 
provided $\sum_{\sigma\in \Sigma_\gp} (k_\sigma-1) \ge 0$.
Therefore it follows from the $q$-expansion Principle that
(under this assumption on $\vec{k}$) $T_{\varpi_\gp}$ restricts
to an endomorphism of $M_{\vec{k}}(U,\CO)$ (and 
of course $S_{\vec{k}}(U,\CO)$) satisfying
$$r_m^t(T_{\varpi_\gp}f)  = 
 r_{\varpi_{\gp}m}^{x^{-1}t}(f) + 
 p^{-f_\gp}
\varpi_{\gp}^{\vec{k}}
 r_{\varpi_{\gp}^{-1}m}^{xt}(S_{\gp}f).$$

Finally we recall that all the operators defined above commute
whenever defined.
See in particular \cite[Cor.~6.8.2]{compact} for the case of
$T_v$ and $T_{v'}$ when $v$ and $v'$ are in $S_p$.

\subsection{Twisting} \label{ssec:twist}
We now explain how the spaces $M_{\vec{k},\vec{m}}(U,R)$ for varying $\vec{m}$ (and fixed $\vec{k}$) are related. 

Recall from \cite[Prop.~3.2.2]{theta} that the action of 
$\GL_2(\A_{F,\f}^{(p)})$ on $\CM_{\vec{0},\vec{\ell}}(R)$ factors
through the determinant, and as a representation of 
$(\A_{F,\f}^{(p)})^\times$ is isomorphic to the smooth
co-induction of the character $\CO_{F,(p),+}^\times \to R^\times$
defined by $\alpha \mapsto \alpha^{\vec{\ell}}$.

It follows (as in \cite[\S4.6]{DS1}) that if 
$\xi: (\A_{F,\f}^{(p)})^\times \to R^\times$
is a continuous character such that\footnote{Note that such characters always exist with $R = E$ sufficiently large.}
$\xi(\alpha) = \alpha^{\vec{\ell}}$ for all 
$\alpha \in \CO^\times_{F,(p),+}$, then the corresponding eigenspace
$$\{\,f \in \CM_{\vec{0},\vec{\ell}}(R)\,|\,
\mbox{$gf = \xi(\det(g))f$ for all $g \in \GL_2(\A_{F,\f}^{(p)})$}\,\}$$
is free of rank one over $R$.  Multiplication by a basis
element thus defines a $\GL_2(\A_{F,\f}^{(p)})$-equivariant isomorphism
$\CM_{\vec{k},\vec{m}}(R)(\xi\circ\det) 
  \stackrel{\sim}{\longrightarrow} \CM_{\vec{k},\vec{\ell}+\vec{m}}(R)$.
In particular if $\xi$ is
trivial on $\det(U^p)$, then we obtain an isomorphism
\begin{equation} \label{eqn:twist}
e_\xi: M_{\vec{k},\vec{m}}(U,R)
  \stackrel{\sim}{\longrightarrow} M_{\vec{k},\vec{\ell}+\vec{m}}(U,R)\end{equation}
such that $T_v(e_\xi f) = e_\xi \xi(\varpi_v)T_v(f)$ and $S_v(e_\xi f) = e_\xi\xi(\varpi_v)^2S_v(f)$ for all $v\nmid p$ such that $\GL_2(\CO_{F,v}) \subset U$.  This holds for example
if $U = U(\gn)$ and $\xi$ is trivial on $V_{\gn,+}$, in which
case we may choose $e_\xi$ so that its effect on $q$-expansions is given by
the formula
$$ r_m^t(e_\xi f) = \xi(t) r_m^t(f).$$

We may similarly relate the spaces $M_{\vec{k},\vec{m}}(U,R)$ and 
$M_{\vec{k}}(U,R)$. (Recall the implicit assumption that either $\vec{k}+2\vec{m}$ is parallel or that $p^NR = 0$ for some $N > 0$ in order to define $M_{\vec{k},\vec{m}}(U,R)$.) For $U$ sufficiently small that $\mu^{\vec{m}+\vec{k}/2}$ is trivial in $R$, we may write $\CA_{\vec{k},\vec{m},R} = \CL_{\vec{k},R}\otimes \CN$ for a line bundle $\CN = \CN_U$ on $Y_{U,R}$, and define in the usual way an action of $\GL_2(\A_{F,\f}^{(p)})$ on 
$\varinjlim_U H^0(\CY_{U,R},\CN_U)$. The same argument as the proof of \cite[Prop.~3.2.2]{theta} then yields the same description of the representation as an inflation of a smooth 
co-induction, but with $\vec{\ell}$ replaced by 
$\vec{m} + \vec{k}/2$.

Suppose now that
$\xi: (\A_{F,\f}^{(p)})^\times \to R^\times$
is a continuous character
such that
$$\xi(\alpha) = \alpha^{\vec{m}+\vec{k}/2}.$$
for all 
$\alpha \in \CO_{F,(p),+}^\times$.
(Again such characters necessarily exist with $R= E$, possibly after enlarging $E$, but
note that they do not extend to continuous characters of
$\A_F^\times/F^\times F_{\infty,+}^\times$ unless
$\vec{k}\in 2\ZZ^\Sigma$.) 
As above we can twist by a basis $e_\xi$
for the corresponding eigenspace to obtain
a $\GL_2(\A_{F,\f}^{(p)})$-equivariant isomorphism
$\CM_{\vec{k}}(R)(\xi\circ\det) 
  \stackrel{\sim}{\longrightarrow} \CM_{\vec{k},\vec{m}}(R)$.
It follows that for sufficiently small $U$ containing 
$\GL_2(\CO_{F,p})$, we have an isomorphism
\begin{equation} \label{eqn:twist2}
e_\xi: M_{\vec{k}}(U,R)
  \stackrel{\sim}{\longrightarrow} M_{\vec{k},\vec{m}}(U,R)\end{equation}
with the same behavior as above with respect to Hecke operators and $q$-expansions.

\subsection{Totally definite quaternionic forms over $\CO$-algebras}\label{ssec:defQO}
Let $B$ be a totally definite quaternion algebra
unramified at all $\gp \in S_p$.  Recall that
$\CO_B$ denotes a maximal order in $B$, and we have fixed isomorphisms $\CO_{B,v} \cong M_2(\CO_{F,v})$ for all $v\not\in \Sigma^B$.

Let $T$ be an $\CO$-module equipped with a smooth action of an open subgroup $U_p$ of $\CO_{B,p}^\times \cong \GL_2(\CO_{F,p})$.
For $U$ an open compact subgroup of $B_{\AA}^\times$ such that
$u_p \in U_p$ for all $u \in U$, we define
$$M^B(U,T) = 
  \{\,f:B_{\AA}^\times \to T \,|\, \mbox{$f(\gamma x u)
  = u_p^{-1} f(x)$ for all $\gamma\in B^\times$,
$x \in B_{\AA}^\times$,
$u \in B_\infty^\times U$} \,\}.$$
In particular, if $U = U_pU^p$ for some open compact subgroup $U^p$ of $(B_{\f}^{(p)})^\times$, then we may identify 
$M^B(U,T)$ with
$$\{\,f:(B_{\f}^{(p)})^\times \to T \,|\, \mbox{$f(\gamma_{\f}^{(p)} x u)
  = \gamma \cdot f(x)$ for all $\gamma\in B^\times \cap U_p$,
$x \in (B_{\f}^{(p)})^\times$,
$u \in U^p$} \,\}.$$
Taking the limit over such $U$, we have a natural action of 
$(B_{\f}^{(p)})^\times$ on
$$\CM^B(U_p,T) := \varinjlim_U M^B(U,T)$$
identifying $M^B(U,T)$ with $\CM^B(U_p,T)^{U^p}$.  In particular, we have a commuting family of Hecke operators $T_v$ and $S_v$ on $M^B(U,T)$ for all
$v \not\in \Sigma^B \cup S_p$ such that $U$ contains 
$\CO_{B,v}^\times$.  

Note that taking $T = \CO$ with trivial action of $U_p = \CO_{B,p}^\times$, we have a natural action of $B_{\f}^\times$ on $\CM^B(\CO) := \varinjlim_U M^B(U,\CO)$ (where the limit is now over all open compact $U$), from which we recover
$$M^B(U,T) = (\CM^B(\CO)\otimes_{\CO} T)^U
  \quad\mbox{and}\quad 
\CM^B(U_p,T) = (\CM^B(\CO)\otimes_{\CO}T)^{U_p}$$
for arbitrary $U_p$, $T$ and $U$ as above.

Suppose that $\vec{k},\vec{m} \in \ZZ^{\Sigma}$ with $k_\sigma \ge 2$ for all $\sigma \in \Sigma$, and let $R$ be an $\CO$-algebra such that $p^NR = 0$ for some $N > 0$.  Consider the $R$-module
$$D_{\vec{k},\vec{m},R} := \bigotimes_{\sigma\in \Sigma}
\det{}^{m_\sigma+1}\otimes_R \Sym^{k_\sigma - 2} R^2 $$
equipped with the action of
$$\CO_{B,p}^\times \cong \GL_2(\CO_{F,p}) \cong \prod_{\gp \in S_p} \GL_2(\CO_{F,\gp})$$
provided by the identification of $\Sigma$ with the set of embeddings $F \hookrightarrow \Qpbar$.  Thus the action on the factor of the tensor product indexed by $\sigma \in \Sigma_{\gp}$ is defined by the corresponding homomorphism
$\GL_2(\CO_{F,\gp}) \to \GL_2(\CO)$.  We then let
$$M^B_{\vec{k},\vec{m}}(U,R) = M^B(U, D_{\vec{k},\vec{m},R}).$$
Taking the limit over open compact $U$ containing $U_p = \CO_{B,p}^\times$, we have a natural action of $(B_{\f}^{(p)})^\times$ on
$$\CM_{\vec{k},\vec{m}}^B(R) := \varinjlim_U M^B_{\vec{k},\vec{m}}(U,R)$$
identifying $M^B_{\vec{k},\vec{m}}(U,R)$ with $\CM^B_{\vec{k},\vec{m}}(R)^{U^p}$.  %In particular, we have a commuting family of Hecke operators $T_v$ and $S_v$ on $M_{\vec{k},\vec{m}}(U,R)$ for all $v \not\in \Sigma^B \cup S_p$ such that $U$ contains a maximal compact subgroup of $B_v^\times$.  
Suppose, as in \S\ref{ssec:twist}, that $\vec{\ell} \in \Z^{\Sigma}$ and $\xi:(\A_{F,\f}^{(p)})^\times \to R^\times$ is a character such that $\xi(\alpha) = \alpha^{\vec{\ell}}$ for all $\alpha \in \CO_{F,(p),+}^\times$.  Multiplication by the element 
$e_\xi = \xi\circ\det \in \CM^B_{\vec{0},\vec{\ell}}(R)$ then defines a
$(B_{\f}^{(p)})^\times$-equivariant
isomorphism $\CM^B_{\vec{k},\vec{m}}(R)(\xi\circ\det) \cong
\CM^B_{\vec{k},\vec{\ell}+\vec{m}}(R)$, and hence isomorphisms
as in (\ref{eqn:twist}) for sufficiently small $U$.

Suppose now that $U = U_{\gp}U^{\gp}$, where $U_{\gp}$ corresponds to either
$\GL_2(\CO_{F,\gp})$ or
$$U_1(\gp)_{\gp} = \left\{\left.\,\smat{a}{b}{c}{d}\in\GL_2(\CO_{F,\gp})\,\right|\, c-1,d\in \gp\CO_{F,\gp}\,\right\}$$
(and $U^{\gp} \subset \GL_2(\AA_{\f}^{(\gp)})$).
In addition to the operators $T_v$ and $S_v$ as above, we may also define $T_{\gp}$, provided\footnote{In fact we shall only make use of $T_{\gp}$ 
in the case that $\vec{m} = -\vec{1}$.}
$\sum_{\sigma\in \Sigma_{\gp}} (m_\sigma+1) \ge 0$
(or $p$ is invertible in $R$).
Indeed in this case the action of $\CO_{B,\gp}^\times$ on
$D'_{\vec{k},\vec{m},R}$ extends to that of the multiplicative monoid $\CO_{B,\gp}$, and we let
$(T_{\gp}f)(x) = \sum h_i f(x h_i)$,
where the sum is over $h_i$ such that $U_{\gp} hU_{\gp} = \coprod_i h_i U_{\gp}$ and $h$ corresponds to $\smat{\varpi_{\gp}}{0}{0}{1}$ under the chosen isomorphism 
$\CO_{B,{\gp}} = M_2(\CO_{F,\gp})$.  It is straighforward to check that this gives a well-defined endomorphism $T_{\gp}$ of $M^B_{\vec{k},\vec{m}}(U,R)$ commuting with the operators $T_v$ and $S_v$
for $v$ as above (or of $\CM^B_{\vec{k},\vec{m}}(R)$ commuting with the action of $(B_{\f}^{(p)})^\times$).  Note that we may similarly define an operator
$S_{\gp}$ if $U_{\gp} = \CO_{B,\gp}^\times$ and $\sum_{\sigma\in \Sigma_{\gp}} (k_\sigma+ 2m_\sigma) \ge 0$ 
(or $p$ is invertible in $R$).

\begin{remark}\label{rmk:TpH0} In anticipation of consideration of the case of indefinite quaternion algebras $B$, we remark that 
$M^B_{\vec{k},\vec{m}}(U,R)$ may be identified with $H^0(Y_U^B,\mathcal{D}_{\vec{k},\vec{m},R})$, where $\mathcal{D}_{\vec{k},\vec{m},R}$ is the sheaf on the (finite discrete) set $Y_U^B = B^\times\backslash B_{\f}^\times/U$ associated to $D_{\vec{k},\vec{m},R}$.  We can then interpret the Hecke operators defined above in terms of pull-back and trace morphisms relative to 
the action of $B_{\f}^\times$ on the inverse system of sets $Y_U^B$. In particular, if $U = U^{\gp}U_{\gp}$ with $U_{\gp} = \CO_{B,\gp}^\times$ and $U^{\gp}$ sufficiently small, we can identify $T_{\gp}$ can with the composite
$$H^0(Y_U^B, \mathcal{D}_{\vec{k},\vec{m},R})
 \to H^0(Y_{U'}^B, \rho_h^*\mathcal{D}_{\vec{k},\vec{m},R})
 \to H^0(Y_{U'}^B, \rho_1^*\mathcal{D}_{\vec{k},\vec{m},R})
\to H^0(Y_{U}^B,\mathcal{D}_{\vec{k},\vec{m},R}),$$
where $U' = U \cap hUh^{-1}$ (for $h$ as above), the first
map is pull-back (relative to the map $\rho_h: Y_{U'}^B \to Y_U^B$
defined by right-multiplication by $h$),
the second is given by the morphism of sheaves induced by $h$ on
$D_{\vec{k},\vec{m},R}$, and the third is the trace relative to the projection $\rho_1:Y_{U'}^B \to Y_U^B$.
\end{remark}

For any $\vec{k} \in \ZZ_{\ge 2}^\Sigma$ and $\CO$-algebra $R$, we may also consider the action of $\CO_{B,(p)}^\times$ on 
$$D_{\vec{k},R} := \det{}^{\vec{1}-\vec{k}/2}
\bigotimes_{\sigma\in \Sigma}\Sym^{k_\sigma - 2} R^2.$$
For $U = U_pU^p$ with $U_p = \CO_{B,p}^\times$, we define $M_{\vec{k}}^B(U,R)$ to be
$$\{\,f:(B_{\f}^{(p)})^\times \to  D_{\vec{k},R} \,|\, \mbox{$f(\gamma_{\f}^{(p)} x u)
  = \gamma \cdot f(x)$ for all $\gamma\in \CO_{B,(p)}^\times$,
$x \in (B_{\f}^{(p)})^\times$,
$u \in U^p$} \,\}.$$
%We again define to $I_{\vec{k}}^B(U,R)$ to be $0$ unless $\vec{k}=\vec{2}$, $I_{\vec{2}}^B(U,R)$ to the submodule of functions 
%in $M_{\vec{2}}^B(U,R)$ factoring through $\det$, and let
%$S_{\vec{k}}(U,R) = M_{\vec{k}}^B(U,R)/I_{\vec{k}}^B(U,R)$.
We again have a natural action of $(B_{\f}^{(p)})^\times$ on
the direct limit $\CM_{\vec{k}}^B(R)$ over open compact $U$ containing $\CO_{B,p}^\times$, from which we recover 
$M_{\vec{k}}^B(U,R)$ as $\CM_{\vec{k}}^B(R)^{U^p}$ and 
equip it Hecke operators $T_v$ and $S_v$ for
$v \not\in \Sigma^B \cup S_p$ such that $\CO_{B,v}^\times \subset U$. The definitions are consistent with those in \S\ref{ssec:defQC} in the sense that there are identifications
\begin{equation}\label{eqn:scalars}
M_{\vec{k}}^B(U,\CC) = \CC\otimes_{\Qbar}M_{\vec{k}}^B(U,\Qbar)
\quad\mbox{and}\quad M_{\vec{k}}^B(U,\Qpbar) = \Qpbar\otimes_{\Qbar}M_{\vec{k}}^B(U,\Qbar)\end{equation}
compatible with Hecke actions, where $M_{\vec{k}}^B(U,\Qbar)$ is defined as above with $R = \Qbar$.

We may also define an operator $T_{\varpi_\gp}$ on 
$M_{\vec{k}}^B(U,R)$ for $U$ containing $\CO_{B,\gp}^\times$.
We choose $\delta \in B^\times$ such that $\delta_p \in h U_p$
(where $h$ is as in the definition of $T_{\gp}$),
and write $\CO_{B,(p)}^\times \delta \CO_{B,(p)}^\times
 = \coprod_i \delta_i \CO_{B,(p)}^\times$.  Note in particular
that each $\det(\delta_i) \in \varpi_{\gp}\CO_{F,(p)}^\times$,
so that $p^{-f_{\gp}}\varpi_{\gp}^{\vec{k}/2}\delta_{i,p}$ defines
an endomorphism of $D_{\vec{k},R}$, and we let
$$(T_{\varpi_{\gp}}f)(x) = \sum_{i} (p^{-f_{\gp}}\varpi_{\gp}^{\vec{k}/2}\delta_{i,p}) \cdot f((\delta_{i,\f}^{(p)})^{-1}x).$$
Just as for $T_{\gp}$ on $M_{\vec{k},\vec{m}}^B(U,R)$, this gives a well-defined operator $T_{\varpi_{\gp}}$ on $M^B_{\vec{k}}(U,R)$ commuting with the $T_v$ and $S_v$
(or on $\CM^B_{\vec{k}}(R)$ commuting with $(B_{\f}^{(p)})^\times$).  Furthermore, the usual Hecke operator $T_{\gp}$
on $M_{\vec{k}}(U,\C)$ corresponds to $p^{f_{\gp}}\varpi_{\gp}^{-\vec{k}/2}T_{\varpi_{\gp}}$ under the identifications of
(\ref{eqn:scalars}).  (For completeness we note that if $U$ contains $\CO^\times_{B,\gp}$, then
the usual Hecke operator $S_{\gp}$ is well-defined on $M_{\vec{k}}(U,R)$ for any $\CO$-algebra $R$.)

If $p^NR = 0$ for some $N>0$, and $\xi:(\A_{F,\f}^{(p)})^\times \to R^\times$ is such that $\xi(\alpha) = \alpha^{\vec{m}+\vec{k}/2}$ for all $\alpha \in \CO_{F,(p),+}^\times$, we can again twist by $e_\xi = \xi\circ\det$ to obtain
$\CM^B_{\vec{k}}(R)(\xi\circ\det) \cong
\CM^B_{\vec{k},\vec{m}}(R)$, and hence isomorphisms
as in (\ref{eqn:twist2}) for sufficiently small $U$.
In particular we have $T_v(e_\xi f) = e_\xi \xi(\varpi_v) T_v(f)$
and $S_v(e_\xi f) = e_\xi \xi(\varpi_v)^2 S_v(f)$ for $v$ as above; furthermore it is straightforward to check that
\begin{equation} \label{eqn:Tptwist}
T_{\gp}(e_\xi f) = \epsilon \xi(x)^{-1} \varpi_{\gp}^{\vec{m} + \vec{1}} e_\xi T_{\varpi_{\gp}}(f),\end{equation}
where $\epsilon = p^{f_{\gp}} \Nm(\varpi_{\gp})^{-1}$
and $x = \varpi_{\gp}^{(p)}$.

Note that $M^B_{\vec{2}}(U,R) = M^B(U,R)$, as defined above with $T = R$ and $U_p = \CO_{B,p}^\times$ acting trivially.  More generally we let $M_{\vec{2}}^B(U,R) = M^B(U,R)$ for arbitrary open compact $U \subset B_{\f}^\times$.  Note that (\ref{eqn:scalars}) holds in this context as well, and that $M^B_{\vec{2}}(U,R)$ coincides with $M^B_{\vec{2},-\vec{1}}(U,R)$ as defined above if $p^NR = 0$ for some $N > 0$.

\subsection{Cohomology of Shimura curves over $\CO$-algebras}
\label{ssec:indefQO}
Now suppose that $B$ is a quaternion algebra over $F$ which is 
unramified at all $\gp \in S_p$ and at a unique $\sigma_0 \in \Sigma_\infty$.  Recall that we have chosen a maximal order $\CO_B$ and isomorphisms $\CO_{B,v} \cong M_2(\CO_{F,v})$ for all finite $v\not\in \Sigma^B$.  For open compact subgroups $U$ of $B_{\A,\f}^\times$, we consider the compact Riemann surface
$$Y_U^B = B_+^\times \backslash ((\uhp\times B_{\f}^\times)/U)$$
defined in \S\ref{ssec:indefQC}.

Let $T$ be an $\CO$-module with a smooth action of an open subgroup $U_p$ of $\CO_{B,p}^\times$.  For sufficiently small $U$
(and in particular satisfying $u_p \in U_p$ for all $u \in U$), consider the locally constant sheaf $\CT = \CT^B_U$ on $Y_U^B$ defined by (sections of)
$$B_+^\times \backslash ((\uhp\times B_{\f}^\times \times T)/U),$$
where the actions of $B_+^\times$ and $U$ are defined by 
$$\gamma\cdot(z,x,t)u = (\gamma_\infty(z), \gamma_\f x u, u_p^{-1}t).$$
Note that if $U = U_pU^p$ for some open compact subgroup $U^p$ of $(B_{\f}^{(p)})^\times$, then we may rewrite 
$Y_U^B$ as $(B_+^\times\cap U_p) \backslash ((\uhp\times 
(B^{(p)}_{\f})^\times)/U^p)$
and $\CT^B_U$ as 
$$(B_+^\times\cap U_p) \backslash ((\uhp\times (B^{(p)}_{\f})^\times \times T)/U^p),$$
where the actions are defined by $\gamma\cdot(z,x,t)u = (\gamma_\infty(z), \gamma^{(p)}_\f x u, \gamma_p t)$.

Note that for $U$, $U'$ as above and $h \in (B^{(p)}_{\f})^\times$ such that $U'\subset hUh^{-1}$, the identity on $T$ induces an isomorphism $\rho_h^*\CT^B_U \stackrel{\sim}{\to} \CT^B_{U'}$, 
where $\rho_h: Y^B_{U'} \to Y^B_U$ is the covering map induced by right multiplication by $h$.  We thus obtain pull-back and trace morphisms
$$H^i(Y_U^B,\CT_U^B) \longrightarrow
 H^i(Y_{U'}^B,\CT_{U'}^B)\quad\mbox{and}\quad
H^i(Y_{U'}^B,\CT_{U'}^B) \longrightarrow
H^i(Y_{U}^B,\CT_{U}^B)$$
relative to $\rho_h$.  More generally, for any $U$, $U'$ as above and $h \in (B^{(p)}_{\f})^\times$, we define the double coset operator $[U'hU]$ to be the composite
$$H^i(Y_U^B,\CT_U^B) \longrightarrow
 H^i(Y_{U''}^B,\CT_{U''}^B)\longrightarrow
H^i(Y_{U'}^B,\CT_{U'}^B),$$
where $U'' = U' \cap hUh^{-1}$, the first map is the pull-back relative to $\rho_h: Y^B_{U''} \to Y^B_U$, and the second is 
$[U'\cap\CO_F^\times:U''\cap \CO_F^\times]$ times\footnote{Note that this normalization factor is $[U':U'']/\deg(\rho_1)$.} the trace relative to $\rho_1:Y^B_{U''}\to Y^B_{U'}$.  These maps satisfy the usual compatibilities, so in particular we obtain a commuting family of operators $T_v$ and $S_v$ on $H^i(Y^B_U,\CT_U^B)$ for 
all $v\not\in \Sigma^B \cup S_p$ such that $U$ contains $\CO_{B,v}^\times$.  We let\footnote{We avoid considering direct limits over $U$ of the $M^B(U,T)$ since the natural maps $M^B(U,T) \to M^B(U',T)^U$ (for $U'$ normal in $U$) are not necessarily isomorphisms.}
 $M^B(U,T) = H^1(Y_U^B,\CT_U^B)$.

Suppose that $\vec{k},\vec{m}\in\ZZ$ with all $k_\sigma \ge 2$, and $R$ is an $\CO$-algebra such that $p^NR = 0$ for some $N>0$.
For sufficiently small $U$, we let $\mathcal{D}_{\vec{k},\vec{m},R}$ denote the locally constant sheaf on $Y_U^B$ associated to $D_{\vec{k},\vec{m},R}$ (defined as in \S\ref{ssec:defQO}), and
let
$$M^B_{\vec{k},\vec{m}}(U,R) = M^B(U,D_{\vec{k},\vec{m},R})
 = H^1(Y_U^B, \mathcal{D}_{\vec{k},\vec{m},R}).$$
Suppose as usual that 
$\vec{\ell} \in \Z^{\Sigma}$ and $\xi:(\A_{F,\f}^{(p)})^\times \to R^\times$ is a character such that 
$\xi(\alpha) = \alpha^{\vec{\ell}}$ for all $\alpha \in \CO_{F,(p),+}^\times$.
Then $\xi\circ\det$ defines an element $e_\xi \in H^0(Y_U^B, \mathcal{D}_{\vec{0},\vec{\ell},R})$ for sufficiently small $U$
(in particular if $U = U_pU^p$ and $\det(U^p) \subset \ker(\xi)$).
Furthermore if $h \in (B^{(p)}_{\f})^\times$ and $U'\subset hUh^{-1}$, then the pull-back 
$$H^0(Y_U^B, \mathcal{D}_{\vec{k},\vec{m},R}) \longrightarrow
  H^0(Y_{U'}^B, \mathcal{D}_{\vec{k},\vec{m},R})$$
relative to $\rho_h$ sends $e_\xi$ to $\xi(\det(h))e_\xi$.
It follows that the cup product with $e_\xi$ defines an
isomorphism as in (\ref{eqn:twist}) for sufficiently small $U$.

As in the case of definite quaternion algebras, we may define operators $T_{\gp}$ for $\gp \in S_p$ such that
$U = U_{\gp}U^{\gp}$, where $U_{\gp}$ corresponds to either $\GL_2(\CO_{F,\gp})$ or $U_1(\gp)_\gp$ and
$\sum_{\sigma\in \Sigma_{\gp}} (m_\sigma+1) \ge 0$
(or $p$ is invertible in $R$).  Indeed consider the
two covering maps $\rho_1,\rho_h:Y_{U'}^B \to Y_U^B$,
where $h \in \CO_{B,\gp}$ corresponds to 
$\smat{\varpi_{\gp}}{0}{0}{1}$
and $U' = U \cap hUh^{-1}$.  The endomorphism $h$
of $D_{\vec{k},\vec{m},R}$ then
induces a morphism 
$$\rho_h^*\mathcal{D}_{\vec{k},\vec{m},R}
  \longrightarrow \rho_1^*\mathcal{D}_{\vec{k},\vec{m},R}$$
of sheaves on $Y_{U'}^B$ (defined by 
$B_+^\times(z,xh,t)U' \mapsto B_+^\times(z,x,ht)U'$
on the fibre over $B_+^\times(z,x)U'$).
We then define $T_{\gp}$ on 
$H^i(Y_U^B, \mathcal{D}_{\vec{k},\vec{m},R})$ as in Remark~\ref{rmk:TpH0}, i.e., as the composite
$$H^i(Y_U^B, \mathcal{D}_{\vec{k},\vec{m},R})
 \to H^i(Y_{U'}^B, \rho_h^*\mathcal{D}_{\vec{k},\vec{m},R})
 \to H^i(Y_{U'}^B, \rho_1^*\mathcal{D}_{\vec{k},\vec{m},R})
\to H^i(Y_{U}^B,\mathcal{D}_{\vec{k},\vec{m},R}).$$
In particular this defines an endomorphism $T_{\gp}$ of $M^B_{\vec{k},\vec{m}}(U,R)$ commuting with the $T_v$ and $S_v$
(for $v$ as above), and we may similarly define 
$S_{\gp}$ (again assuming $U_\gp = \CO_{B,\gp}^\times$ and 
that $\sum_{\sigma\in \Sigma_{\gp}} (k_\sigma+ 2m_\sigma) \ge 0$
or $p$ is invertible in $R$).

For $\vec{k} \in \ZZ_{\ge 2}^\Sigma$ and $R$ any $\CO$-algebra, we define the locally constant sheaf $\mathcal{D}'_{\vec{k},R}$ on $Y_U^B$ by (sections of) 
$$
\CO_{B,(p),+}^\times \backslash ((\uhp\times (B^{(p)}_{\f})^\times \times D_{\vec{k},R})/U^p),$$
where $U = U_pU^p$, $U_p = \CO_{B,p}^\times$ and $D_{\vec{k},R}$ is defined as in \S\ref{ssec:defQO}.  For such $U$, $U'$ and $h \in (B^{(p)}_{\f})^\times$ satisfying $U'\subset hUh^{-1}$, we again have the isomorphism $\rho_h^*\mathcal{D}_{\vec{k},R} \stackrel{\sim}{\to} \mathcal{D}_{\vec{k},R}$ induced by the identity on $D'_{\vec{k},R}$, and hence pull-back and trace morphisms between $M_{\vec{k}}(U,R)$ and $M_{\vec{k}}(U',R)$.
We can therefore define double coset operators $[U'hU]$ as above (without the assumption that $U'\subset hUh^{-1}$), and hence obtain a commuting family of Hecke operators $T_v$ and $S_v$ on $M_{\vec{k}}(U,R)$ for $v\not\in \Sigma^B \cup S_p$ such that $\CO_{B,v}^\times \subset U$.  The definitions are again consistent with those in  
\S\ref{ssec:indefQC} in the sense that there are identifications
as in (\ref{eqn:scalars}) compatible with Hecke actions.

To define $T_{\varpi_\gp}$ for $\gp \in S_p$ and $U$ containing $\CO_{B,\gp}^\times$, choose $\delta \in B_+^\times$ such that 
$\delta_p \in h U_p$ (for $h$ is as in the definition of $T_{\gp}$).  We can then write 
$$Y_{U'}^B = 
\Gamma\backslash((\uhp \times B_{\f}^{(p)})^\times)/U^p),$$
where $\Gamma = B_+^\times \cap U_p' = \delta\CO^\times_{B,(p),+}\delta^{-1} \cap \CO^\times_{B,(p),+}$, so that $\rho_h:Y_{U'}^B \to Y_U^B$ takes the form $\Gamma(z,x)U^p \mapsto \CO_{B,(p),+}^\times(\delta_\infty^{-1}(z), (\delta_\f^{(p)})^{-1} x)U^p$.
We then have a morphism $\rho_h^*\mathcal{D}_{\vec{k},R} \to \rho_1^*\mathcal{D}_{\vec{k},R}$ of sheaves on $Y_{U'}^B$
given by
$$\Gamma(\delta_\infty^{-1}(z), (\delta_\f^{(p)})^{-1} x,t)U^p
\mapsto \Gamma(z, x, (p^{-f_{\gp}}\varpi_{\gp}^{\vec{k}/2}\delta_p) t)U^p,$$
on the fibre over $\Gamma(z,x)U^p$, and we define $T_{\varpi_{\gp}}$ on $H^i(Y_U^B,\mathcal{D}_{\vec{k},R})$ as the resulting composite
$$H^i(Y_U^B, \mathcal{D}_{\vec{k},R})
 \to H^i(Y_{U'}^B, \rho_h^*\mathcal{D}_{\vec{k},R})
 \to H^i(Y_{U'}^B, \rho_1^*\mathcal{D}_{\vec{k},R})
\to H^i(Y_{U}^B,\mathcal{D}_{\vec{k},R}).$$
In particular we obtain an operator
$T_{\varpi_{\gp}}$ on $M_{\vec{k}}(U,R)$, which it is 
straighforward to check commutes with the $T_v$ and $S_v$,
and corresponds to $p^{-f_{\gp}}\varpi_{\gp}^{\vec{k}/2}T_{\gp}$ under the identifications as in
(\ref{eqn:scalars}). (Again for completeness we note that
if $U$ contains $\CO_{B,\gp}^\times$, then the usual definition 
yields an operator $S_{\gp}$ on $M_{\vec{k}}(U,R)$ for any $\CO$-algebra $R$.)

If $p^NR = 0$ and $\xi:(\A_{F,\f}^{(p)})^\times \to R^\times$ is such that $\xi(\alpha) = \alpha^{\vec{m}+\vec{k}/2}$ for all $\alpha \in \CO_{F,(p),+}^\times$, we again have an isomorphism $e_\xi:\mathcal{D}_{\vec{k},R} \stackrel{\sim}{\to} \mathcal{D}_{\vec{k},\vec{m},R} $ defined by 
$\xi\circ\det$, i.e., 
$$\CO_{B,(p),+}^\times (z,x,t)U^p
 \mapsto \CO_{B,(p),+}^\times (z,x,\xi(\det(x))t)U^p$$
 on the fibre over $\CO_{B,(p),+}^\times (z,x,)U^p$.  This
 gives rise to isomorphisms
as in (\ref{eqn:twist2}) for sufficiently small $U$,
and satisfying the same compatibilities with Hecke
operators, and in particular (\ref{eqn:Tptwist}).

Finally, we let $M^B_{\vec{2}}(U,R) = M^B(U,R) = H^1(Y_U^B,R)$ for any sufficiently small open compact $U$, noting again that 
(\ref{eqn:scalars}) still holds, and that $M^B_{\vec{2}}(U,R)= M^B_{\vec{2},-\vec{1}}(U,R)$ if $p^NR = 0$ for some $N > 0$. 

\section{Hilbert modular forms in characteristic $p$}

\subsection{Partial Hasse invariants} \label{ssec:Hasse}
We first recall the existence of certain partial Hasse invariants
(as constructed by \cite{goren}, \cite{ag} and in this generality by \cite{RX}).  

For each $\sigma \in \Sigma$, we let $\vec{h}_\sigma = \nu_\sigma \vec{e}_{\varphi^{-1}\sigma} - \vec{e}_\sigma$, where
$\nu_\sigma = p$ if $\sigma = \sigma_{\gp,i,1}$ for some $\gp \in S_p$, $i \in \Z/f_\gp\Z$, and $\nu_\sigma = 1$ otherwise.  (In particular $n_\sigma = p$ for all $\sigma$ if $p$ is unramified in $F$.) We then have the partial Hasse invariant
$$H_{\sigma} \in M_{\vec{h}_\sigma,\vec{0}}(U,E).$$
Furthermore, multiplication by $H_\sigma$
defines an injective $\GL_2(\A_{F,\f}^{(p)})$-equivariant
homomorphism
$$M_{\vec{k},\vec{m}}(U,E) \hookrightarrow 
 M_{\vec{k}+\vec{h}_\sigma,\vec{m}}(U,E)$$
for all $\vec{k}$, $\vec{m}$ and $U$ as in \S\ref{sec:HMFO}.
Similarly we have elements 
$$G_{\sigma} \in M_{\vec{0},\vec{h}_\sigma}(U,E)$$
inducing $\GL_2(\A_{F,\f}^{(p)})$-equivariant isomorphisms
$$M_{\vec{k},\vec{m}}(U,E) \stackrel{\sim}{\longrightarrow} 
 M_{\vec{k},\vec{m}+\vec{h}_\sigma}(U,E)$$
(see \cite[\S4.1]{theta}).  Note in particular that $H_\sigma$
and $G_\sigma$ commute with the operators $T_v$ and $S_v$ 
(whenever defined) for $v\nmid p$.  Furthermore the
$q$-expansions $H_\sigma$ and $G_\sigma$ at every cusp
are non-zero constants (see \cite[\S8.1]{theta}).

Recall from \cite[\S8]{cone} that for each non-zero $f \in M_{\vec{k},\vec{m}}(U,E)$ and $\sigma \in \Sigma$, there is a maximal integer $r_\sigma$ such that $f$ is divisible by $H_\sigma^{r_\sigma}$ in the ring $M_{\tot}(U,E)$.  Then $f$ is divisible by $\prod_{\sigma\in \Sigma} H_\sigma^{r_\sigma}$, and 
 we define
$$\vec{k}_{\min}(f) = 
\vec{k} - \sum_{\sigma\in \Sigma} r_\sigma \vec{h}_\sigma.$$
The main result of \cite{cone} then states:
\begin{theorem} \label{thm:cone} If $0 \neq f \in M_{\vec{k},\vec{m}}(U,E)$, then 
$\vec{k}_{\min}(f) \in \Xi_{\min}$, where
$$\Xi_{\min} = \{\,\vec{k}\in \ZZ^{\Sigma}\,|\,
\mbox{$n_\sigma k_\sigma\ge k_{\varphi^{-1}(\sigma)}$
for all $\sigma\in \Sigma$}\,\}.$$
\end{theorem}

Since $\Xi_{\min}$ is contained in the cone spanned (rationally) by the partial Hasse invariants, we obtain the following:
\begin{corollary} \label{cor:cone} If $\vec{k} \not\in \{\sum\limits_{\sigma\in \Sigma} r_\sigma \vec{h}_\sigma|
  \mbox{$r_\sigma \in \QQ_{\ge 0}$ for all $\sigma \in \Sigma$}\}$, then $M_{\vec{k},\vec{m}}(U,E) = 0$.
\end{corollary}

\subsection{Weight-shifting operators} \label{ssec:theta}
We recall also the existence and properties of certain 
weight-shifting operators in characteristic $p$.  
For each element $\tau = \tau_{\gp,i} \in \Sigma_0$, we have a 
partial $\Theta$-operator
$$\Theta_\tau: M_{\vec{k},\vec{m}}(U,E) \longrightarrow
  M_{\vec{k} + \vec{t}_\tau, \vec{m} - \vec{e}_\sigma}(U,E),$$
where $\sigma = \sigma_{\gp,i,e_\gp}$ and 
$\vec{t}_\tau = \vec{h}_\sigma + 2\vec{e}_\sigma$
(see \cite[\S5]{theta}).
Then each $\Theta_\tau$ is $\GL_2(\A_{F,\f}^{(p)})$-equivariant, and hence commutes with the Hecke operators 
$T_v$ and $S_v$ (whenever defined) for $v\nmid p$.
Furthermore the $\Theta_\tau$ are $E$-linear derivations
on $ M_{\tot}(U,E)$
which commute with each other (see \cite[\S8.2]{theta}),
as well as with multiplication
by partial Hasse invariants (since $\Theta_\tau(H_{\sigma'}) = 0$
for all $\tau \in \Sigma_0$, $\sigma'\in \Sigma$).
Their effect on $q$-expansions is given by the formula
$$r_m^t(\Theta_\tau(f)) = \pi_\tau(m) r_m^t(f),$$
where $\pi_\tau$ is an isomorphism $(M/\gp M)\otimes_{\F_{\gp}} E
 \stackrel{\sim}{\longrightarrow} E$ (again depending on a choice of basis; see \cite[(30)]{theta} for a more precise statement).

The partial $\Theta$-operators also have the following property (\cite[Thm.~5.2.1]{theta}):
\begin{theorem} \label{thm:thetadiv} Suppose that $\tau = \tau_{p,i}$, $\sigma = \sigma_{\gp,i,e_{\gp}}$ and $f \in M_{\vec{k},\vec{m}}(U,E)$. Then $\Theta_\tau(f)$ is divisible by $H_\sigma$ if and only if $f$ is divisible by $H_\sigma$ or $k_\sigma$ is divisible by $p$.
\end{theorem}

Finally for each $\gp \in \Sigma_p$, we have the partial Frobenius operator
$$V_\gp: M_{\vec{k},\vec{m}}(U,E) \longrightarrow
  M_{\vec{k}', \vec{m}'}(U,E)$$
defined in \cite[\S6]{theta}, where $\vec{k}' = \vec{k} + 
\sum_{\sigma \in \Sigma_\gp} k_\sigma \vec{h}_\sigma$ and
$\vec{m}' = \vec{m} + 
\sum_{\sigma \in \Sigma_\gp} m_\sigma \vec{h}_\sigma$.  These too are $\GL_2(\A_{F,\f}^{(p)})$-equivariant, and hence commute with the Hecke operators 
$T_v$ and $S_v$ (whenever defined) for $v\nmid p$.
Furthermore they commute with each other, and define
$E$-algebra endomorphisms of $M_{\tot}(U,E)$.
Their effect on $q$-expansions is given by the formula
\begin{equation} \label{eqn:Vonq}
r_m^t(V_\gp(f)) = r_{\varpi_{\gp}^{-1}m}^{xt}(f),\end{equation}
where again $x = \varpi_{\gp}^{(p)}$ (with a more precise
description in \cite[\S8.3]{theta}).
Finally, we have $\Theta_{\tau}\circ V_{\gp} = 0$ for all
$\tau \in \Sigma_{\gp,0}$; in fact the following stronger
result holds (\cite[Cor.~9.1.2]{theta}):
\begin{theorem} \label{thm:kertheta} Suppose that $f \in M_{\vec{k},\vec{m}}(U,E)$
and $\tau \in \Sigma_{\gp,0}$.  Then $\Theta_\tau(f) = 0$
if and only if
$$f = V_\gp(g) \prod_{\sigma\in\Sigma} G_\sigma^{s_\sigma} H_\sigma^{r_\sigma}$$
for some $g \in M_{\tot}(U,E)$, $\vec{r} \in \ZZ_{\ge 0}^{\Sigma}$, $\vec{s} \in \ZZ^{\Sigma}$.
\end{theorem}

We now use an idea of Deo, Dimitrov and Wiese (see \cite[Prop.~1.22]{DDW}) to refine Theorem~\ref{thm:cone}.
\begin{theorem} \label{thm:positivity} If $0 \neq f \in M_{\vec{k},\vec{m}}(U,E)$, 
then $\vec{k}_{\min}(f) \in \Xi_{\min}^{+} \cup \{\vec{0}\}$, where
$$\Xi_{\min}^+ = \{\,\vec{k}\in \ZZ_{>0}^{\Sigma}\,|\,
\mbox{$n_\sigma k_\sigma\ge k_{\varphi^{-1}(\sigma)}$
for all  $\sigma\in \Sigma$}\,\}.$$
\end{theorem}
\begin{proof} Let $\vec{k} = \vec{k}_{\min}(f)$.  Note that we may assume $U = U(\gn)$ for some sufficiently small $\gn$ prime to $p$.  Furthermore since the isomorphism of (\ref{eqn:twist}) preserves divisibility by partial Hasse invariants, we may assume $\vec{m} = \vec{0}$.  
We must prove that if $k_\sigma = 0$ for some $\sigma \in \Sigma$, then $\vec{k} = \vec{0}$.  First note that by Theorem~\ref{thm:cone}, we have
$k_\sigma \ge 0$ for all $\sigma \in \Sigma$, and if $k_\sigma = 0$ for some $\sigma \in \Sigma_\gp$, then $k_\sigma = 0$ for all
$\sigma\in \Sigma_\gp$.  

Suppose then that $k_\sigma = 0$ for all $\sigma \in \Sigma_\gp$.
Note that this implies that $V_{\gp}$ is an endomorphism of
$M_{\vec{k},\vec{0}}(U,E)$. We claim that the $q$-expansion is constant at each cusp at $\infty$.  Suppose this is not the case, and let
$$n(f) = \min \{\,v_\gp(m)\,|\,\mbox{$r_m^t(f) \neq 0$ for 
 some $t \in (\A_{F,\f}^{(p)})^\times$}\,\}.$$
(Note that by (\ref{eqn:qexpinv}), it suffices to consider
one representative $t$ for each cusp, or equivalently each
component of $\CY_U$.) Equation (\ref{eqn:Vonq}) implies
that $n(V_{\gp}(f)) = n(f) + 1$, and hence that $n(V^i_{\gp}(f)) = n(f) + i$ for all $i \ge 0$.  It then follows from the $q$-expansion Principle that the elements $V^i_{\gp}(f)$ (for $i \ge 0$) are linearly independent, contradicting the finite-dimensionality of $M_{\vec{k},\vec{0}}(U,E)$.

We now show that $\vec{k} = \vec{0}$.  Let $N = \prod_{\gp \in S_p} (p^{f_{\gp}} - 1) = |\det(A)|$, where $A$ is the $d\times d$-matrix of coefficients of the weights $\vec{h}_\sigma$ of partial Hasse invariants. Then 
$N\vec{k} = \sum_{\sigma} r_\sigma \vec{h}_\sigma$ for some
$\vec{r} \in \ZZ^{\Sigma}$, where each $r_\sigma \ge 0$ by
Corollary~\ref{cor:cone}. Since the $q$-expansion of $H_\sigma$ at each cusp is a non-zero constant, it follows that for each connected component $\CY_t$ of $\CY_{U,E}$, there is a constant $c_t$ such that 
$$f^N|_{\CY_t} = c_t \prod_{\sigma\in \Sigma} H^{r_\sigma}_\sigma|_{\CY_t},$$
and hence $f^N = e \prod_{\sigma\in \Sigma} H^{r_\sigma}_\sigma$ for some $e \in H^0(\CY_{U,E},\CO_{\CY_{U,E}})$.
In particular, if $r_\sigma > 0$ for some $\sigma \in \Sigma$, then $f^N$ is divisible by $H_\sigma$.  Since the vanishing locus of $H_\sigma$ is reduced, it follows that $H_\sigma|f$, contradicting our assumption that $\vec{k} = \vec{k}_{\min}(f)$. Therefore $\vec{r} = \vec{0}$, so $\vec{k} = \vec{0}$.
\end{proof}

\begin{corollary} \label{cor:positivity} 
If $0 \neq f \in S_{\vec{k},\vec{m}}(U,E)$, 
then $\vec{k}_{\min}(f) \in \Xi^+_{\min}$.
\end{corollary}

\subsection{Ampleness} \label{ssec:ample}
We assume that $U$ sufficiently small, in particular ensuring that the coherent sheaves $\CA^{\min}_{\vec{k},\vec{m},E}$ are in fact line bundles.  We let $\vec{\delta} \in \ZZ^{\Sigma}$ be defined by $\delta_{\sigma_{\gp,i,j}} = j-1$. (Note that $\vec{\delta} = \vec{0}$ if $p$ is unramified in $F$.)

We record the following slight variant of \cite[Lemma~4.5]{RX}.
\begin{lemma} \label{lem:ample} The line bundles $\CA^{\min}_{\vec{\ell}+\vec{\delta},\vec{m},E}$ are ample on $Y_E^{\min}$ for sufficiently large $\ell \in \ZZ$.
\end{lemma}
\begin{proof} Consider the morphism 
$\pi:Y_E^{\min} \to Y_{-,E}^{\min}$ defined by forgetting the Pappas--Rapoport filtrations), where  
$Y_{-,E}^{\min}$ is the minimal compactification of the 
Deligne--Pappas model (see \cite[\S2.10]{compact}), so that
$\CA^{\min}_{\vec{1},\vec{0},E} = \pi^*\omega$ for an ample
line bundle $\omega$ on $Y_{-,E}^{\min}$.

Furthermore recall that $\CA^{\min}_{\vec{\delta},\vec{0},E}$ is relatively ample (with respect to $\pi$).  Indeed $\pi$
is an isomorphism in a neighborhood of the cusps (more precisely, the ordinary locus), so we can replace $\pi$ by $Y_E \to Y_{E,-}$, and hence by the morphism
$\widetilde{\pi}: \widetilde{Y}_E \to \widetilde{Y}_{E,-}$,
and $\CA^{\min}_{\vec{\delta},\vec{0},E}$ by 
$\widetilde{\CA}_{\vec{\delta},\vec{0},E}$.
The very definition of $\widetilde{Y}_E$ factors 
$\widetilde{\pi}$ through a closed immersion
$$ \widetilde{Y}_E \hookrightarrow \bigtimes_{\tau \in \Sigma_0} 
\mathrm{Flag}_{\widetilde{Y}_{E,-}} (s_*\Omega_{A/\widetilde{Y}_{E,-}})_{\tau},$$
where the product is fibred over $\widetilde{Y}_{E,-}$
and the factors are the (full) flag varieties associated
to the vector bundles
$\CF^{(e_\gp)}_{\tau_{\gp,i}} = (s_*\Omega_{A/\widetilde{Y}_{E,-}})_{\tau_{\gp,i}}$ (where
$s:A\to\widetilde{Y}_{E,-}$ is the universal abelian scheme).
The relative ampleness then follows from the fact that the closed immersion identifies the pull-back of $\CO(1)$ under the usual projective embedding of the flag variety with 
$$\bigotimes_{j=1}^{e_\gp} \left(\bigwedge{\mathstrut}^{\!\!\!e_\gp - j}\left(\CF_{\tau_{\gp,i}}^{(e_\gp)}/\CF_{\tau_{\gp,i}}^{(j)}\right) \right) \cong \bigotimes_{j=1}^{e_\gp}\CL_{\sigma_{\gp,i,j}}^{j-1}.$$

It follows from \cite[Lemma~0892]{stax} that $\CA^{\min}_{\vec{\ell}+\vec{\delta},\vec{0},E}$ is ample for sufficiently large $\ell$. Finally since $\CA^{\min}_{\vec{0},\vec{m},E}$ is a torsion line bundle on $Y_E^{\min}$, the ampleness of $\CA^{\min}_{\vec{\ell}+\vec{\delta},\vec{m},E}$ follows from
\cite[Lemma~0890]{stax}.
\end{proof}

\section{Associated Galois representations} \label{sec:Galois}
\subsection{Associated Galois representations in characteristic zero} \label{ssec:GalChar0}
Recall that $\mathcal{C}_{\vec{k}}$ is the set of cuspidal automorphic representations of $\GL_2(\AA_F)$ whose local factor
is $D'_{k_{\sigma}}$ at each archimedean place $\sigma \in \Sigma$.
Thus if $\pi\in \mathcal{C}_{\vec{k}}$, then it is
$C$-algebraic (in the terminology
of \cite{BG}) if and only if all the $k_\sigma$ are even,
and $L$-algebraic if and only if all the $k_\sigma$ are odd.
In this case, through the work of many people, we can associate Galois representations  to $\pi$ compatible with its local factors $\pi_v$  under the local Langlands correspondence $\rec_{F_v}$.  More precisely, by the results of many authors (in particular, \cite{carayol}, \cite{BlasRog}, \cite{RogTun}, \cite{Taylor}, \cite{Jarvis}, \cite{Saito} and \cite{BLGGT2}), we have:

\begin{theorem}\label{thm:galois0}
Suppose that $w \in \ZZ$ is such that $w \equiv k_\sigma \bmod 2$
for all $\sigma \in \Sigma$, and let $\pi = \otimes'_v \pi_v$ be an irreducible subrepresentation of $|\det|^{w/2}\CS_{\vec{k}}$.  Then there exists
a unique (up to isomorphism) irreducible representation
$\rho_\pi: G_F \to \GL_2(\Qpbar)$
such that 
\begin{equation} \label{eqn:lgc}
\WD(\rho_\pi|_{G_{F_v}})^{\mathrm{F-ss}} \cong 
\rec_{F_v}(|\det|^{-1/2}\pi_v)\end{equation}
for all but finitely many non-archimedean places $v$
of $F$.  Furthermore if $k_\sigma \ge 2$ for all $\sigma \in \Sigma$, then (\ref{eqn:lgc}) holds
for all\footnote{See \S\ref{subsec:cryslift} below for definitions and conventions of notions from $p$-adic Hodge theory. We remark also that local-global compatibility is even known under mild technical hypotheses without the assumption that all $k_\sigma \ge 2$; see especially \cite[Thm.~1.4]{JN2015} for $v\nmid p$ and \cite[Thm.~6.11.2]{BP} for $v|p$.} $v$, in which case $\rho_\pi|_{G_{F_v}}$ is de Rham with
$\sigma$-labelled Hodge--Tate weights $(\frac{w+k_\sigma-2}{2},
\frac{w-k_\sigma}{2})$ for all $v|p$.
\end{theorem}

Recall that for all but finitely many non-archimedean places
$v$ of $F$, the representation $\pi_v$ is unramified
i.e., $\pi_v^{U_v}$ is non-trivial, and hence
one-dimensional, for $U_v = \GL_2(\CO_{F,v})$.
For such a $v$, let $a_v$ and $d_v$ denote the eigenvalues
of the double-coset operators $T_v$ and $S_v$
on this one-dimensional space (where $T_v$ and $S_v$
are defined as in \S\ref{ssec:Hecke}).  In this case,
(\ref{eqn:lgc}) means that if $v\nmid p$, then
$\rho_\pi$ is unramified at $v$,
and that the image of a geometric Frobenius element at $v$
has characteristic
polynomial $X^2 - a_v X + \Nm_{F/\QQ}(v)d_v$.
On the other hand if $v|p$ (and $\pi_v$ is unramified), then $\rho_\pi|_{G_{F_v}}$ is crystalline and this is the characteristic polynomial of 
$\phi^{f_v}$ on $D_{\crys,F_v}(\rho_\pi|_{G_{F_v}})$
(where $\Nm_{F/\QQ}(v) = p^{f_v}$).
Note also that with our normalizations, the determinant of
$\rho_\pi$ is the product of a finite order character with
$\chi_{\mathrm{cyc}}^{w-1}$. 

\subsection{Associated Galois representations in characteristic $p$} \label{ssec:GalChar0}

We now prove an analogue of Theorem~\ref{thm:galois0} for mod $p$ Hilbert modular eigenforms.  This generalizes \cite[Thm.~6.1.1]{DS1}, which proves such a result under the hypothesis that $p$ is unramified in $F$.  Under an additional parity hypothesis on the weight, this was previously proved independently by Emerton, Reduzzi and Xiao (\cite{ERX}) and Goldring and Koskivirta (\cite{GK}), the latter as a special case of much more general results.  The parity hypothesis was removed in \cite{DS1} using congruences to forms of higher weight and level not necessarily prime to $p$.

Allowing $p$ to ramify in $F$, the approach of \cite{ERX} was generalized by Reduzzi and Xiao (\cite{RX}), and recent work of Shen and Zheng (\cite{SZ}) does the same for that of \cite{GK}, but these results still require a parity hypothesis on the weight.  Along with some additional work, the parity hypothesis can be removed by combining the methods of \cite{DS1} and \cite{RX}, but  we have decided to give a different argument, which is inspired by remarks of David Loeffler and George Boxer.  Again the proof ultimately relies on congruences to forms with level involving primes over $p$, but instead of producing them using the geometry of Hilbert modular varieties with (mildly) bad reduction, we lift (twists) to non-algebraic automorphic forms and apply the Jacquet--Langlands correspondence to transfer the problem to a more amenable setting.

\begin{theorem} \label{thm:galois}  Suppose that $U$ is an open compact subgroup of $\mathrm{GL}_2(\AA_{F,\f})$
containing $\mathrm{GL}_2(\CO_{F,p})$, and $Q$ is a finite set of primes containing all $v|p$ and
all $v$ such that $\mathrm{GL}_2(\CO_{F,v}) \not\subset U$.  Suppose that $\vec{k},\vec{m}\in \mathbb{Z}^{\Sigma}$ and
that $f \in M_{\vec{k},\vec{m}}(U,E)$ is an eigenform for
$T_v$ and $S_v$ (defined in \S\ref{ssec:Hecke}) for all $v \not\in Q$.   Then there is a Galois representation
$$\rho_f : G_F \to \mathrm{GL}_2(E)$$
such that if $v\not\in Q$, then $\rho_f$ is unramified at $v$ and the characteristic
polynomial of $\rho_f(\mathrm{Frob}_v)$ is 
$$X^2 - a_v  X + \mathrm{Nm}_{F/\mathbb{Q}}(v)d_v,$$
where $T_v f = a_v f$ and $S_v f = d_v f$.
\end{theorem}
\begin{remark} \label{rmk:galconv} Recall from Remark~\ref{rmk:conventions} that the Hecke action defined in \cite{DS1} differs by a twist from the one defined here.  More precisely, if $f$ is an eigenform for the operator $T_v$ (resp.~$S_v$) of this paper with eigenvalue $a_v$ (resp.~$d_v$), then it is so for the operators defined in \cite{DS1}, but with eigenvalue $\Nm_{F/\QQ}(v)^{-1}a_v$ (resp.~$\Nm_{F/\QQ}(v)^{-2}d_v$).  Therefore the Galois representation associated to $f$ in \cite{DS1} is the twist by the mod $p$ cyclotomic character of the one in the preceding theorem.
\end{remark}
\begin{proof} First we note that the same arguments as in the start of the proof of \cite[Thm.~6.1.1]{DS1} apply to show that we can replace $E$ by any finite extension; furthermore we may assume that $U = U(\gn)$ for some sufficiently small $\gn$ prime to $p$ and that $Q$ is the set of primes dividing $\gn p$.  In particular, we choose $\gn$ so that if $\alpha \in \CO_F^\times$ is such that $(\alpha - 1)^2 \in \gn$, then $\alpha \equiv 1 \bmod \gp$ for all $\gp|p$.

Let $\TT = \TT^Q_{\vec{k},\vec{m}}(U,E)$ denote the 
$E$-subalgebra of
$\mathrm{End}_{\Fpbar}(M_{\vec{k},\vec{m}}(U,\Fpbar))$ generated by the operators
$T_v$ and $S_v$ for $v\not\in Q$. Restriction to $E\cdot f$ defines an $E$-algebra homomorphism $\TT \onto E$, and hence a maximal ideal $\gm_f \subset \TT$; moreover $\gm_f$ is generated by the elements
$\{\,T_v - a_v,\,S_v - d_v\,|\, v\notin Q\,\}$.
Conversely, every maximal ideal $\gm$ of $\TT$ is the kernel of an $E$-algebra homomorphism $\theta:\TT \to \Fpbar$; replacing $E$ by a finite extension, we may assume it contains the image of 
$\theta$, and hence that $\gm$ is generated by the elements $T_v - \theta(T_v)$, $S_v - \theta(S_v)$ for $v\not\in Q$.
As $\gm$ is in the support of $M_{\vec{k},\vec{m}}(U,E))$, it follows that the eigenspace
$$\begin{array}{rcl} M_{\vec{k},\vec{m}}(U,E)[\gm]
 &=& \{\,g\in M_{\vec{k},\vec{m}}(U,E)\,|\,\mbox{$Tg = 0$ for all
$T \in \gm$}\,\}\\
 &=& \{\,g\in M_{\vec{k},\vec{m}}(U,E)\,|\,\mbox{$T_vg = a_vg$,
 $S_vg = d_vg$ for all
$v\not\in Q$}\,\}\end{array}$$
is non-zero, i.e., that $\gm = \gm_f$ for some $f$ as in the statement of the theorem.

Recall that, under our assumption that $U$ is sufficiently small, $\CA_{\vec{k},\vec{m},E}^{\min} = i_*\CA_{\vec{k},\vec{m},E}$ is a line bundle on $\CY_{U,E}^{\min}$, in which case $\CA_{\vec{k},\vec{m},E}^{\sub} = j_*j^* \CA_{\vec{k},\vec{m},E}^{\min}$, where
$i$ is the open immersion $\CY_{U,E} \hookrightarrow \CY_{U,E}^{\min}$, and $j$ is the closed immersion $\CZ_{U,E} \hookrightarrow \CY_{U,E}^{\min}$.  We thus have an exact
sequence
\begin{equation} \label{eqn:cusp2modular}
0 \longrightarrow S_{\vec{k},\vec{m}}(U,E)
 \longrightarrow M_{\vec{k},\vec{m}}(U,E)
 \longrightarrow C_{\vec{k},\vec{m}}(U,E),
\end{equation}
where $C_{\vec{k},\vec{m}}(U,E) = H^0(\CZ_{U,E},j^* \CA_{\vec{k},\vec{m},E}^{\min})$.  
Just as for $\CM_{\vec{k},\vec{m}}(U,E)$, we obtain an
action of $\GL_2(\A_{F,\f}^{(p)})$ on 
$$C_{\vec{k},\vec{m}}(E) = \varinjlim_U C_{\vec{k},\vec{m}}(U,E),$$
and hence Hecke operators $T_v$ and $S_v$ on 
$C_{\vec{k},\vec{m}}(U,E) = 
C_{\vec{k},\vec{m}}(E)^{U^p}$ for all $v \not\in Q$.
Furthermore these operators commute (for varying $v$) and
are compatible with their action on 
$M_{\vec{k},\vec{m}}(U,E)$.  
We may therefore assume that $\gm$ is a maximal ideal of
either $\TT^0$ or $\TT^c$, where these (respectively)
denote the $E$-subalgebras of the 
endomorphism rings of $S_{\vec{k},\vec{m}}(U,E)$
and $C_{\vec{k},\vec{m}}(U,E)$ generated by
the $T_v$ and $S_v$ for all $v\not\in Q$.

We first treat the case of $\TT^c$ by explicitly describing
the action of $\GL_2(\A_{F,\f}^{(p)})$ on 
$C_{\vec{k},\vec{m}}(\Fpbar) = 
\Fpbar\otimes_E C_{\vec{k},\vec{m}}(E)$.
Recall that 
$$\CZ_{U,\Fpbar} = B(F)_+\backslash \GL_2(\A_{F,\f})/U
 = B(\CO_{F,(p)})_+ \backslash \GL_2(\A_{F,\f}^{(p)})/U^p
$$ (viewed as a scheme over $\Fpbar$).  Furthermore the construction of the minimal compactification provides a trivialization of $j^* \CA_{\vec{k},\vec{m},\Fpbar}^{\min}$, identifying it with the coherent sheaf defined by (sections of)
$$B(\CO_{F,(p)})_+\backslash (\Fpbar \times  \GL_2(\A_{F,\f}^{(p)})/U^p),$$
where $B(\CO_{F,(p)})_+$ acts on $\Fpbar$ via the character
$\xi_{\vec{k},\vec{m}}:\smat{\alpha}{\beta}{0}{\delta}\mapsto {\alpha}^{\vec{m}}{\delta}^{\vec{k}+\vec{m}}$. This identification is compatible with the action of $\GL_2(\A_{F,\f}^{(p)})$, so that $C_{\vec{k},\vec{m}}(\Fpbar)$ is isomorphic to $\Ind_{B(\CO_{F,(p)})_+}^{\GL_2(\A_{F,\f}^{(p)})}\xi_{\vec{k},\vec{m}}$, which we define to be the space of smooth
functions $\phi: \GL_2(\A_{F,\f}^{(p)})\to \Fpbar$ such that
$$\mbox{$\phi(\gamma g u) = \xi_{\vec{k},\vec{m}}(\gamma)\phi(g)$
 for all $\gamma\in B(\CO_{F,(p)})_+$, 
$g\in \GL_2(\A_{F,\f}^{(p)})$}$$
(where by {\em smooth}, we mean that $\phi(gu) = \phi(g)$ for 
all $g \in \GL_2(\A_{F,\f}^{(p)})$ and 
$u$ in some open compact subgroup $U'$ of $\GL_2(\A_{F,\f}^{(p)})$). Letting $[\xi_{\vec{k},\vec{m}}]: B(\CO_{F,(p)})_+ \to
W(\Fpbar)^\times$ denote the Teichm\"uller lift of $\xi_{\vec{k},\vec{m}}$, we thus have $C_{\vec{k},\vec{m}}(\Fpbar) = \Fpbar\otimes_{W}\widetilde{C}$, where $\widetilde{C} = \Ind_{B(\CO_{F,(p)})_+}^{\GL_2(\A_{F,\f}^{(p)})}[\xi_{\vec{k},\vec{m}}]$ and $W = W(\Fpbar)$.  Furthermore our hypothesis that $U$ is sufficiently small implies that
$C_{\vec{k},\vec{m}}(U,\Fpbar) = C_{\vec{k},\vec{m}}(\Fpbar)^{U^p} = \Fpbar\otimes_{W}\widetilde{C}^{U^p}$.  This identification in turn yields a surjective homomorphism $\widetilde{\TT}^c \to \TT^c$, where $\widetilde{\TT}^c$ is the $W$-algebra of
endomorphisms of $\widetilde{C}^{U^p}$ generated by
the $T_v$ and $S_v$ for $v\not\in Q$.  As $\widetilde{\TT}^c$ is free of finite rank over $W$, it follows that there is a $W$-algebra homomorphism $\widetilde{\theta}:\widetilde{\TT}^c \to \Zpbar$ lifting $\theta$, and hence an eigenvector
$\widetilde{f} \in (\Qpbar\otimes_W\widetilde{C})^{U^p}$ such that $T_v \widetilde{f} = \widetilde{a}_v \widetilde{f}$ and
$S_v \widetilde{f} = \widetilde{d}_v \widetilde{f}$ for all $v \not\in Q$, where the $\widetilde{a}_v$ and $\widetilde{d}_v$ 
are lifts of the $a_v$ and $d_v$ to $\Zpbar$.

We now determine the possible homomorphisms $\widetilde{\theta}$ by analyzing the smooth admissible representation
$$\Qpbar\otimes_W \widetilde{C} =
\Ind_{B(\A_{F,\f}^{(p)})}^{\GL_2(\A_{F,\f}^{(p)})}
\Ind_{B(\CO_{F,(p)})_+}^{B(\A_{F,\f}^{(p)})}
[\xi_{\vec{k},\vec{m}}],$$
where $[\xi_{\vec{k},\vec{m}}]$ is now viewed as taking values in $\Qpbar$.  Firstly the density of $\CO_{F,(p)}$ in $\AA_{F,\f}^{(p)}$ implies that
$\Ind_{B(\CO_{F,(p)})_+}^{B(\A_{F,\f}^{(p)})}
[\xi_{\vec{k},\vec{m}}]$ decomposes as the direct sum of all characters $\xi: B(\A_{F,\f}^{(p)})\to \Qpbar^\times$
of the form
$\smat{a}{b}{0}{d}
\mapsto \xi_1(a)\xi_2(d)$,
where $\xi_1$ and $\xi_2$ are smooth characters $(\A_{F,\f}^{(p)})^\times \to \Qpbar^\times$ such that 
$$\xi_1(\alpha) = [\overline{\alpha}^{\vec{m}}]
\,\,\,\mbox{for all}\,\,\,\alpha\in \CO^\times_{F,(p),+},
\quad\mbox{and}
  \quad \xi_1\xi_2(\delta) = [\overline{\delta}^{\vec{k}+2\vec{m}}]
 \,\,\,\mbox{for all}\,\,\,
 \delta\in \CO_{F,(p)}^\times.$$
It follows that $\Qpbar\otimes_W \widetilde{C}$ decomposes as
the direct sum over such $\xi$ of the representations
$$\Ind_{B(\A_{F,\f}^{(p)})}^{\GL_2(\A_{F,\f}^{(p)})} \xi
  \cong \bigotimes_{v \nmid p}{\!}^{^{\displaystyle{\prime}}}
\Ind_{B(F_v)}^{\GL_2(F_v)} \xi_v,$$
where $\otimes'$ denotes the restricted tensor product and
$\xi_v$ the restriction of $\xi$ to $B(F_v)$.  Finally if
$(\Ind_{B(F_v)}^{\GL_2(F_v)} \xi_v)^{\GL_2(\CO_{F,v})}\neq 0$,
then it is one-dimensional, with $T_v$ and $S_v$ acting
by multiplication by 
$$\widetilde{a}_v = \Nm_{F/\QQ}(v)\xi_1(\varpi_v) + \xi_2(\varpi_v)
\quad\mbox{and}\quad \widetilde{d}_v = \xi_1(\varpi_v)\xi_2(\varpi_v),$$
respectively, and the homomorphisms $\widetilde{\theta}:\widetilde{\TT}^c \to \Qpbar$ are precisely those defined by 
$T_v \mapsto  \widetilde{a}_v$, $S_v \mapsto \widetilde{d}_v$ for pairs of characters $(\xi_1,\xi_2)$ as above.

The analogous statement describing the possible homomorphisms
$\theta:\TT^c \to \Fpbar$ now follows.  In particular we conclude that 
$$a_v = \Nm_{F/\QQ}(v)\overline{\xi}_1(\varpi_v) + 
\overline{\xi}_2(\varpi_v)\quad\mbox{and}\quad
d_v = \overline{\xi}_1(\varpi_v)\overline{\xi}_2(\varpi_v),$$
where $\overline{\xi}_1$ and $\overline{\xi}_2$ are smooth characters $(\A_{F,\f}^{(p)})^\times \to \Fpbar^\times$ such that 
$$\overline{\xi}_1(\alpha) = \overline{\alpha}^{\vec{m}}
\,\,\,\mbox{for all}\,\,\,\alpha\in \CO^\times_{F,(p),+},
\quad\mbox{and}
  \quad \overline{\xi}_1\overline{\xi}_2(\delta) = \overline{\delta}^{\vec{k}+2\vec{m}}
 \,\,\,\mbox{for all}\,\,\,
 \delta\in \CO_{F,(p)}^\times.$$
The desired representation $\rho_f$ is therefore given by $\chi_{\mathrm{cyc}}^{-1}\chi_1 \oplus \chi_2$, where
$\chi_i$ corresponds via class field theory to the unique extension of $\xi_i$ to a character $\A_F^\times \to \Fpbar^\times$ trivial on $F_{\infty,+}^\times F^\times$ (and necessarily also the maximal pro-$p$ subgroup of $\CO_{F,p}^\times$).

Suppose now that $\gm$ is a maximal ideal of $\TT^0$, or equivalently that it is associated to an eigenform $f \in S_{\vec{k},\vec{m}}(U,E) = H^0(Y_{U,E}^{\min},\CA_{\vec{k},\vec{m},E}^{\sub})$.  By Lemma~\ref{lem:ample}, we may choose an integer $\ell \ge 0$ such that $\CA^{\min}_{\vec{\ell}+\vec{\delta},\vec{0},E}$ is ample.  We thus have
$$H^1(Y_{U,E}^{\min},\CA_{\vec{k}+N(\vec{\ell}+\vec{\delta}),\vec{m},E}^{\sub}) = 0$$
for sufficiently large integers $N$.  

Since $\vec{\ell}+\vec{\delta} \in \Xi_{\min} \subset \Xi_{\mathrm{Hasse}}$, some positive multiple is the weight of a product of partial Hasse invariants.  More explicitly, we have $(p-1)(\vec{\ell}+\vec{\delta}) = \sum_{\sigma \in \Sigma} r_\sigma \vec{h}_\sigma$, where
$$r_{\sigma_{\gp,i,j}} = \ell(e_\gp + (j-1)(p-1))
  +(e_\gp(e_{\gp}-1) + (j-1)(j-2)(p-1))/2.$$
Thus if $N = M(p-1)$, then multiplication by 
$\prod_\sigma H_\sigma^{M r_\sigma}$ defines a Hecke-equivariant injective map
$$S_{\vec{k},\vec{m}}(U,E) 
  \hookrightarrow S_{\vec{k}+ N(\vec{\ell}+\vec{\delta}),\vec{m}}(U,E).$$
We may therefore replace $\vec{k}$ by $\vec{k}+ N(\vec{\ell}+\vec{\delta})$ for some sufficiently large $N$ divisible by $p-1$ so as to assume $H^1(\CY_{U,E}^{\min},\CA_{\vec{k},\vec{m},E}^{\sub}) = 0$.

Since the character $(\AA_{F,\f}^{(p)})^\times \to E^\times$ defined by $\alpha \mapsto \overline{\alpha}^{\vec{m}+\vec{k}/2}$ has kernel containing $\CO_{F,(p),+}^\times \cap U^p$, and $\CO_{F,(p),+}^\times U^p$ has finite index in $(\AA_{F,\f}^{(p)})^\times$, we may extend it to a character $\xi:\AA_{F,\f}^{(p)} \to E^\times$
(after enlarging $E$ and hence $K$ if necessary).
Multiplication by the section $e_\xi \in H^0(\CY_{U,E},\CN)$ (defined in \S\ref{ssec:twist}) yields an isomorphism of sheaves 
$\CA_{\vec{k},\vec{m},E}^{\sub} \stackrel{\sim}{\longrightarrow}
\CA_{\vec{k},E}^{\sub}$, so we have that $H^1(\CY_{U,E}^{\min},\CA_{\vec{k},E}^{\sub}) = 0$.  Since the closed immersion $t:\CY_{U,E}^{\min} \hookrightarrow \CY_U^{\min}$ induces an isomorphism\footnote{Note that this is false with $\sub$ replaced by $\min$ unless $\vec{k}$ is parallel.} $t^*\CA_{\vec{k}}^{\sub} \stackrel{\sim}{\longrightarrow} \CA_{\vec{k},E}^{\sub}$, we have a short exact sequence
$$ 0 \longrightarrow \CA_{\vec{k}}^{\sub} 
 \stackrel{\varpi\cdot}{\longrightarrow}
\CA_{\vec{k}}^{\sub} \longrightarrow
t_* \CA_{\vec{k},E}^{\sub} \longrightarrow 0,$$
so the vanishing of $H^1(\CY_{U,E}^{\min},\CA_{\vec{k},E}^{\sub})$
implies that of the finitely generated $\CO$-module 
$H^1(\CY_{U}^{\min},\CA_{\vec{k}}^{\sub})$, and hence the surjectivity of the reduction map
$$S_{\vec{k}}(U,\CO) = H^0(\CY_{U}^{\min},\CA_{\vec{k}}^{\sub})
 \to H^0(\CY_{U,E}^{\min},\CA_{\vec{k},E}^{\sub}) 
 = S_{\vec{k}}(U,E).$$
Composition with the isomorphism $e_\xi$ of (\ref{eqn:twist2}) thus yields a surjection 
$$\pi:S_{\vec{k}}(U,\CO)\to
  S_{\vec{k},\vec{m}}(U,E),$$
such that 
$$T_v(\pi(h)) = \xi(\varpi_v)\pi(T_v(h))
 \quad\mbox{and}\quad
S_v(\pi(h)) = \xi(\varpi_v)^{2}\pi(S_v(h))$$
for all $h \in S_{\vec{k}}(U,\CO)$ and $v\not\in Q$.

Letting $\widetilde{\TT}^0$ denote the finite flat
$\CO$-algebra of endomorphisms of $S_{\vec{k}}(U,\CO)$ 
generated by the $T_v$ and $S_v$ for $v\not\in Q$, we obtain
a surjection 
$\widetilde{\TT}^0 \to \TT^0$ sending $T_v$ to $\xi(\varpi_v)T_v$ and $S_v$ to $\xi(\varpi_v)^2S_v$ for all $v\not\in Q$.
The same argument as for $\widetilde{\TT}^c$ thus yields an eigenform $\widetilde{f}' \in S_{\vec{k}}(U,\Qpbar)$ such that $T_v \widetilde{f}' = \widetilde{a}'_v \widetilde{f}$ and
$S_v \widetilde{f}' = \widetilde{d}'_v \widetilde{f}$ for all $v \not\in Q$, where the $\widetilde{a}'_v$ and $\widetilde{d}'_v$ 
are lifts of the $\xi(\varpi_v)^{-1}a_v$ and $\xi(\varpi_v)^{-2}d_v$ to $\Zpbar$.
Viewing the $\widetilde{a}'_v$ and $\widetilde{d}'_v$ as elements of $\CC$ via the chosen embedding $\Qbar \hookrightarrow \Qpbar$, we may replace $\widetilde{f}'$ by an eigenform in $S_{\vec{k}}(U,\CC)$ with the same eigenvalues.  We thus obtain an automorphic representation $\pi' \in \mathcal{C}_{\vec{k}}$ such that the local factor $\pi'_v$ is the unramified principal series with $T_v = \widetilde{a}'_v$ and $S_v = \widetilde{d}'_v$ on 
$(\pi'_v)^{\GL_2(\CO_{F,v})}$ for all $v\not\in Q$.

Suppose now that $d = [F:\QQ]$ is even, and let $B$ be the quaternion algebra over $F$ such that $\Sigma^B = \Sigma_\infty$.
We may then apply the Jacquet-Langlands correspondence (in the form of Theorem~\ref{thm:JL}) to obtain an automorphic representation $\Pi' \in \mathcal{C}_{\vec{k}}^B$ such that $\Pi'_v \cong \pi'_v$ for all $v\notin \Sigma_\infty$, and hence an eigenform $\widetilde{\varphi}' \in S^B_{\vec{k}}(U_B,\CC) = M^B_{\vec{k}}(U_B,\CC)$ with the same eigenvalues as $\widetilde{f}'$ for all $v \not\in Q$, where $U_B$ corresponds to $U$ under the isomorphism $B_{\A,\f}^\times \cong \GL_2(\A_{F,\f})$ obtained from a choice of maximal order $\CO_B$ and $\CO_{B,v} \cong M_2(\CO_{F,v})$ for all $v \not\in \Sigma_\infty$.
It then follows from the identifications~(\ref{eqn:scalars}) and
$\Qpbar\otimes_{\CO}M^B_{\vec{k}}(U_B,\CO) = M^B_{\vec{k}}(U_B,\Qpbar)$ that there is such an eigenform in $M^B_{\vec{k}}(U_B,\CO)$, and hence a non-zero $\varphi' \in M^B_{\vec{k}}(U_B,E)$ such that $T_v(\varphi') = \xi(\varpi_v)^{-1}a_v\varphi'$ and
$S_v(\varphi') = \xi(\varpi_v)^{-2}d_v\varphi'$ for all $v\not\in Q$.  The analogue of~(\ref{eqn:twist2}) for $B$ defined in \S\ref{ssec:defQO} then gives an eigenform $\varphi\in M^B_{\vec{k},\vec{m}}(U_B,E)$ with the same eigenvalues as our original $f$.

Now let $U' = \{\,u\in U\,|\,\mbox{$u_{\gp}\equiv 1\bmod \gp$ for all $\gp \in S_p$}\,\}$ and define ${\TT}^B$ to be the 
$E$-subalgebra of endomorphisms of $M^B_{\vec{2}}(U_B',E)$ generated by the operators $T_v$ and $S_v$ for all $v\not\in Q$.  Similarly letting $\widetilde{\TT}^B$ be the finite flat
$\CO$-algebra of endomorphisms of $M^B_{\vec{2}}(U_B',\CO)$ generated by the same operators, the Hecke equivariant surjection $M^B_{\vec{2}}(U_B',\CO) \onto M^B_{\vec{2}}(U_B',E)$ yields a surjection $\widetilde{\TT}^B \onto \TT^B$. 
Combining this with the Hecke-equivariant inclusion
\begin{equation} \label{eqn:levelp}
M^B_{\vec{k},\vec{m}}(U_B,E) \subset M^B_{\vec{2}}(U'_B,E)\otimes_E D_{\vec{k},\vec{m},E},\end{equation}
we deduce the existence of an $\CO$-algebra homomorphism 
$\widetilde{\TT}^B \to \Qpbar$ sending 
$T_v \mapsto \widetilde{a}_v$ and $S_v \mapsto \widetilde{d}_v$ 
for all $v\not\in Q$, where the $\widetilde{a}_v$ and $\widetilde{d}_v$ are lifts of the $a_v$ and $d_v$ to $\Zpbar$.
Furthermore it follows from (\ref{eqn:scalars}) that 
$\widetilde{a}_v,\widetilde{d}_v \in \Qbar$, and that
there is a Hecke eigenform
$\widetilde{\varphi}\in M^B_{\vec{2}}(U_B',\CC)$ such that 
$T_v\widetilde{\varphi} = \widetilde{a}_v\widetilde{\varphi}$ and $S_v\widetilde{\varphi} = \widetilde{d}_v\widetilde{\varphi}$ for all $v\not\in Q$.

Suppose that $\widetilde{\varphi} \in I_{\vec{2}}^B(U_B',\C) = (\CI^B_{\vec{2}})^{U_B'}$ (as defined in \S\ref{ssec:defQC}).  As a representation of $B_{\A,\f}^\times$,
$\CI_{\vec{2}}^B$ 
decomposes as $\oplus_{\psi} \psi\circ\det$, where $\psi$ runs over the smooth characters of $\A_{F,\f}^\times$ which are trivial on $F_+^\times$. We therefore have
$$I^B_{\vec{2}}(U_B',\CC) = \bigoplus_{\psi} \CC f_\psi$$
as a $\widetilde{\TT}_B$-module, 
where $\psi$ runs over characters $\A_{F,\f}^\times/(F_+^\times\det(U')) \to \Qbar^\times$, and
$$T_vf_\psi = (\Nm_{F/\Q}(v)+1)\psi(\varpi_v)\quad\mbox{and}\quad
S_vf_\psi = \psi(\varpi_v)$$
for all $v\not\in Q$.  It follows that $a_v = (\Nm_{F/\Q}(v)+1)\overline{\psi}(\varpi_v)$ and $d_v = \overline{\psi}(\varpi_v)$ for such a character $\overline{\psi}$, and the conclusion of the theorem is satisfied by $\rho_f = \chi^{-1}_{\mathrm{cyc}}\chi \oplus \chi$ where $\chi$ is the character corresponding to $\psi$ by class field theory.

Finally if $\widetilde{\varphi} \not\in I_{\vec{2}}^B(U_B',\C)$, then its image in $S_{\vec{2}}(U_B',\C)$ generates a cuspidal automorphic representation $\Pi \in \mathcal{C}_{\vec{2}}^B$ such that $\Pi_v$ is the unramified principal series such that $T_v = \widetilde{a}_v$ and $S_v = \widetilde{d}_v$ on $\Pi_v^{\CO_{B,v}^\times}$ for all $v\not\in Q$.  Applying Theorem~\ref{thm:JL} to $\Pi$ and then Theorem~\ref{thm:galois0} to $\pi = \JL(\Pi)$ yields a representation $\rho_\pi$ such that $\rho_f = \overline{\rho}_{\pi}$ satisfies the conclusion of the theorem.

Suppose now that $d=[F:\Q]$ is odd.  The same proof as in the case of even $d$ then carries over, but with $B$ chosen so that 
$\Sigma^B = \Sigma_\infty - \{\sigma_0\}$ for some $\sigma_0 \in \Sigma_\infty$ and only the following other change.

The analogue of the inclusion (\ref{eqn:levelp}) is the Hecke-equivariant map
$$\pi_U^*: M_{\vec{k},\vec{m}}^B(U_B,E)
 = H^1(Y^B_{U_B},\mathcal{D}_{\vec{k},\vec{m},E})
 \longrightarrow 
  H^1(Y^B_{U'_B}, \mathcal{D}_{\vec{k},\vec{m},E})
  = M_{\vec{2}}^B(U'_B,E)\otimes_E D_{\vec{k},\vec{m},E}$$
induced by the projection $\pi_U:Y^B_{U'_B} \to Y^B_{U_B}$,
which is not necessarily injective. If $U^p$ is sufficiently small that $\alpha \equiv 1 \bmod \gp$ for all $\gp \in S_p$ and $\alpha \in U \cap \CO_F^\times$, then 
$\pi_U$ is \'etale with Galois group
$G = U_B/U_B' \cong \prod_{\gp \in S_p} \GL_2(\F_{\gp})$, and
the Hochschild--Serre spectral sequence identifies $\ker(\pi_U^*)$ with
$$H^1(G, H^0(Y^B_{U'_B},E)\otimes_E {D}_{\vec{k},\vec{m},E}),$$
compatibly\footnote{The compatibility is presumably a formal consequence of general functoriality properties of the Grothendieck spectral sequences underlying this identification, but for lack of a clear reference to that effect, we note that it follows from the explicit description of the edge maps given in the Appendix of \cite[\S I.2]{mumford}.} with the action of the Hecke operators.
Moreover note that
$H^0(Y^B_{U'_B},E)$ is isomorphic to the space of functions
$C_{U'} \to E$, where
$$C_{U'} = \A_{F,\f}^\times/F_+^\times\det(U')$$
is identified with the set of components of $Y^B_{U'_B}$.
Assuming further that $U^p$ is sufficiently small that the kernel of the homomorphism $\delta:G \to C_{U'}$ induced by $\det$ is
$G_1 = \prod_{\gp \in S_p} \SL_2(\F_{\gp})$, 
Shapiro's Lemma gives a canonical isomorphism
$$\ker(\pi_U^*) \cong \mathrm{Coind}_{\delta(G)}^{C_{U'}} H^1(G_1, D'_{\vec{k},\vec{m},E}),$$
where $H^1(G_1, D_{\vec{k},\vec{m},E})$ is endowed with the 
natural action of 
$\delta(G) \cong G/G_1 \cong \prod_{\gp \in S_p} \F_{\gp}^\times$. Note in particular that since $\delta(G)$ is abelian of order prime to $p$, the representation $H^1(G_1, D_{\vec{k},\vec{m},E})$ decomposes as a direct sum of characters.
Furthermore if $h \in (B_{\f}^{(p)})^\times$ and $U_2 \subset hU_1h^{-1}$ for a pair of open compact subgroups $U_1,U_2$ as above, then the map
$$\mathrm{Coind}_{\delta_1(G)}^{C_{U_1'}} H^1(G_1, D_{\vec{k},\vec{m},E})
\longrightarrow
\mathrm{Coind}_{\delta_2(G)}^{C_{U_2'}} H^1(G_1, D_{\vec{k},\vec{m},E})$$
induced by $\rho_h$ is defined by composition with multiplication by the image of $\det(h)$ in $C_{U_2'}$.  An argument similar to the ones in the analyses of $C_{\vec{k},\vec{m}}(U,E)$ and $I_{\vec{2}}(U_B,\C)$ then shows that $a_v = (\Nm_{F/\Q}(v)+1)\overline{\psi}(\varpi_v)$ and $d_v = \overline{\psi}(\varpi_v)$ for some character $\overline{\psi}:\A_{F,\f}^\times/(F_+^\times\det(U')) \to \Fpbar^\times$, and hence the conclusion of the theorem is satisfied by a representation $\rho_f$ of the form $\chi^{-1}_{\mathrm{cyc}}\chi \oplus \chi$.

We may therefore assume $\varphi \in M_{\vec{2}}(U_B',E) \otimes_E D_{\vec{k},\vec{m},E}$ and complete the proof as in the case of even $d$.  Note also that having obtained an eigenform $\widetilde{\varphi} \in M_{\vec{2}}^B(U'_B,\C)$ with eigenvalues lifting those of $f$, we immediately obtain the desired cuspidal automorphic representation $\Pi \in \mathcal{C}_{\vec{2}}^B$ without the need to consider the space $I_{\vec{2}}^B(U_B',\C)$.
\end{proof}

We also record the following immediate consequence of the proof of Theorem~\ref{thm:galois}:
\begin{proposition} \label{prop:cuspidal}
If $f$ is as in the statement of Theorem~\ref{thm:galois} and $\rho_f$ is absolutely irreducible, then $f \in S_{\vec{k},\vec{m}}(U,E)$.
\end{proposition}
\begin{proof} If $f \not\in S_{\vec{k},\vec{m}}(U,E)$, then the exact sequence (\ref{eqn:cusp2modular}) produces an eigenform $C_{\vec{k},\vec{m}}(U,E)$ with the same eigenvalues as those of $f$. By construction, the associated Galois representation is then reducible (after possibly extending scalars).
\end{proof}

\begin{remark} If $p > 2$, then any embedding $\overline{F} \hookrightarrow \CC$ gives a complex conjugation $c \in G_F$ such $\rho_f(c)$ has distinct eigenvalues $\pm 1 \in E$.  It follows that absolute irreducibility is equivalent to irreducibility if $p > 2$.
\end{remark}

\subsection{Ordinariness} \label{ssec:ord}

Let us now fix a prime $\gp \in S_p$, and suppose that $\vec{m} = -\vec{1}$ and $k_\sigma\ge 1 $ for all $\sigma \in \Sigma_{\gp}$.  In particular, the inequality (\ref{eqn:Tpinequality}) is satisfied, so there is an endomorphism $T_{\gp}$ of $M_{\vec{k},\vec{m}}(U,E)$; furthermore $T_{\gp}$ commutes with $T_v$ and $S_v$ for all $v$ such that $\GL_2(\CO_{F,v}) \subset U$.

\begin{theorem} \label{thm:ord} Suppose that $k_\sigma \ge 1$ for all $\sigma \in \Sigma_{\gp}$ and that $k_\sigma > 1$ for some $\sigma \in \Sigma_{\gp}$, and let $f \in M_{\vec{k},-\vec{1}}(U,E)$ be as in Theorem~\ref{thm:galois}.  If $T_{\gp}f = a_{\gp}f$ for some $a_{\gp} \in E^\times$, then 
$$\rho_f|_{G_{F_\gp}} \sim \left(\begin{array}{cc}\chi_2&*\\
 0 & \chi_\cyc^{-1}\chi_1\end{array}\right)$$
for some characters $\chi_1,\chi_2: G_{F_{\gp}} \to E^\times$,
where $\chi_2$ is the unramified character sending $\Frob_{\gp}$ to $a_{\gp}$.
\end{theorem}

\begin{remark} \label{rmk:ord} The key step in the proof (maintaining ordinariness in the shift from weight $\vec{k}$ and prime-to-$\gp$ level to weight $\vec{2}$ and level $\gp$) is based on a fundamental principle (and technique) conceived by Hida (see \cite[\S8]{hida_annals}, for example).  Aside from this and some technical adaptation of the preceding proof, the rest of the work is in dealing with cases where we know a priori that the global Galois representation is reducible, and therefore so is its restriction to the local  Galois group, but we still need to show it has an unramified constituent with the correct Frobenius eigenvalue.
\end{remark}

\begin{proof} As usual, we may replace $U$ by an arbitrary small open compact subgroup containing $\GL_2(\CO_{F,p})$.  In particular we may assume $U = U(\gn)$ for some $\gn$ prime to $p$.

Suppose first that $f \not\in S_{\vec{k},-\vec{1}}(U,E)$.  Recall from the proof of Theorem~\ref{thm:galois} that 
$\rho_f = \chi_{\cyc}^{-1}\chi_1 \oplus \chi_2$, where the
$\chi_i$ correspond to characters 
$\overline{\xi}_i: \AA_F^\times \to F^\times$ whose 
restrictions to $\CO_{F,p}^\times$ are determined by 
$\vec{k}$ and $\vec{m}$.  In particular, since 
$\vec{m} = -\vec{1}$, the character $\chi_{\cyc}^{-1}\chi_1$ is unramified at $\gp$ (in fact at all primes over $p$), so the content of the theorem in this case is the characterization of $a_{\gp}$ as a Frobenius eigenvalue.

By assumption, the constant term of the $q$-expansion of $f$ is non-zero for some cusp $c \in \CZ_{U,\Fpbar}$.  Writing $c = B(\CO_{F,(p)})_+ g U^p$  for some $g \in \GL_2(\A_{F,\f}^{(p)})$, we may replace $f$ by $gf$ (shrinking $\gn$ and enlarging $Q$ and $E$ as necessary) so as to assume $r_0^1 \neq 0$ (writing $r_m^t$ for the coefficient of $q^m$ in the $q$-expansion of $f$ at the cusp at $\infty$ associated to an element $t \in  (\AA_{F,\f}^{(p)})^\times$).  Since $T_\gp f = a_\gp f$, the assumption on $\vec{k}$ and the formula for the effect of $T_{\gp}$ on $q$-expansions imply that 
$$r_0^t = a_\gp\epsilon^{-1}r_0^{xt}$$
for all $t \in  (\AA_{F,\f}^{(p)})^\times$, where 
$\epsilon = p^{f_{\gp}}\Nm_{F/\QQ}(\varpi_\gp^{-1})
\in \CO_{F,\gp,+}^\times$ and $x = \varpi_\gp^{(p)} \in (\AA_{F,\f}^{(p)})^\times$.

Choose $v\not\in Q$ representing $\gp^{-1}$ in the strict ray class group $C_{\gn\infty} = F_+^\times\backslash \AA_{F,\f}^\times / U_\gn$ (where $U_{\gn} = \ker(\widehat{\CO}_F^\times \to (\CO_F/\gn)^\times$), so $x = \alpha\varpi_v u$ for some $\alpha \in \CO_{F,(p),+}^\times$ and $u \in U_{\gn}$, and hence (\ref{eqn:qexpinv}) implies that
\begin{equation}\label{eqn:r0t}
r_0^t = a_\gp\epsilon^{-1}r_0^{\alpha \varpi_v t u}  
= a_\gp\epsilon^{-1}\Nm_{F/\QQ}(\alpha^{-1})r_0^{\varpi_v t}
 = a_\gp\Nm_{F/\Q}(v) r_0^{\varpi_v t} \end{equation}
for all $t \in  (\AA_{F,\f}^{(p)})^\times$. 
Since $T_vf = a_vf$ and $S_vf = d_vf$, the formula for the effect of $T_v$ on $q$-expansions implies that
$$a_v r_0^t = \Nm_{F/\QQ}(v)r_0^{\varpi_v t} + d_v r_0^{\varpi_v^{-1} t}$$
for all $t \in  (\AA_{F,\f}^{(p)})^\times$.  The description of the Galois representation in this case renders this equality as
$$(\Nm_{F/\QQ}(v)\overline{\xi}_1(\varpi_v) + \overline{\xi}_2(\varpi_v))r_0^t = \Nm_{F/\QQ}(v)r_0^{\varpi_vt} +
\overline{\xi}_1(\varpi_v) \overline{\xi}_2(\varpi_v)r_0^{\varpi^{-1}_vt}.$$ 
Combining this with (\ref{eqn:r0t}) and the non-vanishing of $r_0^t$ for $t=1$, we deduce that $a^{-1}_\gp$ is a root of
$$(X -\Nm_{F/\QQ}(v) \overline{\xi}_1(\varpi_v))(X - \overline{\xi}_2(\varpi_v)).$$
It follows that either 
$$a_\gp = (\Nm_{F/\QQ}(v)\overline{\xi}_1(\varpi_v))^{-1}
 = \chi_{\cyc}^{-1}\chi_1(\Frob_v^{-1}) = \chi_{\cyc}^{-1}\chi_1(\Frob_\gp),$$
or that $a_\gp = \overline{\xi}_2(\varpi_v)^{-1} = \chi_2(\Frob_v^{-1})$ for all
$v\not\in Q$ representing $\gp^{-1}$ in $C_{\gn\infty}$.
The first possibility gives the desired conclusion with the roles of $\chi_{\cyc}^{-1}\chi_1$ and $\chi_2$ interchanged, and the second\footnote{Note that this is only possible if 
$\sum_{i,j} (k_{\sigma_{\gp,i,j}} - 1) p^i$ is divisible by
$p^{f_\gp} -1$ for all $\gp \in S_p$.}
implies that $\chi_2$ is unramified outside $\gn$ and 
that $a_\gp = \chi_2(\Frob_\gp)$.

We suppose from now on that $f \in S_{\vec{k},-\vec{1}}(U,E)$.  
As in the proof of Theorem~\ref{thm:galois}, we may assume 
$\vec{k}$ is sufficiently large that 
$H^1(\CY_{U,E}^{\min},\CA^{\sub}_{\vec{k},-\vec{1},E})=0$.
(Note that our assumption on $\vec{k}$ implies that multiplication by $\prod_\sigma H_\sigma^{M r_\sigma}$ is compatible with $T_\gp$.) We may assume furthermore that $k_\sigma \ge 2$ for
all $\sigma \in \Sigma$ and that $k_\sigma > 2$ for some $\sigma \in \Sigma_{\gp}$.

Consider now the surjective homomorphism
$\pi:S_{\vec{k}}(U,\CO) \to S_{\vec{k},-\vec{1}}(U,E)$
in the proof of Theorem~\ref{thm:galois}. 
The formulas describing the effect of $T_{\gp}$, 
$T_{\varpi_\gp}$ and $e_\xi$ on $q$-expansions imply that
$$\begin{array}{rcl} r_m^t(T_\gp(\pi(f))& =& \epsilon r_{\varpi_\gp m}^{x^{-1}t}(\pi(f)) = \epsilon \xi(x^{-1}t)r_{\varpi_\gp m}^{x^{-1}t}(\overline{f}) \\
&=&  \epsilon \xi(x^{-1}t)r_{m}^{t}(\overline{ T_{\varpi_\gp}f})
 =  \epsilon \xi(x)^{-1}r_{m}^{t}(\pi(T_{\varpi_\gp}f)) 
\end{array}$$ 
for all $t$ and $m$, so $T_\gp(\pi(f)) = \epsilon\xi(x)^{-1}
 \pi(T_{\varpi_{\gp}}f)$. Arguing as in the proof of Theorem~\ref{thm:galois}, but with $\TT^0$ (respectively $\widetilde{\TT}^0$) generated by $T_\gp$ (respectively, $T_{\varpi_\gp}$) along with the $T_v$ and $S_v$ for $v \not\in Q$, it follows that we may choose $\widetilde{f}'$ so that $T_{\gp}\widetilde{f}' = \widetilde{a}'_{\gp} \widetilde{f}'$, where $\varpi_\gp^{(\vec{k}/2)-\vec{1}} \widetilde{a}'_{\gp}$ is a lift of $\xi(x)a_\gp$ to $\Zpbar$.
It follows that the form $\widetilde{\varphi}'$ obtained in the proof of Theorem~\ref{thm:galois} via the Jacquet--Langlands correspondence satisfies $T_{\gp}\widetilde{\varphi}' = \widetilde{a}'_{\gp}\widetilde{\varphi}'$, and hence that 
$$T_{\varpi_{\gp}}\widetilde{\varphi}' = 
 p^{-f_{\gp}}\varpi^{\vec{k}/2} \widetilde{a}'_{\gp} \widetilde{\varphi}' = 
\epsilon^{-1} \varpi^{(\vec{k}/2) - \vec{1}} \widetilde{a}'_{\gp} \widetilde{\varphi}',$$ which in turns implies
that $T_{\varpi_{\gp}}\varphi' = \epsilon^{-1}\xi(x)a_{\gp} \varphi'$
(where $T_{\varpi_{\gp}}$ was defined in \S\ref{ssec:defQO} and \S\ref{ssec:indefQO}), and therefore by (\ref{eqn:Tptwist}), the form $\varphi = e_\xi \varphi' \in M^B_{\vec{k},-\vec{1}}(U_B,E)$ satisfies $T_{\gp}\varphi = a_{\gp}\varphi$.

Now consider the representation $L_\chi=\mathrm{Coind}_H^{\GL_2(\F_{\gp})}E(\chi)$ of $\GL_2(\F_{\gp})$, where $H$ is the subgroup of upper-triangular matrices and $\chi:H \to E^\times$ is the character defined by 
$$\chi\smat{a}{b}{0}{d} = \prod_{\sigma \in \Sigma_{\gp}}
   \sigma(d)^{k_\sigma - 2}.$$
Let $T_{\vec{k}',\chi}$ denote the representation $D_{\vec{k}',-\vec{1},E}\otimes_E L_\chi$ of $\CO_{B,p}^\times$, where $L_\chi$ is viewed as a representation of $U_{B,\gp} = \CO_{B,\gp}^\times$ via the fixed isomorphism with $\GL_2(\CO_{F,\gp})$ (and inflation from
$\GL_2(\F_{\gp})$), and $\vec{k}'$ is defined by $k'_\sigma = 2$ if $\sigma \in \Sigma_{\gp}$ and $k'_\sigma = k_\sigma$ otherwise.

We thus have the space $M^B(U_B,T_{\vec{k}',\chi})$, equipped with operators $T_v$ and $S_v$ for $v\not\in Q$, on which we now also define an endomorphism $T_{\gp}$.  First note that $L_{\chi}$ may be identified with the space of functions $\psi:\F_{\gp}^2 \to E$ such that $\psi(dx,dy) = \chi(d)\psi(x,y)$ for all $d,x,y \in \F_{\gp}$ (setting $\chi(0) = 0$), where the action of 
$\GL_2(\CO_{F,\gp})$ is defined by $(g\cdot \psi)(x,y) = \psi((x,y)g)$.  (To make the identification explicit, let $\psi$ correspond to the function $\smat{v}{w}{x}{y} \mapsto \psi(x,y)$.)  Letting $L_\chi'$ denote the restriction of $L_\chi$ to $U_{B,\gp}' = U_{B,\gp} \cap hU_{B,\gp}h^{-1}$ along $u \mapsto h^{-1}uh$ (where as usual $h \in \CO_{B,\gp}$ corresponds to $\smat{\varpi_{\gp}}{0}{0}{1}$), we have a $U'_{B,\gp}$-equivariant map 
$$\delta: L_\chi' \longrightarrow \Res_{U_{B,\gp}}^{U'_{B,\gp}}L_\chi$$
defined by $\delta(\psi)(x,y) = \psi(0,y)$.  Tensoring with the identity on $D_{\vec{k'},-\vec{1},E}$ then gives a 
$U_B'$-equivariant map 
$T_{\vec{k}',\chi}' \to \Res_{U_B}^{U'_B}
T_{\vec{k}',\chi}$ (where $U_B' = U_B \cap hU_Bh^{-1}$
and $T_{\vec{k}',\chi}'$ is the restriction of 
$T_{\vec{k}',\chi}$ along $u \mapsto h^{-1}uh$),
and hence a morphism
$$\rho_h^* \CT_{\vec{k}',\chi} \longrightarrow 
  \rho_1^* \CT_{\vec{k}',\chi}$$
of sheaves\footnote{When $[F:\Q]$ is even, so $B$ is totally definite, we may write everything more explicitly in terms of functions on $B_{\f}^\times$, but for the sake of uniformity, 
we view $M^B(U_B,T_{\vec{k}',\chi})$ (for example) as 
$H^0(Y^B_{U_B},\mathcal{T}_{\vec{k}',\chi})$, where $Y_B^{U_B}$ is
the finite set $B^\times \backslash B_{\f}^\times /U$
and $\mathcal{T}_{\vec{k}',\chi}$ is the sheaf associated to
$B^\times \backslash (B_{\f}^\times \times T_{\vec{k}',\chi})/U$ (see Remark~\ref{rmk:TpH0}).} on $Y_{U'_B}^B$.
% (described on fibres by $B^\times(xh,t)U_B' \mapsto B^\times(x,(\eta\otimes 1)t)U_B'$ if $d$ is even and $B_+^\times(z,xh,t)U_B' \mapsto B_+^\times(z,x,(\eta\otimes 1)t)U_B'$ if $d$ is odd).
We then define $T_{\gp}$ on $H^i(Y_{U_B}^B, \mathcal{T}_{\vec{k}',\chi})$ (and in particular
 on $M^B(U_B, T_{\vec{k}',\chi})$) as the resulting composite 
\begin{equation} \label{eqn:TpInd} H^i(Y_{U_B}^B, \mathcal{T}_{\vec{k}',\chi})
 \to H^i(Y_{U_B'}^B, \rho_h^*\mathcal{T}_{\vec{k}',\chi})
 \to H^i(Y_{U_B'}^B, \rho_1^*\mathcal{T}_{\vec{k}',\chi})
\to H^i(Y_{U_B}^B,\mathcal{T}_{\vec{k}',\chi}),\end{equation}
where as usual the first map is pull-back along $\rho_h$ and the third is the trace relative to $\rho_1$.
It is straightforward to check that $T_{\gp}$ commutes with the operators $T_v$ and $S_v$ for $v\not\in Q$.

We will now define a Hecke-equivariant map 
$$M^B_{\vec{k},-\vec{1}} (U_B,E) =
M^B(U_B, D_{\vec{k},-\vec{1},E})
\longrightarrow M^B(U_B, T_{\vec{k}',\chi}).$$
Recall that $\bigotimes_{\sigma \in \Sigma_{\gp}} 
\Sym^{k_\sigma - 2} E^2$ can be identified
%(as a representation of $\GL_2(\F_{\gp})$)
with the space of homogeneous polynomials over $E$ in the variables $\{\,X_\sigma,Y_\sigma\,|\,\sigma \in \Sigma_{\gp}\,\}$ with total degree $k_\sigma - 2$ in each pair $(X_\sigma,Y_\sigma)$, where the action of $\GL_2(\F_{\gp})$ is defined by 
$$(g\cdot\Psi)(X_\sigma,Y_\sigma)_{\sigma\in \Sigma_{\gp}}
  = \Psi((X_\sigma,Y_\sigma)\sigma(g))_{\sigma\in \Sigma_{\gp}}.$$
(To make this identification explicit, the monomial $\prod_{\sigma \in \Sigma_{\gp}} X_\sigma^{m_\sigma}Y_\sigma^{n_\sigma}$ corresponds to 
$\bigotimes_{\sigma\in \Sigma_{\gp}}\left(e_\sigma^{\otimes m_\sigma}
\otimes f_\sigma^{\otimes n_\sigma}\right)$, where 
$(e_\sigma,f_\sigma)$ is the standard basis of $E^2$.)
We then have the $\GL_2(\F_{\gp})$-equivariant map
$$\varepsilon: \bigotimes_{\sigma \in \Sigma_{\gp}} 
\Sym^{k_\sigma - 2} E^2 \longrightarrow L_\chi$$
defined by evaluation, i.e., 
$(\varepsilon(\Psi))(x,y) = \Psi(\sigma(x),\sigma(y))$.
Transporting structure to $\CO_{B,\gp}^\times$ and tensoring with the identity on $D_{\vec{k}',-\vec{1},E}$ then yields an 
$\CO_{B,p}^\times$-equivariant map $D_{\vec{k},-\vec{1},E}
\to T_{\vec{k}',\chi}$, and hence a morphism of sheaves inducing maps
$$H^i(Y^B_{U_B}, \mathcal{D}_{\vec{k},-\vec{1},E})
\longrightarrow H^i(Y^B_{U_B}, \mathcal{T}_{\vec{k}',\chi}),$$
and in particular $M^B_{\vec{k},-\vec{1}} (U_B,E)
\to M^B(U_B, T_{\vec{k}',\chi})$.  The compatibility of the resulting map with the operators $T_v$ and $S_v$ for $v \not\in Q$ is straighforward.  For the compatibility with $T_{\gp}$, note that in the above optic, the morphism 
$\rho_h^* \mathcal{D}_{\vec{k},-\vec{1},E} \to
  \rho_1^* \mathcal{D}_{\vec{k},-\vec{1},E}$ in its definition
is induced by tensoring the identity on $D_{\vec{k}',-\vec{1},E}$ with the $U_{B,\gp}'$-linear map
$$\Delta: \left(\bigotimes_{\sigma \in \Sigma_{\gp}}
\Sym^{k_\sigma - 2} E^2 \right)' \longrightarrow 
\Res_{U_{B,\gp}}^{U_{B,\gp}'}\left(\bigotimes_{\sigma \in \Sigma_{\gp}}
\Sym^{k_\sigma - 2} E^2 \right)$$
defined by 
$\Delta(\Psi)(X_\sigma,Y_\sigma)_{\sigma\in \Sigma_{\gp}} = 
 \Psi(0,Y_\sigma)_{\sigma\in \Sigma_{\gp}}$ (where again we transport structure to $U_{B,\gp} = \CO_{B,\gp}^\times$ and 
the first $'$ denotes restriction along the inclusion $U_{B,\gp}' \to U_{B,\gp}$ defined by conjugation by $h^{-1}$).  The desired compatibility then follows from the fact that $\varepsilon\circ \Delta = \delta \circ \varepsilon$. Furthermore we have 
$\ker(\varepsilon) \subset \ker(\Delta)$ and $\im(\delta) \subset \im(\varepsilon)$.  The first inclusion immediately implies that $T_{\gp}$ annihilates the kernel of the map $M^B_{\vec{k},-\vec{1}} (U_B,E) \to  M^B(U_B, T_{\vec{k}',\chi})$ if $d$ is even.  In the case that $d$ is odd, consider the exact sequences
$$0 \to \CE \to \mathcal{D}_{\vec{k},-\vec{1},E} \to \CF \to 0 \quad\mbox{and}\quad 0 \to \CF \to \CT_{\vec{k}',\chi} \to \CG \to 0$$
of locally constant sheaves on $Y_{U_B}^B$, where $\CE$, $\CF$ and $\CG$ are, respectively, the kernel, image and cokernel of
$\mathcal{D}_{\vec{k},-\vec{1},E} \to \CT_{\vec{k}',\chi}$.
The maps in the resulting long exact sequences in cohomology are then compatible with operators $T_{\gp}$ defined in the usual way, but using the morphisms $\rho_h^*\CH \to \rho_1^*\CH$ induced by $\Delta$ and $\delta$ for $\CH = \CE$, $\CF$ and $\CG$.  The inclusions $\ker(\varepsilon) \subset \ker(\Delta)$ and $\im(\delta) \subset \im(\varepsilon)$ imply\footnote{In fact this argument shows that $H^i(Y^B_{U_B}, \mathcal{D}_{\vec{k},-\vec{1},E})
\to H^i(Y^B_{U_B}, \mathcal{T}_{\vec{k}',\chi})$ is an isomorphism on ordinary components, i.e., after localization at the subset $\{T_{\gp}^i\}$ of $E[T_{\gp}]$.} that this morphism, and hence $T_{\gp}$, is zero in the cases of $\CH = \CE$ and $\CG$, from which it follows that $T_{\gp}$ is nilpotent (in fact $T_{\gp}^2 = 0$) on the kernel of $M^B_{\vec{k},-\vec{1}} (U_B,E) \to  M^B(U_B, T_{\vec{k}',\chi})$.  In particular, the element $\varphi$ is not in its kernel, so its image, which we again denote by $\varphi$, is a non-zero element of $M^B(U_B, T_{\vec{k}',\chi})$ with the same eigenvalues with respect to the operators $T_{\gp}$, $T_v$ and $S_v$ (for $v\not\in Q$).

Now let $U_{B,0}$ denote the subgroup of $U_B$ consisting of elements $u$ such that $u_{\gp} \in \CO_{B,\gp}^\times \cong \GL_2(\CO_{F,\gp})$ is upper-triangular, and let $\CT$ be the locally constant sheaf on $Y^B_{U_{B,0}}$ associated to the representation $T= D_{\vec{k}',-\vec{1},E}\otimes_E E(\chi)$ of $U_{B,0,p}$ (where $\chi$ is viewed as a character of $U_{B,0,\gp}$ by inflation from $H$).  Note that $\CT_{\vec{k}',\chi}$ is canonically isomorphic to 
$\pi_*\CT$, where 
$\pi$ is the natural projection $Y^B_{U_{B,0}} \to Y^B_{U_B}$, so that $H^i(Y^B_{U_{B,0}}, \CT) \cong H^i(Y^B_{U_B}, \CT_{\vec{k}',\chi})$, and in particular $M^B(U_{B,0},T) \cong M^B(U_B,T_{\vec{k}',\chi})$.  Furthermore we claim that $T_{\gp}$ on $H^i(Y^B_{U_B}, \CT_{\vec{k}',\chi})$ corresponds to the composite
\begin{equation} \label{eqn:TpU0} H^i(Y_{U_{B,0}}^B, \mathcal{T})
 \to H^i(Y_{U'_{B,0}}^B, \sigma_h^*\mathcal{T})
 \to H^i(Y_{U'_{B,0}}^B, \sigma_1^*\mathcal{T})
\to H^i(Y_{U_{B,0}}^B,\mathcal{T}),\end{equation}
where $U'_{B,0} = U_{B,0} \cap h U_{B,0} h^{-1}$, $\sigma_1$ and $\sigma_h$ are the projections induced by $1$ and $h$, the first and last maps are pull-back and trace, and the middle one is  defined by the isomorphism of sheaves induced by the identity on $T$.  Indeed the desired compatibility amounts to the commutativity of the outer rectangle in the following diagram:
$$\xymatrix{H^i(Y_{U_B}^B,\pi_*\CT) \ar[r] \ar[d]_{\wr} 
\ar@{}[dr]|{(*)} &
H^i(Y_{U_B'}^B,\rho_h^*\pi_*\CT) \ar[r] \ar@{-->}[d]
\ar@{}[drr]|{(**)} &
H^i(Y_{U_B'}^B,\rho_1^*\pi_*\CT) \ar[r] &
H^i(Y_{U_B}^B,\pi_*\CT) \ar[d]_{\wr} \\
H^i(Y_{U_{B,0}}^B,\CT) \ar[r]  &
H^i(Y_{U_{B,0}'}^B,\sigma_h^*\CT) \ar[r] &
H^i(Y_{U_{B,0}'}^B,\sigma_1^*\CT) \ar[r] &
H^i(Y_{U_{B,0}}^B,\CT),}$$
where the rows are (\ref{eqn:TpInd}) and (\ref{eqn:TpU0}) and the outer horizontal maps are the canonical isomorphisms.  For the dashed arrow, let $\pi'$ denote the projection $Y_{U_{B,0}'} \to Y_{U_B'}$, so that $\rho_1 \circ \pi' = \pi\circ \sigma_1$ and $\rho_h \circ \pi' = \pi \circ \sigma_h$.
In particular, we have the natural transformation $\rho_h^*\pi_* \to \pi'_*\sigma_h^*$, giving the desired map
$$H^i(Y_{U_B'}^B,\rho_h^*\pi_*\CT) \longrightarrow
H^i(Y_{U_B'}^B,\pi'_*\sigma_h^*\CT) \stackrel{\sim}{\longrightarrow} H^i(Y_{U_{B,0}'}^B,\sigma_h^*\CT).$$
Furthermore it follows from general abstract formalism that the 
diagram
$$\begin{array}{rl}
 \pi_*\CT & \longrightarrow \qquad \rho_{h,*}\rho_h^*\pi_*\CT\\
          & \searrow  \qquad \qquad \qquad \downarrow\\
& \pi_*\sigma_{h,*}\sigma_h^*\CT = \rho_{h,*}\pi'_*\sigma_h^*\CT
\end{array}$$
commutes (where the two maps from $\pi_*\CT$ are induced by adjunction and the downward arrow by the morphism
$\rho_h^*\pi_*\CT \to \pi'_*\sigma_h^*\CT$),
and hence so does the square labeled $(*)$.

The commutativity of the rectangle $(**)$ follows from that of the  following diagram of morphisms of sheaves on $Y^B_{U_B}$:
$$\xymatrix{\rho_{1,*}\rho_h^*\pi_*\CT \ar[r] \ar[d] & 
 \rho_{1,*}\rho_1^*\pi_*\CT \ar[r] & \pi_*\CT \ar@{=}[d] \\
 \rho_{1,*}\pi'_*\sigma_1^*\CT \ar[r]^{\sim} & 
 \pi_*\sigma_{1,*}\sigma_1^*\CT \ar[r] & \pi_*\CT,}$$
where the downward arrow is given by the natural transformation $\rho_h^*\pi_* \to \pi'_*\sigma_h^*$, the two leftmost horizontal ones by the morphisms $\rho_h^*\CT_{\vec{k}',\chi} \to
\rho_1^*\CT_{\vec{k}',\chi}$ and $\sigma_h^*\CT \stackrel{\sim}{\to} \sigma_1^*\CT$ in the construction of the operators, and the  other two by trace maps.  To verify this commutativity, we consider the corresponding diagram on stalks, given by tensoring the identity on $D_{\vec{k}',-\vec{1},E}$ with the $U_{B,\gp}$-linear maps
$$\xymatrix{\Coind_{U_{B,\gp}'}^{U_{B,\gp}} L_\chi' \ar[r] \ar[d] & 
\Coind_{U_{B,\gp}'}^{U_{B,\gp}}\Res_{U_{B,\gp}}^{U'_{B,\gp}} L_\chi \ar[r] & L_\chi \ar@{=}[d] \\
\Coind^{U_{B,\gp}}_{U'_{B,0,\gp}}E(\chi') \ar@{=}[r] &
\Coind_{U_{B,0,\gp}}^{U_{B,\gp}}\Coind_{U'_{B,0,\gp}}^{U_{B,0,\gp}}\Res_{U_{B,0,\gp}}^{U'_{B,0,\gp}} E(\chi) \ar[r]  
 &\Coind_{U_{B,0,\gp}}^{U_{B,\gp}}E(\chi),} $$
where 
\begin{itemize}
\item $\chi'$ is the character of $U_{B,0,\gp}'$ defined by restricting $\chi$ along $u \mapsto h^{-1}uh$ (which 
is the same as its restriction along the natural inclusion, justifying the lower left equality),
\item the downward arrow is $\Coind_{U_{B,\gp}'}^{U_{B,\gp}}\varepsilon$, where $\varepsilon: L'_\chi \to \Coind_{U_{B,0,\gp}'}^{U_{B,\gp}'}E(\chi')$ is defined by
$\varepsilon(\psi(u)) = \psi(h^{-1}uh)$ for $u \in U_{B,\gp}'$ 
and $\psi \in L'_{\chi}$ (viewed as a function $U_{B,\gp}\to E$),
\item the top left arrow is 
$\Coind_{U_{B,\gp}'}^{U_{B,\gp}}\delta$, 
\item the top right arrow is the trace map $\tr_{U_{B,\gp}/\U'_{B,\gp}}$, 
\item and the bottom right arrow is
$\Coind_{U_{B,0,\gp}}^{U_{B,\gp}}\left(\tr_{U_{B,0,\gp}/ U'_{B,0,\gp}}\right)$.
\end{itemize}
It is straightforward to check that the composite along the left and bottom of the rectangle sends an element $\xi:U_{B,\gp} \to L_{\chi}'$ to the 
element of $L_\chi$ defined by
$$ u \mapsto \sum_{v \in U_{B,0,\gp}/U_{B,0,\gp}'} 
  \!\!\!\!\!\!\!\!\! \chi(v)\xi(v^{-1}u)(1),$$
and that this is the adjoint of $\delta$, giving the desired commutativity.

Now let $U_{B,1}$ denote the subgroup of $U_B$ consisting of elements $u$ such that $u_{\gp}$ is in the subgroup of $\CO_{B,\gp}^\times$ corresponding to $U_1(\gp)_\gp$. Note that $U_{B,1}$ is a normal subgroup of $U_{B,0}$ of index prime to $p$, and that we may view $\chi$ as a character of $U_{B,0}/U_{B,1} \cong k_{\gp}^\times$.  Furthermore the pull-back of $\CT$ to $Y^B_{U_{B,1}}$ may be identified with 
$\mathcal{D}_{\vec{k}',-\vec{1},E}$, and the injective map
$$M^B(U_{B,0},T) \hookrightarrow M^B_{\vec{k}',-\vec{1}}(U_{B,1},E)$$
is equivariant with respect to $T_{\gp}$ as well as the $T_v$ and $S_v$ for $v\not\in Q$.  We may therefore replace $\varphi$ by its image in $M^B_{\vec{k}',-\vec{1}}(U_{B,1},E)$.

Now let $U_{B,1}' = \{\,u\in U_{B,1},|\,\mbox{$u_{\gq}\equiv 1\bmod \gq$ for all $\gq \in S_p$, $\gq \neq \gp$}\,\}$
and consider the map
\begin{equation} \label{eqn:levelpord}
M^B_{\vec{k}',-\vec{1}}(U_{B,1},E) \longrightarrow M^B(U'_{B,1},E) \otimes_E D_{\vec{k}',-\vec{1},E},
\end{equation}
analogous to (\ref{eqn:levelp}). As in the proof of Theorem~\ref{thm:galois}, this is Hecke-equivariant (now also with respect to $T_{\gp}$), injective if $d$ is even, and has kernel isomorphic to
$$H^1(G',H^0(Y^B_{U'_{B,1}},E)\otimes_E D_{\vec{k}',-\vec{1},E})
\cong \Coind_{\delta(G'_1)}^{C_{U'_{B,1}}} H^1(G_1',D_{\vec{k}',-\vec{1},E})$$
if $d$ is odd, where now $G' = U_{B,1}/U'_{B,1}$ is isomorphic to the product of the $\GL_2(\FF_{\gq})$ for $\gq \in S_p$, $\gq \neq \gp$.  The same analysis as in the preceding proof shows that $T_{\gp}$ is $p^{f_{\gp}}$ times composition with multiplication by the image of $(\varpi_{\gp})_\gp$ in $C_{U'_{B,1}}$, and hence that $T_{\gp}$ annihilates the kernel.  It follows that $\varphi$ is not in this kernel, and we may deduce as before the existence of an eigenform $\widetilde{\varphi} \in M_{\vec{2}}(U'_{B,1},\C)$ with eigenvalues lifting those of $f$, for $T_{\gp}$ as well as $T_v$ and $S_v$ for all $v \not\in Q$.

If $d$ is even and $\widetilde{\varphi} \in I_{\vec{2}}^B(U_{B,1}',\CC)$, then the same analysis as in the proof of Theorem~\ref{thm:galois} shows that $\widetilde{\varphi} $ is in the span of $f_\psi$ for some character $\psi$ of $\A_{F,\f}^\times/F_+^\times\det(U_{B,1}')$ (in particular of conductor prime to $\gp$).  Now however we have
$$T_\gp f_\psi = p^{f_{\gp}}\psi((\varpi_{\gp})_\gp)f_\psi,$$ which contradicts the assumption that $a_{\gp} \neq 0$.
We may therefore replace $\widetilde{\varphi}$ by its image in $S_{\vec{2}}(U'_{B,1},\C)$ (whether $d$ is even or odd).

As in the proof of Theorem~\ref{thm:galois}, we obtain a cuspidal automorphic representation $\pi \in \mathcal{C}_{\vec{2}}$ such that $\pi_v$ is the unramified principal series such that $T_v = \widetilde{a}_v$ (lifting $a_v$) and $S_v = \widetilde{d}_v$ (lifting $d_v$) on $\pi_v^{\GL_2(\CO_{F,v})}$ for all $v\not\in Q$, and moreover that $\pi_{\gp}^{U_1(\gp)}$ contains an eigenvector for $T_{\gp}$ with eigenvalue $\widetilde{a}_{\gp}$ lifting $a_{\gp}$.  In particular $\pi_{\gp}$ has conductor dividing $\gp$, and more precisely is
\begin{itemize}
\item either a principal series representation of the form $I(\psi_1|\cdot|^{1/2},\psi_2|\cdot|^{1/2})$ with $\psi_2$ an unramified character such that $\psi_2(\varpi_{\gp}) = \widetilde{a}_{\gp}$ (and $\psi_1$ at most tamely ramified),
\item or a special representation of the form $\psi\otimes \St$,
where $\psi$ is an unramified character such that $\psi(\varpi_{\gp}) = \widetilde{a}_{\gp}$ and $\St$ is the Steinberg representation.
\end{itemize}
In the first case, Theorem~\ref{thm:galois0} implies that
$\rho_{\pi}|_{G_{F_{\gp}}}$ is potentially crystalline with
$\sigma$-labelled Hodge-Tate weights $(0,1)$ for all $\sigma \in \Sigma_{\gp}$ and $\WD(\rho_{\pi}|_{G_{F_{\gp}}})\cong \psi_1 \oplus \psi_2$.  It is then a standard exercise in $p$-adic Hodge theory to conclude that $\rho_{\pi}|_{G_{F_{\gp}}}$ has the form
$$\left(\begin{array}{cc}{\widetilde{\chi}_2}&{*}\\{0}&{\widetilde{\chi}_1\chi_{\cyc}^{-1}}\end{array}\right),$$
where $\widetilde{\chi}_2$ corresponds to $\psi_2$ and 
$\widetilde{\chi}_1$ to $\psi_1|\cdot|$ by local class field theory.  Similarly in the second case, we get that 
$\rho_{\pi}|_{G_{F_{\gp}}}$ has the form
$$\widetilde{\chi}\otimes \left( \begin{array}{cc}{1}&{*}\\0&{\chi_{\cyc}^{-1}}\end{array}\right),$$
where $\widetilde{\chi}$ corresponds to $\psi$.  It therefore follows in either case that $\rho_f|_{G_{F_{\gp}}} \cong
 \overline{\rho}^{{\mathrm{ss}}}_{\pi}|_{G_{F_{\gp}}}$ has the
desired form.
\end{proof}

\section{Geometric weight conjectures}
In this section we generalize the geometric Serre weight conjecture of \cite{DS1} to the case where $p$ is ramified in $F$, and discuss the relation with the corresponding generalization due to Gee (see~\cite[\S4]{Gtypes}) of the algebraic Serre weight conjecture of \cite{BDJ}.

\label{sec:conj}

\subsection{Geometric modularity}
Let 
$$\rho: G_F = \mathrm{Gal}(\overline F/F)\rightarrow \mathrm{GL}_2(\overline{\mathbb{F}}_p)$$
be an irreducible, continuous, totally odd representation of the absolute Galois group of $F$. 

We make the following definition as in \cite{DS1}:

\begin{definition} \label{def:geomod} We say that $\rho$ is {\em geometrically modular of weight $(\vec{k},\vec{m})$} if
$\rho$ is equivalent to the extension of scalars of $\rho_f$
for some open compact subgroup $U$
and eigenform $f \in M_{\vec{k},\vec{m}}(U;E)$ as in the statement of Theorem~\ref{thm:galois}.\end{definition}

Thus $\rho$ is geometrically modular of weight $(\vec{k},\vec{m})$ if there
is a non-zero element $f \in M_{\vec{k},\vec{m}}(U;E)$, for some $U\supset \mathrm{GL}_2(\CO_{F,p})$
and $E\subset \overline{\mathbb{F}}_p$, such that
$$T_v f = \mathrm{tr}(\rho(\mathrm{Frob}_v))f\quad\mbox{and} \quad\mathrm{Nm}_{F/\mathbb{Q}}(v) S_v f = \det(\rho(\mathrm{Frob}_v))f$$
for all but finitely many primes $v$.

Note also that by Proposition~\ref{prop:cuspidal} (and our running assumption that $\rho$ is irreducible), we can replace $M_{\vec{k},\vec{m}}(U,E)$ by $S_{\vec{k},\vec{m}}(U,E)$ in the definition of geometric modularity.

As in the setting of \cite{DS1}, it is a folklore conjecture that every $\rho$ as above is geometrically modular of {\em some}
weight $(\vec{k},\vec{m})$, and our aim is to predict exactly {\em which} weights in terms of
of the local behaviour of $\rho$ at primes  over $p$.

We first note an immediate consequence of the existence of partial Hasse invariants.  By the discussion in \S\ref{ssec:Hasse}, we have injective Hecke-equivariant maps
$$M_{\vec{k},\vec{m}}(U,E) \hookrightarrow M_{\vec{k}+\vec{\ell},\vec{m}}(U,E)$$
for all $\vec{\ell} \in \Xi_{\mathrm{Hasse}}$, where
$$\Xi_{\mathrm{Hasse}} = 
\{\,\sum_{\sigma\in \Sigma} r_\sigma \vec{h}_\sigma\,|\,
  \mbox{$r_\sigma \in \ZZ_{\ge 0}$ for all $\sigma \in \Sigma$}\,\}$$
is the (integral) Hasse cone.  We therefore have the following:
\begin{proposition} \label{prop:Hasse_effect}  Suppose that $\vec{k},\vec{m} \in \ZZ^\Sigma$ and that $\vec{\ell} \in \Xi_{\mathrm{Hasse}}$.  If $\rho$ is geometrically modular of weight $(\vec{k},\vec{m})$, then $\rho$ is geometrically modular of weight $(\vec{k}+\vec{\ell},\vec{m})$.
\end{proposition}

\subsection{Crystalline lifts and labelled Hodge-Tate weights}  \label{subsec:cryslift}
Suppose now that $\rho: G_{F_{\gp}} \to \Aut_{\Qpbar}(V) \cong \GL_d(\Qpbar)$ is a continuous representation of $G_{F_{\gp}}$, given by its
action on a $d$-dimensional vector space $V$ over $\Qpbar$.
Recall that $D_{\mathrm{crys}}(V) = (V \otimes_{\mathbb{Q}_p} B_{\mathrm{crys}})^{G_{F_{\gp}}}$ is a finitely-generated module over
$$(\Qpbar \otimes_{\mathbb{Q}_p} B_{\mathrm{crys}})^{G_{F_{\gp}}} = \Qpbar \otimes_{\mathbb{Q}_p} F_{\gp,0} \cong \prod_{\tau\in \Sigma_{\gp,0}} \Qpbar,$$
and hence decomposes as the direct sum of finite-dimensional $\Qpbar$-vector spaces $D_{\crys}(V)_\tau$ for $\tau \in \Sigma_{\gp,0}$ (on which $F_{\gp,0}$ acts via $\tau$).
Furthermore $D_\crys(V)$ is equipped with a semi-linear automorphism $\phi$, sending $D_{\crys}(V)_{\tau_{\gp,i}}$ to 
$D_{\crys}(V)_{\tau_{\gp,i-1}}$, so that $D_{\crys}(V)$ is in fact free of finite rank over $\prod_{\tau\in \Sigma_{\gp,0}}\Qpbar$, and we say that $\rho$ is {\em crystalline} if this is rank is $d$.

Similarly $\rho$ is {\em de Rham} (resp.~{\em Hodge--Tate}) if the 
filtered (resp.~graded)
module $D_{\mathrm{dR}}(V) = (V \otimes_{\mathbb{Q}_p} B_{\mathrm{dR}})^{G_{F_{\gp}}}$ (resp.~$D_{\mathrm{HT}}(V) =
(V \otimes_{\mathbb{Q}_p} B_{\mathrm{HT}})^{G_{F_{\gp}}}$) 
is free of rank $d$ over $\Qpbar \otimes_{\mathbb{Q}_p} F_{\gp}
 \cong \prod_{\sigma \in \Sigma_{\gp}} \Qpbar$.
Furthermore, if $\rho$ is crystalline, then it is de Rham, and hence also Hodge--Tate, in which case 
$D_{\mathrm{dR}}(V)\cong D_{\mathrm{crys}} \otimes_{F_{\gp,0}} F_{\gp}$
and $D_{\mathrm{HT}}(V) \cong \gr(D_{\mathrm{dR}}(V))$.
In particular, for each
$\sigma \in \Sigma_{\gp}$, the corresponding component 
$D_{\mathrm{HT}}(V)_\sigma$ of $D_{\mathrm{HT}}(V)$
is a graded $d$-dimensional vector space over $\Qpbar$.  
\begin{definition}  \label{def:HTtype}
If $\rho$ is Hodge--Tate and $\sigma \in \Sigma_{\gp}$, then the $\sigma$-{\em labelled weights}
of $V$ are the $d$-tuple of integers $(w_1,w_2,\ldots,w_d) \in \mathbb{Z}^d$ such that
$w_1 \ge w_2 \ge \cdots \ge w_d$
and $D_{\mathrm{HT}}(V)_\sigma$ is isomorphic to $\oplus_{i=1}^d \Qpbar[-w_i]$,
where $\Qpbar[-w_i]$ has degree $-w_i$.
\end{definition}

We restrict our attention now to the case $d=2$.

\begin{definition}  \label{def:weightlift} Suppose that 
$\rho: G_{F_{\gp}} \to \GL_2(\Fpbar)$ is a continuous representation and that $(\vec{k},\vec{m}) \in \mathbb{Z}_{\ge 1}^{\Sigma_{\gp}} \times \mathbb{Z}^{\Sigma_{\gp}}$.
We say that $\rho$ has a {\em crystalline lift} of weight $(\vec{k},\vec{m})$ if there exists a continuous representation:
$$\widetilde{\rho}:  G_{F_{\gp}} \to  \mathrm{GL}_2(\Zpbar)$$
such that $\widetilde{\rho}\otimes_{\Zpbar} \overline{\mathbb{F}}_p$ is isomorphic 
to $\rho$, and $\widetilde{\rho}\otimes_{\Zpbar}\Qpbar$ is crystalline with $\sigma$-labelled
Hodge--Tate weights $(-1-m_\sigma,-k_\sigma-m_\sigma)$ for all
$\sigma \in \Sigma_{\gp}$.\end{definition}

Note that this definition differs slightly from the one in \cite[Def.~7.2.2]{DS1}, reflecting the difference in conventions in the definitions of the Hecke action and
$\sigma$-labelled Hodge--Tate weights.

\subsection{Statement of the conjecture}

First recall the definition of the (positive) minimal cone appearing in Theorem~\ref{thm:positivity}:
$$\Xi_{\min}^+ = \left\{\,\vec{k}\in \ZZ_{>0}^{\Sigma}\,\left|
\begin{array}{c} k_{\gp,i,1} \le k_{\gp,i,2} \le \cdots \le k_{\gp,i,e_{\gp}}
 \le p k_{\gp,i+1,1} \\
\mbox{for all  $\gp\in S_p$, $i\in \Z/f_{\gp}\Z$}
\end{array}\right.\right\},$$
where we have written $k_{\gp,i,j}$ for $k_{\sigma_{\gp,i,j}}$.

Recall also that for $\vec{k} \in \ZZ^\Sigma$, we let $\vec{k}_\gp$ denote its image in $\ZZ^{\Sigma_\gp}$ under the natural projection.

Our generalization of Conjecture~7.3.1 of \cite{DS1} is then:

\begin{conjecture} \label{conj:geomweights}
Let $\rho: G_F \rightarrow \mathrm{GL}_2(\overline{\mathbb{F}}_p)$
be an irreducible, continuous, totally odd representation, and let $\vec{m} \in \mathbb{Z}^\Sigma$.
There exists $\vec{k}_{\mathrm{min}} = \vec{k}_{\mathrm{min}}(\rho,\vec{m}) \in \Xi_{\mathrm{min}}^+$ such that the following hold:
\begin{enumerate}
\item $\rho$ is geometrically modular of weight $(\vec{k},\vec{m})$ if and only if $\vec{k} - \vec{k}_{\mathrm{min}} \in \Xi_{\mathrm{Hasse}}$;
\item if $\vec{k} \in \Xi_{\mathrm{min}}^+$, then $\vec{k} - \vec{k}_{\mathrm{min}} \in \Xi_{\mathrm{Hasse}}$ if and only if $\rho|_{G_{F_{\gp}}}$
has a crystalline lift of weight $(\vec{k}_\gp,\vec{m}_\gp)$ for all $\gp \in S_p$.
\end{enumerate}
\end{conjecture}

We can view the conjecture as comprising several parts:
\begin{itemize}
\item Firstly, we incorporate the folklore conjecture that every continuous,  irreducible, totally odd
$\rho$ is modular, in the sense that it arises as $\rho_f$ from a mod $p$ Hilbert modular eigenform $f$ of {\em some} weight.  While this can be called the ``weak version'' of (the analogue of) Serre's Conjecture in this context, it is almost certainly the hardest part.

We remark on a possible alternative formulation of the weak version of the conjecture. Note that it follows from the construction of $\rho_f$ in the proof of Theorem~\ref{thm:galois} that if it is irreducible, then it arises as $\barrho_\pi$ for some $\pi \in \mathcal{C}_{\vec{2}}$; note however that the conductor of $\pi$ need not be prime to $p$.
Conversely, we will see below (Corollary~\ref{cor:reductions}) that if $\pi$ is as in the statement of Theorem~\ref{thm:galois0} (for any $\vec{k}$ and $w$ all of the same parity), and $\barrho_\pi$ is irreducible, then it arises as $\rho_f$ for some $f$ as in Theorem~\ref{thm:galois}.  The weak version is therefore equivalent to the assertion that every continuous,  irreducible, totally odd $\rho$ is of the form $\barrho_\pi$ for some $\pi \in \mathcal{C}_{\vec{2}}$.

\item Assuming $\rho$ is geometrically modular (of some weight), the existence of $\vec{k}_{\min}$ as in (1) is an assertion purely about mod $p$ Hilbert modular forms.  Recall from Theorem~\ref{thm:positivity} that for each eigenform $f$ giving rise to $\rho$, we have $\vec{k}_{\min}(f) \in \Xi_{\min}^+$ (the irreducibility of $\rho_f$ ruling out the possibility that
$\vec{k}_{\min}(f) = \vec{0}$).  The conjecture thus predicts that among all such $f$, i.e., those (for fixed $\vec{m}$) with the same systems of Hecke eigenvalues outside a finite set of primes, there are those for which $\vec{k}_{\min}(f)$ is uniquely minimal with respect to the partial ordering induced by $\Xi_{\mathrm{Hasse}}$.

The difficulty in proving this arises when $\rho_f|_{G_{F_\gp}}$ is reducible for multiple primes $\gp \in S_p$.  More precisely, if $f$ is ordinary at $\gp$, in the sense of the hypothesis of Theorem~\ref{thm:ord}, then $\rho_f$ will also arise from Hecke  eigenforms $g$, also with $\vec{m} = -\vec{1}$ but different $\vec{k}_{\min}$, such that $T_{\gp}(g) = 0$.  The conjectured existence of $\vec{k}_{\min}$ as in (1) is based on the expectation that $\rho$ arises from an $f$ (with $\vec{m} = -\vec{1}$) which is simultaneously ordinary at all $\gp \in S_p$ for which $\rho|_{G_{F_{\gp}}}$ is ramified and ordinary (in the sense of the conclusion of Theorem~\ref{thm:ord}), and with $k_\sigma = 1$ for all $\sigma \in S_{\gp}$ for which $\rho|_{G_{F_{\gp}}}$ is unramified.  The fact that this holds (at least under mild technical hypotheses) follows from the main results of \cite{BLGG1, BLGG2} and forthcoming work of the authors.

\item Complementarily, the existence of $\vec{k}_{\min}$ as in (2) is a statement purely about integral $p$-adic Hodge theory.  Indeed let $\rho:G_L \to \GL_2(\Fpbar)$ be any continuous representation, where $L$ is any finite extension of $\Q_p$,
and (re)define $\Xi_{\min}^+$ and $\Xi_{\mathrm{Hasse}}$ as subsets of $\ZZ^{\Sigma}$, where $\Sigma$ is the set of embeddings $\{\,\sigma:L \to \Qpbar\,\}$.  The assertion then amounts\footnote{The local statement is a priori stronger, but in fact equivalent by \cite[Prop.~A.1]{GK}.  Note also that we have assumed here that $\vec{m} = -\vec{1}$, to which the general case reduces by twisting arguments discussed below.}
to the existence of a unique minimal $\vec{k} \in \Xi_{\min}^+$ 
(with respect to the partial ordering induced by 
$\Xi_{\mathrm{Hasse}}$) such that $\rho$ has a crystalline lift 
of weight $(\vec{k},-\vec{1})$.  The existence of such a $\vec{k}$ is strongly suggested by its consistency with the Breuil--M\'ezard Conjecture in conjunction with modular representation theoretic considerations, and in particular Conjecture~2.3 of Wiersema's thesis~\cite{HWPhD} (in the case that $p$ is unramified in $F$).

\item Finally, the conjecture predicts that the minimal weights characterized in the preceding two points coincide; in particular if $\rho$ arises from a mod $p$ Hilbert modular eigenform of some weight, then the set of such weights is determined by the local Galois representations $\rho|_{G_{F_\gp}}$ for $\gp \in S_p$.
More precisely, granting the ``weak'' conjecture and the existence of minimal weights as in parts (1) and (2), then their coincidence is equivalent to the following generalization of \cite[Conj.~7.3.2]{DS1}, which we may view as the ``weight part'' of the analogue of Serre's conjecture in this context:
\end{itemize}
\begin{conjecture} \label{conj:geomweights2}
Suppose that $\rho: G_F \rightarrow \mathrm{GL}_2(\overline{\mathbb{F}}_p)$ is irreducible and
geometrically modular some weight, and that $\vec{k} \in \Xi_{\mathrm{min}}^+$.  Then
$\rho$ is geometrically modular of weight $(\vec{k},\vec{m})$ if and only if $\rho|_{G_{F_\gp}}$
has a crystalline lift of weight $(\vec{k}_\gp,\vec{m}_\gp)$ for all $\gp \in S_p$.
\end{conjecture}

\begin{remark} \label{rmk:conj:geomweights2}
We remark that the ``only if'' in Conjecture~\ref{conj:geomweights2} may fail for certain $\vec{k} \not\in \Xi_{\min}^+$; see \cite[Example~3.2.21]{SYPhD} for an example due to Bartlett.  On the other hand, the results in S.~Yang's thesis (see \cite[Rmk.~4.1.3]{SYPhD}) show that, for the purpose of ensuring the ``only if'' implication, the inequalities in the definition of the $\Xi_{\mathrm{min}}^+$ are not sharp. (See also Remark~\ref{rmk:notsharp} below.)

In any case, for a version that covers all $\vec{k}$, we could replace ``weight $(\vec{k}_\gp,\vec{m}_\gp)$ for all $\gp \in S_p$'' with ``weight $(\vec{k}'_\gp,\vec{m}_\gp)$ for all $\gp \in S_p$ and some $\vec{k}' \in \Xi_{\min}^+$ such that $\vec{k} - \vec{k}' \in \Xi_{\mathrm{Hasse}}$.''
\end{remark}

\subsection{Dependence on $\vec{m}$} \label{ssec:galoistwist}
We first note that the conjectural characterizations of 
$\vec{k}_{\min}(\rho,\vec{m})$ depend only on a mod $p$ character associated to $\vec{m}$.
\begin{proposition} \label{prop:chars} 
Let $\rho: G_F \rightarrow \mathrm{GL}_2(\overline{\mathbb{F}}_p)$
be an irreducible, continuous representation.
Suppose that $\vec{k},\vec{m},\vec{n} \in \ZZ^{\Sigma}$ are such that $\prod_{\sigma \in \Sigma_\gp} \omega_\sigma^{m_\sigma} = \prod_{\sigma \in \Sigma_\gp} \omega_\sigma^{n_\sigma}$ for all $\gp \in S_p$,
where $\omega_\sigma:\F_{\gp}^\times \to \Fpbar^\times$ is the character induced by $\sigma$.
\begin{enumerate}
\item $\rho$ is geometrically modular of weight $(\vec{k},\vec{m})$ if and only if $\rho$ is geometrically modular of weight $(\vec{k},\vec{n})$;
\item if $\gp \in S_p$, then $\rho|_{G_{F_\gp}}$ has a crystalline lift of weight $(\vec{k}_\gp,\vec{m}_\gp)$ if and only if it has a crystalline lift of weight $(\vec{k}_\gp,\vec{n}_\gp)$.
\end{enumerate}
\end{proposition}
\begin{proof} First note that the condition on $\vec{m}$ and $\vec{n}$ means that $\vec{n}-\vec{m} \in \ker(\lambda)$, where
$$\lambda:\ZZ^\Sigma \longrightarrow
  \prod_{\gp \in S_p} \Hom(\FF_\gp^\times,\Fpbar^\times)$$
is the surjective homomorphism defined by $\lambda(\vec{\ell})_{\gp}
 = \prod_{\sigma\in \Sigma_\gp} \omega_\sigma^{\ell_\sigma}$.
Since $\vec{h}_\sigma \in \ker(\lambda)$ for all $\sigma \in \Sigma$,
and
$$[\ZZ^\Sigma:\sum_{\sigma \in \Sigma} \ZZ\vec{h}_\sigma]
 = |\det(A)| = \prod_{\gp \in S_p} (p^{f_\gp} - 1),$$
where $A$ is the $d\times d$-matrix with columns $\vec{h}_\sigma$, it follows that $\ker(\lambda) = \sum_{\sigma \in \Sigma} \ZZ\vec{h}_\sigma$, and therefore $\vec{n} = \vec{m} + \sum_{\sigma \in \Sigma} r_\sigma \vec{h}_\sigma$ for some $\vec{r} \in \ZZ^\Sigma$.

Recall from \S\ref{ssec:Hasse} that we have Hecke-equivariant isomorphisms 
$$M_{\vec{k},\vec{m}}(U,E) \stackrel{\sim}{\to} M_{\vec{k},\vec{m}+\vec{h}_\sigma}(U,E)$$ 
for all $\sigma \in \Sigma$, so it follows that $M_{\vec{k},\vec{m}}(U,E) \cong M_{\vec{k},\vec{n}}(U,E)$, which implies (1).

To prove (2), recall that a character $\xi:G_{F_\gp} \to \Qpbar^\times$ (necessarily $\Zpbar^\times$-valued) is crystalline with $\sigma$-labelled Hodge--Tate weights $(w_\sigma)$ for $\sigma\in \Sigma_\gp$ if and only if the restriction 
$\CO_{F_\gp}^\times \to \Qpbar^\times$ of the character corresponding to $\xi$ via local class field theory has the form 
$\prod_{\sigma\in \Sigma_\gp} \sigma^{w_\sigma}$ (see \cite[App.~B]{brian}).  
The condition on $\vec{m}$ and $\vec{n}$ therefore implies that there is such a character $\xi$ with trivial reduction and $\sigma$-labelled Hodge--Tate weights $(m_\sigma - n_\sigma)$ for $\sigma \in \Sigma_\gp$.  Thus if $\widetilde{\rho}$ is a crystalline lift of $\rho|_{G_{F_\gp}}$ of weight $(\vec{k}_\gp,\vec{m}_\gp)$, then $\xi\otimes\widetilde{\rho}$ is 
a crystalline lift of $\rho|_{G_{F_\gp}}$ of weight $(\vec{k}_\gp, \vec{n}_\gp)$.
\end{proof}

Let $Q$ be a finite set of primes of $F$ containing $S_p$, and
suppose that $\chi:G_F \to \Fpbar^\times$ is a continuous character unramified outside $Q$.  Write $\xi:\A_F^\times/F^\times F_{\infty,+}^\times \to \Fpbar^\times$
for the character corresponding to $\chi$ via class field theory, so in particular, $\xi(\varpi_v) = \chi(\Frob_v)$ for all $v \not \in Q$.  Note also that $\chi$ is tamely ramified at all $\gp \in S_p$, and that $\xi|_{\CO_{F,\gp}^\times}$ induces the character $\F_{\gp}^\times \to \Fpbar^\times$ corresponding to $\chi|_{I_{\gp}}$ via local class field theory. 

\begin{proposition} \label{prop:chartwist} Let $\rho: G_F \rightarrow \mathrm{GL}_2(\overline{\mathbb{F}}_p)$ be an irreducible, continuous representation, and let $\chi: G_F \rightarrow\overline{\mathbb{F}}_p^\times$ be a continuous character. Suppose that $\vec{k},\vec{\ell},\vec{m} \in \ZZ^{\Sigma}$, and that $\xi|_{\CO_{F,\gp}^\times} = 
\prod_{\sigma \in \Sigma_\gp} \overline{\sigma}^{-\ell_\sigma}$for all $\gp \in S_p$, where $\xi$ is the character corresponding to $\chi$ via class field theory.
\begin{enumerate}
\item $\rho$ is geometrically modular of weight $(\vec{k},\vec{m})$ if and only if $\chi\otimes\rho$ is geometrically modular of weight $(\vec{k},\vec{\ell}+\vec{m})$;
\item if $\gp \in S_p$, then $\rho|_{G_{F_\gp}}$ has a crystalline lift of weight $(\vec{k}_\gp,\vec{m}_\gp)$ if and only if $(\chi\otimes\rho)|_{G_{F_\gp}}$ has a crystalline lift of weight $(\vec{k}_\gp,\vec{\ell}_\gp+\vec{m}_\gp)$.
\end{enumerate}
In particular, Conjectures~\ref{conj:geomweights} and~\ref{conj:geomweights2} hold for $\rho$ if and only if they hold for $\chi\otimes\rho$.
\end{proposition}
\begin{proof} (1) Suppose that $\rho$ is geometrically of weight $(\vec{k},\vec{m})$, and let $f \in M_{\vec{k},\vec{m}}(U,E)$ be an eigenform such that $\rho_f$ is isomorphic to $\rho$.  Note that since $\xi$ is trivial on $F^\times F_{\infty,+}^\times$, its restriction to $(\AA_{F,\f}^{(p)})^\times$ satisfies $\xi(\alpha) = \alpha^{\vec{\ell}}$ for all $\alpha \in \CO_{F,(p),+}^\times$.  Therefore, shrinking $U$ and enlarging $E$ as necessary, we have the isomorphism
$$e_\xi: M_{\vec{k},\vec{m}}(U,E)
  \stackrel{\sim}{\longrightarrow} M_{\vec{k},\vec{\ell}+\vec{m}}(U,E)$$
of (\ref{eqn:twist}), and hence an eigenform $\varphi = e_\xi f  \in M_{\vec{k},\vec{\ell}+\vec{m}}(U,E)$ such that
$T_v(\varphi) = \xi(\varpi_v) a_v \varphi $ and $S_v(\varphi) = \xi(\varpi_v)^2 d_v \varphi$ for all but finitely many primes $v$ of $F$, where $a_v = \tr(\rho(\Frob_v))$ and $\Nm_{F/\Q}(v)d_v = \det(\rho(\Frob_v))$. It follows that $\rho_{\varphi}$ is isomorphic to $\chi\otimes \rho_f$, and therefore that $\chi\otimes \rho$ is modular of weight $(\vec{k},\vec{\ell}+\vec{m})$.

The converse holds by symmetry.

(2) We claim that $\chi|_{G_{F_\gp}}$ has a crystalline lift $\widetilde{\chi}_{\gp}$ whose $\sigma$-labelled Hodge--Tate weight is $-\ell_\sigma$.  Indeed, this follows from \cite[App.~B]{brian} 
on taking $\widetilde{\chi}_{\gp}$ to correspond via local class field theory to a lift of $\xi|_{F_\gp}^\times$ whose restriction to $\CO_{F,\gp}^\times$ is $\prod_{\sigma\in \Sigma_\gp}\sigma^{-\ell_\sigma}$.  Therefore if $\widetilde{\rho}_\gp$ is a crystalline lift of $\rho|_{G_{F_\gp}}$ of weight $(\vec{k}_\gp,\vec{m}_\gp)$, then 
$\widetilde{\chi}_\gp \otimes \widetilde{\rho}_\gp$ is a crystalline lift of $(\chi\otimes\rho)|_{G_{F_\gp}}$ of weight $(\vec{k}_\gp,\vec{\ell}_\gp+\vec{m}_\gp)$.  This gives one direction of the implication, and the converse again holds by symmetry.
\end{proof}

\begin{remark} Recall that conjectures along the lines of those above are formulated in \cite{DS1} under the assumption that $p$ is unramified in $F$.  However, due to the different conventions in the definition of the Hecke operators, the Galois representation associated to an eigenform $f$ in~\cite{DS1} is $\chi\otimes \rho_f$ in the notation of this paper, where $\chi$ is the cyclotomic character (see Remark~\ref{rmk:galconv}).  Therefore $\rho$ is modular of weight $(\vec{k},\vec{m})$ in the notation of~\cite{DS1} if and only if $\chi^{-1}\otimes \rho$ is modular of weight $(\vec{k},\vec{m})$ in the notation of this paper, which, by the preceding proposition, is equivalent to $\rho$ being modular of weight $(\vec{k},\vec{m}-\vec{1})$

The analogous statement holds for our conventions with respect to weights of crystalline lifts, i.e., $\rho|_{G_{F_{\gp}}}$ has a crystalline lift of weight $(\vec{k}_\gp,\vec{m}_\gp)$ in the notation of \cite{DS1} if and only if it has a crystalline lift of weight $(\vec{k}_\gp,\vec{m}_\gp-1)$ in the notation of this paper.

It follows that if $p$ is unramified in $F$, then Conjectures~\ref{conj:geomweights} and~\ref{conj:geomweights2} above are equivalent to Conjectures~7.3.1 and~7.3.2 of \cite{DS1}, but with $\vec{k}_{\min}(\rho,\vec{m})$ of that paper equal to 
$\vec{k}_{\min}(\chi^{-1}\otimes\rho,\vec{m}) = \vec{k}_{\min}(\rho,\vec{m}-\vec{1})$ of this paper (where $\chi$ is the cyclotomic character).
\end{remark}

Finally we note the following consequence of the effect of 
$\Theta$-operators.
\begin{proposition} \label{prop:theta_effect} Let $\rho: G_F \rightarrow \mathrm{GL}_2(\overline{\mathbb{F}}_p)$ be an irreducible, continuous representation.  Suppose that 
$\vec{k},\vec{m} \in \ZZ^{\Sigma}$ and $\varsigma \in \Sigma$. If $\rho$ is geometrically modular of weight $(\vec{k},\vec{m})$, 
then $\rho$ is geometrically modular of weight
$(\vec{k}+\vec{h}_\varsigma+2\vec{e}_\varsigma,\vec{m}-\vec{e}_\varsigma)$. 
\end{proposition}
\begin{proof} Suppose that $\rho$ is geometrically modular of weight $(\vec{k},\vec{m})$, and let $f \in M_{\vec{k},\vec{m}}(U,E)$ be an eigenform giving rise to $\rho$.  We first treat that the case where $\varsigma = \sigma_{\gp,i,e_{\gp}}$ for some $\tau = \tau_{\gp,i}\in \Sigma_0$. We wish to reduce to the case that $f \not\in \ker(\Theta_\tau)$.  Recall that $\ker(\Theta_\tau)$ is described in Theorem~\ref{thm:kertheta} in terms of the images of the maps
$V_\gp: M_{\vec{\ell},\vec{n}}(U,E) \hookrightarrow
 M_{\vec{\ell}',\vec{n}'}(U,E)$,
where $\vec{\ell}' = \vec{\ell} + \sum_{\sigma\in \Sigma_\gp} \ell_\sigma\vec{h}_\sigma$ and $\vec{n}' = \vec{n} + \sum_{\sigma\in \Sigma_\gp} n_\sigma\vec{h}_\sigma$.

First we claim that we may assume $f$ is not divisible by any partial Hasse invariant $H_\sigma$.  Indeed, we may write $f = \phi\prod_{\sigma\in \Sigma} H_\sigma^{r_\sigma}$ for some $\vec{r} \in \ZZ^{\Sigma}_{\ge 0}$ and $\phi \in M_{\vec{k}',\vec{m}}(U,E)$ not divisible by any $H_\sigma$, where $\vec{k}' = \vec{k}_{\min}(f)$ and $\vec{k} - \vec{k}' \in \Xi_{\mathrm{Hasse}}$.  If the proposition holds for $\phi$, then we have that $\rho$ is geometrically modular of weight $(\vec{k}'+\vec{h}_\varsigma+2\vec{e}_\varsigma,\vec{m}-\vec{e}_\varsigma)$, and hence of weight $(\vec{k}+\vec{h}_\varsigma+2\vec{e}_\varsigma,\vec{m}-\vec{e}_\varsigma)$ by Proposition~\ref{prop:Hasse_effect}.

Next we show that there is a maximal non-negative integer $n$ such that 
$$f = V_\gp^n(\phi) \prod_{\sigma\in \Sigma} G_\sigma^{r_\sigma}$$ for some $\phi \in M_{\vec{k}',\vec{m}'}(U,E)$, $\vec{k}',\vec{m}',\vec{r}\in \ZZ^{\Sigma}$.
Indeed since $V_\gp$ and multiplication by the $G_\sigma$ are Hecke-equivariant, the formula implies that $\phi$ is an eigenform giving rise to $\rho$, whose irreducibility then implies that $\vec{k}' \neq \vec{0}$.  Furthermore since $f$ is not divisible by any $H_\sigma$, neither is $\phi$ (see \cite[\S9.3]{theta} for the effect of $V_\gp$ on the $H_\sigma$ and $G_\sigma$), so Theorem~\ref{thm:positivity} implies that $\vec{k}' = \vec{k}_{\min}(\phi) \in \Xi_{\min}^+$.  The boundedness of $n$ then follows from the fact that in the definition of $V_\gp$, if each $\ell_\sigma > 0$, then $\sum_{\sigma\in \Sigma_\gp} \ell'_\sigma > \sum_{\sigma\in \Sigma_\gp} \ell_\sigma$.

Now note that for $\phi$ as above, we have $\Theta_\tau(\phi)\neq 0$.  Indeed, if $\phi \in \ker(\Theta_\tau)$, then Theorem~\ref{thm:kertheta} implies that 
$$\phi = V_\gp(g) \prod_{\sigma\in \Sigma} G_\sigma^{s_\sigma}$$
for some $g \in M_{\vec{\ell},\vec{n}}(U,E)$, $\vec{\ell},\vec{n},\vec{s}\in \ZZ^{\Sigma}$ (since $\phi$ is not divisible by any $H_\sigma$), which contradicts the characterization of $n$.
Therefore $\rho$ arises from $\Theta_\tau(\phi)$, and hence is geometrically modular of weight $(\vec{k}'+\vec{h}_\varsigma+2\vec{e}_\varsigma,\vec{m}'-\vec{e}_\varsigma)$.
Since $\vec{k} -\vec{k}' \in \Xi_{\mathrm{Hasse}}$ and $\lambda(\vec{m}) = \lambda(\vec{m}')$ (in the notation of the proof of Proposition~\ref{prop:chars}), it follows 
that $\rho$ is geometrically modular of weight $(\vec{k}+\vec{h}_\varsigma+2\vec{e}_\varsigma,\vec{m}-\vec{e}_\varsigma)$
(by Propositions~\ref{prop:Hasse_effect} and ~\ref{prop:chars}).

Finally suppose that $\varsigma = \sigma_{\gp,i,j}$ for some $j < e_\gp$.  Note that
$$(\vec{h}_\varsigma + 2\vec{e}_\varsigma) - (\vec{h}_{\varsigma'} + 2\vec{e}_{\varsigma'}) = \vec{h}_\varsigma + \vec{h}_{\varsigma'} \in \Xi_{\mathrm{Hasse}}\quad\mbox{and}\quad \lambda(\vec{e}_\varsigma)
 = \lambda(\vec{e}_{\varsigma'}),$$
where $\varsigma =  \sigma_{\gp,i,j+1}$.
The desired conclusion therefore follows (by induction on $e_\gp - j$) from Propositions~\ref{prop:Hasse_effect} and ~\ref{prop:chars}.
\end{proof}

\begin{remark} \label{rmk:BM} Under the assumption that $k_\sigma \ge 2$ for all $\sigma \in \Sigma$, a statement analogous to Proposition~\ref{prop:theta_effect} for crystalline lifts follows from the Breuil--M\'ezard Conjecture and a modular representation theory calculation.  In the same vein, so does an analogue of Proposition~\ref{prop:Hasse_effect}, under the further assumption that $\vec{k}$ and $\vec{k} + \vec{\ell} \in \Xi_{\min}$.
\end{remark}

\begin{remark} \label{rmk:theta_effect} If there exist $\vec{k}_{\min}(\rho,\vec{m})$ as in part (1) of Conjecture~\ref{conj:geomweights}, then Proposition~\ref{prop:theta_effect} implies that
$$\vec{k}_{\min}(\rho,\vec{m} - \vec{e}_\sigma)
 \le \vec{k}_{\min}(\rho,\vec{m}) + \vec{h}_\sigma + 2\vec{e}_\sigma$$
where $\sigma = \sigma_{\gp,i,e_\gp}$ and the partial ordering is defined by $\vec{k} \le \vec{k}'$ if $\vec{k}' - \vec{k} \in \Xi_{\mathrm{Hasse}}$.  (The same holds for 
$\sigma = \sigma_{\gp,i,j}$ if $j < e_\gp$, but in view of the final step in the proof of Proposition~\ref{prop:theta_effect}, this is a weaker assertion.) Furthermore, it follows from \cite[Thm.~5.2.1]{theta} and the proof of Proposition~\ref{prop:theta_effect} that if $\vec{k} =  \vec{k}_{\min}(\rho,\vec{m})$ satisfies $p|k_\sigma$, then in fact
$$\vec{k}_{\min}(\rho,\vec{m} - \vec{e}_\sigma)
 \le \vec{k}(\rho,\vec{m}) + 2\vec{e}_\sigma.$$

Recall also from Proposition~\ref{prop:chars} (and its proof) that $\vec{k}_{\min}(\rho,\vec{m})$ will depend only on the image $\lambda(\vec{m}) \in \Hom(T,\Fpbar^\times)$, where $T = \prod_{\gp \in S_p} \F_{\gp}^\times$.  We may view the resulting function
$\Hom(T,\Fpbar^\times) \to \Xi_{\min}^+$
(defined by $\lambda(\vec{m}) \to \vec{k}_{\min}(\rho,\vec{m})$)
as a generalized ``$\Theta$-cycle.''  Note however that the information is slightly cruder than what is provided by the original notion in the classical setting (as in \cite[\S3]{edixhoven}), since it fails to distinguish between ordinary and non-ordinary forms giving rise to the same Galois representation.
\end{remark}

\subsection{Algebraic modularity} \label{ssec:algmod}
We now discuss the relation with ``algebraic'' Serre weight conjectures in this context, as formulated in \cite{BDJ} for $p$ unramified in $F$, and then generalized to the ramified case by Schein~\cite{Schein} and Gee~\cite{Gtypes}.   Recall that (for $F$ unramified at $p$) a weight $V$ in \cite{BDJ} is an $\Fpbar$-representation of $\GL_2(\CO_F/p)$, and the notion of a representation $\rho:G_F \to \GL_2(\Fpbar)$ being modular of weight $V$ is defined terms of the manifestation of $\rho$ in Jacobians of certain Shimura curves over $F$, or equivalently, the appearance of $\rho$ in their cohomology with coefficients in an \'etale sheaf determined by $V$ (see~\cite[Lemma~2.4]{BDJ}).
Furthermore, this is in turn equivalent to the presence of the corresponding system of Hecke eigenvalues arising in such cohomology (see~\cite[Lemmas~4.10, 4.11]{BDJ}).  One can make similar definitions in terms of Hecke eigensystems on forms associated to totally definite quaternion algebras (as in \cite{gee_inv}).

To adapt these notions to our setting, let $\CO_B$ be a maximal order in a quaternion algebra $B$ over $F$, where $B$ is as in \S\ref{ssec:defQO} or \S\ref{ssec:indefQO}, so $B$ is unramified at all $\gp \in S_p$ and at most one archimedean place of $F$.  In either case, we defined a space $M_{\vec{k},\vec{m}}^B(U,\Fpbar)$ for any $\vec{k}\in \ZZ_{\ge 2}^\Sigma$, $\vec{m} \in \ZZ^\Sigma$ and sufficiently small open compact $U \subset B_{\A,\f}^\times$ containing $\CO_{B,p}^\times$, and equipped it with an action of (commuting) Hecke operators $T_v$ and $S_v$ for all primes $v \not\in S_p$ such that $B$ is unramified at $v$ and $\CO_{B,v}^\times \subset U$.  We thus have an action of $\TT^Q$ on $M_{\vec{k},\vec{m}}^B(U,\Fpbar)$, where $\TT^Q$ is the polynomial ring over $\Fpbar$ in the variables $T_v$ and $S_v$ for all $v \not\in Q$, and $Q$ is any finite set of primes of $F$ containing all $v$ such that $v\in S_p$, $B_v^\times$ is ramified, or $\CO_{B,v}^\times \not\subset U$.

Suppose now that $\rho:G_F \to \GL_2(\Fpbar)$ is any continuous, irreducible, totally odd representation, and $Q$ is a finite set of primes containing all those in $S_p$ and those at which $\rho$ is ramified.  We let $\gm_\rho^Q$ denote the maximal ideal of $\TT^Q$ 
generated by the elements
$$T_v - \tr(\rho(\Frob_v)) \quad\mbox{and}\quad 
 \Nm_{F/\Q}(v)S_v - \det(\rho(\Frob_v))$$
for all $v\not\in Q$.

We then make the following definition:
\begin{definition} \label{def:algmod} For $\rho$, $\vec{k}$ and $\vec{m}$ as above, we say that $\rho$ is {\em definitely modular} (resp.~{\em algebraically modular}) of weight $(\vec{k},\vec{m})$ if 
$$M_{\vec{k},\vec{m}}^B(U,\Fpbar)[\gm_\rho^Q] \neq 0$$
for some $B$, $U$ and $Q$ as above, where $B$ is totally definite (resp.~unramified at exactly one archimedean place).
\end{definition}
Note that since $M_{\vec{k},\vec{m}}^B(U,\Fpbar)$ is finite-dimensional, the action on it of $\TT^Q$ factors through an Artinian quotient.  It follows that the
condition in the definition is equivalent to
the non-vanishing of $M_{\vec{k},\vec{m}}^B(U,\Fpbar)_{\gm_\rho^Q}$, or equivalently the containment $\Ann_{\TT^Q}(M_{\vec{k},\vec{m}}^B(U,\Fpbar)) \subset \gm_\rho^Q$.

We note also that if $\gm_\rho^Q$ is in the support of $M_{\vec{k},\vec{m}}^B(U,\Fpbar)$, then so is $\gm_\rho^{Q'}$ for all finite $Q' \supset Q$.  Furthermore, if $\CO_{B,p}^\times \subset U' \subset U$ and $Q'$ contains all $v$ such that $\CO_{B,v}^\times \not\subset U'$, then $\gm_\rho^{Q'}$ is also in the support of $M_{\vec{k},\vec{m}}^B(U',\Fpbar)$.  This is clear if $B$ is totally definite; otherwise, we need to show that $\gm_\rho^{Q'}$ is not in the kernel of
the map $M_{\vec{k},\vec{m}}^B(U,\Fpbar) \to M_{\vec{k},\vec{m}}^B(U',\Fpbar)$ induced by the projection $Y_{U'}^B \to Y_U^B$, for which we may assume that $U'$ is normal in $U$.  As in the proof of Theorem~\ref{thm:galois}, the Hochschild--Serre spectral sequence identifies this kernel with
$H^1(G,H^0(Y_{U'}^B,\mathcal{D}_{\vec{k},\vec{m},\Fpbar}))$ as a $\TT^{Q'}$-module, where $G = U/U'(U\cap \CO_F^\times)$, and a standard argument shows that $H^0(Y_{U'}^B,\mathcal{D}_{\vec{k},\vec{m},\Fpbar})_{\gm_{\rho}^{Q'}} = 0$ (see for example the proof of Lemma~\ref{lem:JHfactors} below).

The above notion of algebraic modularity generalizes the one in \cite{BDJ} and \cite{DS1}, but with different conventions.
More precisely, if $p$ is unramified in $F$, then $\rho$ is algebraically modular of weight $(\vec{k},\vec{m})$ in the sense above if and only if it is algebraically modular of $(\vec{k},\vec{m}+\vec{1})$ in the sense of \cite{DS1}, or equivalently modular of weight $D_{\vec{k},-\vec{k}-\vec{m}-\vec{1},\Fpbar}$
in the sense of \cite{BDJ}, where 
$$D_{\vec{k},-\vec{k}-\vec{m}-\vec{1},\Fpbar} = \bigotimes_{\sigma \in \Sigma} (\det\nolimits^{-k_\sigma-m_\sigma}\otimes_{\Fpbar} \Sym^{k_\sigma-2} \Fpbar^2)$$ 
(as defined in \S\ref{ssec:defQO}) is viewed as a representation of $\GL_2(\CO_F/p\CO_F)$.

We note that part (1) of both Propositions~\ref{prop:chars} and~\ref{prop:chartwist} hold with ``geometrically'' replaced by either ``definitely'' or ``algebraically.'' Indeed in the case of Proposition~\ref{prop:chars}, this is immediate from the definitions, and for Proposition~\ref{prop:chartwist}, the same proof as for geometric modularity carries over.

Recall that Conjecture~7.5.2 of \cite{DS1} 
predicts (for $p$  unramified in $F$) that if $\rho$ is algebraically modular of weight $(\vec{k},\vec{m})$ (where $\vec{k} \in \ZZ^\Sigma_{\ge 2})$, then it is algebraically modular of the same weight, and that the converse holds if in addition $\vec{k} \in \Xi_{\rm min}$.
We extend the conjecture to the current setting as follows:
\begin{conjecture}  \label{conj:defalggeom} Suppose that $\rho:G_F \to \GL_2(\Fpbar)$ is a continuous, irreducible, totally odd representation, and that $(\vec{k},\vec{m}) \in \ZZ^\Sigma_{\ge 2} \times \ZZ^\Sigma$.
\begin{enumerate}
\item $\rho$ is definitely modular of weight $(\vec{k},\vec{m})$ if and only if $\rho$ is algebraically modular of weight $(\vec{k},\vec{m})$.
\item If $\rho$ is algebraically modular of weight $(\vec{k},\vec{m})$ then $\rho$ is geometrically modular of weight $(\vec{k},\vec{m})$; furthermore the converse holds if $\vec{k} \in \Xi_{\rm min}$.
\end{enumerate}
\end{conjecture}
Under a Taylor--Wiles hypothesis on $\rho$, part (1) follows from results of Gee--Liu--Savitt~\cite{GLS} (see Theorem~\ref{thm:GLSmain} and Corollary~\ref{cor:GLSmain} below). As explained in \cite{DS1}, the first assertion in (2) is straightforward if all $k_\sigma$ have the same parity.  Here we prove the forward directions of (1) and (2) unconditionally.  The other direction of (2) is more difficult, but some cases are proved in S.~Yang's PhD thesis~\cite{SYPhD} (see also~\cite{SY}).
  
Before proving the implications, we recall the following well-known behavior of weights with respect to definite (resp.~algebraic) modularity.  Note that the action of $\GL_2(\CO_{F,p})$ on $D_{\vec{k},\vec{m},\Fpbar}$ factors through $\prod_{\gp \in S_p}\GL_2(\FF_\gp)$, and recall that this representation is irreducible if and only if $\vec{k}$ has the property that for each $\tau = \tau_{\gp,i} \in \Sigma_0$, we have 
$$\mbox{$2 \le k_{\gp,i,j_\tau} \le p + 1$ for some $j_\tau \in {1,\ldots,e_\gp}$,}\quad \mbox{and $k_{\gp,i,j}=2$ for all $j \neq j_\tau$}.$$
Observe also that such a $D_{\vec{k},\vec{m},\Fpbar}$ depends only on the pair $(\lambda(\vec{m}),\vec{k}^0)$, where 
$\lambda(\vec{m})$ is a character of $\prod_{\gp \in S_p}\F_{\gp}^\times$ (as in the proof of Proposition~\ref{prop:chars}) and $\vec{k}^0 \in \ZZ^{\Sigma_0}$ is defined by $k^0_{\gp,i} = \max\{k_{\gp,i,1},\ldots,k_{\gp,i,e_\gp} \}$.  In particular, it does not depend on the $j_\tau$ such that $k^0_{\gp,i} = k_{\gp,i,j_\tau}$ (but such a $\vec{k}$ is in $\Xi_{\rm min}$ if and only if $k^0_{\gp,i} = k_{\gp,i,e_\gp}$ for all $\tau \in \Sigma_0$). Recall also that for such $\vec{k}$, the representation $D_{\vec{k},\vec{m},\Fpbar}$ is dual to $D_{\vec{k},-\vec{k}-\vec{m},\Fpbar}$.  The following lemma reduces the question of definite (resp.~algebraic) modularity to consideration of such weights.

\begin{lemma}  \label{lem:JHfactors} Suppose that $\rho:G_F \to \GL_2(\Fpbar)$ is a continuous, irreducible, totally odd representation, and that $(\vec{k},\vec{m}) \in \ZZ^\Sigma_{\ge 2} \times \ZZ^\Sigma$.  Then $\rho$ is definitely (resp.~algebraically) modular of weight $(\vec{k},\vec{m})$ if and only if $\rho$ is definitely (resp.~algebraically) modular of some weight $(\vec{\ell},\vec{n})$ such that $D_{\vec{\ell},\vec{n},\Fpbar}$ is a Jordan--H\"older factor of $D_{\vec{k},\vec{m},\Fpbar}$.
\end{lemma}
\begin{proof} We only sketch the argument, which is standard.

Consider first the case of algebraic modularity.  Adopting the notation of \S\ref{ssec:indefQO}, an exact sequence $0 \to T \to T' \to T'' \to 0$ of $\CO$-modules with smooth actions of $U_p = \CO_{B,p}^\times$ yields an exact sequence
$$ 0 \longrightarrow \CT \lra \CT' \lra \CT'' \lra 0$$
of locally constant sheaves on $Y_B^U$, where $U = U_pU^p$ for sufficiently small $U^p$.  Furthermore the morphisms in the resulting long exact sequence
$$H^0(Y_U^B,\CT'')\to H^1(Y_U^B,\CT) \to H^1(Y_U^B,\CT') \to H^1(Y_U^B,\CT'') \to H^2(Y_U^B,\CT)$$ 
are $\TT^Q$-equivariant.  It follows that if $\gm_\rho^Q$ is in the support of $M_{\vec{k},\vec{m}}(U,\Fpbar) = H^1(Y_U^B,\mathcal{D}_{\vec{k},\vec{m},\Fpbar})$, then it is in the support of 
$M_{\vec{\ell},\vec{n}}(U,\Fpbar) = H^1(Y_U^B,\mathcal{D}_{\vec{\ell},\vec{n},\Fpbar})$ for some Jordan--H\"older factor  $D_{\vec{\ell},\vec{n},\Fpbar}$ of $D_{\vec{k},\vec{m},\Fpbar}$.

To prove the converse, it suffices to show that $\gm_\rho^Q$ is not in the support of $H^i(Y_U^B,\mathcal{D}_{\vec{\ell},\vec{n},\Fpbar})$ for any such Jordan-H\"older factor and $i=0$ or $2$. Furthermore, we claim that these spaces vanish for such $\ell$, unless $\vec{\ell} = \vec{0}$.  For $i=0$, this follows from the fact that the sections on each connected component are identified with the invariants in $D_{\vec{\ell},\vec{n},\Fpbar}$ under the action of $B_+^\times \cap gUg^{-1}$ for some $g \in (B_{\f}^{(p)})^\times$, and the image of this group in 
$\prod_{\gp\in S_p}\GL_2(\FF_\gp)$ contains $\prod_{\gp\in S_p}\SL_2(\FF_\gp)$ by the Strong Approximation Theorem.  The vanishing for $i = 2$ then follows from Poincar\'e Duality (and the duality between $D_{\vec{k},\vec{m},\Fpbar}$ and $D_{\vec{k},-\vec{k}-\vec{m},\Fpbar}$).  

We now assume $\vec{\ell} = \vec{0}$ and analyze the Hecke action on $H^0(Y_U^B, \mathcal{D}_{\vec{0},\vec{n},\Fpbar})$ by identifying it with the space of $U^p$-invariants in the smooth $\Fpbar$-representation
$$\varinjlim H^0(Y_U^B, \mathcal{D}_{\vec{0},\vec{n},\Fpbar})
 = \varinjlim H^0(C_U^B, \pi_{U,*}\mathcal{D}_{\vec{0},\vec{n},\Fpbar})$$
of $(B_{\f}^{(p)})^\times$, where the limit is over sufficiently small $U^p$ and $\pi_U:Y_U^B \to C_U^B = \AA_{F,\f}^\times/F_+^\times\det(U)$ is the projection to the set of connected components.  Furthermore $\pi_{U,*}\mathcal{D}_{\vec{0},\vec{n},\Fpbar}$ may be identified with the locally constant sheaf over $C_U$ defined by sections of
$$ (\AA_{F,\f}^\times \times \Fpbar(\xi_{\vec{n}}))/F_+^\times\det(U),$$
where the character $\xi_{\vec{n}}$ of $F_+^\times\det(U)$ is trivial on $F_+^\times\det(U^p)$ and (the inflation of) $x \mapsto x^{\vec{n}}$ on $\CO_{F,p}^\times$.  In particular, it follows that if $v\not\in Q$ splits completely in the abelian extension $L$ of $F$ corresponding via class field theory to $\ker(\xi_{\vec{n}})$, then
$$T_v = 1 + \Nm_{F/\QQ}(v)\quad\mbox{and}\quad S_v = 1$$
on $H^0(Y_U^B, \mathcal{D}_{\vec{0},\vec{n},\Fpbar})$.
So if $\gm_\rho^Q$ lies in its support, then the Cebotarev Density and Brauer--Nesbitt Theorems imply that $\rho|_{G_L}^{\mathrm{ss}}$ is isomorphic to $1 + \chi_{\cyc}^{-1}$, which in turn implies that $\rho$ is reducible, yielding the desired contradiction.
Using Poincar\'e Duality, we similarly obtain a contradiction to $\gm_\rho^Q$ being in the support of $H^2(Y_U^B, \mathcal{D}_{\vec{0},\vec{n},\Fpbar})$.

The proof for definite modularity is similar, but much easier, since it only involves $H^0$.  The result therefore follows directly from the exactness of sequences of the form
$$0 \lra  H^0(Y_U^B,\CT) \lra H^0(Y_U^B,\CT') \lra H^0(Y_U^B,\CT'') \lra 0.$$ 
\end{proof}

We recall also an analogue of Ihara's Lemma in the setting of totally definite quaternion algebras, essentially due to Taylor~\cite{Taylor}. Let $B$ be a totally definite quaternion algebra over $F$ unramified at all $\gp \in S_p$, $U$ a sufficiently small open compact subgroup of $B_{\A,\f}^\times$ containing $\CO_{B,p}^\times$, and $Q$ a sufficiently large finite set of primes of $F$.  For $w \not\in Q$, choose an element  $h \in \CO_{B,w}\cong M_2(\CO_{F,w})$ such that $\det(h)$ is a uniformizer, let $U_0(w) = U \cap hUh^{-1}$, and  consider the level-raising map 
$$\beta:M^B_{\vec{k},\vec{m}}(U,\Fpbar)^2 \lra M^B_{\vec{k},\vec{m}}(U_0(w),\Fpbar)$$
defined by $\beta(f_1,f_2)(x) = f_1(x) + f_2(x h)$.  Note that $\beta$ is $\TT^{Q'}$-linear, where $Q' = Q \cup \{w\}$.

The analogue of Ihara's Lemma is the following, for which our proof is based on the one in \cite[Lemma~2]{DT}.  
\begin{lemma} \label{lem:ihara} Suppose that $B$, $U$, $Q$, $w$ and $\beta$ are as above, and $\rho$, $\vec{k}$ and $\vec{m}$ are as in Conjecture~\ref{conj:defalggeom}.  Then $\gm_{\rho}^{Q'}$ is not in the support of $\ker(\beta)$.
\end{lemma}
\begin{proof} If $f_1$, $f_2 \in M_{\vec{k},\vec{m}}^B(U,\Fpbar)$ are such that ${\beta}(f_1,f_2) = 0$, then $f_1(x) = -f_2(xh)$ for all $x \in B_{\A}^\times$.  It follows that $f_1(\gamma_{\f} x u) = u_p^{-1}f_1(x)$ for all $\gamma \in B^\times$, $x \in B_{\f}^\times$ and $u \in \langle U, h^{-1}U h \rangle = U B_w^1$,
where $B_w^1$ is the kernel of $\det:B_w^\times \to F_w^\times$. In particular, $f_1(\gamma_{\f} x u) = f_1(x)$ for all $\gamma \in B^\times$, $x \in B_{\f}^\times$ and $u \in U' B_w^1$,
where $U' = U^p U'_p$ and $U_p' \subset \CO_{B_p}^\times$ is the principal congruence subgroup of level $\prod_{\gp\in S_p}\gp$.

Suppose now that $v\not\in Q'$ splits completely in the abelian extension $L$ of $F$ corresponding via class field theory to $F_+^\times\det(U')$.  Then for any $x\in B_{\f}^\times$ and $g \in B_v^\times$, the Strong Approximation Theorem implies that $xg = \gamma_{\f} xu$ for some $\gamma \in B^\times$, $u \in U'B_w^1$, from which it follows that $T_vf_1 = (\Nm_{F/\QQ}(v)+1)f_1 = 2f_1$ and $S_vf_1 = f_1$. Thus if $\gm_{\rho}^{Q'}$ is in the support of $\ker(\beta)$, then $T_v - 2, S_v - 1 \in \gm_\rho^{Q'}$ for all such $v$, which (by the Chebotarev Density and Brauer--Nesbitt Theorems) implies that $\rho|_{G_L}$ has trivial semi-simplification, contradicting the irreducibility of $\rho$.
\end{proof}

We now prove one direction of each part of Conjecture~\ref{conj:defalggeom}.

\begin{theorem} \label{thm:defalggeom}
Suppose that $\rho$, $\vec{k}$ and $\vec{m}$ are as in Conjecture~\ref{conj:defalggeom}.
\begin{enumerate}
\item If $\rho$ is definitely modular of weight $(\vec{k},\vec{m})$,
then $\rho$ is algebraically modular of weight $(\vec{k},\vec{m})$.
\item If $\rho$ is algebraically modular of weight $(\vec{k},\vec{m})$, then $\rho$ is geometrically modular of weight $(\vec{k},\vec{m})$.
\end{enumerate}
\end{theorem}
\begin{proof}
(1) Firstly, by Lemma~\ref{lem:JHfactors}, we may assume that $D_{\vec{k},\vec{m},\Fpbar}$ is irreducible, and in particular that $2\le k_\sigma \le p+1$ for all $\sigma \in \Sigma$.

Let $B$, $U$ and $Q$ be as in Definition~\ref{def:algmod}, with $B$ totally definite, so there exists a non-zero $f\in M_{\vec{k},\vec{m}}(U,\Fpbar)$ (for some sufficiently small open compact subgroup $U$ of level prime to $p$ and sufficiently large finite set of primes $Q$) such that $T_v f = a_v f$ and $S_v f = d_v f$ for all $v \not\in Q$, where
$$a_v  = \tr(\rho(\Frob_v))\quad\mbox{and}\quad d_v = \Nm_{F/\QQ}(v)^{-1}\det(\rho(\Frob_v)).$$

As usual, we choose a character $\xi:(\A_{F,\f}^{(p)})^\times \to \Fpbar^\times$ such that $\xi(\alpha) = \alpha^{\vec{m}+\vec{k}/2}$ for all $\alpha \in \CO_{F,(p),+}^\times$ and (after shrinking $U$ and enlarging $Q$ as needed) apply the isomorphism
$$e_\xi: M^B_{\vec{k}}(U,\Fpbar) \stackrel{\sim}{\longrightarrow}
M^B_{\vec{k},\vec{m}}(U,\Fpbar)$$
to obtain an eigenform $\varphi \in M^B_{\vec{k}}(U,\Fpbar)$ such that $T_v \varphi = a_v'\varphi$ and $S_v \varphi = d_v'\varphi$ for all $v \not\in Q$,  where $a_v' = \xi(\varpi_v)^{-1}a_v$ and $d_v' = \xi(\varpi_v)^{-2}d_v$.

Let $L = \overline{F}^{\ker(\rho)}$ denote the splitting field of $\rho$.  By the Chebotarev Density Theorem, we may choose a prime $w \not\in Q$ such that the conjugacy class of $\Frob_w$ in $\Gal(L(\zeta_p)/F)$ is that of a complex conjugation.  Since $\rho$ is totally odd and $\chi_{\mathrm{cyc}}(\Frob_w) = -1$, we have $a_w = \tr(\rho(\Frob_w)) = 0$ and $\Nm_{F/\Q}(w) \equiv - 1 \bmod p$.

Now consider the level-raising map
$$\beta':M^B_{\vec{k}}(U,\CO)^2 \lra M^B_{\vec{k}}(U_0(w),\CO)$$
defined by $\beta'(\varphi_1,\varphi_2)(x) = \varphi_1(x) + \varphi_2(x h)$, where $U_0(w)$ and $h$ are as in the definition of $\beta$ (preceding Lemma~\ref{lem:ihara}).  Thus $\beta'$ is $\TT$-linear, where $\TT$ is the polynomial ring over $\CO$ in the variables $T_v$ and $S_v$ for $v \not\in Q' = Q \cup \{w\}$.  Let $\gm$ denote the kernel of the homomorphism $\TT \to E$ defined by $T_v \mapsto a_v'$, $S_v \mapsto d_v'$.  It follows from Lemma~\ref{lem:ihara} and the commutativity of the diagram
$$\xymatrix{
M^B_{\vec{k}}(U,\Fpbar)_{\gm}^2 
\ar[r]^-{\beta'_{\gm}\otimes_{\CO}\Fpbar} 
\ar[d]^{\wr}_{(e_\xi,\,\xi(\varpi_w)^{-1}e_{\xi})} 
& M^B_{\vec{k}}(U_0(w),\Fpbar)_{\gm} 
\ar[d]^{\wr}_{e_\xi} \\
M^B_{\vec{k},\vec{m}}(U,\Fpbar)_{\gm_\rho^{Q'}}^2 
\ar[r]^-{\beta_{\gm_\rho^{Q'}}}
& M^B_{\vec{k},\vec{m}}(U_0(w),\Fpbar)_{\gm_\rho^{Q'}}}$$
that the reduction
$$\overline{\beta}'_{\gm} : M^B_{\vec{k}}(U,E)^2_{\gm} \longrightarrow M^B_{\vec{k}}(U_0(w),E)_{\gm}$$
of $\beta'_{\gm}$ is injective, and hence that $\beta'_{\gm}$ is injective with torsion-free cokernel.

Let us note also that the same argument, applied with $\vec{m}$ replaced by $-\vec{k}-\vec{m}$, $\rho$ replaced by $\chi_\cyc^{-1}\otimes \rho^\vee$ and $\xi$ replaced by $\xi^{-1}$ shows that the same conclusions hold with $\gm$ replaced by $\gn$, where $\gn$ is the kernel of the homomorphism $\TT \to E$ defined by $T_v \mapsto d_v'^{-1}a_v'$, $S_v \mapsto d_v'^{-1}$.

Next recall that since $2\le k_\sigma  \le p + 1$ for all $\sigma$, the $\CO_{B,(p)}^\times$-module $D_{\vec{k},E}$ is equipped with a perfect $E$-bilinear pairing
$$\langle \,\,\,,\,\, \rangle: D_{\vec{k},E} \times D_{\vec{k},E} \to E$$
satisfying $\langle \gamma a ,\gamma b \rangle = \langle a , b\rangle$ for all $\gamma \in \CO_{B,(p)}^\times$, $a,b\in D_{\vec{k},E}$ (see \cite{frazer} for an explicit description).
We thus have a perfect pairing on $M^B_{\vec{k}}(U',E)$, for any sufficiently small $U'$ containing $\CO_{B,p}^\times$, defined by
$$\langle \varphi_1,\varphi_2 \rangle = \sum_x \langle \varphi_1(x),\varphi_2(x) \rangle$$
where the sum is over double coset representatives, i.e., 
$(B_{\f}^{(p)})^\times = \coprod_x \CO_{B,(p)}^\times x U'^p$.
One sees also that for any $g \in (B_{\f}^{(p)})^\times$, the adjoint of the double coset operator $[UgU']$ is $[U'g^{-1}U]$.
In particular, the adjoint of $T_v$ (resp.~$S_v$) on $M^B_{\vec{k}}(U,E)$ is $S_v^{-1}T_v$ (resp.~$S_v^{-1}$) for all $v \not\in Q$.
Furthermore, extending the pairing diagonally to a 
perfect pairing on $M^B_{\vec{k}}(U,E)^2$, the adjoint of 
$\overline{\beta}'$ is the map
$$\overline{\alpha}' = ([U_0(w) 1 U],[U_0(w)h^{-1}U]): M^B_{\vec{k}}(U_0(w),E) \longrightarrow M^B_{\vec{k}}(U,E)^2$$
In particular $\overline{\alpha}'$ is $\TT$-linear, and $\overline{\alpha}'_{\gm}$ is surjective (since $\overline{\beta}'_{\gn}$ is injective).

Now consider the composite 
$\overline{\alpha}'\circ\overline{\beta}'$, which is the $\TT$-linear endomorphism of $M_{\vec{k}}(U,E)^2$ given by the matrix
$$\left(\begin{array}{cc} \Nm_{F/\QQ}(w) + 1 & T_w \\
 S_w^{-1}T_w & \Nm_{F/\Q}(w) + 1 \end{array} \right).$$
Since $(\varphi,0) \in \ker(\overline{\alpha}'\circ\overline{\beta}')[\gm]$, the localization $\overline{\alpha}'_{\gm}\circ\overline{\beta}'_{\gm}$ is not an isomorphism, so its cokernel is non-trivial.  Since $\overline{\alpha}'_{\gm}$ is surjective, it follows that $\beta'_{\gm}$ has non-trivial cokernel.  Furthermore, since this cokernel is torsion-free, it follows that there is a prime ideal $\gp$ of $\TT$ in the support of $\coker(\beta')$ such that $\gp \cap \CO = 0$, and hence (enlarging $K$ if necessary), an eigenform 
$\widetilde{\varphi} \in \coker(\beta')\otimes_{\CO}K$ such that that $T_v\widetilde{\varphi} = \widetilde{a}_v'\widetilde{\varphi}$ and $S_v\widetilde{\varphi} = \widetilde{d}_v'\widetilde{\varphi}$ for some $\widetilde{a}_v'\in \CO$ lifting $a_v'$ and $\widetilde{d}_v'\in \CO$ lifting $d_v'$.

Letting $\Pi \in \mathcal{C}_{\vec{k}}^B $ denote the (necessarily cuspidal) automorphic representation generated by 
$\widetilde{\varphi}$, we have that $\Pi_w^{U_0(w)_w} \neq 0$,
but $\Pi_w^{U_w} = 0$, from which it follows that $\Pi_w$ is a discrete series representation of $B_w^\times \cong \GL_2(F_w)$. Theorem~\ref{thm:JL} therefore implies the existence of a cuspidal automorphic representation $\Pi' \in \mathcal{C}_{\vec{k}}^{B'}$ such that $\Pi'_v \cong \Pi_v$ for all finite $v \neq w$, where $B'$ is a quaternion algebra over $F$ such that $\Sigma^{B'} = \Sigma^B \cup \{w\} - \{\sigma_0\}$ for some archimedean place $\sigma_0$. This in turn yields an eigenform $\widetilde{\varphi}' \in M_{\vec{k}}^{B'}(U',\CO)$, where $U'$ contains $\CO_{B',p}^\times$ and $\CO_{B',v}^\times$ for all $v \not\in Q'$,  with $T_v \widetilde{\varphi}' = \widetilde{a}_v'\widetilde{\varphi}'$ and $S_v \widetilde{\varphi}' = \widetilde{d}_v'\widetilde{\varphi}'$ for all such $v$.  Reducing mod $\varpi$ thus yields an eigenform $\varphi' \in M^{B'}_{\vec{k}}(U',E)$ such that $T_v {\varphi}' = {a}_v'{\varphi}'$ and $S_v {\varphi}' = {d}_v'{\varphi}'$ for all $v \not\in Q'$, and finally applying $e_\xi$ yields the desired eigenform $f' \in M^{B'}_{\vec{k},\vec{m}}(U',E)$ such that $T_v f' = {a}_vf'$ and $S_vf' = {d}_vf'$ for all $v\not\in Q'$.

(2) Now let $B$, $U$ and $Q$ be as in Definition~\ref{def:algmod}, but with $B$ indefinite, so there exists a non-zero $f\in M_{\vec{k},\vec{m}}(U,\Fpbar)$ for some sufficiently small $U$ of level prime to $p$ and sufficiently large $Q$ such that $T_v f = a_v f$ and $S_v f = d_v f$ for all $v \not\in Q$, where
$a_v  = \tr(\rho(\Frob_v))$ and $d_v = \Nm_{F/\QQ}(v)^{-1}\det(\rho(\Frob_v))$.

Again choose $\xi:(\A_{F,\f}^{(p)})^\times \to \Fpbar^\times$ such that $\xi(\alpha) = \alpha^{\vec{m}+\vec{k}/2}$ for all $\alpha \in \CO_{F,(p),+}^\times$ (and shrink $U$ and enlarge $Q$ and $E$), so that applying $e_\xi$ yields $\varphi \in M^B_{\vec{k}}(U,E)$ such that $T_v \varphi = a_v'\varphi$ and $S_v \varphi = d_v'\varphi$ for all $v \not\in Q$,  where $a_v' = \xi(\varpi_v)^{-1}a_v$ and $d_v' = \xi(\varpi_v)^{-2}d_v$.

The next step is to lift $\varphi$ to an eigenform in $M^B_{\vec{k}}(U,\CO)$.  To that end, let $\TT$ denote the 
polynomial ring over $\CO$ in the variables $T_v$ and $S_v$ for $v\notin Q$, and $\gm$ the kernel of the homomorphism $\TT \to E$ defined by $T_v \mapsto a_v'$ and $S_v \mapsto d_v'$ for $v\not\in Q$.  We claim that the homomorphism $M^B_{\vec{k}}(U,\CO)_\gm \to M^B_{\vec{k}}(U,E)_{\gm}$ induced by reduction mod $\varpi$ is surjective. Indeed from the long exact cohomology sequence 
$$H^1(Y_U^B,\mathcal{D}_{\vec{k},\CO}) \to 
H^1(Y_U^B,\mathcal{D}_{\vec{k},E}) \to 
H^2(Y_U^B,\mathcal{D}_{\vec{k},\CO}) \stackrel{\varpi\cdot}{\to} 
H^2(Y_U^B,\mathcal{D}_{\vec{k},\CO}) \to 
H^2(Y_U^B,\mathcal{D}_{\vec{k},E})$$
and Nakayama's Lemma, we see that it suffices to prove that 
$H^2(Y_U^B,\mathcal{D}_{\vec{k},E})_{\gm} = 0$.
Applying $e_\xi: H^2(Y_U^B,\mathcal{D}_{\vec{k},\Fpbar})\stackrel{\sim}{\to} H^2(Y_U^B,\mathcal{D}_{\vec{k},\vec{m},\Fpbar})$, we are reduced to proving that $H^2(Y_U^B,\mathcal{D}_{\vec{k},\vec{m},\Fpbar})_{\gm_\rho^Q} = 0$, which was already established in the proof of Lemma~\ref{lem:JHfactors}.
Similarly, we see that $H^0(Y_U^B,\mathcal{D}_{\vec{k},E})_{\gm} = 0$, and it follows that the sequence
$$0 \longrightarrow 
M^B_{\vec{k}}(U,\CO)_{\gm}  \stackrel{\varpi\cdot}{\longrightarrow}
M^B_{\vec{k}}(U,\CO)_{\gm}  \longrightarrow 
M_{\vec{k}}(U,E)_{\gm}  \longrightarrow 0$$
is exact.

We have now shown that $M^B_{\vec{k}}(U,\CO)_{\gm}$ is non-zero and torsion-free, so the usual argument gives an eigenform $\widetilde{\varphi} \in M^B_{\vec{k}}(U,K)$ (after enlarging $K$ if necessary) such that $T_v(\widetilde{\varphi}) = \widetilde{a}_v'\widetilde{\varphi}$ and $S_v(\widetilde{\varphi}) = \widetilde{d}_v'\widetilde{\varphi}$ for all $v\not\in Q$, where $\widetilde{a}_v'\in \CO$ lifts $a_v'$ and $\widetilde{d}_v'$ lifts $d_v'$.  This in turns yields a cuspidal automorphic representation $\Pi \in \mathcal{C}_{\vec{k}}^B$, to which we can apply Theorem~\ref{thm:JL} to obtain an eigenform $\widetilde{\varphi}' \in M_{\vec{k}}(U,K)$ with the same eigenvalues as $\widetilde{\varphi}$ for $T_v$ and $S_v$ for all $v\not\in Q$.  Finally, the usual argument of reducing mod $\varpi$ and applying $e_\xi$ yields an eigenform $f' \in M_{\vec{k},\vec{m}}(U,E)$ with the same eigenvalues as $f$.
\end{proof}

\begin{remark} \label{rmk:GLSmain} As we will recall below (Corollary~\ref{cor:GLSmain}), the converse to (1) in the preceding theorem follows from the main result of \cite{GLS} under a Taylor--Wiles hypothesis on $\rho$.

Alternatively, one can try to argue more directly along the lines of our proof of (1) above.  However this requires an analogue of Ihara's Lemma in the context of Shimura curves over $F$, for which the strongest results to date are due to Manning and Shotton \cite{MS}, which also require a Taylor--Wiles hypothesis on $\rho$.

Finally we remark that the argument we gave for part (2) of the theorem can be easily adapted to give a more direct proof that definite modularity of weight $(\vec{k},\vec{m})$ implies geometric modularity of weight $(\vec{k},\vec{m})$.
\end{remark}

\begin{corollary} \label{cor:reductions} Suppose that $\pi$ is as in the statement of Theorem~\ref{thm:galois} and that $\rho := \overline{\rho}_\pi$ is irreducible. Then $\rho$ is geometrically modular of some weight $(\vec{\ell},\vec{n})$ with $\vec{\ell} \in [2,p+1]^\Sigma$ such that $\ell_{\gp,i,j} = 2$ for all $\gp \in S_p$, $i \in \ZZ/f_\gp \ZZ$ and $1 \le j \le e_\gp-1$.
\end{corollary}
\begin{proof} Suppose first that $[F:\QQ]$ is odd, and let $B$ be a quaternion algebra over $F$ which is unramified at all finite places and exactly one archimedean place of $F$. By the Jacquet--Langlands correspondence, there is an irreducible subrepresentation $\Pi$ of $|\det|^{w/2}\mathcal{S}_{\vec{k}}$
such that $\Pi_v \cong \pi_v$ for all finite primes $v$ of $F$. For some sufficiently small open compact $U' = U'_pU'^p \subset B_{\mathbf{f}}^\times$, we have that $T_v$ (resp.~$\Nm_{F/\QQ}(v)S_v$) acts as $\tr(\rho_\pi(\Frob_v))$ (resp.~$\det(\rho_\pi(\Frob_v))$) on $\Pi^{U'}$ for all $v$ not in some finite set of non-archimedean places $Q$ of $F$, from which it follows that
$\gm_\rho^Q$ is in the support of 
$$M^B_{\vec{k},\vec{m}}(U',\Fpbar) = H^1(Y^B_{U'},\mathcal{D}_{\vec{k},\vec{m},\Fpbar}),$$
where $m_\sigma = -(w+k_\sigma)/2$ for $\sigma \in \Sigma$.

Letting $U = \CO_{B,p}^\times U_p'$, the projection $\varrho:Y_{U'}^B \to Y_U^B$ is \'etale (for sufficiently small $U_p'$), and $\varrho_*\mathcal{D}_{\vec{k},\vec{m},\Fpbar}$ is isomorphic to the locally constant sheaf $\mathcal{T}$ on $Y^B_U$ associated to the smooth representation $T = \Ind_{U_p'}^{\CO_{B,p}^\times}D_{\vec{k},\vec{m},\Fpbar}$ of $\CO_{B,p}^\times$.  We thus obtain an isomorphism $ H^1(Y^B_{U'},\mathcal{D}_{\vec{k},\vec{m},\Fpbar})
 \cong H^1(Y^B_{U},\mathcal{T})$
of $\TT^Q$-modules, and (the easier direction of) the proof of Lemma~\ref{lem:JHfactors} then implies that $\gm_\rho^Q$ is in the support of $M_{\vec{\ell},\vec{n}}(U,\Fpbar)$ for some Jordan--H\"older factor $D_{\vec{\ell},\vec{n},\Fpbar}$ of $T$. 

It follows that $\rho$ is algebraically modular of some weight $(\vec{\ell},\vec{n})$ as in the statement of the corollary. Theorem~\ref{thm:defalggeom}(2) now implies that $\rho$ is geometrically modular of such a weight. 

The proof in the case that $[F:\QQ]$ is even is similar, but using  forms on a definite quaternion algebra over $F$ and parts (1) and (2) of the theorem.
\end{proof}

\begin{corollary} \label{cor:wtredn} Suppose that $\rho$ is irreducible and geometrically modular of some weight $(\vec{k},\vec{m})$.  Then it is geometrically modular for some weight $(\vec{\ell},\vec{n})$ as in the conclusion of Corollary~\ref{cor:reductions}.
\end{corollary}
\begin{proof}  From the construction of the Galois representation $\rho_f$ associated to an eigenform $f \in M_{\vec{k},\vec{m}}(U,\Fpbar)$, we have that $\rho_f \cong \rhobar_\pi$ for some $\pi \in \mathcal{S}_{\vec{2}}$.  Applying Corollary~\ref{cor:reductions} therefore gives the desired conclusion.
\end{proof}

In view of Theorem~\ref{thm:defalggeom}, one can obtain results towards Conjecture~\ref{conj:geomweights2} using ones on the algebraic Serre weight conjecture, which was proved by Gee--Liu--Savitt~\cite{GLS} (combined with work of Gee--Kisin~\cite{GK} or Newton~\cite{JN}) under a Taylor--Wiles hypothesis.  We recall the statement, or more precisely a variant which is immediate from Theorems~4.2.1 and~6.1.8 of \cite{GLS} and the fact that if $\rho$ is geometrically modular of some weight, then it is algebraically modular of some weight:
\begin{theorem}[Gee--Liu--Savitt] \label{thm:GLSmain} Suppose that $\rho$ is irreducible and geometrically modular of some weight $(\vec{k},\vec{m})$. Suppose further that $p >2$, $\rho|_{F(\zeta_p)}$ is irreducible, and if $p=5$, then the projective image of $\rho|_{F(\zeta_5)}$ is not isomorphic to $A_5$.  If $(\vec{\ell},\vec{n}) \in \ZZ^{\Sigma}_{\ge 2} \times \ZZ^\Sigma$ is such that $D'_{\vec{\ell},\vec{n},\Fpbar}$ is irreducible, then the following are equivalent:
\begin{enumerate}
\item $\rho$ is definitely modular of weight $(\vec{\ell},\vec{n})$;
\item $\rho$ is algebraically modular of weight $(\vec{\ell},\vec{n})$;
\item $\rho|_{G_{F_{\gp}}}$ has a crystalline lift of weight $(\vec{\ell}_\gp,\vec{n}_\gp)$ for all $\gp \in S_p$.
\end{enumerate}
\end{theorem}

Combined with Lemma~\ref{lem:JHfactors}, this implies:
\begin{corollary} \label{cor:GLSmain} Suppose that $p > 2$, $(\vec{k},\vec{m}) \in \ZZ^{\Sigma}_{\ge 2} \times \Sigma$, and $\rho:G_F \to \GL_2(\Fpbar)$ is such that $\rho|_{F(\zeta_p)}$ is irreducible, and if $p=5$, then the projective image of $\rho|_{F(\zeta_5)}$ is not isomorphic to $A_5$.  Then $\rho$ is definitely modular of weight $(\vec{k},\vec{m})$ if and only if 
$\rho$ is algebraically modular of weight $(\vec{k},\vec{m})$.
\end{corollary}

\subsection{Ordinary modularity} \label{ssec:ordmod}
Suppose now that $\vec{k} \in \ZZ_{\ge 1}^\Sigma$, so 
the Hecke operator $T_\gp$ on $M_{\vec{k},-\vec{1}}(U,E)$ is defined for all $\gp \in S_p$.  Since the operators $T_{\gp}$ commute with each other as well as the $T_v$ and $S_v$ for $v\not\in Q$, it follows that if $\rho \sim \rho_f$ for some 
eigenform $f \in M_{\vec{k},-\vec{1}}(U,E)$ as in the statement of Theorem~\ref{thm:galois}, then we may in fact take $f$ to be an eigenform for the $T_{\gp}$ for all $\gp \in S_p$ as well.

Recall from Theorem~\ref{thm:ord} that if $f$ is ordinary at a prime $\gp \in S_p$ in the sense that $T_{\gp}f = a_{\gp}f$ for some $a_{\gp} \in E^\times$, then so is the associated local Galois representation $\rho_f|_{G_{F_\gp}}$ in the sense that it has an unramified subrepresentation (with our conventions), provided $k_\sigma > 1$ for some $\sigma \in S_{\gp}$.

On the other hand, if $k_\sigma = 1$ for all $\sigma \in \Sigma_\gp$, then one direction of Conjecture~\ref{conj:geomweights2} amounts to the well-known expectation that $\rho_f$ be unramified at $\gp$.
Using the argument of~\cite[Prop.~2.5]{DDW}, we at least have
the following:
\begin{proposition} \label{prop:ppw1} Suppose that $f \in M_{\vec{k},-\vec{1}}(U,E)$ is an eigenform (as in Theorem~\ref{thm:galois}), and moreover that $T_\gp f = a_\gp f $ for some $a_\gp \in E$.  Then $\rho_f|_{G_{F_\gp}}^{\mathrm{ss}}$ is unramified and $a_\gp(f) = \tr(\rho_f(\gamma))$ for any
$\gamma$ lifting $\Frob_\gp$ to $G_{F_{\gp}}$.
\end{proposition}
\begin{proof} We first note that $\det(\rho_f)$ is unramified at 
$\gp$, and moreover that 
$$\det(\rho_f(\Frob_\gp)) = \delta\epsilon\varpi_{\gp}^{\vec{k}-\vec{1}} \in E^\times,$$
where $\delta$ is the eigenvalue of $S_{\varpi_{\gp}}$ on $f$ and $\epsilon$ is the reduction of $p^{f_\gp}\Nm(\varpi_\gp)^{-1}$ (see \S\ref{ssec:Hecke}).

Letting $\vec{k}' = \vec{k} + \sum_{\sigma \in \Sigma_{\gp}} \vec{h}_\sigma$ and relating the actions of $T_\gp$ under the ``doubling map''
$$ M_{\vec{k},-\vec{1}}(U,E)^2 \lra  M_{\vec{k}',-\vec{1}}(U,E)$$
as in \cite[Prop.~2.5]{DDW}, one deduces the existence of an eigenform $f_\alpha \in M_{\vec{k}',-\vec{1}}(U,E)$ such that $T_\gp f_\alpha = \alpha f_\alpha$ and $\rho_{f_\alpha} \sim \rho_f$ for each root $\alpha$ of the polynomial
$$X^2 - a_\gp(f) X + \det(\rho_f(\Frob_\gp)).$$
It then follows from Theorem~\ref{thm:ord} that $\rho_f$
has a non-trivial unramified subrepresentation $L_\alpha$ on which 
$\Frob_\gp$ acts as $\alpha$.

If the polynomial above has distinct roots $\alpha \neq \alpha'$, then it follows that $\rho|_{F_{\gp}}$ is isomorphic to $L_\alpha \oplus L_{\alpha'}$.  Therefore $\rho_f$ is unramified at $\gp$ and 
$$\tr(\rho_f(\Frob_{\gp})) = \alpha + \alpha' = a_{\gp}.$$
On the other hand, if the polynomial factors as $(X-\alpha)^2$, then $\det(\rho_f)$ is the unramified character sending $\Frob_{\gp}$ to $\alpha^2$.  Therefore the quotient of $\rho_f|_{G_{F_\gp}}$ by $L_\alpha$ is also isomorphic to $L_\alpha$, and $\rho|_{F_{\gp}}^{\mathrm{ss}}$ is isomorphic to $L_\alpha \oplus L_{\alpha}$, again giving the desired conclusion.
\end{proof}

\begin{remark} Note that for $f$ as in the proposition, it is immediate that $\rho_f$ is unramified at $\gp$ if the polynomial in the proof has distinct roots (or if $\rho_f$ itself is reducible).  If it has repeated roots (and $\rho_f$ is irreducible), then the unramifiedness is proved by De~Maria (\cite[Thm.~D]{DM}) under the additional assumption that the $k_\sigma$ are odd (for all $\sigma \in \Sigma$).
\end{remark}

In view of the proposition, if $k_\sigma = 1$ for all $\sigma \in \Sigma_{\gp}$, then we view any $T_{\gp}$-eigenform $f \in M_{\vec{k},-\vec{1}}(U,E)$ as being ordinary at $\gp$.

We extend the notions of ordinariness to more general weights 
$(\vec{k},\vec{m})$ with $\vec{k} \in \ZZ_{\ge 1}^\Sigma$ by twisting by characters as in \S\ref{ssec:galoistwist}.  Suppose that $f \in M_{\vec{k},\vec{m}}(U,E)$ is an eigenform
for $T_v$ and $S_v$ for all $v\not\in Q$
(as in Theorem~\ref{thm:galois}).  Enlarging $E$ if necessary, let $\xi: \AA_{F}^\times/F^\times F_{\infty,+}^\times \to E^\times$ be a continuous character such that $\xi(u_p) =  \overline{u}_p^{\vec{m}+\vec{1}}$ for all $u \in \CO_{F,p}^\times$, so its restriction to $(\AA_{F,\f}^{(p)})^\times$ is as in \S\ref{ssec:twist} with $\vec{\ell} = -\vec{m}-\vec{1}$.
Shrinking $U$ and enlarging $Q$ as necessary, we have that 
$e_\xi f \in  M_{\vec{k},-\vec{1}}(U,E)$ is an eigenform for $T_v$ and $S_v$ for all $v\not\in Q$.  We say that $f$ is {\em $\gp$-ordinary} if either $k_\sigma = 1$ for all $\sigma \in \Sigma_{\gp}$ or $T_\gp(e_\xi f) = a_{\gp,\xi} e_{\xi}f$ for some $a_{\gp,\xi} = a_{\gp,\xi}(f) \in E^\times$.  If $\xi_1$ and $\xi_2$ are two such characters, then
$$T_{\gp}(e_{\xi_2}f) = T_{\gp}(e_\mu e_{\xi_1} f)
 = \mu(\varpi_{\gp,\gp})e_{\mu} T_{\gp}(e_{\xi_1} f)$$
where $\mu = \xi_1^{-1}\xi_2$.
It follows that the notion is independent of the choice of $\xi$,
and that if $f$ is $\gp$-ordinary, then the $a_{\gp,\xi}$ as $\xi$ varies are related by the formula 
$a_{\gp,\xi_2} = \mu(\Frob_{\gp})a_{\gp,\xi_1}$, where we have written $\mu$ also for the character $G_F \to E^\times$ to which it corresponds via class field theory.
\begin{definition} \label{def:ordgeommod} Suppose that $S$ is a subset of $S_p$, $\vec{k} \in \ZZ_{\ge 1}^{\Sigma}$ and $\vec{m} \in \ZZ^{\Sigma}$. We say that a continuous, irreducible $\rho:G_F \to \GL_2(\Fpbar)$ is 
{\em $S$-ordinarily geometrically modular} if $\rho$ arises from some $f \in M_{\vec{k},\vec{m}}(U,E)$ (as in Theorem~\ref{thm:galois}) which is $\gp$-ordinary for all $\gp \in S$.
\end{definition} 

We may similarly define ordinarily definite and algebraic modularity.  Recall that for $\vec{k} \in \ZZ_{\ge 2}^\Sigma$, 
$\gp \in S_p$ and $B$ and $U$ as in Definition~\ref{def:algmod}, 
the Hecke operator $T_\gp$ is defined on $M_{\vec{k},-\vec{1}}^B(U,\Fpbar)$ (see \S\ref{ssec:defQO}, \ref{ssec:indefQO}). Recall also that for arbitrary $\vec{m}$, we have (after possibly shrinking $U$) isomorphisms 
$$ e_{\xi}: M^B_{\vec{k},\vec{m}}(U,\Fpbar) \stackrel{\sim}{\lra} M^B_{\vec{k},-\vec{1}}(U,\Fpbar) $$
associated to characters $\xi: \AA_{F}^\times/F^\times F_{\infty,+}^\times \to \Fpbar^\times$ such that $\xi(u_p) =  \overline{u}_p^{\vec{m}+\vec{1}}$ for all $u_p \in \CO_{F,p}^\times$.
\begin{definition} \label{def:ordalgmod} Suppose that $S$ is a subset of $S_p$, $\vec{k} \in \ZZ_{\ge 2}^{\Sigma}$ and $\vec{m} \in \ZZ^{\Sigma}$. We say that a continuous, irreducible $\rho:G_F \to \GL_2(\Fpbar)$ is 
{\em $S$-ordinarily definitely} (resp.~{\em algebraically}) {\em modular} if there exists a non-zero $\varphi \in M^B_{\vec{k},\vec{m}}(U,\Fpbar)[\gm^Q_\rho]$ for some $B$, $U$ and $Q$ as in Definition~\ref{def:algmod}, such that $B$ is totally definite (resp.~unramified at exactly one archimedean place) and for all $\gp \in S$, we have
$$T_\gp(e_\xi \varphi) = a_{\gp,\xi} e_\xi \varphi$$
for some $a_{\gp,\xi} = a_{\gp,\xi}(\varphi) \in \Fpbar^\times$ and character $\xi: \AA_{F}^\times/F^\times F_{\infty,+}^\times \to \Fpbar^\times$ as above.
\end{definition} 

Just as for ordinary geometric modularity, if the condition on $e_\xi \varphi$ in the definition holds for some character $\xi$ such that $\xi(u_p) =  \overline{u}_p^{\vec{m}+\vec{1}}$ for all $u_p \in \CO_{F,p}^\times$ above, then in fact the condition holds for all such $\xi$.  Furthermore, the $a_{\gp,\xi}$ for varying $\xi$ are related in the same way as for ordinarily geometric modularity.

The analogue of Theorem~\ref{thm:ord} holds in the setting for definite or algebraic modularity.  Furthermore, we can extend the description of the local behavior at $\gp$ to all $\gp$-ordinarily modular Galois representations as follows (writing $\gp$ instead of $\{\gp\}$).
\begin{theorem} \label{thm:ord2} Suppose that $\gp \in S_p$, $\vec{k},\vec{m} \in \ZZ^\Sigma$, with $k_\sigma \ge 1$ (resp.~$\ge 2$) for all $\sigma \in \Sigma$, and that $\rho:G_F \to \GL_2(\Fpbar)$ is irreducible and $\gp$-ordinarily geometrically (resp.~definitely or algebraically) modular of weight $(\vec{k},\vec{m})$.
Then
$$\rho|_{G_{F_{\gp}}} \sim \left(\begin{array}{cc} \chi_2 & * \\ 0 &\chi_{\mathrm{cyc}}^{-1}\chi_1\end{array}\right),$$
for some characters $\chi_1,\chi_2:G_{F_{\gp}} \to \Fpbar^\times$, where $\chi_2$ corresponds via class field theory to the character $F_{\gp}^\times \to \Fpbar^\times$ defined by 
$x \mapsto \xi^{-1}(x)a_{\gp,\xi}^{v_{\gp}(x)}$ for any $\xi$ and $a_{\gp,\xi}$ as in Definition~\ref{def:ordgeommod} (resp.~\ref{def:ordalgmod}), unless $k_\sigma = 1$ for all $\sigma \in \Sigma_{\gp}$, in which case the same conclusion holds with $a_{\gp,\xi}$ replaced by any root of the polynomial
$$X^2 - a_{\gp,\xi} X + (\chi^2\det(\rho))(\Frob_{\gp}),$$
where $\chi$ corresponds via class field theory to $\xi$ (so $\chi^2\det(\rho)$ is unramified at $\gp$).
\end{theorem}
\begin{proof} Consider first the case of $\gp$-ordinarily geometric modularity, and let $f$ and $\xi$ be as in Definition~\ref{def:ordgeommod}. Since $\rho_f \sim \chi^{-1}\rho_{e_{\xi}f}$ (where $\chi$ corresponds to $\xi$), we can replace $f$ by $e_{\xi}f$ and reduce to the case $\vec{m} = - \vec{1}$ and $\xi = 1$.  The assertions are then given by Theorem~\ref{thm:ord} and Proposition~\ref{prop:ppw1}.

Consider now the case of $\gp$-ordinarily definite or algebraic modularity, and let $\varphi$ and $\xi$ be as in Definition~\ref{def:ordalgmod}. Since $e_\xi\varphi \in M^B_{\vec{k},-\vec{1}}(U,\Fpbar)[\gm^Q_{\chi\otimes\rho}]$, we can again reduce to the case $\vec{m} = - \vec{1}$ and $\xi = 1$.  An inspection of the proof of Theorem~\ref{thm:ord} (starting from the introduction of the representation $L_\chi$) shows that this is precisely the input required in order to obtain the desired conclusion, except that it is assumed at that point\footnote{The formula $\varepsilon\circ\Delta = \delta\circ \varepsilon$ in the proof fails if $k_\sigma = 2$ for all $\sigma \in \Sigma_{\gp}$.} that $k_\sigma > 2$ for some $\sigma \in \Sigma_{\gp}$.   However if $k_\sigma = 2$ for all $\sigma \in \Sigma_{\gp}$, then one can a give a similar (and much simpler) argument by making the following changes:  

Firstly, in the setup, replace the representation $L_\chi$ by the trivial representation $E$, $\varepsilon$ by the identity, $U_{B,0}$ and $U_{B,1}$ by $U_B$ and $\pi$ by the identity. 

One can then skip directly to the introduction of the subgroup $U_{B,1}'$ in the proof. The formulas describing $T_{\gp}$ on the kernel of (\ref{eqn:levelpord}) in the algebraic case\footnote{Note also that whether we are in the algebraic and definite case is no longer determined by whether the degree $d = [F:\QQ]$ is even or odd.} and on $I_{\vec{2}}^B(U_{B,1}',\C)$ in the definite case are no longer valid (the factor of $p^{f_{\gp}}$ being replaced by $p^{f_{\gp}} + 1$).  However if $\varphi$ is in the kernel of (\ref{eqn:levelpord}) or $\widetilde{\varphi} \in I_{\vec{2}}^B(U_{B,1}',\C)$, then the proof of Theorem~\ref{thm:galois} gives that $\rho \sim \chi_{\mathrm{cyc}}^{-1}\psi \oplus \psi$ for some character $\psi$, contradicting our assumption here that $\rho$ is irreducible.

Finally we obtain a cuspidal automorphic representation $\pi \in \mathcal{C}_{\vec{2}}$ such that $\rho \sim \rhobar_{\pi}$ and $\pi_{\gp}$ is an unramified principal series $I(\psi_1|\cdot|^{1/2},\psi_2|\cdot|^{1/2})$ such that $\psi_1(\varpi_\gp)$ and $\psi_2(\varpi_\gp)$ are the roots of
$$X^2 - \widetilde{a}_{\gp}X + \widetilde{d}_{\gp} p^{f_{\gp}}$$
for some lift $\widetilde{a}_{\gp}$ of $a_{\gp}$ and root of unity $\widetilde{d}_{\gp}$ (both viewed in $\CO^\times \cap \Qbar \subset \CC$).  We may therefore assume that $\psi_2(\varpi_{\gp})$ is a lift of $a_{\gp}$ and conclude exactly as in the proof of Theorem~\ref{thm:ord}. 
\end{proof}

We conjecture the following converse to Theorem~\ref{thm:ord2}:
\begin{conjecture} \label{conj:ord} Suppose that 
$\rho:G_F \to \GL_2(\Fpbar)$ is irreducible, $S \subset S_p$ and
$(\vec{k},\vec{m}) \in \ZZ_{\ge 1}^\Sigma \times \ZZ^\Sigma$.  Suppose that $\rho|_{G_{F_{\gp}}}$ has a crystalline lift of weight $(k_\sigma,m_\sigma)_{\sigma\in \Sigma_{\gp}}$ for all $\gp \in S_p$, and moreover that for all $\gp \in S$, we have
$$\rho|_{G_{F_{\gp}}} \sim \left(\begin{array}{cc} \chi_{2,\gp} & * \\ 0 &\chi_{\mathrm{cyc}}^{-1}\chi_{1,\gp}\end{array}\right)$$
for some characters $\chi_{1,\gp}$, $\chi_{2,\gp}$ such that $\chi_{2,\gp}|_{I_{F_{\gp}}}$ corresponds via class field theory to the character $\CO_{F_{\gp}}^\times \to \Fpbar^\times$ defined by $ \prod_{\sigma\in \Sigma_{\gp}} \overline{\sigma}^{-1-m_\sigma}$. Then $\rho$ is $S$-ordinarily geometrically modular of weight $(\vec{k},\vec{m})$, and furthermore is $S$-ordinarily definitely and algebraically modular of weight $(\vec{k},\vec{m})$ if $\vec{k} \in \ZZ^\Sigma_{\ge 2}$.
\end{conjecture}

\begin{remark} Note that the hypotheses on $\rho|_{G_{F_{\gp}}}$ imply also that the characters $\chi_{1,\gp}|_{I_{F_{\gp}}}$ correspond to $ \prod_{\sigma\in \Sigma_{\gp}} \overline{\sigma}^{1-k_\sigma - m_\sigma}$ for $\gp \in S$.  Conversely, if 
$\rho|_{G_{F_{\gp}}}$ has this form for such characters $\chi_{1,\gp}$, $\chi_{2,\gp}$, then it necessarily has  a crystalline lift of weight $(k_\sigma,m_\sigma)_{\sigma\in \Sigma_{\gp}}$, provided $k_\sigma \ge 2$ for all $\sigma \in \Sigma_{\gp}$ and some $k_\sigma > 2$.
\end{remark}

We expect that the proof of Theorem~\ref{thm:GLSmain} can be modified (by imposing an ordinariness condition in the relevant local deformation rings) to establish the preceding conjecture in the definite and algebraic settings under the additional hypotheses of that theorem.  However, we have not checked the details, so we content ourselves with the following result, which can be deduced more directly from results of Gee--Kisin~\cite{GK}, and will suffice for the purpose of the intended applications to the ramified quadratic case in the next section.

\begin{theorem} 
\label{thm:GLSord} Suppose that $\rho$ is irreducible and geometrically modular of some weight $(\vec{k},\vec{m})$. Suppose further that $p >2$, $\rho|_{F(\zeta_p)}$ is irreducible, and if $p=5$, then the projective image of $\rho|_{F(\zeta_5)}$ is not isomorphic to $A_5$.  Let $(\vec{\ell},\vec{n}) \in \ZZ^{\Sigma}_{\ge 2} \times \ZZ^\Sigma$ be such that $D'_{\vec{\ell},\vec{n}}$ is irreducible and for each $\gp \in S_p$, 
\begin{itemize}
\item $\max_{1 \le j \le e_{\gp}} \ell_{\gp,i,j} < p + 1$ for some $i \in \ZZ/f_{\gp}\ZZ$;
\item $\rho|_{G_{F_{\gp}}}$ has a crystalline lift of weight $(\vec{\ell}_\gp,\vec{n}_{\gp})$ ;
\item $\rho|_{G_{F_{\gp}}} \sim
\left(\begin{array}{cc} \chi_{2,\gp} & * \\ 0 &\chi_{\mathrm{cyc}}^{-1}\chi_{1,\gp}\end{array}\right)$
for some characters $\chi_{1,\gp}$, $\chi_{2,\gp}$ such that $\chi_{2,\gp}|_{I_{F_{\gp}}}$ corresponds to the character $\CO_{F_{\gp}}^\times \to \Fpbar^\times$ defined by $ \prod_{\sigma\in \Sigma_{\gp}} \overline{\sigma}^{-1-n_\sigma}$.
\end{itemize} Then $\rho$ is $S_p$-ordinarily definitely and algebraically modular of weight $(\vec{\ell},\vec{n})$.
\end{theorem}
\begin{proof} Firstly, we can reduce to the case $\vec{n} = -\vec{1}$ by twisting.

Note that the hypothesis that $\rho|_{G_{F_{\gp}}}$ has a crystalline lift of weight $(\ell_\sigma,n_\sigma)_{\sigma\in \Sigma_{\gp}}$ implies that $\chi_{1,\gp}|_{I_{F_{\gp}}}$ corresponds to $\prod_{\sigma \in \Sigma_\gp} \overline{\sigma}^{2-\ell_\sigma}$. Recall also that the hypothesis that $D'_{\vec{\ell},\vec{n}}$ is irreducible means that $2 \le \ell_{\gp,i,j} \le p + 1$ for all $i=1,\ldots,e_{\gp}$ and $j \in \ZZ/f_{\gp}\ZZ$, and that $\ell_{\gp,i,j} > 2$ for at most one $j$ for each $i$.
Let $S_0$ denote the set of primes in $S_p$ such that
$k_\sigma \neq 2$ for some $\sigma \in \Sigma_{\gp}$. 

For each $\gp \in S_p$, let $\tau_{\gp}$ denote the inertia type $1 \oplus \xi_{\gp}$, where $\xi_{\gp}:I_{F_{\gp}} \to \CO^\times$ (for some sufficiently large $\CO$) is the Teichm\"uller lift of $\chi_{1,\gp}|_{I_{F_{\gp}}}$.

As in the proof of Theorem~\ref{thm:defalggeom}, let us also choose an auxiliary prime $w \not\in S_p$ at which $\rho$ is unramified and the conjugacy class of $\Frob_w$ in $\Gal(L(\zeta_p)/F)$ is that of a complex conjugation, where $L$ is the splitting field of $\rho$.

By \cite[Lemma~4.2.5]{GK} (applied to $\chi_{\mathrm{cyc}}\otimes \rho$), there is an automorphic representation $\pi \in \mathcal{C}_{\vec{2}}$ such that
\begin{itemize}
\item $\overline{\rho}_\pi \sim \rho$;
\item $\pi_w \sim \psi_w \otimes \mathrm{St}$ for some unramified character $\psi_w$;
\item if $\gp \in S_p$, then $\pi_{\gp}$ is principal series of the form $I(\psi_1|\cdot|^{1/2},\psi_2|\cdot|^{1/2})$ where $\psi_1|_{\CO_{F,\gp}^\times}$ corresponds to $\xi_{\gp}$ and $\psi_2$ is unramified,
\item if $\gp \in S_p$, then $\rho|_{G_{F_{\gp}}}  \sim \left(\begin{array}{cc} \widetilde{\chi}_{2,\gp} & * \\ 0 &\chi_{\mathrm{cyc}}^{-1}\widetilde{\chi}_{1,\gp}\end{array}\right) $, where $\widetilde{\chi}_{2,\gp}$ is an unramified lift of $\chi_{2,\gp}$ and $\widetilde{\chi}_{1,\gp}$ is a tamely ramified lift of $\chi_{1,\gp}$ (so $\widetilde{\chi}_{1,\gp}|_{I_{F_{\gp}}} = \xi_{\gp}$).
\end{itemize}
Since $\pi_{\gp}$ corresponds via local Langlands to $\mathrm{WD}(\rho_\pi|_{G_{F_{\gp}}}) \cong \widetilde{\chi}_{2,\gp}\oplus \widetilde{\chi}_{1,\gp}|\cdot|^{-1}$, we have $\psi_1 = \widetilde{\chi}_{1,\gp}|\cdot|^{-1}$ and $\psi_2 =\widetilde{\chi}_{2,\gp}$, after possibly interchanging $\psi_1$ and $\psi_2$ if $\gp \not\in S_0$ (conflating characters of $W_{F_{\gp}}$ and $F_{\gp}^\times$).

Let $B$ denote the totally definite quaternion algebra over $F$ ramified at either $S_\infty$ or $S_\infty \cup \{w\}$, according to whether $d$ is even or odd.  Then by Theorem~\ref{thm:JL}, there is an automorphic representation $\Pi \in \mathcal{C}_{\vec{2}}^B$ such that $\pi_v \cong \Pi_v$ for all finite places $v \neq w$.

For $\gp \in S_p$, define $U_{\gp,0}$ (resp.~$U_{\gp,1}$) to be $\CO_{B,\gp}^\times$ to be the open compact subgroup of $\CO_{B,\gp}^\times$ corresponding to 
$$ U_0(\gp)_{\gp} = \left\{\left.\,\smat{a}{b}{c}{d}\in \GL_2(\CO_{F,\gp})\,\right|\,c \in \gp\CO_{F,\gp}\,\right\}$$
(resp.~$U_1(\gp)_{\gp}$); otherwise let $U_{\gp,0} = U_{\gp,1} = \CO_{B,\gp}^\times$.  Let $U_1 \subset U_0 \subset U$ be open compact subgroups of $B_{\f}^\times$ of the form 
$$U_{1,p} U^p \subset U_{0,p} U^p \subset U_pU^p,$$
where $U_{i,p} = \prod_{\gp \in S_p} U_{i,\gp}$ for $i=0,1$, $U_p= \CO_{B,p}^\times$, and (the same) $U^p \subset (B_{\f}^{(p)})^\times$ is sufficiently small that $\Pi^{U_1} \neq 0$ (in addition to the usual sense).

For a sufficiently large finite set of primes of $Q$ containing $S_p$, we have $T_v = \widetilde{a}_v$ and $S_v = \widetilde{d}_v$ on $\Pi^{U_1}$ for all $v \not\in Q$, where $\widetilde{a}_v$ is a lift of $\tr(\rho(\Frob_v))$ and $\widetilde{d}_v$ is a lift of $\Nm_{F/\QQ}(v)^{-1}\det(\rho(\Frob_v))$.  Furthermore, for $\gp \in S_0$, we have $T_{\gp} = \widetilde{a}_{\gp}$ on $\Pi^{U_1}$, where $\widetilde{a}_{\gp} \in \CO^\times$ is a lift of $\chi_{2,\gp}(\Frob_{\gp})$ (more precisely,  $\widetilde{\chi}_{2,\gp}(\Frob_{\gp})$ or $\widetilde{\chi}_{2,\gp}(\Frob_{\gp}) + p^{f_{\gp}} \widetilde{\chi}_{1,\gp}(\Frob_{\gp})$, according to whether or not $\gp \in S_0$). Furthermore for each $\gp \in S_0$, $U_{0,\gp}/U_{1,\gp} \cong \FF_{\gp}^\times$ acts on $\Pi^{U_1}$ via the Teichm\"uller lift of $\prod_{\sigma \in \Sigma_\gp} \overline{\sigma}^{2-\ell_\sigma}$.  

Let $\widetilde{\TT}^Q$ denote the polynomial ring over $\CO$ in the variables $T_v$ and $S_v$ for $v \not\in Q$, and $T_\gp$ for $\gp \in S_p$, and let $\widetilde{\gm}$ denote the kernel of the homomorphism $\widetilde{\TT}^Q \to E$ defined by
$$\begin{array}{c}
T_v \mapsto \tr(\rho(\Frob_v))\quad\mbox{and}\quad S_v \mapsto \Nm_{F/\QQ}(v)^{-1}\det(\rho(\Frob_v)),\quad\mbox{for $v\not\in Q$;}\\
T_{\gp} \mapsto \chi_{2,\gp}(\Frob_{\gp}), \quad\mbox{for $\gp\in S_p$.}
\end{array}$$
Let $\chi = \prod_{\gp \in S_0} \overline{\sigma}^{\ell_\sigma-2}$ and $\widetilde{\chi}$ its Teichm\"uller lift, viewed as characters of $U_{0,p}/U_{1,p} \cong \FF_{\gp}^\times$, inflated to characters of $U_{0,p}$, and consider the action of $\widetilde{\TT}^Q$ on
$$M^B(U_0,\CO(\widetilde{\chi})) = (M^B(U_1,\CO)\otimes_{\CO}\CO(\widetilde{\chi}))^{U_0}.$$
Since the surjective map
$$M^B(U_1,\CO) \otimes_{\CO} \CC = M^B(U_1,\CC)  \onto \Pi^{U_1}$$
is compatible with the action of $U_0$ and $\widetilde{\TT}$, it follows that there is an eigenform in $M^B(U_0,\CO(\widetilde{\chi}))$ for $\widetilde{\TT}^Q$ on which $T_v$ acts as $\widetilde{a}_v$ and $S_v$ as $\widetilde{d}_v$ for $v \not\in Q$, and $T_{\gp}$ as $\widetilde{a}_{\gp}$ for $\gp \in S_p$.  Therefore the maximal ideal $\widetilde{\gm}^Q$ is in the support of $M^B(U_0,\CO(\widetilde{\chi}))$, and hence in that of
$M^B(U_0,E(\chi))$.

Now let $T = \mathrm{Coind}_{U_{0,p}}^{U_p}E(\chi)$.
As in the proof of Theorem~\ref{thm:ord}, we have a $U_p$-equivariant homomorphism 
$\varepsilon: D'_{\vec{\ell},-\vec{1}} \to T$;
in this case it is even injective as a consequence of the irreducibility of $D'_{\vec{\ell},-\vec{1}}$.  Furthermore, we have an endomorphism $T_{\gp}$ on $M^B(U,T)$ for each $\gp \in S_p$ which is compatible with the resulting maps
$$M^B_{\vec{\ell},-\vec{1}}(U,E) \longrightarrow M^B(U,T) \stackrel{\sim}{\lra} M^B(U_0,E(\chi)).$$
We claim that $\widetilde{\gm}_\rho$ is not in the support of the cokernel.  In order to prove this, write 
 $S_0 = \{\gp_1,\ldots,\gp_r\}$, $D_{\vec{\ell},-\vec{1}} = \bigotimes_{i=1}^r D_i$ where $D_i = \bigotimes_{\sigma \in \Sigma_{\gp_i}} \Sym^{\ell_{\sigma}-2} E^2$, $T = \bigotimes_i T_i$ where $T_i = \Coind_{U_{0,\gp}}^{U_{\gp}}\chi_i$ and $\chi_i = \prod_{\sigma\in\Sigma_{\gp_i}} \overline{\sigma}^{\ell_\sigma-2}$, and factor $\varepsilon$ as the composite of morphisms of the form
$$\varepsilon_i: (\bigotimes_{i=1}^{j-1} T_i) \otimes (\bigotimes_{i=j}^r D_i)
 \longrightarrow (\bigotimes_{i=1}^{j} T_i) \otimes (\bigotimes_{i=j+1}^r D_i)$$
for $j=1,\ldots,r$ (defined as the tensor product of the morphism $D'_j \to T_j$ with the identity in the other factors).
The argument in the proof of Theorem~\ref{thm:ord} then shows that $T_{\gp_i}$ annihilates the cokernel of the map on cohomology induced by $\varepsilon_i$, from which it follows that $\widetilde{\gm}_\rho$ is not in its support.  Therefore $\widetilde{\gm}_{\rho}$ is not in the support of the cokernel of the resulting composite morphism.  It follows that $\widetilde{\gm}_\rho$ is in the support of $M^B_{\vec{\ell},-\vec{1}}(U,E)$ and hence that $M^B_{\vec{\ell},-\vec{1}}(U,E)[\widetilde{\gm}_\rho] \neq 0$, which implies that $\rho$ is $S_p$-ordinarily definitely modular.

The proof of $S_p$-ordinarily algebraic modularity is exactly the same, starting instead with the quaternion algebra over $F$ ramified at either $S_\infty \cup \{w \}- \{\sigma_0\}$ or $S_\infty - \{\sigma_0\}$ for some $\sigma_0 \in \Sigma_\infty$. 
\end{proof}

We will also make use of the following refinement of some of the implications in Theorem~\ref{thm:defalggeom}:
\begin{theorem} \label{thm:ordmods} Suppose that $\rho$, $\vec{k}$ and $\vec{m}$ are as in Conjecture~\ref{conj:defalggeom}, and let $S$ be a subset of $S_p$. If $\rho$ is $S$-ordinarily definitely or algebraically modular of weight $(\vec{k},\vec{m})$, then $\rho$ is $S$-ordinarily geometrically modular of weight $(\vec{k},\vec{m})$.
\end{theorem}
\begin{proof} Firstly, as usual we may twist to reduce to the case 
$\vec{m} = - \vec{1}$ and suppose the existence of a non-zero $\varphi \in M^B_{\vec{k},-\vec{1}}(U,\Fpbar)[\gm_\rho^Q]$ (with $B$ as in Definition~\ref{def:ordalgmod}) such that for each $\gp \in S$, we have $T_\gp \varphi = a_\gp \varphi$ 
from some $a_\gp \in \CO^\times$.

We then proceed exactly as in the proof of part (2) of Theorem~\ref{thm:defalggeom}, except that instead of $\TT$, we use the polynomial ring  $\widetilde{\TT}$ over (some sufficiently large) $\CO$ in the variables $T_v$ and $S_v$ for $v\not\in Q$, and $T_{\varpi_{\gp}}$ for $\gp \in S$.  Letting $\xi$, $a_v'$ and $d_v'$ (for $v\not\in Q$) be as in the proof of Theorem~\ref{thm:defalggeom} and setting $a_{\gp}' = \epsilon^{-1}\xi(\varpi_{\gp}^{(p)})a_{\gp}$ for $\gp \in S$, 
we get that $\widetilde{\gm}$ is in the support of $M^B_{\vec{k}}(U,E)$, where
$\widetilde{\gm}$ is the $\CO$-algebra homomorphism defined by
$$
T_v \mapsto a_v' \,\,\mbox{and}\,\, S_v \mapsto d_v',\,\,\mbox{for $v\not\in Q$;}\quad
T_{\gp} \mapsto a_{\gp}', \,\,\mbox{for $\gp\in S_p$.}
$$
It then follows as in the proof there that $M^B_{\vec{k}}(U,\CO)_{\widetilde{\gm}}$ is non-zero and torsion-free. This  implies the existence of an eigenform in $M^B_{\vec{k}}(U,\CO)$ for the operators in $\widetilde{\TT}$, with eigenvalues lifting those above (after possibly enlarging $\CO$).  This in turn gives an automorphic representation to which we may apply the Jacquet--Langlands correspondence to obtain an eigenform in $M_{\vec{k}}(U,\CO)$ with the same property.  Finally reducing mod $\varpi$ and untwisting yields an eigenform $f \in M_{\vec{k},-\vec{1}}(U,E)$ with the same eigenvalues as the original $\varphi$ for $T_v$ and $S_v$ for all $v\not\in Q$, as well as the $T_{\gp}$ for $\gp \in S$.  In particular $\rho_f \sim \rho$ and $f$ is $\gp$-ordinary for all $\gp \in S$, as required.
\end{proof}

\begin{remark} \label{rmk:defalgord}  The proof of part (1) of Theorem~\ref{thm:defalggeom} can be similarly modified to show that if $\rho$ is $S$-ordinarily definitely modular of weight $(\vec{k},\vec{m})$ and $D_{\vec{k},\vec{m}}$ is irreducible, then $\rho$ is $S$-ordinarily algebraically modular of weight $(\vec{k},\vec{m})$.
\end{remark}

\section{Partial weight one modularity} \label{ssec:pp1}
In the final sections, we will specialize to the case where $F$ is a real quadratic field in which $p$ ramifies.
We will determine precisely when the conjecture predicts that a representation $\rho$ arises from a Hilbert modular form of partial weight one and prove strong results in this direction, analogous to those of \cite[\S11]{DS1} for the quadratic inert case. 

In order to do this, we need some extensions of results from \cite{DS1} and \cite{GLS}.  We do this in more general settings, as it requires no extra work and could be useful in extending our 
methods beyond the ramified quadratic case.

\subsection{Normalized and stabilized eigenforms} \label{ssec:eigenforms}
We start with slight generalizations of various notions and results from Chapter~10 of \cite{DS1}.  Recall that it is assumed there that $p$ is unramified in $F$, and also that $k_\sigma \ge 2$ for all $\sigma \in \Sigma$ in statements involving Hecke operators at primes over $p$.  Using the results of \cite{FD:KS} (and \cite{compact}), we may remove the assumption that $p$ be unramified and relax the assumption on the weight to require only that $k_\sigma \ge 1$ for all $\sigma \in \Sigma$.
The proofs in \cite{DS1} carry over essentially without change, except to adapt them to our notation and conventions. 

Firstly, we have the following strengthening of \cite[Lemma~10.2.1]{DS1}:
\begin{lemma} \label{lem:levelU1}
If $\rho:G_F \to \GL_2(\Fpbar)$ is irreducible and geometrically modular of weight $(\vec{k},\vec{m})$, then there exist an ideal $\gn$ prime to $p$ and a finite extension $E$ of $\F_p$ in $\Fpbar$
such that $\rho$ is equivalent to $\rho_f$ for some 
$f \in S_{\vec{k},\vec{m}}(U_1(\gn),E)$, where $f$ is an eigenform for the operators $T_v$ for all $v\nmid p$ and $S_v$ for all $v\nmid \gn p$.  Furthermore if $\vec{k} \in \ZZ_{\ge 1}^{\Sigma}$, $\vec{m} = -\vec{1}$, $S\subset S_p$ and $\rho$ is $S$-ordinarily geometrically modular of weight $(\vec{k},\vec{m})$, then we may take $f$ to be an eigenform for $T_\gp$ for all $\gp \in S_p$, with $T_{\gp}f \neq 0$ for all $\gp \in S$.
\end{lemma} 
Indeed the argument in \cite{DS1} starts with an eigenform, say  $g$, in $M_{\vec{k},\vec{m}}(U(\gn),E)$ for the operators $T_v$ and $S_v$ for all $v \nmid \gn p$ giving rise to $\rho$, and produces one in $M_{\vec{k},\vec{m}}(U_1(\gn),E)$ after possibly shrinking $\gn$ and enlarging $E$. By the remark after Definition~\ref{def:geomod}, we may replace $M_{\vec{k},\vec{m}}(U_1(\gn),E)$ by $S_{\vec{k},\vec{m}}(U_1(\gn),E)$, and the fact that we may take $f$ to be an eigenform for the $T_v$ for $v|\gn$  follows from their commutativity with each other and the $T_v$ and $S_v$ for $v\nmid \gn p$.  Under the additional hypothesis, the commutativity of the $T_v$ and $S_v$ for $v\nmid \gn p$ with the $T_{\gp}$ for $\gp \in S$ allows us to assume $g$ is also an eigenform for the $T_{\gp}$, with eigenvalue in $E^\times$ if $\gp \in S$.  The fact that the same will be true for the resulting eigenform in $M_{\vec{k},-\vec{1}}(U_1(\gn),E)$ follows from the commutativity between the $T_{\gp}$ and the double coset operators in the construction of $f$. (Note that we allow $S$ to be empty.)

As in \cite[\S10.3]{DS1}, we may refine the twisting construction of \S\ref{ssec:twist} in order to work with eigenforms of level $U_1(\gn)$ (for varying $\gn$). To align with the conventions of this paper, let us fix $t_1 \in (\AA_{\f}^{(p)})^\times$ and 
$m_1 \in F_+^\times$ such that 
$m_1 t_1$ generates $\widehat{\gd}^{-1} = \gd^{-1} \otimes_{\CO_F}\widehat{\CO}_F$, and let $c = (t_1 m_1)^{-1}$.

Suppose now that $\ell \in \ZZ^\Sigma$ and $\xi: \A_F^\times/F^\times F_{\infty,+}^\times \to E^\times$ is a continuous character such that $\xi(u_p) = \overline{u}_p^{-\vec{\ell}}$ for all $u_p \in \CO_{F,p}^\times$, and let $\gm$ denote its conductor, i.e., the largest (prime-to-$p$) ideal $\gm$ of $\CO_F$ such that $\xi$ is trivial on $\ker ((\widehat{\CO}_F^{(p)})^\times \onto (\CO_F/\gm)^\times )$. Note that the restriction of $\xi$ to $(\A_{F,\f}^{(p)})^\times$ is a character as in \S\ref{ssec:twist}, i.e., $\xi(\alpha) = \alpha^{\vec{\ell}}$ for all $\alpha \in \CO_{F,(p),+}^\times$.

The construction in \cite[\S10.3]{DS1} then yields the following:
\begin{lemma} \label{lem:U1twist}  Suppose that $f \in S_{\vec{k},\vec{m}}(U_1(\gn),E)$.  Then there exists an element $f_{\xi} \in S_{\vec{k},\vec{\ell}+\vec{m}}(U_1(\gm^2\gn),E)$ with $q$-expansion coefficients given by
$$r^t_m(f_\xi) = \left\{\begin{array}{cl}
\xi((ctm)^{(\gm)})r^t_m(f),& \mbox{if $(ctm)_v \in \CO_{F,v}^\times$ for all $v|\gm$;}\\
0,& \mbox{otherwise;}\end{array}\right.$$
for all $t \in (\AA_{F,\f}^{(p)})^\times$ and $m \in t^{-1}\widehat{\gd}^{-1} \cap F^\times_+$ (where $\cdot^{(\gm)}$ denotes the projection to components prime to $\gm$).  In particular, if 
$r^t_m(f) \neq 0$ for some $t,m$ such that $(ctm)_v \in \CO_{F,v}^\times$ for all $v|\gm$, then $f_\xi \neq 0$, in which case if $f$ is an eigenform for $T_v$ for all $v\nmid p$ and $S_v$ for all $v\nmid \gn p$, then so is $f_\xi$ for $T_v$ for all $v\nmid p$ and $S_v$ for all $v\nmid \gm\gn p$, and $\rho_{f_\xi} \sim \chi \otimes \rho_f$ where $\chi$ is the character corresponding to $\xi$ via class field theory.
\end{lemma}
In transporting the argument from \cite[\S10.3]{DS1}, note that our $\xi$ is the character denoted there by $\xi'$.  We also point out two typos there: $\xi$ should be $\xi^{-1}$ in the definition of $\xi_{\gm}^{-1}$ just before (10.7), and the level $\gm\gn^2$ should $\gm^2\gn$ just before (10.8).  Furthermore the assumption that $\gm|\gn$ has implicitly been dropped at that point.  (In fact, if we maintain the assumption that $\gm|\gn$, then the level could be taken as $\gm\gn$ there, as well as the conclusions of \cite[Lemma~10.3.2]{DS1} and Lemma~\ref{lem:U1twist} above.)

\begin{remark} \label{rmk:U1twist} Note that we necessarily have $ctm \in \widehat{\CO}_F$ in the formula for the $q$-expansion coefficient, and the condition prescribing the first case is that the ideal it defines be prime to $\gm$.  With this interpretation, the lemma holds also for $f \in M_{\vec{k},\vec{m}}(U_1(\gn),E)$ (where now $f_{\xi} \in M_{\vec{k},\vec{\ell}+\vec{m}}(U_1(\gm^2\gn),E)$); more precisely, in the case $\gm = \CO_F$, replace $(ctm)^{(\gm)}$ by $ct$ to obtain a formula that applies also to $m=0$.
\end{remark}

We generalize \cite[Defn.~10.5.1]{DS1} as follows (restricted to cusp forms and adapted to our conventions):
\begin{definition} \label{def:normalized}  Suppose that $\vec{k},\vec{m} \in \ZZ^\Sigma$, with $k_\sigma \ge 1$ for all $\sigma \in \Sigma$.  We say that an element $f \in S_{\vec{k},\vec{m}}(U_1(\gn),E)$ is a {\em normalized eigenform} if the following hold:
\begin{itemize}
\item $r_{t_1}^{m_1}(f) = 1$;
\item $f$ is an eigenform for $T_v$ for all $v\nmid p$, and $S_v$ for all $v\nmid \gn p$;
\item $f_\xi \in  S_{\vec{k},-\vec{1}}(U_1(\gm^2\gn),E')$ is an eigenform for $T_\gp$ for all $\gp \in S_p$ and all characters $\xi:\A_F^\times/F^\times F_{\infty,+}^\times \to E'^\times$ such that $\xi(u_p) = \overline{u}_p^{\vec{m}+\vec{1}}$ for all $u_p \in \CO_{F,p}^\times$ (where $E'$ is any finite extension of $E$ and $\gm$ is the prime-to-$p$ conductor of~$\xi$).
\end{itemize}
\end{definition}
Note that the first condition implies that the forms $f_\xi$ in the third are non-zero. Furthermore, if $f$ is a normalized eigenform, then so is $f_{\xi'} \in S_{\vec{k},\vec{\ell}+\vec{m}}(U_1(\gm'^2\gn),E')$ for any character $\xi': \A_F^\times/F^\times F_{\infty,+}^\times \to E'^\times$  
such that $\xi'(u_p) = \overline{u}_p^{-\vec{\ell}}$ for all $u_p \in \CO_{F,p}^\times$ (where $\gm'$ is the prime-to-$p$ conductor of~$\xi'$).

\begin{proposition} \label{prop:normalized}
If $\rho:G_F \to \GL_2(\Fpbar)$ is irreducible and $S$-ordinarily geometrically modular of weight $(\vec{k},\vec{m})$ with $k_\sigma \ge 1$ for all $\sigma \in \Sigma$, then there exist an ideal $\gn$ prime to $p$ and a finite extension $E$ of $\F_p$ in $\Fpbar$
such that $\rho$ is equivalent to $\rho_f$ for some 
normalized eigenform $f \in S_{\vec{k},\vec{m}}(U_1(\gn),E)$ such that $T_{\gp} f_\xi \neq 0$ for all $\gp \in S$ and $\xi$ as in Definition~\ref{def:normalized}.
\end{proposition}
\begin{proof} The argument is essentially the same as the proof of \cite[Prop.~10.5.2]{DS1}.  

Suppose first that $\vec{m} = - \vec{1}$.  Then by Lemma~\ref{lem:levelU1}, there exist $\gn$ and $E$ such that $\rho$ arises from an eigenform $f \in S_{\vec{k},-\vec{1}}(U_1(\gn),E)$ for $T_v$ for all $v$ and $S_v$ for all 
$v\nmid \gn p$, with $T_{\gp}f \neq 0$ for all $\gp \in S$.  Therefore it suffices to prove that 
$r_{t_1}^{m_1}(f) \neq 0$, as then Lemma~\ref{lem:U1twist} 
implies the third condition in Definition~\ref{def:normalized} holds (and we may rescale $f$ to achieve the first condition).

We prove that if $r_{t_1}^{m_1}(f) = 0$, then $r_t^m(f) = 0$ for all $t \in (\A_{F,\f}^{(p)})^\times$ and $m \in t^{-1}\widehat{\gd}^{-1} \cap F^\times_+$, contradicting that $f\neq 0$.  We proceed 
by induction on $||ctm||^{-1} = \mathrm{Disc}_{F/\QQ}^{-1}||tm||^{-1} \in \ZZ_{>0}$ (where $c = (t_1m_1)^{-1}$).

Firstly if $||ctm|| = 1$, then letting $u = (ctm)^{(p)} \in (\widehat{\CO}_F^{(p)})^\times$ and $\alpha = m^{-1}m_1 \in \CO_{F,(p),+}^\times$, we have $r_m^t(f) = r_{m_1}^{t_1}(f) = 0$ by (\ref{eqn:qexpinv}).

For the induction step, suppose that $||ctm|| < 1$ and that $r^{t'}_{m'}(f) = 0$ for all $t',m'$ such that $||tm|| < ||t'm'||$.  Since $ctm \in \widehat{\CO}_F$ and $||ctm|| < 1$, there is a prime $v$ such that $m \in v t^{-1}\widehat{\gd}^{-1}$.

Suppose first that $v \nmid p$. If $v|\gn$ or $m \not\in v^2 t^{-1}\widehat{\gd}^{-1}$, then the formula for the effect of $T_v$ on $q$-expansions gives
$$\Nm_{F/\Q}(v) r_m^t(f) = r_m^{\varpi_v^{-1}t}(T_vf ) =
 a_v r_m^{\varpi_v^{-1}t}(f) = 0$$
by the induction hypothesis, where $a_v$ is the eigenvalue of $T_v$ on $f$.
 
If $v\nmid \gn$ and $m \in v^2 t^{-1}\widehat{\gd}^{-1}$, then
we proceed similarly using instead that 
$$\Nm_{F/\Q}(v) r_m^t(f) = r_m^{\varpi_v^{-1}t}(T_vf )
 - r_m^{\varpi_v^{-2}t}(S_vf).$$

If $v|p$, then the argument is similar to the case $v\nmid p$, with the subcase where $k_{\sigma} = 1$ for all $\sigma \in \Sigma_v$ corresponding to the one where $v\nmid \gn$.

This completes the proof in the case $\vec{m} = -\vec{1}$, and the general case follows by twisting as in \cite{DS1}.  
\end{proof}

\begin{definition} \label{def:stabilized}  Suppose that $\vec{k},\vec{m} \in \ZZ^\Sigma$, with $k_\sigma \ge 1$ for all $\sigma \in \Sigma$.  We say that a normalized eigenform $f \in S_{\vec{k},\vec{m}}(U_1(\gn),E)$ is {\em stabilized} if $T_vf = 0$ for all $v|\gn$.
\end{definition}

We have the following refinement of \cite[Lemma10.6.2]{DS1}
(by the same proof):
\begin{lemma} \label{lem:stabilized} Suppose that $\vec{k},\vec{m} \in \ZZ^\Sigma$, with $k_\sigma \ge 1$ for all $\sigma \in \Sigma$, and that $\gn \subset \gm$ are non-zero ideals of $\CO_F$ prime to $p$.  
If $\rho \sim \rho_f$ for some normalized eigenform 
$f \in S_{\vec{k},\vec{m}}(U_1(\gm),E)$, then $\rho \sim \rho_{f'}$ for some normalized eigenform $f' \in S_{\vec{k},\vec{m}}(U_1(\gn),E')$ such that $a_{\gp,\xi}(f) = a_{\gp,\xi}(f')$ for all characters $\xi$ as in Definition~\ref{def:normalized}, where $E'$ is (at most) a quadratic extension of $E$.
Furthermore if $\ord_v(\gn) \ge \min(2,1+\ord_v(\gm))$ for all $v|\gn$, then we may take $f'$ as above to be a stabilized eigenform.
\end{lemma}

Note that if $f$ is a stabilized eigenform, then $\rho_f$ and $\gn$ determine its Hecke eigenvalues for all $T_v$ for $v\nmid p$ and $S_v$ for $v\nmid \gn p$, and this in turn determines the $q$-expansion coefficients $r_m^t(f)$ for all $t$ and $m$ such that $(ctm)_p \in \CO_{F,p}^\times$.  However this information does not necessarily determine the eigenvalue of $T_\gp$ on $f_\xi$ for $\gp \in S_p$ and characters $\xi$ as
in Definition~\ref{def:normalized}. 
Let us denote this eigenvalue by $a_{\gp,\xi}(f)$.  As in the discussion before Definition~\ref{def:ordgeommod}, these are related by the formula
$$ a_{\gp,\xi_2}(f) = \mu(\Frob_\gp) a_{\gp,\xi_1}(f)$$
as $\xi$ varies, where $\mu$ is the unramified character $G_{F_{\gp}} \to E^\times$ corresponding via local class field theory to $\xi_2\xi_1^{-1}|_{F_{\gp}^\times}$.  

If $k_\sigma > 1$ for some $\sigma \in \Sigma_{\gp}$, then Theorem~\ref{thm:ord} implies that the $a_{\gp,\xi}(f)$ can only be non-zero if the representation $\chi\otimes\rho_f$ has a non-trivial unramified subrepresentation, or equivalently, if $\rho_f$ has a non-trivial subrepresentation on which $I_{F_\gp}$ acts as $\prod_{\sigma \in \Sigma_\gp}\omega_\sigma^{-1-m_\sigma}$ (in the notation of  \S\ref{ssec:galoistwist}).  Note also that these subrepresentations (of which there are at most two) determine the possible (non-zero) systems of values of $a_{\gp,\xi}(f)$.

On the other hand, if $k_\sigma = 1$ for all $\sigma \in \Sigma_{\gp}$, then $\rho_f|^{\mathrm{ss}}_{G_{F_{\gp}}}$ is unramified and the $a_{\gp,\xi}(f)$ are determined
by Proposition~\ref{prop:ppw1}.

\begin{definition} \label{def:strong}  Suppose that $\vec{k},\vec{m} \in \ZZ^\Sigma$, with $k_\sigma \ge 1$ for all $\sigma \in \Sigma$, and that $\gp \in S_p$ is such that $k_\sigma > 1$ for some $\sigma \in \Sigma_\gp$.   We say that a stabilized eigenform $f \in S_{\vec{k},\vec{m}}(U_1(\gn),E)$ is {\em strongly $\gp$-stabilized} if $T_\gp(f_\xi) = 0$ for all $\xi$ as in Definition~\ref{def:normalized}, or equivalently if $r_m^t(f) = 0$ for all $t$ and $m$ such that $(ctm)_\gp \not\in \CO_{F,\gp}^\times$.  We say that $f$ is {\em strongly stabilized} if it is strongly $\gp$-stabilized for all $\gp \in S_p$ is such that $k_\sigma > 1$ for some $\sigma \in \Sigma_\gp$.
\end{definition}

By the discussion preceding Definition~\ref{def:strong}, if $f \in S_{\vec{k},\vec{m}}(U_1(\gn,E))$ is a stabilized eigenform and $\gp \in S_p$ is such that $k_\sigma > 1$ for some $\sigma \in \Sigma_\gp$, then it is automatically strongly $\gp$-stabilized unless $\rho_f$ has a non-trivial subrepresentation on which $I_{F_\gp}$ acts as $\prod_{\sigma \in \Sigma_\gp}\omega_\sigma^{-1-m_\sigma}$. On the other hand, if $\rho_f$ has such a subrepresentation, then it does not necessarily arise from 
a strongly $\gp$-stabilized eigenform of the same weight 
$(\vec{k},\vec{m})$.  Note however that $\Theta_\tau(f)$ is strongly $\gp$-stabilized for any $\tau \in \Sigma_{\gp,0}$.

The same argument as for \cite[Lemma~10.6.5]{DS1}, together with Proposition~\ref{prop:ppw1}, gives the following:
\begin{lemma} \label{lem:strong} Suppose that $\vec{k},\vec{m} \in \ZZ^\Sigma$, with $k_\sigma \ge 1$ for all $\sigma \in \Sigma$, and that $\gn$ is a non-zero ideal of $\CO_F$ prime to $p$.  If 
$\rho:G_F \to \GL_2(E)$ is absolutely irreducible, then 
$\rho \sim \rho_f$ for at most one strongly stabilized eigenform $f \in S_{\vec{k},\vec{m}}(U_1(\gn),E)$.
\end{lemma}

\subsection{Quado-Barsotti--Tate representations} \label{subsec:qBT}
In order to describe local Galois representations with repeated 
$\sigma$-labelled Hodge--Tate weights, we need a slight generalization of results of Gee, Liu and Savitt~\cite{GLS}.  (See  \cite[\S3]{hanneke_new} for an extension along these lines in the unramified case.)

Fix a finite extension $L$ of $\QQ_p$, write $L_0$ for its maximal unramified subextension, and let  $f = [L_0:\QQ_p]$, $e = [L:L_0]$.  Let $\Sigma_L$ denote the set of embeddings $\sigma: L \hookrightarrow \Qpbar$, which we index as $\sigma_{i,j}$ for $i \in \ZZ/f\ZZ$, $j = 1,\ldots,e$, with our usual convention that for each $i$, the embeddings $\sigma_{i,1},\ldots,\sigma_{i,e}$ have the same restriction to $L_0$, say $\tau_i$, and that $\tau_{i+1} = \phi\circ\tau_i$. 

Recall that \cite{GLS} gives a complete description of reductions of certain crystalline representations $\widetilde{\rho}:G_L \to \GL_2(K)$ which they call {\em pseudo-Barsotti--Tate}, meaning that for each $i$, the $\sigma_{i,j}$-labelled Hodge--Tate weights have the form $(r_{i,j},0)$, where $r_{i,e} \in [1,p]$ and $r_{i,j} = 1$ for $j < e$. We wish to allow some of the $r_{i,j}$ to equal $0$, and we will refer to such representations as {\em quado-Barsotti--Tate} (a concatenation of {\em quasi-pseudo-Barsotti--Tate}). Since the results are independent of the choice of ordering of $\sigma_{i,1},\ldots,\sigma_{i,e}$ for each $i$, we may further assume that
$$ 0 \le r_{i,1} \le r_{i,2} \le \cdots \le r_{i,e} \le p,$$
so that the quado-Barsotti--Tate condition becomes $r_{i,e-1} \le 1$ for each $i$.  For notational convenience, we set $r_{i,0} = 0$.

Suppose now that $\widetilde{\rho}$ is quado-Barsotti--Tate with $\sigma_{i,j}$-labelled Hodge--Tate weights $(r_{i,j},0)$ as above.   We define the {\em degeneracy degree} of $\widetilde{\rho}$ to be the element $\vec{\delta} \in [0,e]^f$ such that $\delta_i = \max\{\,j\,|\,r_{i,j} = 0\,\}$.  (Thus a quado-Barsotti--Tate representation is pseudo-Barsotti--Tate if and only if it has degeneracy degree $\vec{0}$.)

We then have the following extension of \cite[Cor.~2.3.10]{GLS} (as extended in \cite{wang} to include the case $p=2$) to the quado-Barsotti--Tate setting.  We adopt the same notation as in \cite{GLS} (except that our $L$ is their $K$, our $K$ is their $E$, their indexing of embeddings is the reverse of ours, and we let $\prod_{i=1}^n A_i$ denote the matrix product $A_nA_{n-1}\cdots A_1$). 
\begin{proposition} \label{prop:GLS.2.3.10} Let $\mathfrak{M}$ be the Kisin module corresponding to a lattice in a quado-Barsotti--Tate representation $\widetilde{\rho}$ as above.  There exist matrices $Z'_{i,j} \in \GL_2(\CO_K)$ for $j=0,\ldots,e-1$ such that $\mathrm{Fil}^{p,p,\ldots,p}\mathfrak{M}_i^* = \mathfrak{S}_{\CO_K,i}\alpha_{i,e-1} \oplus \mathfrak{S}_{\CO_K,i}\beta_{i,e-1}$ with
$$(\alpha_{i,e-1},\beta_{i,e-1}) = (\mathfrak{e}_i',\mathfrak{f}_i')\Lambda'_{i,e}\left(\prod_{j=1}^{e-1} Z_{i,j}'\Lambda_{i,j}'\right),$$
where $\mathfrak{e}_i'$, $\mathfrak{f}_i'$ is a basis of $\mathfrak{M}_i^*$ (as in \cite[Prop.2.3.5(2)]{GLS} if $\delta_i < e$) and
$$\Lambda_{i,j}' = \left(\begin{array}{cc} (u-\pi_{i,j})^p & 0 \\
0 & (u-\pi_{i,j})^{p-r_{i,j}} \end{array}\right).$$
\end{proposition}
We omit the proof, which is essentially the same as that of \cite[Cor.2.3.10]{GLS}, except that the induction step (now downwards on $m$ due to our different conventions) becomes easier for $m \le \delta_i$, as does the base case if $\delta_i = e$.  Indeed, if $\delta_i = e$, then the base case is immediate from \cite[Lemma~2.2.2(3)]{GLS}.  Furthermore, if $m \le \delta_i$, then the arguments in the third through fifth paragraphs of the proof are unnecessary. To obtain the analogue of the displayed equation in the final paragraph, one may simply take $Z_{i,m}'$ to be the identity matrix, and apply\footnote{Note that the reference there to Lemma~2.2.2(2) should be to Lemma~2.2.2(3).}~\cite[Lemma~2.2.2(3)]{GLS} using that $\Fil^1D_{L,i,m} = 0$.

We then obtain from this the obvious generalization of \cite[Thm.~2.4.1]{GLS} with $\Lambda_{i,j} = \left(\begin{array}{cc} 1 & 0 \\
0 & (u-\pi_{i,j})^{r_{i,j}} \end{array}\right)$.  The proof of Proposition~3.1.3 of \cite{GLS} then carries over without change as well, giving the same statement, but with $ s_i = \sum_{j \in J_i} r_{i,j}$ for some subset $J_i \subset \{1,\ldots,e\}$, so that $s_i = x_i$ or $r_{i,e} + x_i$ for some $x_i \in [0,e-1-\delta_i]$, unless $r_{i,e}= 0$ (i.e., $\delta_i = e$), in which case $s_i = 0$. This in turn yields the following generalization of \cite[Thm.~3.1.4]{GLS}:
\begin{theorem} \label{thm:GLS.3.1.4} Let $T$ be a lattice in a quado-Barsotti--Tate representation as above, with $\sigma$-labelled Hodge--Tate weights $(r_\sigma,0)$ for $\sigma \in \Sigma_L$.  If $\overline{T} = T\otimes_{\CO_K} E$ is reducible,  then there is a subset $J \subset \Sigma_L$ and a basis for $\overline{T}$ with respect to which the action of $I_L$ has the form
$$\left(\begin{array}{cc} 
\prod_{\sigma \in J} \omega_\sigma^{r_\sigma} & * \\
0 & \prod_{\sigma \not\in J} \omega_\sigma^{r_\sigma} \end{array}\right),$$
where $\omega_\sigma:I_L \to E^\times$ is the fundamental character corresponding via local class field theory to the character $\CO_L^\times \to \CO_K^\times \to E^\times$ induced by $\sigma$.
\end{theorem}
Note that $\omega_{\sigma_{i,j}}$ depends only on $i$, and we henceforth denote it $\omega_i$, so $\omega_{i} = \omega_{i-1}^p$ has order $p^f - 1$ for all $i$.

The same arguments as in \cite{GLS} give the following generalization of their Theorem~5.1.5 (using their notation, modified as above):
\begin{theorem}\label{thm:GLS.5.1.5} Let $T$ be a $G_L$-stable lattice in a quado-Barsotti--Tate representation as above, with $\sigma_{i,j}$-labelled Hodge--Tate weights $(r_{i,j},0)$ for $\sigma_{i,j} \in \Sigma_L$.  Let $\gM$ be the Kisin module associated to $T$, and let $\overline{\gM} = \gM\otimes_{\CO_K} E$.

Suppose that $\overline{T} := T \otimes_{\CO_K} E$ is reducible, so there exist rank one Kisin modules $\overline{\gN} = \gM(s_0,\ldots,s_{f-1};a)$ and $\overline{\gP} = \gM(t_0,\ldots,t_{f-1};b)$ such that $\overline{\gM}$ is an extension of $\overline{\gN}$ by $\overline{\gP}$.  Then for each $i$ there is an integer $x_i \in [0,\varepsilon_i]$ such that $\{\,s_i,t_i\,\} = \{\,r_i + x_i, \varepsilon_i - x_i\,\}$, where $r_i = r_{i,e}$ and $\varepsilon_i = \displaystyle\sum_{j=1}^{e-1} r_{i,j} = \max\{\,0,e-1-\delta_i\,\}$.

Furthermore we can choose bases $e_i,f_i$ of $\overline{\gM}_i$ so that $\varphi$ has the form
$$\begin{array}{rcl} \varphi(e_{i+i}) & = & (b)_i u^{t_i} e_i \\
\varphi(f_{i+1}) & = & (a)_iu^{s_i} f_i + y_i e_i \end{array}$$
where
\begin{itemize}
\item $y_i \in E[[u]]$ is a polynomial with $\deg(y_i) < s_i$;
\item if $t_i < r_i$, then the nonzero terms of $y_i$ have degrees in the set $[t_i] \cup [r_i,s_i - 1]$;
\item except that when there is a nonzero map $\overline{\gN} \to \overline{\gP}$, we must also allow $y_\iota$ to have a term of degree $s_\iota + \alpha_\iota(\overline{\gN}) - \alpha_\iota(\overline{\gP})$ for any one choice of $\iota$.
\end{itemize}
\end{theorem}

Finally, we need that the proofs of Lemmas~6.1.2 and~6.1.3 of \cite{GLS} carry over\footnote{These are in fact only needed in the case that the semisimplification of $\overline{T}$ is a direct sum of two characters whose ratio is cyclotomic.}
without change for quado-Barsotti--Tate representations. Thus in the setting of Theorem~\ref{thm:GLS.5.1.5}, the exact sequence $0 \to \overline{\gP} \to \overline{\gM} \to \overline{\gN} \to 0$ extends to one of $(\varphi,\widehat{G})$-modules
$$ 0 \lra \widehat{\overline{\gP}} \lra \widehat{\overline{\gM}}
\lra \widehat{\overline{\gN}} \lra 0$$
with $E$-action such that $\overline{T} \cong \widehat{T}(\widehat{\overline{\gM}})$ and for all $x \in \overline{\gM}$, there exist $\alpha \in R$ and $y \in R\otimes_{\varphi,\gS} \overline{\gM}$ such that $\tau(x) - x = \alpha y$ and $\nu_R(\alpha) \ge \frac{p}{p-1} + \frac{p}{e}$. Furthermore this $(\varphi,\widehat{G})$-module is unique, except in the case $\vec{r} = \vec{p}$, $\vec{\delta} = \vec{t} = \vec{0}$.

To make the statements involving $(\varphi,\widehat{G})$-modules less cumbersome, we will say that a $(\varphi,\widehat{G})$-module with $E$-action is {\em typical} if it satisfies the condition that for all $x \in \overline{\gM}$, there exist $\alpha \in R$ and $y \in R\otimes_{\varphi,\gS} \overline{\gM}$ such that $\tau(x) - x = \alpha y$ and $\nu_R(\alpha) \ge \frac{p}{p-1} + \frac{p}{e}$.

\subsection{The ramified quadratic case: $p$-adic Hodge theory} \label{subsec:hodge}
Now we assume that $L$ is a ramified quadratic extension of $\QQ_p$.  We simplify some of the notation of the preceding section as follows:  We write $\Sigma = \{\sigma_1,\sigma_2\}$ for $\Sigma_L$,  $(k_1,k_2)$ for the  element $(k_{\sigma_1},k_{\sigma_2}) \in \ZZ^\Sigma$, and $\omega:I_L \to \FF_p^\times$ for the fundamental character (which was $\omega_1 = \omega_2$ in the preceding section).

For $(\vec{k},\vec{m})\in \ZZ^\Sigma \times \ZZ^\Sigma$, where $\vec{k} = (k_1,k_2)$ is of the form $(1,w)$ with $2\le w \le p+1$, we will describe precisely which representations $\rho:G_L \to \GL_2(\Fpbar)$ have crystalline lifts of weight $(\vec{k},\vec{m})$ in the sense of Definition~\ref{def:weightlift}, i.e., with $\sigma_i$-labelled Hodge--Tate weights $-(1+m_i,k_i+m_i)$ for $i=1,2$.

In order to provide the desired characterization in the case that $\rho$ is reducible, we first define a certain subspace of
$H^1(G_L,\Fpbar(\chi))$ for any character $\chi:G_L \to \Fpbar^\times$.  Note that we may write
$\chi|_{I_L} = \omega^{n-1}$ for a unique $n \in [1,p-1]$. Let $\widetilde{\chi}:G_L \to \Zpbar^\times$ be any crystalline lift of $\chi$ with $\sigma_1$-labeled weight $-1$ and $\sigma_2$-labeled weight $n$.

Recall that $H^1_f(G_L,\Qpbar(\widetilde{\chi}))$ is defined as the kernel of the natural map
$$H^1(G_L,\Qpbar(\widetilde{\chi})) \lra
 H^1(G_L,B_{\mathrm{crys}}\otimes_{\QQ_p}\Qpbar(\widetilde{\chi})), $$
and $H^1_f(G_L,\Zpbar(\widetilde{\chi}))$ as the preimage of 
$H^1_f(G_L,\Qpbar(\widetilde{\chi}))$ in
$H^1(G_L,\Zpbar(\widetilde{\chi}))$.  We then define $V_\chi$ as 
image of $H^1_f(G_L,\Zpbar(\widetilde{\chi}))$ in $H^1(G_L,\Fpbar(\chi))$.  Since $H^1_f(G_L,\Qpbar(\widetilde{\chi}))$ is one-dimensional over $\Qpbar$ (by \cite[3.3.11]{FPR}, for example), it follows that
$$H^1_f(G_L,\Zpbar(\widetilde{\chi}))/H^1(G_L,\Zpbar(\widetilde{\chi}))_{\mathrm{tor}}$$
is free of rank one over $\Zpbar$, and therefore that
$$ \dim_{\Fpbar} V_\chi = \left\{\begin{array}{ll} 1,&\mbox{if $\chi \neq 1$;}\\
2,&\mbox{if $\chi = 1$.}\end{array}\right. $$

We then have the following characterization of the classes\footnote{See \cite{BS} for a more explicit description of such spaces in terms of Kummer theory.} in $V_\chi$ in terms of $(\varphi,\widehat{G})$-modules:
\begin{proposition} \label{prop:Vchi} Suppose that $\rho: G_L \to \GL_2(E)$ is a representation of the form $\psi \otimes \left(\begin{array}{cc}\chi&c\\0&1\end{array}\right)$, where $\psi|_{I_L} = \omega$, $\chi|_{I_L} = \omega^{n-1}$ for some $n \in [1,p-1]$, and
$c \in Z^1(G_L,E(\chi))$. Then $c$ represents a class in $V_\chi$ if and only if $\rho$ arises from a typical $(\varphi,\widehat{G})$-module with $E$-action $\widehat{\overline{\gM}}$ fitting in an exact sequence 
$$ 0 \lra \widehat{\overline{\gP}} \lra \widehat{\overline{\gM}}
\lra \widehat{\overline{\gN}} \lra 0,$$
where the underlying Kisin modules $\overline{\gP}$, $\overline{\gM}$ and $\overline{\gN}$ are as in the conclusion of Theorem~\ref{thm:GLS.5.1.5} with $\overline{\gN} = \gM(n;a)$ and $\overline{\gP} = \gM(1;b)$ for some $a,b\in E^\times$. 

In particular, $V_\chi$ is independent of the choice of the character $\widetilde{\chi}$ in its definition.
\end{proposition}
\begin{proof} Suppose first that $c$ represents a class in $V_\chi$.  Letting $\widetilde{\psi}:G_L \to \CO_K^\times$ be a crystalline lift of $\psi$ with $\sigma_1$ (resp.~$\sigma_2$)-labelled Hodge--Tate weight $0$ (resp.~$1$).  Enlarging $K$ if necessary so that $\widetilde{\chi}$ is valued in $\CO_K^\times$, it follows from the definition of $V_\chi$ that $\rho$ is given by the reduction of a $G_L$-stable lattice $T$ in a crystalline representation $\rho:G_L \to \GL_2(K)$ with $\sigma_1$ (resp.~$\sigma_2$)-labelled Hodge--Tate weights $(1,0)$ (resp.~$(n,0)$).
It follows (from cases already covered by \cite{GLS} and \cite{wang}) that the reduction $\overline{\gM}$ of the Kisin module $\gM$ associated to $T$ is as in the conclusion of
Theorem~\ref{thm:GLS.5.1.5} with $\{s,t\} = \{n,1\}$ or $\{n+1,0\}$. 

First suppose that $\{s,t\} = \{n+1,0\}$, in which case we may assume that $s=n+1$ and $t=0$ (as otherwise the extension of Kisin modules splits and we can exchange $s$ and $t$).  This implies that $\rho$ has a subrepresentation whose restriction to $I_L$ is $\omega^{n+1}$, which is only possible if $n=p-1$ and either $\rho$ is decomposable or $p=2$.  In the first case, the conclusion of the proposition is clear, so suppose that $p=2$.  In that case we obtain a contradiction using Kisin's classification of connected $p$-divisible groups and finite flat group schemes in \cite{K2}.   More precisely, using the notation and terminology of \cite[\S1]{K2}, we have that $T$ is the Tate module of a bi-connected $p$-divisible group, and therefore $\gM$ is formal and $\overline{\gM}$ is connected, contradicting that 
$\psi(\overline{\gP}) = \varphi^*(\overline{\gP})$.

We may therefore assume that $\{s,t\} = \{n,1\}$.  Furthermore if $s=1$ and $t=n > 1$, then the extension of Kisin modules splits, so we can exchange $s$ and $t$.  Finally applying \cite[Lemma~6.1.2]{GLS} (as extended to $p=2$ in \cite{wang}) yields the desired conclusion.

To prove the converse, let $V_\chi'$ denote the set of extension classes arising from $(\varphi, \widehat{G})$-modules as in the statement of the proposition. We have thus already shown that $V_{\chi,E} \subset V_\chi'$ (where $V_{\chi,E} = V_\chi \cap H^1(G_L,E(\chi))$), so it suffices to prove that $\#V_\chi' \le \#V_{\chi,E}$.  Note that $a$ and $b$ are determined by $\psi$ and $\chi$, and \cite[Lemma~6.1.3]{GLS} (again extended to $p=2$ by \cite{wang}) implies that each extension $\overline{\gM}$ as above is the underlying Kisin module for at most one $(\varphi,\widehat{G})$-module as in the statement of the proposition.  Therefore it suffices to note that the number of equivalence classes of extensions of $\gM(n;a)$ by $\gM(1;b)$ as in the statement of Theorem~\ref{thm:GLS.5.1.5} is $\# E$, unless $n = 1$ and $a=b$, in which case it is $\# E^2$; in either case, this coincides with $\#V_{\chi,E}$.
\end{proof}

We need to introduce some more notation to handle the case where $\rho$ is irreducible, let $M$ denote the unramified quadratic extension of $L$, denote the embeddings of its residue field into $\Fpbar$ by  $\tau$ and $\tau'$, and the corresponding characters $I_L = I_M \to \Fpbar^\times$ by $\omega_\tau$ and $\omega_{\tau'}$, so that $\omega_{\tau'} = \omega_\tau^p$ and $\omega = \omega_\tau\omega_{\tau'} = \omega_\tau^{p+1}$.  (Note that $\omega_\tau$ has order $p^2-1$, $\omega$ has order $p-1$ and $\omega^2$ is the cyclotomic character.)
\begin{proposition} \label{prop:repshape} Suppose that $\vec{k},\vec{m} \in \ZZ^\Sigma$ are as above, i.e., $k_1 = 1$ and $k_2 = w$ for some $2 \le w \le p+1$, and let $m = m_1 + m_2$.  A representation $\rho:G_L \to \GL_2(\Fpbar)$ has a crystalline lift of weight $(\vec{k},\vec{m})$ if and only if
$$\rho \sim \psi\otimes\left(\begin{array}{cc}\chi&c\\0&1\end{array}\right)$$
for some characters $\psi$ and $\chi$ such that $\psi|_{I_L} = \omega^{-1-w-m}$, $\chi|_{I_L} = \omega^{w-1}$ and $c \in V_\chi$, or 
$$\rho \sim \Ind_{G_M}^{G_L} \xi$$
for some character $\xi:G_M \to \Fpbar^\times$ such that 
$\xi|_{I_M} = \omega^{-1-w-m}\omega_\tau^{w-1}$.
\end{proposition}
\begin{proof} Firstly, by part (2) of Proposition~\ref{prop:chartwist} (for any $\gp$ and $F$ such that $F_\gp$ is isomorphic to $L$), we may replace $\rho$ by $\mu \otimes \rho$ for any character $\mu:G_L \to \Fpbar^\times$ such that $\mu|_{I_L} = \omega^{m+w+1}$ so as to assume that $\vec{m}= - \vec{k}$.  In this case, the assertion characterizes $\rho$ with  quado-Barsotti--Tate lifts with $\sigma_j$-labelled Hodge--Tate weights $(r_j,0)$, where $r_1 = 0$ and $r_2 = w-1$.

Suppose first that $\rho$ is reducible and that it has such a lift $\widetilde{\rho}:G_L \to \GL_2(K)$ for some sufficiently large finite extension $K$ of $\Q_p$. By Theorem~\ref{thm:GLS.3.1.4}, we have that 
$$\rho|_{I_L} \sim \left(\begin{array}{cc} 
\omega^{w-1} & * \\
0 & 1 \end{array}\right) \quad\mbox{or}\quad
\left(\begin{array}{cc} 
1 & * \\
0 & \omega^{w-1} \end{array}\right),$$
according to whether or not $\sigma_2 \in J$.
Furthermore, in the notation of Theorem~\ref{thm:GLS.5.1.5} (suppressing the subscript $i=0$), we have $\delta=1$, $\varepsilon=0$, $r = w-1$ and $\{s,t\} = \{w-1,0\}$.  If $s = 0$ and $t=w-1$, then the extension of Kisin modules splits, so we may assume $s=w-1$ and $t=0$.  We thus have an exact sequence
$$ 0 \lra \gM(0;b) \lra \overline{\gM} \lra \gM(w-1;a) \lra 0$$
for some $a,b \in E^\times$, extending to one of typical $(\varphi,\widehat{G})$-modules with $E$-action. 
In particular, we have
$$\rho \sim \psi\otimes\left(\begin{array}{cc}\chi&c\\0&1\end{array}\right)$$
for some characters $\psi$ and $\chi$ such that $\psi$ is unramified and $\chi|_{I_L} = \omega^{w-1}$, and it remains to show that $c$ represents a class in $V_\chi$. 

Consider the $(\varphi,\widehat{G})$-module $\widehat{\overline{\gM}'}$ obtained by tensoring $\widehat{\overline{\gM}}$ with $\widehat{\gM}(1;1)$; i.e., its underlying Kisin module is $\overline{\gM}' = \overline{\gM} \otimes_{\gS_E} \gM(1;1)$, and the $\widehat{G}$-action is obtained from those on the tensor product via the canonical isomorphism
$$\widehat{\mathcal{R}}\otimes_{\varphi,\gS}
(\overline{\gM} \otimes_{\gS_E} \gM(1;1))
\cong (\widehat{\mathcal{R}}\otimes_{\varphi,\gS}\overline{\gM})
 \otimes_{\widehat{\mathcal{R}}_E} (\widehat{\mathcal{R}}\otimes_{\varphi,\gS}\gM(1;1)).$$
Thus $\widehat{T}(\widehat{\overline{\gM}'}) \sim (\psi'\otimes\rho)$ for some character $\psi':G_L \to E^\times$ such that $\psi'|_{I_L} = \omega$, and we have an exact sequence
$$0 \lra \widehat{\gM}(1;b) \lra \widehat{\overline{\gM}'} \lra \widehat{\gM}(w;a) \lra 0$$
of typical $(\varphi,\widehat{G})$-modules with $E$-action.
In particular, if $w < p$, then $\widehat{\overline{\gM}'}$ is as in the statement of Proposition~\ref{prop:Vchi} (with $n=w$), so $c$ represents a class in $V_\chi$.  

Suppose now that $w=p$, so that $n=1$ in the definition of $V_\chi$.  In this case, we will define an injective morphism $\widehat{\overline{\gM}'} \to \widehat{\overline{\gM}''}$ for some $(\varphi,\widehat{G})$-module $\widehat{\overline{\gM}''}$ as in the statement of Proposition~\ref{prop:Vchi}.  Note that Theorem~\ref{thm:GLS.5.1.5} provides a basis $(e',f')$ for $\overline{\gM}'$ such that
$$\begin{array}{rcl} \varphi(e') & = & bue' \\
\varphi(f') & = & au^pf' + cue'\, (+\, du^{p+1}e'), \end{array}$$
for some $c$ (and $d$) in $E$, where the term in parentheses occurs only if $a=b$.  We then define $\overline{\gM}''$ to be the Kisin module with basis $(e'',f'')$ and
$$\begin{array}{rcl} \varphi(e'') & = & bue'' \\
\varphi(f'') & = & auf'' + ab^{-1}ce''\, (+\, due''). \end{array}$$
The map $\overline{\gM}' \to \overline{\gM}''$ defined by $e' \mapsto e''$ and $f' \mapsto uf'' + b^{-1}c e''$ is then compatible with $\varphi$, hence defines the desired injective morphism of Kisin modules.  Letting $\widehat{\overline{\gM}''}$ denote the unique extension to a typical $(\varphi,\widehat{G})$-module as in the statement of Proposition~\ref{prop:Vchi}, it is straightforward to check (as in the proof of \cite[Prop.~6.1.7]{GLS}) that $\widehat{\mathcal{R}}\otimes_{\varphi,\gS} \overline{\gM}'$ is stable under the action of $\tau$.  It follows that the map $\overline{\gM}' \to \overline{\gM}''$ extends to a morphism of typical $(\varphi,\widehat{G})$-modules with $E$-action, which is necessarily of the form $\widehat{\overline{\gM}'} \to \widehat{\overline{\gM}''}$ by \cite[Lemma~6.1.3]{GLS}. This implies that $\widehat{T} (\widehat{\overline{\gM}'}) \cong \widehat{T}(\widehat{\overline{\gM}''})$, and hence that $c$ represents a class in $V_\chi$.

Finally suppose that $w = p + 1$.  In this case, we define an injective morphism $\widehat{\overline{\gM}} \to \widehat{\overline{\gM}''}$ for some typical $(\varphi,\widehat{G})$-module $\widehat{\overline{\gM}''}$ as in the description of $\widehat{\overline{\gM}}$, but with $w=2$ instead of $w=p+1$.
Indeed by Theorem~\ref{thm:GLS.5.1.5}, we have a basis $(e,f)$ for $\overline{\gM}$ such that
$$\begin{array}{rcl} \varphi(e) & = & be \\
\varphi(f) & = & au^pf + ce\, (+\, du^4e), \end{array}$$
for some $c$ (and $d$) in $E$, where the term in parentheses occurs only if $a=b$ and $p=2$.  We then define $\overline{\gM}''$ to be the Kisin module with basis $(e'',f'')$ and
$$\begin{array}{rcl} \varphi(e'') & = & be'' \\
\varphi(f'') & = & auf'' + ab^{-1}ce''\, (+\, du^2e''). \end{array}$$
The map $\overline{\gM} \to \overline{\gM}''$ defined by $e \mapsto e''$ and $f \mapsto uf'' + b^{-1}c e''$ is then compatible with $\varphi$, hence defines the desired injective morphism of Kisin modules.  Arguing exactly as in the case $w=p$, we find that the morphism extends to one of typical $(\varphi,\widehat{G})$-modules with $E$-action, and hence that $\widehat{T} (\widehat{\overline{\gM}}) \cong \widehat{T}(\widehat{\overline{\gM}''})$.  It now follows from the case $w=2$ (already treated above) that $c$ represents a class in $V_\chi$.

Conversely, suppose that
$$\rho \sim \left(\begin{array}{cc}\chi&c\\0&1\end{array}\right)$$ 
with $\psi$ unramified, $\chi|_{I_L} = \omega^{w-1}$ and $c$ representing a class in $V_\chi$. Choosing an unramified lift $\widetilde{\psi}$ of $\psi$ and a crystalline lift $\widetilde{\chi}'$ of $\chi$ with $\sigma_i$-labelled Hodge--Tate weight $r_i$, we again have that $H^1_f(G_L,\Qpbar(\widetilde{\chi}'))$ is one-dimensional.  Letting $V_\chi''$ denote the image of $H^1_f(G_L,\Zpbar(\widetilde{\chi}'))$ in $H^1(G_L,\Fpbar(\chi))$, we find that $\dim_{\Fpbar}(V_\chi'') = \dim_{\Fpbar}(V_\chi)$.  Since the classes in $V_\chi''$ correspond to reductions of lattices in crystalline representations of the form
$$\widetilde{\psi} \otimes \left(\begin{array}{cc}\widetilde{\chi}'&c\\0&1\end{array}\right),$$
we have already proved that $V_\chi'' \subset V_\chi$, and it follows that $V_\chi = V''_\chi$, and hence that $\rho$ has a (reducible) crystalline lift with the prescribed labelled Hodge--Tate weights.

Suppose now that $\rho$ is irreducible, and therefore that $\rho\sim \Ind_{G_M}^{G_L}\xi$ for some character $\xi: G_M \to E^\times$. Write $\Sigma_M = \{\,\sigma_{0,1}, \sigma_{0,2}, \sigma_{1,1}, \sigma_{1,2}\,\}$ where $\sigma_{i,j}$ is the embedding whose restriction to $L$ is $\sigma_j$ and whose reduction is $\tau$ (resp.~$\tau'$) if $i=0$ (resp.~$1$).

If $\rho$ has a crystalline lift $\widetilde{\rho}:G_L \to \GL_2(\Qpbar)$ with $\sigma_j$ labelled Hodge--Tate weights $(r_j,0)$, then $\widetilde{\rho}|_{G_M}$ is a crystalline lift of $\rho|_{G_M}$ with $\sigma_{i,j}$ labelled Hodge--Tate weights $(r_j,0)$.  Since $\rho|_{G_M}$ is decomposable, Theorem~\ref{thm:GLS.3.1.4} implies that
$$\rho|_{I_M} \sim \omega^{w-1} \oplus 1 \quad\mbox{or }\quad 
\omega_\tau^{w-1} \oplus \omega_{\tau'}^{w-1}.$$
The first possibility contradicts the irreducibility of $\rho$, and the second implies that $\xi|_{I_M} = \omega_\tau^{w-1}$ (after replacing $\xi$ with its conjugate if necessary).

Conversely, if $\xi|_{I_M} = \omega_\tau^{w-1}$, then we may choose a crystalline lift $\widetilde{\xi}:G_M \to \Qpbar^\times$ whose $\sigma_{i,j}$-labelled Hodge--Tate weight is $w-1$ if $(i,j) = (0,2)$ and $0$ otherwise. Then $\Ind_{G_M}^{G_L}\widetilde{\xi}$ is a crystalline lift of $\rho$ with $\sigma_1$-labelled Hodge--Tate weights $(0,0)$ and $\sigma_2$-labelled Hodge--Tate weights $(w-1,0)$, as required.
\end{proof}

\begin{remark} In view of \cite[Lemma~5.4.2]{GLS}, the arguments involving $(\varphi,\widehat{G})$-modules are only needed in the case that $\chi$ is the cyclotomic character, which implies that either $w=3$ or $w=p=2$.  Thus it is only for $p \le 3$ that there is any overlap between these cases and those for which we need the additional arguments in the proof of the proposition comparing extensions (i.e., $w \ge p$).
\end{remark}

\begin{remark} \label{rmk:notsharp} It follows from the proposition that $\rho$ has a crystalline lift of weight $((1,2),\vec{m})$ if and only if it has one of weight $((1,p+1),\vec{m})$ (or indeed of weight $((2,1),\vec{m})$).  Furthermore it follows from Theorem~\ref{thm:cone} that the same equivalence holds with  ``has a crystalline lift'' replaced by ``is geometrically modular.'' This is further indication that the equivalence in Conjecture~\ref{conj:geomweights2} extends to weights slightly outside $\Xi_{\min}^+$.
\end{remark}

\begin{remark} \label{rmk:boundaries} Taking $w=1$ in the statement of Proposition~\ref{prop:repshape}, the corresponding condition on $\rho$ is that it be unramified.

We remark also that in the case $w=p$ of the proposition (in which case $\chi$ is unramified), we could have taken $n=p$ instead of $n=1$ in the definition of $V_\chi$, provided $p > 2$. That this gives the same space of extensions follows for example from the main result of \cite{BS} (see the proof below of Theorem~\ref{thm:elimination}).
\end{remark}

We need one further observation in another special case.  Recall that if $\chi = \chi_{\mathrm{cyc}}$ is the cyclotomic character, then $\dim_{\Fpbar} H^1(G_L,\Fpbar(\chi)) = 3$, unless $\chi$ is also trivial, in which case $\dim_{\Fpbar} H^1(G_L,\Fpbar(\chi)) = 4$.  (Note the latter occurs only if $p=2$, or if $p=3$ and $L = \Q_3(\zeta_3)$.)  In this case, letting $\widetilde{\chi}':G_L \to \Zpbar^\times$ be the cyclotomic character (or indeed any character whose restriction to inertia is cyclotomic), we have that the image of $H^1_f(G_L,\Zpbar(\widetilde{\chi}')$ in $H^1(G_L,\Fpbar(\chi))$ has codimension one (consisting of what are commonly called the {\em peu ramifi\'ees} classes).
Let us denote this subspace by $V^{\mathrm{ty}}_\chi$.

\begin{proposition} \label{prop:typical} If $\chi$ is the cyclotomic character, then $V_\chi \subset V^{\mathrm{ty}}_\chi$.
\end{proposition}
\begin{proof} If $c$ represents a class in $V_\chi$, then the proof of Proposition~\ref{prop:repshape} shows that the representation $\rho:= \left(\begin{array}{cc}\chi&c\\0&1\end{array}\right)$ arises from a typical $(\varphi,\widehat{G})$-module $\widehat{\overline{\gM}}$ as in the statement of Theorem~\ref{thm:GLS.5.1.5}\footnote{In particular, with $E$-action for some sufficiently large $E$.} with $s=r=2$, 
$\delta = 1$ and $t = \varepsilon = x = 0$.
In particular, it follows immediately that $\rho$ arises from 
such a $(\varphi,\widehat{G})$-module $\widehat{\overline{\gM}}$ with $s=2$, $r=x=\varepsilon=1$ and $t = \delta = 0$.

We claim that this is equivalent to $c$ representing a class in $V_\chi^{\mathrm{ty}}$.  Indeed if $c$ represents a class in $V_\chi^{\mathrm{ty}}$, then Theorem~\ref{thm:GLS.5.1.5} and \cite[Lemma~6.1.2]{GLS} imply that $\rho$ arises from such a $(\varphi,\widehat{G})$-module if $p > 2$.  If $p=2$, then Wang's extension of these results (\cite{wang}) implies that either 
$\rho$ arises from such a $(\varphi,\widehat{G})$-module, or from one for which $\overline{\gM}$ is as in the statement of Theorem~\ref{thm:GLS.5.1.5} with $s = t = r = \varepsilon = 1$ and 
$x = \delta = 0$.  In the latter case, Proposition~\ref{prop:Vchi} implies that $c \in V_\chi$, and hence we have already shown in the preceding paragraph that $\rho$ arises from a $(\varphi,\widehat{G})$-module of the desired form.  This proves one direction of the claimed equivalence, and the other follows from a counting argument as in the last paragraph of the proof of Proposition~\ref{prop:Vchi}. 
\end{proof}

\subsection{The ramified quadratic case: Partial weight one modularity} \label{ssec:pp1}
We specialize again to the case where $F$ is a real quadratic field in which $p$ ramifies.  We will now prove cases of Conjecture~\ref{conj:geomweights2} involving partial weight one, i.e., where $\vec{k} = (1,w)$ for some $2 \le w \le p$.

We first assume $w \ge 3$, since the argument for $w = 2$ will require additional ingredients and hypotheses. We remark though that a particular case of the following theorem requires the result for $w=2$ as input (in the form of Corollary~\ref{cor:final}).

To begin with, in the modularity direction we have the following:
\begin{theorem} \label{thm:3toPmod} Let $F$ be a real quadratic field in which $p$ ramified, and let $\gp$ be the prime of $F$ over $p$.
Suppose that $\vec{k},\vec{m} \in \ZZ^\Sigma$, with $\vec{k} = (1,w)$ for some $w \in [3,p]$.  Let $\rho:G_F \to \GL_2(\Fpbar)$ be irreducible and geometrically modular of some weight.   Suppose further that $\rho|_{G_{F(\zeta_p)}}$ is irreducible, and if $p=5$, then the projective image of $\rho|_{G_{F(\zeta_5)}}$ is not isomorphic to $A_5$. If $\rho|_{G_{F_\gp}}$ has a crystalline lift of weight $(\vec{k},\vec{m})$, then $\rho$ is geometrically modular of weight $(\vec{k},\vec{m})$.
\end{theorem}
\begin{proof} Let $L$ denote $F_\gp$, $M$ its unramified quadratic extension, and write $\Sigma_{M_0} = \{\tau,\tau'\}$ and $\Sigma_M = \{\,\sigma_{0,1},\sigma_{0,2},\sigma_{1,1},\sigma_{1,2},\}$ as in the proof of Proposition~\ref{prop:repshape}.  By that proposition, if $\rho|_{G_L}$ has a crystalline lift of weight $(\vec{k},\vec{m})$, then
$$\rho|_{G_L} \sim \psi\otimes\left(\begin{array}{cc}\chi&c\\0&1\end{array}\right)$$
for some characters $\psi$ and $\chi$ such that $\psi|_{I_L} = \omega^{-1-w-m}$, $\chi|_{I_L} = \omega^{w-1}$ and $c \in V_\chi$, or 
$$\rho \sim \Ind_{G_M}^{G_L} \xi$$
for some character $\xi:G_M \to \Fpbar^\times$ such that 
$\xi|_{I_M} = \omega^{-1-w-m}\omega_\tau^{w-1}$.% = \omega_\tau^{-2-m}\omega_{\tau'}^{-1-w-m}$.  

We claim that in either case, $\rho|_{G_L}$ has crystalline lifts of weight $(\vec{k}',\vec{m}')$ and of weight $(\vec{k}'',\vec{m})$ (in the sense of Definition~\ref{def:weightlift}), where $\vec{k}' = (2,w+1)$, $\vec{k}'' = (2,w-1)$ and $\vec{m}' = \vec{m} - (0,1)$.

Indeed in the first case, it follows from the definition of $V_\chi$ that $\rho|_{G_L}$ has a crystalline lift of the form 
 $\widetilde{\psi}\otimes\left(\begin{array}{cc}\widetilde{\chi}&*\\0&1\end{array}\right)$, where $\widetilde{\chi}$ has 
$\sigma_1$ (resp.~$\sigma_2$)-labelled Hodge--Tate $-1$ (resp.~$w$) and $\widetilde{\psi}$ has 
$\sigma_1$ (resp.~$\sigma_2$)-labelled Hodge--Tate $-1-m_1$ (resp.~$-w-m_2$).  Therefore $\widetilde{\rho}|_{G_L}$ has $\sigma_1$ (resp.~$\sigma_2$)-labelled Hodge--Tate weights $(-1-m_1,-2-m_1)$ (resp.~$(-m_2,-w-m_2)$), i.e., $\sigma_i$-labelled Hodge--Tate weights $(-1-m_i',-k_i'-m_i')$.
Similarly, we see that $\rho|_{G_L}$ has a crystalline lift of the form 
 $\widetilde{\psi}'\otimes\left(\begin{array}{cc}\widetilde{\chi}'&*\\0&1\end{array}\right)$ (using also Proposition~\ref{prop:typical} if $\chi$ is cyclotomic), where $\widetilde{\chi}'$ is a crystalline lift of $\chi$ with $\sigma_1$ (resp.~$\sigma_2$)-labelled Hodge--Tate weight $1$ (resp.~$w-2$)
and $\widetilde{\psi}'$ is a crystalline lift of $\psi$ with $\sigma_1$ (resp.~$\sigma_2$)-labelled Hodge--Tate weight $-m_1-2$ (resp.~$1-w-m_2$).  Therefore $\rho|_{G_L}$ has a crystalline lift with $\sigma_i$-labelled Hodge--Tate weights $(-1-m_i,-k_i''-m_i)$.

In the second case, we may choose a crystalline lift $\widetilde{\xi}$ of $\xi$ with $\sigma_{0,1}$ (resp.~$\sigma_{0,2}$, $\sigma_{1,1}$, $\sigma_{1,2}$)-labelled Hodge--Tate weight $-2-m_1$ (resp.$-m_2$, $-1-m_1$, $-w-m_2$), so that $\Ind_{G_M}^{G_L}$ is a crystalline lift of $\rho|_{G_L}$ of weight $(\vec{k}',\vec{m}')$.  Similarly choosing $\widetilde{\xi}'$ with $\sigma_{i,j}$-labelled Hodge--Tate weights $-1-m_1$ (resp.~$-1-m_2$, $-2-m_1$, $1-w-m_2$) yields a crystalline of lift of weight $(\vec{k}'',\vec{m})$.

We may now apply Theorem~\ref{thm:GLSmain} to conclude that $\rho$ is algebraically modular of weights $(\vec{k}',\vec{m}')$ and $(\vec{k}'',\vec{m})$.  Applying Theorem~\ref{thm:defalggeom}, it follows that $\rho$ is geometrically modular of those weights.

By Proposition~\ref{prop:normalized} and Lemma~\ref{lem:stabilized}, we then have that $\rho \sim \rho_{f'} \sim \rho_{f''}$ for some stabilized eigenforms $f' \in S_{\vec{k}',\vec{m}'}(U_1(\gn),E)$ and $f'' \in S_{\vec{k}'',\vec{m}}(U_1(\gn),E)$ for some sufficiently small level $\gn$ prime to $p$ and sufficiently large finite subfield $E$ of $\Fpbar$.  

We claim furthermore that we may take $f'$ to be strongly stabilized. Indeed if $f'$ is not strongly stabilized, then applying Theorem~\ref{thm:ord} to $f'_\xi$, where $\xi$ is as in Definition~\ref{def:normalized}, implies that $\rho_{f'}|_{I_L} \sim \omega^{-2-w-m} \otimes\left(\begin{array}{cc}\omega^{w+1}&*\\0&1\end{array}\right)$.  This contradicts the explicit form above for $\rho|_{G_L}$, unless $w=p-1$ and $\rho|_{I_L} \sim \omega^{-1-m} \oplus \omega^{-2-m}$, in which case the claim follows from Corollary~\ref{cor:final} (and a further shrinking of $\gn$ using Lemma~\ref{lem:stabilized} if necessary).

Now consider the forms\footnote{We will systematically write $H_i$ for $H_{\sigma_i}$ (and similarly $G_i$ for $G_{\sigma_i}$), and omit the subscript for the unique $\Theta$.} $H_2 f'$ and $\Theta f''\in S_{\vec{k}''',\vec{m}'}(\gn,E)$, where $\vec{k}''' = (3,w)$. Since both are strongly stabilized and give rise to $\rho$, Lemma~\ref{lem:strong} implies they are the same, and hence $H_2 | \Theta f''$. Since $p \nmid (w -1)$, it follows from \cite[Thm.~A]{theta} that $f'' = H_2 f$ for some $f \in S_{\vec{k},\vec{m}}(U_1(\gn),E)$, as desired.
\end{proof}

For the argument in the other direction, we will need to assume part (2) of Conjecture~\ref{conj:defalggeom}.  On the other hand, we can remove the Taylor--Wiles hypothesis from the ``easier'' implications in Theorem~\ref{thm:GLSmain} in our setting.  In particular, we have the following:
\begin{theorem} \label{thm:elimination} Suppose that $(\vec{k},\vec{m}) \in \ZZ^{\Sigma}_{\ge 2} \times \Sigma$, with 
$k_1 = 2 $ and $2 \le k_2 \le p + 1$. 
If $\rho$ is definitely or algebraically modular of weight 
$(\vec{k},\vec{m})$, then $\rho|_{G_{F_{\gp}}}$ has a crystalline lift of weight $(\vec{k},\vec{m})$.
\end{theorem}
\begin{proof} By Theorem~\ref{thm:defalggeom}, we may assume\footnote{Alternatively, a similar (but slightly simpler) argument than we are about to give applies in the case of definite modularity.} that $\rho$ is algebraically modular of weight $(\vec{k},\vec{m})$.
Let $B$, $U$ and $Q$ be as in Definition~\ref{def:algmod} (with $B$ ramified at exactly one archimedean place), so that $\gm_\rho^Q$ is in the support of $M_{\vec{k},\vec{m}}^B(U,\Fpbar)= H^1(Y_U^B,\mathcal{D}_{\vec{k},\vec{m},\Fpbar})$.

Suppose that $\theta: \GL_2(\F_p) \to \Aut_K(V)$
 (for some sufficiently large $K$) 
is an absolutely irreducible representation that has $D_{\vec{k},\vec{m},\Fpbar}$ as a Jordan--H\"older factor in its reduction, i.e., in that of any $\GL_2(\FF_{\gp})$-stable lattice $T$ in $V$. 
Viewing $\theta$ as a representation of $U_\gp = \CO_{B,\gp}^\times \cong \GL_2(\CO_{F,\gp})$ via inflation, let $\mathcal{T}$ be the locally constant sheaf associated to $T$ as in \S\ref{ssec:indefQO}. (Note that we may assume $U^p$ is sufficiently small that this is defined.)
The same argument as in the proof of Lemma~\ref{lem:JHfactors} then shows that $\gm_\rho^Q$ is in the support of $H^1(Y_U^B,\mathcal{T}\otimes_{\CO}E)$, and hence in that of $H^1(Y_U^B,\mathcal{T})$.  Similarly, it follows that $\gm_\rho^Q$ is in the support of its maximal torsion-free quotient $H^1(Y_U^B,\mathcal{T})_{\mathrm{tf}}$, and hence in that of
$$H^1(Y_U^B, \mathcal{T}) \otimes_{\CO} K 
 \cong (H^1(Y_{U'}^B,K) \otimes_K V)^{\GL_2(\FF_p)},$$
where $U' = U^p U_{\gp}'$ and $U_{\gp}'$ is the kernel of the projection $U_{\gp} \to \GL_2(\FF_p)$.  Therefore there is a $\TT^Q$-eigenform 
$$f \in (M^B_{\vec{2}}(U',\Qpbar)\otimes_K V)^{\GL_2(\FF_p)}$$
such that $T_v f = \widetilde{a}_v f$ and $S_v f = \widetilde{d}_v$ for some lifts $\widetilde{a}_v$ of $\tr(\rho(\Frob_v))$ and 
$\widetilde{d}_v$ of $\Nm_{F/\QQ}(v)\det(\rho(\Frob_v))$ for all $v\not\in Q$, and hence a cuspidal automorphic representation $\pi \in \mathcal{C}_{\vec{2}}^B$ (and embedding $K \hookrightarrow \CC$) such that 
$$(\pi_\gp \otimes_K V)^{U_{\gp}} \neq 0$$
and $T_v = \widetilde{a}_v$ and $S_v = \widetilde{d}_v$ on $\pi_v^{\CO_{B,v}^\times}$ for all $v\not\in Q$.  

Let us again adopt notation from the proof of Proposition~\ref{prop:repshape}, so $L = F_{\gp}$, $M$ is its unramified quadratic extension, and $\omega$, $\omega_\tau$ and $\omega_{\tau'}$ are fundamental characters.

Applying Theorems~\ref{thm:JL} and~\ref{thm:galois0}, we now have $\rho \sim \overline{\rho}_\pi$ for some $\pi \in \mathcal{C}_{\vec{2}}$ with the property that $\rho_{\pi}|_{G_{L}}$ is potentially semistable with $\sigma_i$-labelled Hodge--Tate weights $(0,-1)$ for $i=1,2$, and $\mathrm{WD}(\rho_\pi|_{G_L})^{\mathrm{F-ss}}$ corresponds under local Langlands to a representation 
$\pi_{\gp}$ such that $(\pi_\gp \otimes_K V)^{\GL_2(\CO_L)} \neq 0$.
In particular, if $\theta \sim [\det]^{n}$ (resp.~$[\det]^{n} \otimes \mathrm{St}$), where $[\cdot]$ denotes the Teichm\"uller lift (and $\mathrm{St}$ the Steinberg representation), then 
$\psi \otimes \rho_{\pi}|_{G_L}$ is crystalline (resp.~semistable) with 
$\sigma_i$-labelled Hodge--Tate weights $(1,0)$ (for $i=1,2$), where $\psi$ is any character whose restriction to inertia is 
$\chi_{\mathrm{cyc}}[\omega]^n$.  Otherwise, we have that 
$\chi_{\mathrm{cyc}} \otimes \rho_{\pi}|_{G_L}$ is potentially Barsotti--Tate with inertial type corresponding under local Langlands to the dual of $\theta$. 

To complete the proof, suppose first that $k_2 = 2$.  In that case, we may apply the conclusion above with $\theta \sim [\det]^{m+2}$, where $m = m_1 + m_2$.  Since $\overline{\psi} \otimes \rho$ has a crystalline lift with $\sigma_i$-labelled Hodge--Tate weights $(1,0)$, and $\overline{\psi}|_{I_{F_\gp}} = \omega^{m+4}$, it follows that $\rho$ has a crystalline lift with $\sigma_i$-labelled Hodge--Tate weights $(-1-m_i,-2-m_i)$, as required.  For consistency with the discussion below, note that this means $D_{\vec{2},-\vec{3}-\vec{m}} = \det^{-m-4} \in W^{\mathrm{cris}}(\rho|_{G_{L}})$, where $W^{\mathrm{cris}}$ is as in \cite[Def.~4.1.3]{GLS}.  In general, the task at hand is to prove that
$$ D_{\vec{k},-\vec{1}-\vec{k}-\vec{m}} = 
  \det\nolimits^{-m-k_2-2} \otimes \Sym^{k_2-2}\Fpbar^2 \in W^{\mathrm{cris}}(\rho|_{G_{L}}).$$

Suppose next that $k_2 = p+1$. Taking $\theta \sim [\det]^{m+2} \otimes \mathrm{St}$ and letting $n=-m-4$, it follows that either $\det^n \in W^{\mathrm{cris}}(\rho|_{G_L})$ or that 
$\rho|_{G_L}$ has the form $\psi \otimes \mat{\overline{\chi}_{\mathrm{cyc}}}{*}{0}{1}$ with 
${\psi}|_{I_L} = \omega^n$.
In either case, we deduce that that $\det^n\Sym^{p-1}\Fpbar^2 \in W^{\mathrm{cris}}(\rho|_{G_{L}})$ from the equality 
$W^{\mathrm{cris}}(\rho|_{G_L}) = W^{\mathrm{explicit}}(\rho|_{G_{L}})$ (\cite[Thm.~6.1.8]{GLS} and \cite[Thm.~5.4]{wang})
and the description of $W^{\mathrm{explicit}}(\rho|_{G_{L}})$ in \cite{BS}.  Indeed this is already immediate from the definition of $W^{\mathrm{explicit}}(\rho|_{G_{L}})$ in the second case, so suppose that $\det^n \in W^{\mathrm{cris}}(\rho|_{G_L})$.  The only irreducible representations with $\det^n$ as a Serre weight have the form $\Ind_{G_M}^{G_L}\xi$, where $\xi|_{I_M} = \omega_\tau^2$, and one then finds that $\det^n\Sym^{p-1}\Fpbar^2$ is also a Serre weight.  Reducible representations with $\det^n$ as a Serre weight have the form $\psi\otimes\mat{\chi}{c}{0}{1}$, where either $\psi|_{I_L} = \omega^n$ and $\chi|_{I_L} = \omega^2$ (with $c$ representating a class in $V_\chi^{\mathrm{ty}}$ if $\chi = \overline{\chi}_{\mathrm{cyc}}$), or $\psi|_{I_L} = \omega^{n+1}$, $\chi$ unramified and $c$ representing a class in $V_\chi$.  In the first case, or if $p=2$, we again have $\det^n\Sym^{p-1}\Fpbar^2  \in W^{\mathrm{explicit}}(\rho|_{G_{L}})$ by definition, so suppose  
$p > 2$ and we are in the second case.  By \cite[Thm.~5.6]{BS}, the space $V_\chi$ is spanned by the subspaces
$\Psi_{\det^n,J,x}$ for $(J,x) \in \{(\Sigma_0,0),(\emptyset,1)\}$ 
defined in \cite[Def.~5.2]{BS}.  Replacing $\det^n$ by $\det^n\Sym^{p-1}$ in \cite[Def.~5.2]{BS} gives the same two possible values for $(J,x)$, but interchanges the corresponding subspaces (the exponents of $v$ in their definition still being $0$ and $1-p$), so another application of \cite[Thm.~5.6]{BS} implies that $\det^n\Sym^{p-1}\Fpbar^2  \in W^{\mathrm{explicit}}(\rho|_{G_{L}})$.

Now suppose that $3 \le k_2 \le p$.  We have two choices for $\theta$: one of dimension $p+1$ (induced from a Borel), and  one of dimension $p-1$ (cuspidal).  We denote these $\theta_1$ and $\theta_2$, and let $\theta_i^*$ denote the dual of $\theta_i$ for $i=1,2$.  It follows from the existence of the lift $\chi_{\mathrm{cyc}}\otimes \rho_\pi|_{G_L}$ and the definition of the set $W^{\mathrm{BT}}(\rho|_{G_L})$ (see \cite[Thm.~A]{GK}) and its equality with $W^{\mathrm{cris}}(\rho|_{G_L})$
(\cite[Thm.~6.1.8]{GLS}) that
$$W^{\mathrm{cris}}(\overline{\chi}_{\mathrm{cyc}}\otimes\rho|_{G_L}) \cap \mathrm{JH}(\overline{\theta}^*_i) \neq \emptyset$$ 
(where $\mathrm{JH}(\overline{\vartheta})$ denotes the set of Jordan--H\"older factors of the reduction of any stable lattice for $\vartheta$).  Since 
$\theta_i$ has central character $[\det]^{2m+k_2+2}$, we have $\theta_i^* \sim [\det]^{n-m}\otimes \theta_i$ where $n = -m-k_2-2$, and the above intersection is equivalent to
$$W^{\mathrm{cris}}(\rho|_{G_L}) \cap \mathrm{JH}([\det]^{n-m-2}\overline{\theta}_i) \neq \emptyset.$$
Since $\overline{\theta}_i$ has $D_{\vec{k},\vec{m}} = \det^{m+2}\Sym^{k_2-2}\Fpbar^2$ as a Jordan--H\"older factor, we have
$$\begin{array}{rcl}\mathrm{JH}([\det]^{n-m-2}\overline{\theta}_1)
 &= &\{\,\det\nolimits^n\Sym^{k_2-2}\Fpbar^2,\det\nolimits^{n+k_2-2}\Sym^{p + 1 - k_2}\Fpbar^2\,\}\\
\mbox{and}\quad
\mathrm{JH}([\det]^{n-m-2}\overline{\theta}_2)
 &= &\{\,\det\nolimits^n\Sym^{k_2-2}\Fpbar^2,\det\nolimits^{n+k_2-1}\Sym^{p -1  - k_2}\Fpbar^2\,\},\end{array}
$$
with the last factor absent if $k_2 = p$.  In particular, if $k_2 = p$, it follows immediately that $\det^n\Sym^{p-2}\Fpbar^2\in W^{\mathrm{cris}}(\rho|_{G_L})$.

Suppose now that $3\le k_2 \le p-1$ and
 $\det^n\Sym^{k_2-2}\Fpbar^2 \not\in W^{\mathrm{cris}}(\rho|_{G_L})$.  We must then have
$$\,\{\det\nolimits^{a}\Sym^b\Fpbar^2, \det\nolimits^{a-1}\Sym^{b+2}\Fpbar^2\,\} \subset W^{\mathrm{cris}}(\rho|_{G_L}),$$
where $a = n+k_2-1$ and $0 \le b = p-1-k_2 \le p-4$.
We will again appeal to the explicit description of the set of Serre weights for the local Galois representation.  Firstly, if $\rho|_{G_L}$ is irreducible, then the assumption that $\det\nolimits^{a}\Sym^b\Fpbar^2 \in W^{\mathrm{cris}}(\rho|_{G_L})$ implies $\rho|_{G_L}$ is of the form $\Ind_{G_M}^{G_L}\xi$ where 
$\xi|_{I_M} = \omega^a\omega_\tau^{b+2}$ or $\omega^{a+1}\omega_\tau^{b}$.  Similarly since $\det\nolimits^{a-1}\Sym^{b+2}\Fpbar^2 \in W^{\mathrm{cris}}(\rho|_{G_L})$, we have that $\rho|_{G_L}$ is of the form $\Ind_{G_M}^{G_L}\xi'$ where 
$\xi'|_{I_M} = \omega^{a-1}\omega_\tau^{b+4}$ or $\omega^{a}\omega_\tau^{b+2}$.  This is only possible if 
$$\xi|_{I_M} = \omega^a\omega_\tau^{b+2}
 =  \omega^{a+b+1}\omega_{\tau'}^{p-b-1} = \omega^n\omega_{\tau'}^{k_2},$$
which contradicts the assumption that $\det^n\Sym^{k_2-2}\Fpbar^2 \not\in W^{\mathrm{cris}}(\rho|_{G_L})$.
Similarly, if $\rho|_{G_L}$ is reducible, then an analysis of the possibilities shows that it must have the form $\psi\otimes\mat{\chi}{c}{0}{1}$ with $\psi|_{I_L} = \omega^a$, $\chi|_{I_L} = \omega^{b+2}$ and $c$ representing a class in $V_\chi$. Appealing again to \cite[Thm.~5.6]{BS}, we find that $V_\chi$ is the subspace $\Psi_{\sigma,J,x}$ in \cite[Def.~5.2]{BS} with $\sigma = \det^{a-1}\Sym^{b+2}\Fpbar^2$, $J = \Sigma_0$ and $x=0$, giving $-(b+2)$ as the exponent of $v$.  Furthermore taking $\det^n\Sym^{k_2-2}\Fpbar^2$, $J = \emptyset$ and $x = 1$ gives the same subspace, contradicting assumption that $\det^n\Sym^{k_2-2}\Fpbar^2 \not\in W^{\mathrm{cris}}(\rho|_{G_L})$.
\end{proof}

\begin{theorem} \label{thm:3toPcrys} Let $F$ be a real quadratic field in which $p$ is ramified, and let $\gp$ be the prime of $F$ over $p$.
Suppose that $\vec{k},\vec{m} \in \ZZ^\Sigma$, with $\vec{k} = (1,w)$ for some $w \in [3,p]$. 
Suppose further that Conjecture~\ref{conj:defalggeom}(2) holds.
If $\rho$ is irreducible and geometrically modular of weight $(\vec{k},\vec{m})$, then $\rho|_{G_{F_\gp}}$ has a crystalline lift of weight 
$(\vec{k},\vec{m})$.
\end{theorem}
\begin{proof} %As in the proof of Theorem~\ref{thm:3toPmod}, we may twist so as to assume $m=-w-2$.  
Suppose that $\rho$ is geometrically modular of weight $(\vec{k},\vec{m})$.  Then by Proposition~\ref{prop:theta_effect}, it is
also geometrically modular of weight $(\vec{k}',\vec{m}')$, where $\vec{k}' = (2,w+1)$ and $\vec{m}' = \vec{m} - (0,1)$ as in the proof of Theorem~\ref{thm:3toPmod}.  Furthermore $\rho$ is also geometrically modular of weight $(\vec{k}'',\vec{m})$, where $\vec{k}'' = \vec{k} + \vec{h}_2 = (2,w-1)$ .

Under our assumption of Conjecture~\ref{conj:defalggeom}(2), it follows that $\rho$ is also algebraically modular of weights $(\vec{k}',\vec{m}')$ and $(\vec{k}'',\vec{m})$.  Therefore by Theorem~\ref{thm:elimination}, $\rho|_{G_{F_{\gp}}}$ has crystalline lifts of those weights, and hence
$$\{\,\det\nolimits^{-m-w-1}\otimes\Sym^{w-3}\Fpbar^2, \det\nolimits^{-m-w-2}\otimes \Sym^{w-1}\Fpbar^2\,\} \subset W^{\mathrm{cris}}(\rho|_{G_{F_{\gp}}}).$$
The same analysis as in the last paragraph of the proof of Theorem~\ref{thm:elimination}, but with $a = -m-w-1$ and 
$b = w-3$ then shows that $\rho|_{G_{F_\gp}}$ is of the form described in Proposition~\ref{prop:repshape}, unless $w=p$ and
$$\rho|_{G_{F_{\gp}}} \sim \psi \otimes \left(\begin{array}{cc}\chi&c\\ 0 & 1\end{array}\right) $$
for some characters $\psi$ and $\chi$ such that 
$\psi|_{I_{F_{\gp}}} = \omega^{-3-m}$ and $\chi|_{I_{F_{\gp}}} = \omega^{2}$ (with $c$ defining a class in $V_{\chi}$, but we will not make use of this).  

Suppose then that this is the case, i.e., $\rho$ has the form just described and $\rho \sim \rho_f$ for some eigenform $f \in S_{(1,p),\vec{m}}(U,E)$, and we will derive a contradiction. 
By Lemma~\ref{lem:stabilized}, we may furthermore assume that $U = U_1(\gn)$ for some $\gn$ prime to $\gp$ and that $f$ is stabilized, strongly so by Theorem~\ref{thm:ord}.  Note that $k_2 = p$ is divisible by $p$, we have that $\Theta(f)$ is divisible by $H_2$ by Theorem~\ref{thm:thetadiv}.
Therefore we have $\Theta(f) = H_2 f_1$ for some $f_1\in M_{(1,p+2),\vec{m}'}(U,E)$ where $\vec{m}' = \vec{m} - (0,1)$.  Furthermore it follows from Theorem~\ref{thm:cone} that $H_1|f_1$, so $\Theta(f) = H_1H_2 f_2$ for some $f_2 \in M_{\vec{2},\vec{m}'}(U,E)$, and note that (a multiple of) $f_2$ is strongly stabilized.

In particular, $\rho$ is geometrically modular of weight $(\vec{2},\vec{m}')$, and our hypothesis therefore implies that $\rho$ is also algebraically modular of weight $(\vec{2},\vec{m}')$.  We claim that it is in fact $\gp$-ordinarily so. To that end consider a quaternion algebra $B$ and non-zero element $\varphi \in M_{\vec{2},\vec{m}'}^B(U',E)[\gm_\rho^Q]$, where $B$ is ramified at exactly one infinite place (and unramified at $\gp$) and $U'$ for some $U'$ containing $\CO_{B,\gp}^\times$. We then have that $e_{\xi}\varphi  \in M_{\vec{2},-\vec{1}}^B(U',E)[\gm_{\mu\otimes\rho}^{Q'}]$ for sufficiently large $Q'$, where $\xi(u_p) = \overline{u}_p^{m+1}$ for all $u_p\in \CO_{F,\gp}^\times$ and $\mu$ corresponds to $\xi$ via class field theory. We may then lift $e_{\xi}\varphi$ to an eigenvector $\widetilde{\varphi}_{\xi} \in M_{\vec{2},-\vec{1}}^B(U',\CO)$ for $T_v$ and $S_v$ for all $v\not\in Q'$ (enlarging $\CO$ if necessary).  We thus obtain an automorphic representation $\Pi \in \mathcal{C}^B_{\vec{2}}$ such that $\Pi_\gp$ is an unramified principal series, as are the $\Pi_v$ for all $v\not\in Q'$, with the same eigenvalues as $\widetilde{\varphi}_\xi$ for $T_v$ and $S_v$ on $\Pi_v^{\CO_{B,v}^\times}$.  We thus have $\mu\otimes\rho \sim \overline{\rho}_\pi$, where $\pi \in \mathcal{C}_{\vec{2}}$ corresponds to $\Pi$ under Theorem~\ref{thm:JL}, and $\Pi_{\gp} \cong \pi_{\gp}$ is a principal series representation of the form $I(\psi_1|\cdot|^{1/2},\psi_2|\cdot|^{1/2})$ for some unramified characters $\psi_1$, $\psi_2$. Furthermore $\rho_{\pi}|_{G_{F_{\gp}}}$ is crystalline with $\sigma$-labelled Hodge--Tate weights $(1,0)$ for both $\sigma \in \Sigma$, so $\chi_{\mathrm{cyc}}\otimes\rho_{\pi}|_{G_{F_{\gp}}}$ arises from a $p$-divisible group over $\CO_{F_{\gp}}$ with an action of $\CO$. Noting that
$e_{\gp} = 2 \le p-1$ and
$$\chi_{\cyc}\otimes\overline{\rho}_\pi |_{G_{F_{\gp}}}
 \sim \psi' \otimes \left(\begin{array}{cc}\chi&c\\ 0 & 1\end{array}\right)$$
where $\psi' = \psi\chi_{\cyc}\mu|_{G_{F_{\gp}}}$ is unramified and $\psi|_{I_{F_{\gp}}} = \omega^2$ is cyclotomic, it follows that the $p$-divisible cannot be local-local, and therefore that
$$\chi_{\cyc}\otimes {\rho}_\pi |_{G_{F_{\gp}}}
 \sim \widetilde{\psi}
' \otimes \left(\begin{array}{cc}\widetilde{\chi}&*\\ 0 & 1\end{array}\right)
$$
for some lifts $\widetilde{\psi}'$ of $\psi$ and $\widetilde{\chi}$ of $\chi$ such that $\widetilde{\psi}'$ is unramified and $\widetilde{\psi}'|_{I_{F_{\gp}}}$ is cyclotomic.  Therefore $\mathrm{WD}({\rho}_\pi |_{G_{F_{\gp}}}) \cong \widetilde{\chi}_1 \oplus \widetilde{\chi}_2|\cdot|^{-1}$ for a pair of unramified characters 
$\widetilde{\chi}_1,\widetilde{\chi}_2:W^{\mathrm{ab}}_{F_{\gp}}\cong F_{\gp}^\times \to \CO^\times$.  By local-global compatibility, these coincide with $\psi_1, \psi_2$, from which it follows that $T_{\gp}$ acts on $\Pi_{\gp}^{\CO_{B,\gp}^\times}$, and hence on $\widetilde{\varphi}_{\xi}$, by $\psi_1(\varpi_{\gp,\gp}) + p\psi_2(\varpi_{\gp,\gp}) \in \CO^\times$. This in turn implies that $e_{\xi}\varphi$ is an eigenform for $T_{\gp}$ with non-zero eigenvalue, and hence that $\rho$ is $\gp$-ordinarily algebraically modular of weight $(\vec{2},\vec{m}')$.

By Theorem~\ref{thm:ordmods}, we then have that $\rho$ is also $\gp$-ordinarily geometrically modular of weight $(\vec{2},\vec{m}')$. By Proposition~\ref{prop:normalized} and Lemma~\ref{lem:stabilized}, there exist an ideal $\gn'$ prime to $\gp$ and a stabilized eigenform $f_3 \in S_{\vec{2},\vec{m}'}(U_1(\gn'),E)$ (enlarging $E$ if necessary) such that $\rho \sim \rho_{f_3}$ and $T_{\gp} f_{3,\xi} \neq 0$ for all $\xi$ as in Definition~\ref{def:ordgeommod}. By another application of Lemma~\ref{lem:stabilized}, we may assume $\gn = \gn'$.

Now consider $f_4 = f_3 - f_2 \in S_{\vec{2},\vec{m}'}(U_1(\gn'),E)$. Since $f_2$ and $f_3$ are both stabilized, we have $r_m^t(f_4) = 0$ for all $t \in (\AA_{F,\bf{f}}^{(p)})^\times $ and $m \in t^{-1}\widehat{\gd}^{-1} \cap F_+^\times$ such that $ctm \not\in \gp\widehat{O}_F$.  On the other hand, since 
$f_2$ is strongly stabilized, we have $f_2 \neq f_3$, so $f_4 \neq 0$. It follows that $f_4 \in \ker(\Theta)$, so by Theorem~\ref{thm:kertheta}, we have
$$f_4 = H_1^{s_1} H_2^{s_2} G_1^{t_1} G_2^{t_2} V_{\gp}(f_5)$$
for some $s_1,s_2 \in \ZZ_{\ge 0}$, $t_1,t_2\in \ZZ$ and $f_5 \in M_{\mathrm{tot}}(U,E)$ (where $U = U_1(\gn)$).  Furthermore, by Theorem~\ref{thm:positivity}, we may assume $f_5 \in M_{\vec{k}''',\vec{m}''}(U,E)$ for some $\vec{k}'''\in \Xi_{\min}^+$, $\vec{m}''\in \ZZ^2$.  We thus have
$$(2,2) = s_1(-1,p) + s_2(1,-1) + (k'''_2,pk_1''')$$
with $1 \le k_2''' \le pk_1''' \le pk_2'''$.  One sees easily that the first equation implies that $p = 3$ and $\vec{k}''' = \vec{1}$, so assume this is the case.  Note also that 
$$m'' := m_1'' + m_2'' \equiv m'_1 + m'_2  = m - 1 \bmod (p-1) = 2.$$

Since $V_{\gp}$ is injective and compatible with $T_v$ for all $v \neq \gp$ and $S_v$ for all $v\nmid \gn \gp$, it follows that $f_5$ is an eigenform for all these operators, with the same eigenvalues as $f_4$.  In particular, we have $\rho_{f_5} \sim \rho$ and $T_v f_5 = 0 $ for all $v|\gn$.  The same argument as in the proof of Proposition~\ref{prop:normalized} then shows that if $r^{t_1}_{m_1}(f_5) = 0$, then in fact $r_t^m = 0$ for all 
$t \in (\AA_{F,\bf{f}}^{(p)})^\times $ and $m \in t^{-1}\widehat{\gd}^{-1} \cap F_+^\times$ such that $ctm \not\in \gp\widehat{O}_F$, which implies that $f_5 \in \ker(\Theta)$. In this case, applying Theorem~\ref{thm:kertheta} gives a contradiction, so we conclude that $r^{t_1}_{m_1}(f_5) \neq 0$.

Now consider $f_6 = \Theta(f_5) \in S_{\vec{2},\vec{m}'''}(U,E)$, where $\vec{m}''' = \vec{m}'' - (0,1)$.  Since $T_v f_6 = 0$ for all $v|\gn \gp$, some multiple of $f_6$ is a strongly stabilized eigenform.  Note also that since $m_1''' + m_2''' \equiv m \bmod 2$, there exist $u_1,u_2 \in \ZZ$ such that 
$$(m_1,m_2) = (m_1''',m_2''') + u_1(-1,3) + u_2(1,-1).$$
Therefore some multiple of $G_1^{u_1}G_2^{u_2}f_6$ is a strongly stabilized eigenform in $S_{\vec{2},\vec{m}}(U,E)$ giving rise to $\rho$, as is $H_2 f$.  Therefore $f_6 = \Theta(f_5)$ is divisible by $H_2$.  Since $k_2''' = 1$ is not divisible by $p$, Theorem~\ref{thm:thetadiv} implies that $f_5$ is divisible by $\mathrm{Ha}_2$, i.e., $f_5 = H_2 f_7$ for some $f_7 \in S_{(0,2),\vec{m}''}(U,E)$.  Finally, this contradicts Theorem~\ref{thm:positivity} and completes the proof.
\end{proof}

We now turn our attention to the case of $w=2$, where our arguments require ordinariness hypotheses.  We first treat the easier implication relating geometric modularity to existence of crystalline lifts:
\begin{theorem} \label{thm:2crys} Let $F$ be a real quadratic field in which $p$ is ramified, and let $\gp$ be the prime of $F$ over $p$. Suppose that $\vec{k},\vec{m} \in \ZZ^\Sigma$, with $\vec{k} = (1,2)$. Suppose further that Conjecture~\ref{conj:defalggeom}(2) holds.
If $\rho$ is irreducible and $\gp$-ordinarily geometrically modular of weight $(\vec{k},\vec{m})$, then $\rho|_{G_{F_\gp}}$ has a crystalline lift of weight $(\vec{k},\vec{m})$.
\end{theorem}
\begin{proof} By Theorem~\ref{thm:ord2}, we have
$$\rho \sim \psi \otimes \left(\begin{array}{cc}\chi&c \\ 0 & 1\end{array}\right)$$
for some characters $\psi$ and $\chi$ such that $\psi|_{I_{F_{\gp}}} = \omega^{-m-2}$ and $\chi|_{I_{F_\gp}} = \omega$.
Furthermore, by Proposition~\ref{prop:theta_effect}, $\rho$ is also geometrically modular of weight $(\vec{k}',\vec{m}')$, where $\vec{k}' = (2,3)$ and $\vec{m}' = \vec{m} - (0,1)$.
Our assumption of Conjecture~\ref{conj:defalggeom}(2) therefore implies that $\rho$ is algebraically modular of weight $(\vec{k}',\vec{m}')$. By Theorem~\ref{thm:elimination}, $\rho$ has a crystalline lift of weight $(\vec{k}',\vec{m}')$.
It therefore follows from \cite[Thm.~5.4.1]{GLS} that $c$ defines a class in $V_\chi$.  Finally, Proposition~\ref{prop:repshape} implies that $\rho$ has a crystalline lift of weight $(\vec{k},\vec{m})$.
\end{proof}
\begin{remark} The ordinariness hypothesis is needed for our argument since there are irreducible representations of $G_{F_{\gp}}$ with crystalline lifts of weight $((2,p),\vec{m})$ and $((2,3),\vec{m}')$ (where $\vec{m}' = \vec{m} - (0,1)$ as usual), but no crystalline lift of weight $((1,2),\vec{m})$.  This fact also underpins our reliance on such a hypothesis in the proof of Theorem~\ref{thm:2mod}.
\end{remark}

We have the following result in the direction of partial weight one modularity:
\begin{theorem} \label{thm:2mod} Let $F$ be a real quadratic field in which $p$ ramified, and let $\gp$ be the prime of $F$ over $p$.
Suppose that $\vec{k} = (1,2)$ and that $\vec{m} \in \ZZ^\Sigma$. Let $\rho:G_F \to \GL_2(\Fpbar)$ be irreducible and geometrically modular of some weight.   Suppose further that $p > 2$, $\rho|_{G_{F(\zeta_p)}}$ is irreducible, and if $p=5$, then the projective image of $\rho|_{G_{F(\zeta_5)}}$ is not isomorphic to $A_5$. If $\rho|_{G_{F_\gp}}$ is reducible and has a crystalline lift of weight $(\vec{k},\vec{m})$, then $\rho$ is geometrically modular of weight $(\vec{k},\vec{m})$.
\end{theorem}
\begin{proof} If $\rho|_{G_{F_\gp}}$ is reducible and has a crystalline lift of weight $(\vec{k},\vec{m})$, then by Proposition~\ref{prop:repshape} we have
$$\rho \sim \psi \otimes \left(\begin{array}{cc}\chi&c \\ 0 & 1\end{array}\right)$$
for some characters $\psi$ and $\chi$ such that $\psi|_{I_{F_{\gp}}} = \omega^{-m-2}$, $\chi|_{I_{F_\gp}} = \omega$ and $c \in V_\chi$. In particular, $\rho$ has a crystalline lift of weight $(\vec{k}'',\vec{m})$, where $\vec{k}'' = (2,p)$, and Theorem~\ref{thm:GLSord} implies that $\rho$ is 
$\gp$-ordinarily algebraically modular of weight $(\vec{k}'',\vec{m})$.  Therefore $\rho$ is also $\gp$-ordinarily geometrically modular of weight $(\vec{k}'',\vec{m})$ by Theorem~\ref{thm:ordmods}. By Proposition~\ref{prop:normalized} and Lemma~\ref{lem:stabilized}, it follows that $\rho$ arises from a stabilized eigenform $f \in S_{\vec{k}'',\vec{m}}(U_1(\gn),E)$ for some $\gn$ not divisible by $\gp$, and moreover with $T_{\gp}f_{\xi} \neq 0$ for all $\xi$ as in Definition~\ref{def:normalized}.

Since $c \in V_\chi$, we see as in the proof of Theorem~\ref{thm:3toPmod} that $\rho$ also has a crystalline lift of weight $(\vec{k}',\vec{m}')$, where $\vec{k}' = (2,3)$ and $\vec{m}' = \vec{m} - (0,1)$.  Therefore Theorems~\ref{thm:GLSmain} and~\ref{thm:defalggeom} imply that $\rho$ is geometrically modular of weight $(\vec{k}',\vec{m}')$. By Proposition~\ref{prop:normalized} and Lemma~\ref{lem:stabilized}, $\rho$ arises from a stabilized eigenform $f' \in S_{\vec{k}',\vec{m}}(U,E)$, where $U = U(\gn')$ for some $\gn'$ not divisible by $\gp$.  Furthermore another application of Lemma~\ref{lem:stabilized} allows us to replace $\gn$ and $\gn'$ by a common multiple, and hence assume $\gn' = \gn$.

We claim that we can in fact take $f'$ to be strongly stabilized. Indeed this follows from Theorem~\ref{thm:ord}, unless $p=3$ and $\rho|_{I_{F_\gp}} \sim 1 \oplus \omega$.  In that case, note that Theorem~\ref{thm:thetadiv} implies that $H_2|\Theta(f)$, so $\rho$ arises from a strongly stabilized $g \in S_{(2,5),\vec{m}'}(U,E)$.
Similarly $H_2|\Theta(f')$, so $\rho$ arises from a strongly stabilized $g' \in S_{(2,5),\vec{m}''}(U,E)$, where $\vec{m}'' = \vec{m} - (0,2)$.  It follows that $\Theta(g) = H_1H_2^2g'$ is divisible by $H_2$.  Theorem~\ref{thm:thetadiv} then implies that $g = H_2g''$ for some $g'' \in S_{(1,6),\vec{m}'}(U,E)$, and Theorem~\ref{thm:positivity} then implies $g'' = H_1 f''$, where now $f'' \in S_{\vec{k}',\vec{m}'}(U,E)$ is (up to scalar) a strongly stabilized eigenform giving rise to $\rho$.  We can therefore replace $f'$ by $f''$ to complete the proof of the claim.

Now write $\Theta^{p-2}f' = G_1^{-1}G_2^{-1}H_1^{s_1}H_2^{s_2} f'''$, where $f''' \in S_{\vec{k}''',\vec{m}}(U,E)$ is (up to scalar) a strongly stabilized eigenform not divisible by $H_1$ or $H_2$.  By Theorem~\ref{thm:positivity}, we have
$$\vec{k}''' = (p,p+1) - a_1(-1,p) - a_2(1,-1) \in \Xi_{\min}^+$$
for some $a_1,a_2\in \ZZ_{\ge 0}$.  Furthermore $\Theta^{p-1}f' = G_1G_2H_1^2H_2^{p+1}$, so Theorem~\ref{thm:thetadiv} implies that either $a_2 \ge p+1$ or $k_2''' = p+1-pa_1+a_2$ is divisible by $p$.  Using that $\vec{k}''' \in \Xi_{\min}^+$, the first possibility implies that $a_1=2$ and $a_2 = p+1$, so $\vec{k}''' = (1,2) = \vec{k}$ and we are done.  

So suppose that $a_2 < p+1$ and $p|k_2'$, i.e., $a_2 \equiv -1\bmod p$, so $a_2 = p-1$. Again using that $\vec{k}''' \in \Xi_{\min}^+$, we find that $a_1 = 1$, so $\vec{k}''' = (2,p) = \vec{k}''$.  Since $f$ and $f'''$ are both stabilized eigenforms in $S_{\vec{k}'',\vec{m}}(U,E)$, but only $f'''$ is strongly stabilized, we have $f - f'''$ in $\ker(\Theta)$.  Applying Theorem~\ref{thm:kertheta} now implies that $f-f''' = G_1^{t_1}G_2^{t_2}V_{\gp}(g''')$ for some $t_1,t_2 \in \ZZ$ and eigenform $g'''\in S_{\vec{k},\vec{m}}(U,E)$ giving rise to $\rho$, and we are done.
\end{proof}
\begin{remark} Recall that $\rho|_{G_{F(\zeta_p)}}$ is irreducible unless $\rho$ is induced from a character of $G_L$, where $L$ is the quadratic extension of $F$ contained in $F(\zeta_p)$.  From the shape of the local Galois representation, one sees easily that (under the other assumptions of Theorem~\ref{thm:2mod}) this is only possible if $\gp$ splits in $L$.
\end{remark}

Finally, we record the following consequence used in the proof of Theorem~\ref{thm:3toPmod}:
\begin{corollary} \label{cor:final} Under the hypotheses of Theorem~\ref{thm:2mod}, there exist an ideal $\gn$ prime to $\gp$ and strongly stabilized eigenform $f\in S_{(2,p),\vec{m}}(U_1(\gn),E)$ such that $\rho \sim \rho_f$.
\end{corollary}
This was in fact established in the course of the proof of Theorem~\ref{thm:2mod}, but can also be deduced from the statement of the theorem by taking a linear combination of $G^{-m_1}G_2^{-m_2}V_{\gp}(g)$ and $H_1H_2^2g$, where $g\in S_{(1,2),\vec{m}}(U_1(\gn),E)$ is a stabilized eigenform giving rise to $\rho$.

\medskip

\noindent {\bf Acknowledgements:}  We are grateful to Robin Bartlett, George Boxer, Toby Gee, Payman Kassaei, David Loeffler, Gabriel Micolet, Hanneke Wiersema and Siqi Yang for helpful correspondence and conversations related to this work.  We also benefitted from numerous discussions at a workshop on ``Geometric Serre Weight Conjectures,'' held at King's College London in February~2024 and sponsored by a Heilbronn Institute Focused Research Grant.  We also thank the Insitute for Advanced Study, where some of the research was conducted by the first author during a visit in March--April 2024.  The final stage of this work was supported by the Engineering and Physical Sciences Research Council grant EP/Z53609X/1.

\bibliographystyle{amsalpha} 
\bibliography{DS2_ref} 

\bigskip

\end{document}